\documentclass[10pt]{amsart}

\usepackage[left=2.8cm, right=2.8cm, top=2.2cm, bottom=2.2cm]{geometry}

\usepackage{amsmath,amssymb,amsthm,mathrsfs,stmaryrd,latexsym}
\usepackage{amsfonts}
\usepackage{accents}

\usepackage{tcolorbox}
\tcbuselibrary{skins, breakable, theorems}

\usepackage{enumerate} 
\usepackage{tikz}
\usepackage{etoolbox}

\def\Luoma#1{\uppercase\expandafter{\romannumeral#1}}
\def\luoma#1{\romannumeral#1}

\usepackage{url}

\newtheorem{mythm}{Theorem}[section]
\newtheorem{mylem}[mythm]{Lemma}
\newtheorem{myprop}[mythm]{Proposition}
\newtheorem{mycor}[mythm]{Corollary}

\theoremstyle{definition}
\newtheorem{mydefn}[mythm]{Definition}
\newtheorem{myexample}[mythm]{Example}

\theoremstyle{remark}
\newtheorem{myrem}[mythm]{Remark}
\newtheorem{mypara}[mythm]{}

\usepackage[all]{xy}
\usepackage{graphicx}
\usepackage{subfigure}

\newcommand{\bb}{\mathbb}
\newcommand{\ca}{\mathcal}
\newcommand{\ak}{\mathfrak}
\newcommand{\scr}{\mathscr}

\newcommand{\mbf}{\mathbf}
\newcommand{\mrm}{\mathrm}

\newcommand{\trm}{\textrm}

\def\op#1{\mathop{\mathrm{#1}}}

\newcommand{\ho}{\mrm{Hom}}
\newcommand{\tor}{\mrm{Tor}}
\newcommand{\ke}{\mrm{Ker}}
\newcommand{\cok}{\mrm{Coker}}
\newcommand{\im}{\mrm{Im}}
\newcommand{\df}{\mrm{d}}
\newcommand{\id}{\mrm{id}}
\newcommand{\spec}{\op{Spec}}
\newcommand{\colim}{\op{colim}}

\newcommand{\rr}{\mrm{R}}
\newcommand{\dl}{\mrm{L}}

\newcommand{\iso}{\stackrel{\sim}{\longrightarrow}}

\newcommand{\plim}{\varprojlim}

\newcommand{\oppo}{\mrm{op}}

\newcommand{\triv}{\mrm{tr}}
\newcommand{\sen}{\mrm{Sen}}
\newcommand{\ari}{\mrm{ari}}
\newcommand{\geo}{\mrm{geo}}

\newcommand{\lie}{\mrm{Lie}}
\newcommand{\gal}{\mrm{Gal}}

\newcommand{\repn}{\mbf{Rep}_{\mrm{cont}}}
\newcommand{\repnpr}{\mbf{Rep}^{\mrm{proj}}_{\mrm{cont}}}

\def\repnan#1{\mbf{Rep}^{\mrm{proj}}_{\mrm{cont},{#1}\trm{-}\mrm{an}}}
\def\repnsm#1{\mbf{Rep}^{#1\trm{-}\mrm{s}}}
\newcommand{\higgs}{\mbf{HB}_{\mrm{nilp}}}

\def\ff#1{\scr{F}^{\mrm{fini}}_{\overline{#1}/#1}}

\newcommand{\et}{\mrm{\acute{e}t}}

\newcommand{\hyodoring}{\scr{C}_{\mrm{HT}}}

\title[Sen Operators over $p$-adic Varieties]{Sen Operators and Lie Algebras arising from Galois Representations over $p$-adic Varieties}
\author{Tongmu He}
\date{\today}
\address{Tongmu He, Institut des Hautes \'Etudes Scientifiques, 35 route de Chartres, 91440 Bures-sur-Yvette, France}
\email{hetm15@ihes.fr}

\usepackage{hyperref}
\hypersetup{colorlinks,
citecolor=black,
filecolor=black,
linkcolor=black,
urlcolor=black,
pdftitle={Sen Operators over p-adic Varieties},
pdfauthor={Tongmu He},
pdftex}

\numberwithin{equation}{mythm}

\begin{document}
\maketitle

\begin{abstract}
	Any finite-dimensional $p$-adic representation of the absolute Galois group of a $p$-adic local field with imperfect residue field is characterized by its arithmetic and geometric Sen operators defined by Sen and Brinon. We generalize their construction to the fundamental group of a $p$-adic affine variety with a semi-stable chart, and prove that the module of Sen operators is canonically defined, independently of the choice of the chart. Our construction relies on a descent theorem in the $p$-adic Simpson correspondence developed by Tsuji. When the representation comes from a $\bb{Q}_p$-representation of a $p$-adic analytic group quotient of the fundamental group, we describe its Lie algebra action in terms of the Sen operators, which is a generalization of a result of Sen and Ohkubo. These Sen operators can be extended continuously to certain infinite-dimensional representations. As an application, we prove that the geometric Sen operators annihilate locally analytic vectors, generalizing a result of Pan.
\end{abstract}


\tableofcontents

\section{Introduction}

\begin{mypara}
	Let $K$ be a complete discrete valuation field extension of $\bb{Q}_p$, $\overline{K}$ an algebraic closure of $K$, $\widehat{\overline{K}}$ the $p$-adic completion of $\overline{K}$, $G$ the Galois group of $\overline{K}$ over $K$. When the residue field of $K$ is perfect, for any finite-dimensional (continuous semi-linear) $\widehat{\overline{K}}$-representation $W$ of $G$, Sen \cite{sen1980sen} associates a canonical $\widehat{\overline{K}}$-linear endomorphism on $W$, called the Sen operator, which determines the isomorphism class of $G$-representations on $W$. Moreover, if $W$ is the base change of a $\bb{Q}_p$-representation $V$ of $G$, Sen \cite[Theorem 11]{sen1980sen} relates the infinitesimal action of the inertia subgroup of $G$ on $V$ to the Sen operator on $W$. When the residue field of $K$ is imperfect with a $p$-basis of cardinality $d\geq 1$, Brinon \cite{brinon2003sen} defines $1+d$ (non-canonical) operators on $W$, which also determine the isomorphism class of $G$-representations on $W$. Moreover, if $W$ is the base change of a $\bb{Q}_p$-representation $V$ of $G$, Ohkubo \cite{ohkubo2014sen} relates the space generated by these $1+d$ operators to the infinitesimal action of the inertia subgroup of $G$ on $V$ as Sen did for $d=0$. In this article, we construct Sen operators for representations of the fundamental group of a $p$-adic affine variety with semi-stable chart. We show that the module of Sen operators is canonically defined, independent of the choice of the chart. Indeed, we associate to each representation a canonical Lie algebra action which gives all the Sen operators. Moreover, when the representation comes from $\bb{Q}_p$, we relate the Sen operators to the infinitesimal action of the inertia subgroups at height-$1$ primes, generalizing the results of Sen-Ohkubo. As an application, we prove that the geometric Sen operators annihilate locally analytic vectors, generalizing a result of Pan.
\end{mypara}

\begin{mypara}\label{para:review-brinon-intro}
	In fact, our strategy for constructing the Sen operators in the relative situation is to glue the Sen operators defined in the case of valuation fields. Hence, we firstly take a brief review on Brinon's construction of Sen operators. We take $t_1,\dots,t_d\in\ca{O}_K$ whose images in the residue field form a $p$-basis. We fix a compatible system of primitive $p^n$-th roots of unity $\zeta=(\zeta_{p^n})_{n\in\bb{N}}$ and a compatible system of $p^n$-th roots $(t_{i,p^n})_{n\in\bb{N}}$ of $t_i$ for $1\leq i\leq d$. We also put $t_{0,p^n}=\zeta_{p^n}$ for consistency. For any $n,m\in\bb{N}\cup\{\infty\}$, consider the field extension $K_{n,m}=K(\zeta_{p^n},t_{1,p^m},\dots,t_{d,p^m})$ of $K$ contained in $\overline{K}$. We simply set $K_{n,0}=K_n$ and we name some Galois groups as indicated in the following diagram
	
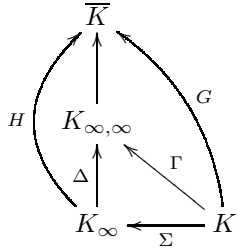
\begin{align}\label{diam:intro-tower}
		\xymatrix{
			\overline{K}&\\
			K_{\infty,\infty}\ar[u]&\\
			K_\infty\ar[u]^-{\Delta}\ar@/^2pc/[uu]^-{H}&K\ar[l]^-{\Sigma}\ar[lu]_-{\Gamma}\ar@/_1pc/[luu]_{G}
		}
	\end{align}
	Any finite-dimensional $\widehat{\overline{K}}$-representation $W$ of $G$ descends to a $K_{\infty,\infty}$-representation $V$ of $\Gamma$ by a theorem of Brinon (cf. \ref{thm:sen-brinon}). We remark that it can be descended further to a $K_\infty$-representation of $\Gamma$ on which $\Delta$ acts analytically by a theorem of Tsuji (cf. \ref{prop:gamma-analytic}). Here, acting analytically means that the action of any element of $\Delta$ is given by the exponential of its infinitesimal action (cf. \ref{defn:analytic}). The topological group $\Gamma$ is indeed a $p$-adic analytic group, to which one can associate a Lie algebra $\lie(\Gamma)$ over $\bb{Q}_p$. Then, the infinitesimal action of $\lie(\Gamma)$ on $V$ extends $\widehat{\overline{K}}$-linearly to a (non-canonical) Lie algebra action of $\lie(\Gamma)$ on $W$, which defines $1+d$ operators of $W$ by Brinon as $\Gamma$ is locally isomorphic to $\bb{Z}_p\ltimes\bb{Z}_p^d$. 
	
	This action of $\lie(\Gamma)$ depends on the choice of $t_1,\dots,t_d$, which prevents the generalization to relative situation. The first question is whether we can define a canonical Lie algebra action on $W$, which gives the Sen operators defined by Brinon by choosing a basis. We answer it positively by considering the Faltings extension of $\ca{O}_K$ defined in \cite{he2021faltingsext} (cf. \ref{thm:fal-ext}), that is, a canonical exact sequence of $\widehat{\overline{K}}$-representations of $G$,
	\begin{align}\label{eq:fal-ext-intro}
		0\longrightarrow \widehat{\overline{K}}(1)\stackrel{\iota}{\longrightarrow} \scr{E}_{\ca{O}_K}\stackrel{\jmath}{\longrightarrow}\widehat{\overline{K}}\otimes_{\ca{O}_K}\widehat{\Omega}^1_{\ca{O}_K}\longrightarrow 0,
	\end{align}
	where $\widehat{\overline{K}}(1)$ denotes the first Tate twist of $\widehat{\overline{K}}$, $\scr{E}_{\ca{O}_K}=\plim_{x\mapsto px}\Omega^1_{\ca{O}_{\overline{K}}/\ca{O}_K}$ is a $(1+d)$-dimensional $\widehat{\overline{K}}$-space with a basis $\{(\df\log(t_{i,p^n}))_{n\in\bb{N}}\}_{0\leq i\leq d}$. Taking duals and Tate twists, we obtain a canonical exact sequence
	\begin{align}\label{eq:fal-ext-dual-intro}
		0\longrightarrow \ho_{\ca{O}_K}(\widehat{\Omega}^1_{\ca{O}_K}(-1),\widehat{\overline{K}})\stackrel{\jmath^*}{\longrightarrow} \scr{E}^*_{\ca{O}_K}(1)\stackrel{\iota^*}{\longrightarrow}\widehat{\overline{K}}\longrightarrow 0
	\end{align} 
	where $\scr{E}^*_{\ca{O}_K}=\ho_{\widehat{\overline{K}}}(\scr{E}_{\ca{O}_K},\widehat{\overline{K}})$. There is a canonical $\widehat{\overline{K}}$-linear Lie algebra structure on $\scr{E}^*_{\ca{O}_K}(1)$ associated to the linear form $\iota^*$ defined by $[f_1,f_2]=\iota^*(f_1)f_2-\iota^*(f_2)f_1$ for any $f_1,f_2\in \scr{E}^*_{\ca{O}_K}(1)$. This will be the canonical Lie algebra replacing $\lie(\Gamma)$, so that we obtain the following canonical definition of Sen operators.
\end{mypara}

\begin{mythm}[{cf. \ref{thm:sen-brinon-operator}, \ref{prop:sen-brinon-operator-func}}]\label{thm:sen-brinon-operator-intro}
	Let $K$ be a complete discrete valuation field extension of $\bb{Q}_p$ whose residue field admits a finite $p$-basis, $G$ its absolute Galois group. For any finite-dimensional $\widehat{\overline{K}}$-representation of $G$, there is a canonical $G$-equivariant homomorphism of $\widehat{\overline{K}}$-linear Lie algebras (where we put adjoint action of $G$ on $\mrm{End}_{\widehat{\overline{K}}}(W)$),
	\begin{align}
		\varphi_{\sen}|_W:\scr{E}^*_{\ca{O}_K}(1)\longrightarrow \mrm{End}_{\widehat{\overline{K}}}(W),
	\end{align}
	which is functorial in $W$ and satisfies the following properties:
	\begin{enumerate}
		\renewcommand{\labelenumi}{{\rm(\theenumi)}}
		\item Let $t_1,\dots,t_e\in K$ with compatible systems of $p$-power roots $(t_{i,p^n})_{n\in\bb{N}}\subseteq \overline{K}$ such that $\df t_1,\dots,\df t_e$ are $K$-linearly independent in $\widehat{\Omega}^1_{\ca{O}_K}[1/p]$. Consider the tower $(K_{n,m})_{n,m\in\bb{N}}$ defined by these elements analogously to \eqref{diam:intro-tower} and take the same notation for Galois groups, and assume that there is a $K_\infty$-representation $V$ of $\Gamma$ on which $\Delta$ acts analytically ({\rm\ref{defn:analytic}}) such that $W=\widehat{\overline{K}}\otimes_{K_\infty} V$. Then, $\Gamma$ is naturally locally isomorphic to $\bb{Z}_p\ltimes\bb{Z}_p^e$, and if we take the standard basis $\partial_0,\dots,\partial_e$ of $\lie(\Gamma)\cong \lie(\bb{Z}_p\ltimes\bb{Z}_p^e)$, then for any $f\in \scr{E}^*_{\ca{O}_K}(1)$,
		\begin{align}\label{eq:thm:sen-brinon-operator-intro}
			\varphi_{\sen}|_W(f)=\sum_{i=0}^e f((\df\log(t_{i,p^n}))_{n\in\bb{N}}\otimes\zeta^{-1})\otimes \varphi_{\partial_i}|_V,
		\end{align}
		where $\varphi_{\partial_i}|_V$ is the infinitesimal action of $\partial_i$ on $V$.
		\item Let $K'$ be a complete discrete valuation field extension of $K$ whose residue field admits a finite $p$-basis, $W'=\widehat{\overline{K'}}\otimes_{\widehat{\overline{K}}}W$. Assume that $K'\otimes_K\widehat{\Omega}^1_{\ca{O}_K}[1/p]\to \widehat{\Omega}^1_{\ca{O}_{K'}}[1/p]$ is injective. Then, there is a natural commutative diagram
		\begin{align}\label{diam:sen-brinon-operator-func-intro}
			\xymatrix{
				\scr{E}^*_{\ca{O}_{K'}}(1)\ar[rr]^-{\varphi_{\sen}|_{W'}}\ar[d]&&\mrm{End}_{\widehat{\overline{K'}}}(W')\\
				\widehat{\overline{K'}}\otimes_{\widehat{\overline{K}}}\scr{E}^*_{\ca{O}_K}(1)
				\ar[rr]_-{\id_{\widehat{\overline{K'}}}\otimes\varphi_{\sen}|_W}&&\widehat{\overline{K'}}\otimes_{\widehat{\overline{K}}}\mrm{End}_{\widehat{\overline{K}}}(W)
				\ar[u]_-{\wr}
			}
		\end{align}
	Moreover, if $K'$ is a finite extension of $K$, then the left vertical arrow is an isomorphism.
	\end{enumerate}
\end{mythm}

The key of its proof is to show that the map $\varphi_{\sen}|_W$ defined by the formula \eqref{eq:thm:sen-brinon-operator-intro} does not depend on the choice of $V$ and $t_i$. For this, we use the variant of $p$-adic Simpson correspondence developed by Tsuji \cite{tsuji2018localsimpson} over $\ca{O}_K$ (cf. \ref{thm:simpson-K}). One clue is that the period ring used in this correspondence is constructed as the filtered colimit of symmetric tensor products of the Faltings extension \eqref{eq:fal-ext-intro} (called the Hyodo ring, cf. \ref{defn:hyodo-ring}). We remark that the assumption on $K'\otimes_K\widehat{\Omega}^1_{\ca{O}_K}[1/p]\to \widehat{\Omega}^1_{\ca{O}_{K'}}[1/p]$ for the functoriality is a technical condition for its proof, and we don't know how to remove this (cf. \ref{prop:sen-brinon-operator-func}).

\begin{mypara}\label{para:notation-A-intro}
	Now we can generalize the construction of Sen operators in the relative situation. Let $K$ be a complete discrete valuation field extension of $\bb{Q}_p$ with perfect residue field, $\pi$ a uniformizer of $K$. For simplicity, we consider a Noetherian normal domain $A$ flat over $\ca{O}_K$ with $A/pA\neq 0$ such that there exists an \'etale ring homomorphism for some integers $0\leq r\leq c\leq d$,
	\begin{align}\label{eq:semi-stable-chart-intro}
		\ca{O}_K[T_0,\dots,T_r,T_{r+1}^{\pm 1},\dots,T_c^{\pm 1},T_{c+1},\dots,T_d]/(T_0\cdots T_r-\pi)\longrightarrow A.
	\end{align}
	Thus, $\spec(A)$ is endowed with a strictly normal crossings divisor defined by $T_0\cdots T_d$. We set $A_{\triv}=A[1/T_0\cdots T_d]$. We call $(A_{\triv},A)$ a semi-stable pair, and call \eqref{eq:semi-stable-chart-intro} a semi-stable chart of it. Let $\ca{K}$ be the fraction field of $A$, $\ca{K}_{\mrm{ur}}$ the maximal unramified extension of $\ca{K}$ with respect to $(A_{\triv},A)$, i.e. the union of finite field extensions $\ca{K}'$ of $\ca{K}$ in an algebraic closure of $\ca{K}$ such that the integral closure $A'$ of $A$ in $\ca{K}'$ is finite \'etale over $A_{\triv}$, and let $\overline{A}$ be the integral closure of $A$ in $\ca{K}_{\mrm{ur}}$. We remark that this article considers more general pairs (called \emph{quasi-adequate}) than semi-stable pairs, so that the directed system of finite subextensions of $\ca{K}_{\mrm{ur}}/\ca{K}$ admits a cofinal subsystem consisting of elements  $\ca{K}'$ such that the pair $(A'_{\triv},A')$ is quasi-adequate, where $A'_{\triv}=A_{\triv}\otimes_A A'$ (cf. \ref{defn:quasi-adequate-alg}, \ref{cor:quasi-adequate-cofinal}). As before, we also consider the Faltings extension of $A$ (cf. \ref{thm:B-fal-ext}), that is, a canonical exact sequence of $\widehat{\overline{A}}[1/p]$-representations of $\gal(\ca{K}_{\mrm{ur}}/\ca{K})$,
	\begin{align}\label{eq:thm:B-fal-ext-intro}
		0\longrightarrow \widehat{\overline{A}}[\frac{1}{p}](1)\stackrel{\iota}{\longrightarrow}\scr{E}_A\stackrel{\jmath}{\longrightarrow} \widehat{\overline{A}}[\frac{1}{p}]\otimes_A\Omega^1_{(A_{\triv},A)}\longrightarrow 0,
	\end{align}
	where $\Omega^1_{(A_{\triv},A)}$ denotes the $A$-module of logarithmic $1$-differentials of the pair $(A_{\triv},A)$ over $(K,\ca{O}_K)$, which is finite free of rank $d$. The canonical $\widehat{\overline{A}}[1/p]$-module $\scr{E}_A$ is finite free of rank $1+d$, which satisfies the following property (cf. \ref{thm:B-fal-ext}, \ref{rem:B-fal-ext}): there is a canonical $\overline{A}[1/p]$-linear map $\plim_{x\mapsto px}\Omega^1_{\overline{A}/A}\to \scr{E}_A$ such that for any element $s\in \overline{A}[1/p]\cap \overline{A}_{\triv}^\times$ with a compatible system of $p$-power roots $(s_{p^n})_{n\in\bb{N}}\subseteq \overline{A}[1/p]$, there is a unique element $\omega\in \scr{E}_A$ such that the image of $(s_{p^n}^{p^n-1}\df s_{p^n})_{n\in\bb{N}}$ is equal to $s\omega$ (we thus denoted $\omega$ by $(\df\log(s_{p^n}))_{n\in\bb{N}}$). As before, we obtain a canonical exact sequence by taking duals and Tate twists,
	\begin{align}\label{eq:B-fal-ext-dual-intro}
		0\longrightarrow \ho_A(\Omega^1_{(A_{\triv},A)}(-1),\widehat{\overline{A}}[\frac{1}{p}])\stackrel{\jmath^*}{\longrightarrow} \scr{E}^*_A(1)\stackrel{\iota^*}{\longrightarrow}\widehat{\overline{A}}[\frac{1}{p}]\longrightarrow 0,
	\end{align} 
	and we endow $\scr{E}^*_A(1)$ with the canonical $\widehat{\overline{A}}[1/p]$-linear Lie algebra structure associated to the linear form $\iota^*$. Now we can state the construction of Sen operators in the relative situation.
\end{mypara}

\begin{mythm}[{cf. \ref{thm:sen-brinon-B}, \ref{prop:sen-brinon-B-func}, \ref{cor:quasi-adequate-cofinal}}]\label{thm:sen-brinon-B-intro}
	With the notation in {\rm\ref{para:notation-A-intro}}, for any finite projective $\widehat{\overline{A}}[1/p]$-representation $W$ of an open subgroup $G$ of $\gal(\ca{K}_{\mrm{ur}}/\ca{K})$, there is a canonical $G$-equivariant homomorphism of $\widehat{\overline{A}}[1/p]$-linear Lie algebras (where we put adjoint action of $G$ on $\mrm{End}_{\widehat{\overline{A}}[1/p]}(W)$),
	\begin{align}\label{eq:sen-brinon-B-intro}
		\varphi_{\sen}|_W:\scr{E}_A^*(1)\longrightarrow \mrm{End}_{\widehat{\overline{A}}[\frac{1}{p}]}(W),
	\end{align}
	which is functorial in $W$ and $G$, depends only on the pair $(A_{\triv},A)$ not on the choice of the chart \eqref{eq:semi-stable-chart-intro}, and satisfies the following properties:
	\begin{enumerate}
		\renewcommand{\labelenumi}{{\rm(\theenumi)}}
		\item Let $A'$ be the integral closure of a finite field extension $\ca{K}'$ of $(\ca{K}_{\mrm{ur}})^G$ contained in $\ca{K}_{\mrm{ur}}$, and let $t_1,\dots,t_e\in A'[1/p]\cap A_{\triv}'^\times$ with compatible systems of $p$-power roots $(t_{i,p^n})_{n\in\bb{N}}\subseteq \overline{A}[1/p]$ such that $\df t_1,\dots,\df t_e$ are $\ca{K}'$-linearly independent in $\Omega^1_{\ca{K}'/K}$. Consider the tower $(\ca{K}'_{n,m})_{n,m\in\bb{N}}$ defined by these elements analogously to \eqref{diam:intro-tower} and take the same notation for Galois groups, and let $A_{n,m}'$ be the integral closure of $A$ in $\ca{K}'_{n,m}$, $\widetilde{A}'_\infty=\colim_{n} \widehat{A'_n}$. Assume that there is a finite projective $\widetilde{A}'_\infty[1/p]$-representation $V$ of $\Gamma$ on which $\Delta$ acts analytically ({\rm\ref{defn:analytic}}) such that $W=\widehat{\overline{A}}\otimes_{\widetilde{A}'_\infty}V$. Then, $\Gamma$ is naturally locally isomorphic to $\bb{Z}_p\ltimes\bb{Z}_p^e$, and if we take the standard basis $\partial_0,\dots,\partial_e$ of $\lie(\Gamma)\cong \lie(\bb{Z}_p\ltimes\bb{Z}_p^e)$, then for any $f\in \scr{E}^*_A(1)$,
		\begin{align}\label{eq:thm:sen-brinon-B-intro}
			\varphi_{\sen}|_W(f)=\sum_{i=0}^e f((\df\log(t_{i,p^n}))_{n\in\bb{N}}\otimes\zeta^{-1})\otimes \varphi_{\partial_i}|_V,
		\end{align}
		where $\varphi_{\partial_i}|_V$ is the infinitesimal action of $\partial_i$ on $V$.\label{item:thm:sen-brinon-B-intro-1}
		\item Let $K'$ be a complete discrete valuation field extension of $K$ with perfect residue field, $(A'_{\triv},A')$ a semi-stable pair over $\ca{O}_{K'}$ with fraction field $\ca{K}'$, $\overline{A}\to \overline{A'}$ an injective ring homomorphism over $\ca{O}_K\to\ca{O}_{K'}$ which induces an inclusion $(A_{\triv},A)\to(A'_{\triv},A')$, $W'=\widehat{\overline{A'}}\otimes_{\widehat{\overline{A}}}W$. Assume that $\ca{K}'\otimes_{\ca{K}}\Omega^1_{\ca{K}/K}\to \Omega^1_{\ca{K'}/K'}$ is injective. Then, there is a natural commutative diagram
		\begin{align}\label{diam:sen-brinon-B-func-intro}
			\xymatrix{
				\scr{E}^*_{A'}(1)\ar[rr]^-{\varphi_{\sen}|_{W'}}\ar[d]&&\mrm{End}_{\widehat{\overline{A'}}[\frac{1}{p}]}(W')\\
				\widehat{\overline{A'}}\otimes_{\widehat{\overline{A}}}\scr{E}^*_A(1)
				\ar[rr]_-{\id_{\widehat{\overline{A'}}}\otimes\varphi_{\sen}|_W}&&\widehat{\overline{A'}}\otimes_{\widehat{\overline{A}}}\mrm{End}_{\widehat{\overline{A}}[\frac{1}{p}]}(W)
				\ar[u]_-{\wr}
			}
		\end{align}
	\end{enumerate}
\end{mythm}

The situation described in \ref{thm:sen-brinon-operator-intro}.(\ref{item:thm:sen-brinon-B-intro-1}) is not special. Indeed, by a descent theorem of Tsuji \cite[14.2]{tsuji2018localsimpson} when he developed the $p$-adic Simpson correspondence for $(A_{\triv},A)$, the representation $W$ of $G$ can be descended to $V$ for some $A'$. We remark that Tsuji proved the case when $G=\gal(\ca{K}_{\mrm{ur}}/\ca{K})$, and we prove the general case by transferring his arguments to a more general class of pairs (cf. \ref{thm:descent}). The key to the proof of \ref{thm:sen-brinon-B-intro} is still checking that the map $\varphi_{\sen}|_W$ defined by the formula \eqref{eq:thm:sen-brinon-B-intro} does not depend on the choice of $A'$, $V$ and $t_i$. We reduce this problem to the case of valuation rings \ref{thm:sen-brinon-operator-intro} by localizing at height-$1$ prime ideals of $\overline{A}$ containing $p$.

\begin{mydefn}\label{defn:sen-brinon-B-intro}
	We call the image $\Phi(W)$ of $\varphi_{\sen}|_W$ the module of \emph{Sen operators} of $W$. We call the image $\Phi^{\geo}(W)$ of $\ho_A(\Omega^1_{(A_{\triv},A)}(-1),\widehat{\overline{A}}[\frac{1}{p}])$ under $\varphi_{\sen}|_W$ the module of \emph{geometric Sen operators} of $W$. And we call the image of $1\in \widehat{\overline{A}}[\frac{1}{p}]$ under $\varphi_{\sen}|_W$ in $\Phi^{\ari}(W)=\Phi(W)/\Phi^{\geo}(W)$ the \emph{arithmetic Sen operator} of $W$.
\end{mydefn}

The following evidence supports such a definition of arithmetic Sen operator: any two lifts of it in $\mrm{End}_{\widehat{\overline{A}}[\frac{1}{p}]}(W)$ have the same characteristic polynomial (cf. \ref{prop:sen-B-char-poly}).

\begin{mypara}\label{para:inertia-subgroup-intro}
	We denote by $\ak{S}_p(\overline{A})$ the set of height-$1$ prime ideals of $\overline{A}$ containing $p$. For any $\ak{q}\in \ak{S}_p(\overline{A})$ with image $\ak{p}\in\spec(A)$, let $E_{\ak{p}}$ be the $p$-adic completion of the discrete valuation field $A_{\ak{p}}[1/p]$, $\overline{E}_{\ak{q}}$ an algebraic closure of $E_{\ak{p}}$ with an embedding of valuation rings $\overline{A}_{\ak{q}}\to \ca{O}_{\overline{E}_{\ak{q}}}$. Let $I_{\ak{q}}\subseteq \gal(\ca{K}_{\mrm{ur}}/\ca{K})$ be the image of the inertial subgroup of $\gal(\overline{E}_{\ak{q}}/E_{\ak{p}})$. We have the following generalization of Sen-Ohkubo's result, which follows from the same reduction strategy as above.
\end{mypara}

\begin{mythm}[{cf. \ref{thm:sen-lie-B}, \ref{cor:quasi-adequate-cofinal}}]\label{thm:sen-lie-B-intro}
	Let $G$ be an open subgroup of $\gal(\ca{K}_{\mrm{ur}}/\ca{K})$, $(V,\rho)$ a finite-dimensional $\bb{Q}_p$-representation of $G$, $W=\widehat{\overline{A}}[1/p]\otimes_{\bb{Q}_p}V$. Then, $\sum_{\ak{q}\in\ak{S}_p(\overline{A})}\lie(\rho(I_{\ak{q}}))$ is the smallest $\bb{Q}_p$-subspace $S$ of $\mrm{End}_{\bb{Q}_p}(V)$ such that the $\widehat{\overline{A}}[1/p]$-module of Sen operators $\Phi(W)$ is contained in $\widehat{\overline{A}}[1/p]\otimes_{\bb{Q}_p} S$.
\end{mythm}

As a corollary, one can lift the Sen operators of $\bb{Q}_p$-representations to a universal Lie algebra homomorphism.

\begin{mycor}[{cf. \ref{thm:sen-lie-lift-B}, \ref{cor:sen-lie-lift-B}}]\label{thm:sen-lie-lift-B-intro}
	Let $\ca{G}$ be a quotient of an open subgroup of $\gal(\ca{K}_{\mrm{ur}}/\ca{K})$ which is a $p$-adic analytic group. Then, there exists a canonical homomorphism of $\widehat{\overline{A}}[1/p]$-linear Lie algebras $\varphi_{\sen}|_{\ca{G}}:\scr{E}^*_A(1)\to \widehat{\overline{A}}[1/p]\otimes_{\bb{Q}_p}\lie(\ca{G})$ making the following diagram commutative for any finite-dimensional $\bb{Q}_p$-representation $V$ of $\ca{G}$,
	\begin{align}\label{diam:thm:sen-lie-lift-B-intro}
		\xymatrix{
			\scr{E}^*_A(1)\ar[r]^-{\varphi_{\sen}|_{\ca{G}}}\ar[d]_-{\varphi_{\sen}|_W}&\widehat{\overline{A}}[\frac{1}{p}]\otimes_{\bb{Q}_p}\lie(\ca{G})\ar[d]^-{\id_{\widehat{\overline{A}}[\frac{1}{p}]}\otimes\varphi|_V}\\
			\mrm{End}_{\widehat{\overline{A}}[\frac{1}{p}]}(W)&\widehat{\overline{A}}[\frac{1}{p}]\otimes_{\bb{Q}_p}\mrm{End}_{\bb{Q}_p}(V)\ar[l]_-{\sim}
		}
	\end{align}
	where $W=\widehat{\overline{A}}[1/p]\otimes_{\bb{Q}_p}V$ is the associated object of $\repnpr(G,\widehat{\overline{A}}[1/p])$, and $\varphi|_V$ is the infinitesimal Lie algebra action of $\lie(\ca{G})$ on $V$
\end{mycor}

Now we can give our generalization of Pan's result \cite[3.1.2]{pan2021locally}.

\begin{mythm}[{cf. \ref{thm:locally-analytic-vector}}]\label{thm:locally-analytic-vector-intro}
	Let $\ca{G}$ be a quotient of $\gal(\ca{K}_{\mrm{ur}}/\ca{K})$ which is a $p$-adic analytic group, $\ca{G}_H\subseteq \ca{G}$ the image of $\gal(\ca{K}_{\mrm{ur}}/\ca{K}_\infty)$, $\Phi^{\geo}_{\ca{G}}\subseteq \widehat{\overline{A}}[1/p]\otimes_{\bb{Q}_p}\lie(\ca{G}_H)$ the image of $\ho_A(\Omega^1_{(A_{\triv},A)}(-1),\widehat{\overline{A}}[1/p])$ under $\varphi_{\sen}|_{\ca{G}}$. Then, the infinitesimal action of $\Phi^{\geo}_{\ca{G}}$ annihilates the $\ca{G}_H$-locally analytic vectors in $\widehat{\overline{A}}[1/p]$ (see {\rm\ref{para:locally-analytic-vector-B}} for a precise definition).
\end{mythm}

For its proof, we need to extend Sen operators on the infinite-dimensional representations of analytic functions on sufficiently small open subgroups of $\ca{G}$. This is the reason why we insist to consider open subgroups of $\gal(\ca{K}_{\mrm{ur}}/\ca{K})$ in the previous theorems, which enables us to prove properties related to Lie algebras but leads us to a general class of pairs more than semi-stable pairs. 

Previously, we always work with representations with rational coefficients, since a finite extension $A'$ of $A$ is not a nice integral model for $A'[1/p]$. But in order to investigate the continuity of Sen operators on infinite-dimensional representations, we need to consider representations with integral coefficients as ``lattices'' to bound these operators. Nice properties of the Sen operators are preserved by continuation if we have good descent and decompletion theory for integral representations over $A'$. But it has not been well developed yet as $A'$ is not a nice integral model. However, we don't encounter such a problem if $A'$ is a valuation ring (at least for the geometric part)! So we still follow the previous strategy: reduce the problem to the case of valuation rings by localizing at height-$1$ prime ideals of $\overline{A}$ containing $p$; and for the latter case, we can apply the descent results for small representations with integral coefficients of the geometric fundamental group, developed by Faltings \cite{faltings2005simpson}, Abbes-Gros \cite[\Luoma{2}.14]{abbes2016p} and Tsuji \cite[\textsection11, \textsection12]{tsuji2018localsimpson}. We plan to investigate in the future whether or not the image $\Phi_{\ca{G}}$ of $\varphi_{\sen}|_{\ca{G}}$ annihilates the $\ca{G}$-locally analytic vectors in $\widehat{\overline{A}}[1/p]$.

\begin{mypara}
	The article is structured as follows. In section \ref{sec:lie}, we briefly review the theory of $p$-adic analytic groups from a purely algebraic view following \cite{ddms1999lie}. In section \ref{sec:kummer}, we study the tower $(K_{n,m})_{n,m\in\bb{N}}$ \eqref{diam:intro-tower} and the infinitesimal actions of representations arising from this tower in a general setting. Then, we revisit Brinon's generalization of Sen's theory in section \ref{sec:brinon} using the $p$-adic Simpson correspondence developed by Tsuji, and give our canonical definition of Sen operators. For the generalization in the relative situation, we firstly introduce the main objects, quasi-adequate algebras, in section \ref{sec:quasi}. They share nice properties with semi-stable pairs up to a $p$-power torsion by some preparation lemmas in section \ref{sec:bound}. Especially, we can also define Faltings extension for such general algebra. A priori, the construction of Faltings extension in the relative situation is not canonical. We show the canonicity of Faltings extension by reducing to the case of valuation rings, cf. \ref{thm:B-fal-ext}. To glue the Sen operators defined over valuation rings, we need a ``global model'' on a quasi-adequate algebra, that is, a descent of representation of the fundamental group. This is a generalization of Tsuji's result and done in the section \ref{sec:descent}. We construct the Sen operators in section \ref{sec:sen-B}, and discuss their relation with Lie algebras. Finally, we extend Sen operators to infinite-dimensional representations in section \ref{sec:inf} and the end of section \ref{sec:sen-B}, and give an application on locally analytic vectors in the last section \ref{sec:analytic}.
\end{mypara}

\subsection*{Acknowledgements}
This work is part of my thesis prepared at Universit\'e Paris-Saclay and Institut des Hautes \'Etudes Scientifiques. 
I would like to express my sincere gratitude to my doctoral supervisor, Ahmed~ Abbes, for his guidance to this project, his thorough review of this work and his plenty of helpful suggestions on both research and writing. 
I would like to thank Longke~ Tang for his help on the proof of \ref{prop:cotangent-bound}. I would like to thank Lue~ Pan, Takeshi~ Tsuji, Zhixiang~ Wu, and Yicheng~ Zhou for useful discussions.

\section{Notation and Conventions}

\begin{mypara}\label{para:product}
	Let $d\in\bb{N}$ be a natural number. We endow the set $(\bb{N}\cup\{\infty\})^d$ with the partial order defined by $\underline{m}=(m_1,\dots,m_d)\leq \underline{m'}=(m_1',\dots,m_d')$ if $m_i\leq m'_i$ for any $1\leq i\leq d$. We put $|\underline{m}|=m_1+\cdots+m_d\in \bb{N}\cup\{\infty\}$. For any $r\in \bb{N}\cup\{\infty\}$, we set $\underline{r}=(r,\dots,r)\in (\bb{N}\cup\{\infty\})^d$ and $\underline{r}_i=(0,\dots,r,\dots,0)\in (\bb{N}\cup\{\infty\})^d$ where $r$ appears at the $i$-th component.
	
	We endow the set $\bb{N}_{>0}^d$ with the partial order defined by $\underline{N}| \underline{N'}$ if $N_i$ divides $N'_i$ for any $1\leq i\leq d$, where $\underline{N}=(N_1,\dots,N_d)$ and $\underline{N'}=(N_1',\dots,N_d')$, and we put $\underline{N'}/\underline{N}=(N_1'N_1^{-1},\dots,N_d'N_d^{-1})\in \bb{N}_{>0}^d$.
\end{mypara}

\begin{mypara}\label{para:notation-Tate-mod}
	All rings considered in this article are unitary and commutative. We fix a prime number $p$. For a ring $R$, we denote by $\Omega^1_R$ the $R$-module $\Omega^1_{R/\bb{Z}}$ of $1$-differentials of $R$ over $\bb{Z}$, and by $\widehat{\Omega}^1_R$ its $p$-adic completion. For an abelian group $M$, we set
	\begin{align}
		T_p(M)=&\plim_{x\mapsto px}M[p^n]=\ho_{\bb{Z}}(\bb{Z}[1/p]/\bb{Z},M),\\
		V_p(M)=&\plim_{x\mapsto px}M=\ho_{\bb{Z}}(\bb{Z}[1/p],M).
	\end{align}
	We remark that $T_p(M)$ is a $p$-adically complete $\bb{Z}_p$-module (\cite[4.4]{jannsen1988cont}), and that if $M=M[p^\infty]$ (i.e., $M$ is $p$-primary torsion) then $V_p(M)=T_p(M)\otimes_{\bb{Z}_p}\bb{Q}_p$. We fix an algebraic closure $\overline{\bb{Q}}_p$ of $\bb{Q}_p$, and we set $\bb{Z}_p(1)=T_p(\overline{\bb{Q}}_p^\times)$ which is a free $\bb{Z}_p$-module of rank $1$ and any compatible system of primitive $p^n$-th roots of unity $\zeta=(\zeta_{p^n})_{n\in \bb{N}}$ in $\overline{\bb{Q}}_p$ (i.e. $\zeta_{p^{n+1}}^p=\zeta_{p^n}$, $\zeta_1=1$, $\zeta_p\neq 1$) gives a basis of it. We endow $\bb{Z}_p(1)$ with the natural continuous action of the Galois group $\gal(\overline{\bb{Q}}_p/\bb{Q}_p)$. For any $\bb{Z}_p$-module $M$ and $r\in \bb{Z}$, we set $M(r)=M\otimes_{\bb{Z}_p}\bb{Z}_p(1)^{\otimes r}$, the $r$-th Tate twist of $M$.
\end{mypara}

\begin{mypara}\label{para:canonical-top}
	Let $A$ be a topological ring, $M$ a finitely generated $A$-module. For any $A$-linear surjective homomorphism $A^n\to M$ with $n\in\bb{N}$, if we endow $A^n$ with the product topology, then the quotient topology on $M$ does not depend on the choice of the surjection. We call this topology on $M$ the \emph{canonical topology} (\cite[page 820]{tsuji2018localsimpson}). It is clear that any homomorphism of finitely generated $A$-modules is continuous with respect to the canonical topology.
	
	If the topology on $A$ is linear, then the canonical topology on $M$ is also linear. For another finitely generated $A$-module $N$, the canonical topology on $M\otimes_A N$ coincides with the tensor product topology of the canonical topologies on $M$ and $N$. Moreover, let $A\to A'$ be a continuous homomorphism of linearly topologized rings. Then, the canonical topology on $A'\otimes_A M$ as a finitely generated $A'$-module coincides with the tensor product topology of the topology on $A'$ and the canonical topology on $M$.
\end{mypara}

\begin{mypara}\label{defn:repn}
	Let $G$ be a topological group, $A$ a topological ring endowed with a continuous action by $G$. An \emph{$A$-representation $(W,\rho)$ of $G$} is a topological $A$-module $W$ endowed with a continuous semi-linear action $\rho$ of $G$. A morphism $W\to W'$ of $A$-representations of $G$ is a continuous $A$-linear homomorphism compatible with the action of $G$. We denote by $\repn(G,A)$ the category of $A$-representations of $G$. Let $A_0$ be a $G$-stable subring of $A$. The \emph{$(G,A_0)$-finite part} of an $A$-representation $W$ of $G$ is the sum of all $G$-stable finitely generated $A_0$-submodules of $W$.
	
	We say that an $A$-representation $W$ of $G$ is \emph{finite projective} if $W$ is a finite projective $A$-module endowed with the canonical topology. We denote by $\repnpr(G,A)$ the full subcategory of  $\repn(G,A)$ consisting of finite projective $A$-representations of $G$. 
	
	Assume that the topology on $A$ is linear. For any two $A$-representations $W,W'$ of $G$ with linear topologies, the diagonal action of $G$ on $W\otimes_A W'$ is continuous with respect to the tensor product topology. If moreover $W$ and $W'$ are finite projective, then so is $W\otimes_A W'$ by \ref{para:canonical-top}. This makes $\repnpr(G,A)$ into an additive tensor category. Moreover, let $A'$ be a linearly topologized ring endowed with a continuous action of a topological group $G'$, $G'\to G$ a continuous group homomorphism, $A\to A'$ a continuous ring homomorphism compatible with the actions of $G$ and $G'$. Then, the tensor product defines a natural functor 
	\begin{align}
		\repnpr(G,A)\longrightarrow \repnpr(G',A'),\ W\mapsto A'\otimes_A W.
	\end{align}
\end{mypara}

\section{Brief Review on $p$-adic Analytic Groups}\label{sec:lie}
The theory of $p$-adic analytic groups (which are often referred to as ``$p$-adic Lie groups'') was developed by Lazard \cite{lazard1965lie}. We mainly follow \cite{ddms1999lie} to give a brief review.

\begin{mydefn}[{\cite[Theorem 4.5]{ddms1999lie}}]\label{defn:uniform-pro-p}
	A pro-$p$ group $G$ is called \emph{uniform} if $G$ is topologically finitely generated, torsion free and $G/\overline{G^p}$ (resp. $G/\overline{G^4}$) is abelian if $p$ is odd (resp. $p=2$), where $\overline{G^n}$ denotes the closed subgroup of $G$ generated by $n$-th powers for $n\in \bb{N}$. 
\end{mydefn}
In fact, the subset of $p^n$-th powers in a uniform pro-$p$ group $G$ forms a uniform and open characteristic subgroup $G^{p^n}$ of $G$; these open subgroups $\{G^{p^n}\}_{n\in\bb{N}}$ form a fundamental system of neighbourhoods of $1\in G$; and the map $G\to G^{p^n}$ sending $x$ to $x^{p^n}$ is a homeomorphism of topological spaces (\cite[Theorems 3.6, 4.10]{ddms1999lie}).

\begin{mydefn}[{\cite[Section 9.4]{ddms1999lie}}]\label{defn:powerful-lie-alg}
	A Lie algebra $L$ over $\bb{Z}_p$ is called \emph{powerful} if $L$ is a finite free $\bb{Z}_p$-module and $[L,L] \subseteq pL$ (resp. $[L,L] \subseteq 4L$) if $p$ is odd (resp. $p=2$).
\end{mydefn}

For a powerful Lie algebra $L$ over $\bb{Z}_p$ and any $n\in\bb{N}$, it is clear that the Lie sub-algebra $p^nL$ is also powerful.

\begin{mypara}\label{para:L_G}
	We associate to a uniform pro-p group $G$ a powerful Lie algebra $L_G=(G,+_G,[\ ,\ ]_G)$ over $\bb{Z}_p$ as follows: 
	\begin{enumerate}
		\renewcommand{\labelenumi}{{\rm(\theenumi)}}
		\item The underlying set of $L_G$ is that of $G$.
		\item The additive structure on $L_G$ is given by (\cite[Definition 4.12]{ddms1999lie})
		\begin{align}\label{eq:para:L_G-add}
			x +_G y=\lim_{n\to \infty} (x^{p^n}y^{p^n})^{p^{-n}},\ \forall x,y\in G,
		\end{align}
		where taking $p^n$-th root is well-defined as the map $G\to G^{p^n}$ sending $x$ to $x^{p^n}$ is a homeomorphism.
		\item The Lie bracket on $L_G$ is given by (\cite[Definition 4.29]{ddms1999lie})
		\begin{align}\label{eq:para:L_G-bracket}
			[x , y]_G=\lim_{n\to \infty} (x^{-p^n}y^{-p^n}x^{p^n}y^{p^n})^{p^{-2n}},\ \forall x,y\in G.
		\end{align}
	\end{enumerate}
	The Lie algebra $L_G$ over $\bb{Z}_p$ is well-defined and powerful, and a minimal topological generating set $(g_1,\dots,g_d)$ of $G$ forms a $\bb{Z}_p$-linear basis of $L_G$ (\cite[Theorems 4.17, 4.30 and Exercise 4.2.(\luoma{2})]{ddms1999lie}). We denote by $\log:G\to L_G$ and $\exp:L_G\to G$ the identity maps. Then, the map 
	\begin{align}\label{eq:coordinates-first}
		\bb{Z}_p^d\longrightarrow G,\ (a_1,\dots,a_d)\mapsto \exp(a_1\log(g_1)+_G\cdots+_G a_d\log(g_d))
	\end{align}
	is a homeomorphism of topological spaces such that the image of $(p^n\bb{Z}_p)^d$ is $G^{p^n}$. This map is called the \emph{system of coordinates of the first kind}. Moreover, for $x,y\in G$ such that $xy=yx$ (resp. for $u,v\in L_G$ such that $[u,v]=0$), we have $\log(xy)=\log(x)+_G\log(y)$ (resp. $\exp(u+_Gv)=\exp(u)\exp(v)$). In particular, for $a\in \bb{Z}_p$, we have $\log(x^a)=a\log(x)$ for any $x\in G$, and we have $\exp(au)=\exp(u)^a$ for any $u\in L_G$. On the other hand, the map
	\begin{align}\label{eq:coordinates-second}
		\bb{Z}_p^d\longrightarrow G,\ (b_1,\dots,b_d)\mapsto g_1^{b_1}\cdots g_d^{b_d},
	\end{align}
	is also a homeomorphism of topological spaces such that the image of $(p^n\bb{Z}_p)^d$ is $G^{p^n}$ (\cite[Theorems 4.9, 4.10]{ddms1999lie}). This map is called the \emph{system of coordinates of the second kind}.
\end{mypara}

\begin{mypara}\label{para:L-*}
	We associate to a powerful Lie algebra $L$ over $\bb{Z}_p$ a uniform pro-p group $(L,*)$ as follows: we endow $L$ with a group structure given by the Baker-Campbell-Hausdorff formula (\cite[Section 9.4]{ddms1999lie})
	\begin{align}\label{eq:para:L-*}
		x*y=x+y+\frac{1}{2}[x,y]+\frac{1}{12}([x,[x,y]]+[y,[y,x]])+\cdots.
	\end{align}
	The group $(L,*)$ is well-defined and uniform pro-$p$, and a $\bb{Z}_p$-linear basis of $L$ forms a minimal topological generating set of $(L,*)$ (\cite[Theorem 9.8]{ddms1999lie}).
\end{mypara}

\begin{mythm}[{\cite[Theorem 9.10]{ddms1999lie}}]\label{thm:powerful-lie}
	The assignments $G\mapsto L_G$ and $L\mapsto (L,*)$ defined in {\rm\ref{para:L_G}} and {\rm\ref{para:L-*}} are mutually inverse isomorphisms between the category of uniform pro-$p$ groups and the category of powerful Lie algebras over $\bb{Z}_p$.
\end{mythm}

\begin{myexample}\label{exmp:gl_n}
	The subgroup $G=\id+p^\epsilon\mrm{M}_d(\bb{Z}_p)$ of the general linear group $\mrm{GL}_d(\bb{Z}_p)$ of degree $d$ over $\bb{Z}_p$ is a uniform pro-$p$ group, where $\epsilon=1$ if $p$ is odd, and $\epsilon=2$ if $p=2$. We have $G^{p^n}=\id+p^{n+\epsilon}\mrm{M}_d(\bb{Z}_p)$ for any $n\in\bb{N}$. In fact, the matrix exponential and logarithm,
	\begin{align}
		\exp: p^\epsilon\mrm{M}_d(\bb{Z}_p)&\longrightarrow\id+p^\epsilon\mrm{M}_d(\bb{Z}_p),\ X\mapsto \sum_{n= 0}^{\infty}\frac{1}{n!}X^n,\label{eq:matrix-exp}\\
		\log:\id+p^\epsilon\mrm{M}_d(\bb{Z}_p)&\longrightarrow p^\epsilon\mrm{M}_d(\bb{Z}_p),\ 1+X\mapsto \sum_{n= 1}^{\infty}\frac{(-1)^{n-1}}{n}X^n,\label{eq:matrix-log}
	\end{align}
	are mutually inverse homeomorphisms, which identify $\id+p^{n+\epsilon}\mrm{M}_d(\bb{Z}_p)$ with $p^{n+\epsilon}\mrm{M}_d(\bb{Z}_p)$. Moreover, they induce an isomorphism of $\bb{Z}_p$-linear Lie algebras $L_G\iso p^\epsilon\mrm{M}_d(\bb{Z}_p)$, where $p^\epsilon\mrm{M}_d(\bb{Z}_p)$ is endowed with the usual matrix Lie algebra structure. We can extend the matrix logarithm to $\log:\mrm{GL}_d(\bb{Z}_p)\to \mrm{M}_d(\bb{Q}_p)$ by setting $\log(X)=\log(X^r)/r$ for some $r\in \bb{N}$ such that $X^r\in \id+p^\epsilon\mrm{M}_d(\bb{Z}_p)$. Especially, for $d=1$, we can take $r=p(p-1)$ so that
	\begin{align}\label{eq:p-adic-log}
		\log: \bb{Z}_p^\times \longrightarrow \bb{Z}_p.
	\end{align}
\end{myexample}

\begin{mylem}[{\cite[Proposition 4.31]{ddms1999lie}}]\label{lem:uniform-ext}
	Let $G$ be a uniform pro-$p$ group, $N$ a closed normal subgroup of $G$ such that $G/N$ is a uniform pro-$p$ group. Then, $N$ is also a uniform pro-$p$ group, and the following natural sequence of powerful Lie algebras over $\bb{Z}_p$ is exact
	\begin{align}
		0\longrightarrow L_N\longrightarrow L_G\longrightarrow L_{G/N}\longrightarrow 0.
	\end{align} 
\end{mylem}

\begin{mydefn}[{\cite[Theorems 8.32, 9.4]{ddms1999lie}}]\label{defn:p-analytic-group}
	A \emph{$p$-adic analytic group} is a topological group which contains a uniform pro-$p$ open subgroup. A morphism between $p$-adic analytic groups is a continuous group homomorphism.
\end{mydefn}

\begin{mythm}[{\cite[Theorem 7.19]{ddms1999lie}}]\label{thm:compact-lie-group}
	Any compact $p$-adic analytic group is isomorphic to a closed subgroup of $\mrm{GL}_d(\bb{Z}_p)$ for some $d\in\bb{N}$.
\end{mythm}

\begin{mylem}[{\cite[proof of Theorem 4.8]{ddms1999lie}}]\label{lem:uniform-ext-2}
	Let $G$ be a $p$-adic analytic group, $N$ a closed normal subgroup of $G$. Then, there exists an open subgroup $G_0$ of $G$ such that $G_0$, $N_0=N\cap G_0$ and $G_0/N_0$ are all uniform pro-$p$ groups.
\end{mylem}

\begin{mythm}[{\cite[Theorems 9.6, 9.7]{ddms1999lie}}]\label{thm:lie-group-ext}
	Let $G$ be a separated topological group, $N$ a closed normal subgroup of $G$, $H$ a closed subgroup of $G$.
	\begin{enumerate}
		\renewcommand{\labelenumi}{{\rm(\theenumi)}}
		\item If $G$ is a $p$-adic analytic group, then so is $H$ and $G/N$.
		\item If $N$ and $G/N$ are $p$-adic analytic groups, then so is $G$.
	\end{enumerate}
\end{mythm}

\begin{mypara}
	The uniform pro-$p$ open subgroups $H$ of a $p$-adic analytic group $G$ form a fundamental system of open neighbourhoods of $1\in G$. Moreover, given such an $H$, $\{H^{p^n}\}_{n\in\bb{N}}$ is initial in this system. Thus, for a uniform pro-$p$ open subgroup $H'$ of $H$, the corresponding Lie algebra $L_{H'}$ over $\bb{Z}_p$ is a $\bb{Z}_p$-submodule of $L_H$ with finite index. In particular, the natural morphism 
	\begin{align}
		L_{H'}\otimes_{\bb{Z}_p}\bb{Q}_p\longrightarrow L_{H}\otimes_{\bb{Z}_p}\bb{Q}_p
	\end{align}
	is an isomorphism of Lie algebras over $\bb{Q}_p$.
\end{mypara}

\begin{mydefn}[{\cite[Section 9.5]{ddms1999lie}}]\label{defn:lie-alg}
	Let $G$ be a $p$-adic analytic group. The filtered colimit of Lie algebras over $\bb{Q}_p$,
	\begin{align}
		\lie(G)=\colim_{H} L_H\otimes_{\bb{Z}_p}\bb{Q}_p,
	\end{align}
	where $H$ is a uniform pro-$p$ open subgroup of $G$ with the corresponding Lie algebra $L_H$ over $\bb{Z}_p$, is called the \emph{Lie algebra} of $G$ over $\bb{Q}_p$. We denote by $\dim(G)$ the dimension of $\lie(G)$ over $\bb{Q}_p$ and call it the \emph{dimension} of $G$.
	
	Moreover, if $G$ is compact, then there is a canonical continuous map, called the \emph{logarithm map} of $G$,
	\begin{align}\label{eq:log-lie}
		\log_G:G\longrightarrow \lie(G)
	\end{align}
	sending $g$ to $\log_H(g^{r})/r$, where $r$ is the index of a uniform pro-$p$ open subgroup $H$ of $G$, $\log_H:H\to L_H$ is defined in \ref{para:L_G}, and this definition does not depend on the choice of $H$.
\end{mydefn}

\begin{mylem}\label{lem:lie-ext}
	Let $G$ be a $p$-adic analytic group, $N$ a closed normal subgroup of $G$. Then, there is a canonical exact sequence of $\bb{Q}_p$-linear Lie algebras
	\begin{align}
		0\longrightarrow \lie(N)\longrightarrow\lie(G)\longrightarrow\lie(G/N)\longrightarrow 0.
	\end{align}
\end{mylem}
\begin{proof}
	It follows from \ref{lem:uniform-ext-2} and \ref{lem:uniform-ext}.
\end{proof}

\begin{mypara}\label{para:adjoint}
	Let $G$ be a $p$-adic analytic group. For any $g\in G$, the conjugation on $G$ sending $x$ to $gxg^{-1}$ is continuous, and thus induces an automorphism $\mrm{Ad}_g$ of the Lie algebra $\lie(G)$. The map
	\begin{align}\label{eq:adjoint}
		\mrm{Ad}:G\longrightarrow \mrm{Aut}_{\bb{Q}_p}(\lie(G)),\ g\mapsto \mrm{Ad}_g,
	\end{align}
	is a continuous group homomorphism, which makes $\lie(G)$ into a finite projective $\bb{Q}_p$-representation of $G$, which we call the \emph{adjoint representation} of $G$ (cf. \cite[Exercise 9.11]{ddms1999lie}). This construction is functorial in $G$. We remark that for $G=\mrm{GL}_d(\bb{Z}_p)$, the adjoint action is given by $\mrm{Ad}_X(Y)=XYX^{-1}$ for any $X\in \mrm{GL}_d(\bb{Z}_p)$ and $Y\in \lie(G)=\mrm{M}_d(\bb{Q}_p)$.
\end{mypara}

\section{Infinitesimal Actions of Representations arising from Kummer Towers}\label{sec:kummer}
\begin{mydefn}\label{defn:ht1-prime}
	Let $A$ be a ring, $\pi$ an element of $A$. We denote by $\ak{S}_{\pi}(A)$ the set of prime ideals $\ak{p}$ of height $1$ containing $\pi$. 
\end{mydefn}
We remark that for a Noetherian normal domain $A$ with a non-zero element $\pi$, the set $\ak{S}_{\pi}(A)$ coincides naturally with the finite set of generic points of $\spec(A/\pi A)$, and $A_{\ak{p}}$ is a discrete valuation ring for any $\ak{p}\in \ak{S}_{\pi}(A)$.

\begin{mylem}\label{lem:ht1-prime-map}
	Let $A\to B$ be an injective and integral homomorphism of domains with $A$ normal, $\pi$ an element of $A$. Then, the inverse image of $\ak{S}_{\pi}(A)$ via the map $\spec(B)\to \spec(A)$ is $\ak{S}_{\pi}(B)$, and the induced map $\ak{S}_{\pi}(B)\to \ak{S}_{\pi}(A)$ is surjective.
\end{mylem}
\begin{proof}
	Firstly, we note that $\spec(B)\to \spec(A)$ is surjective. For any $\ak{q}\in\spec(B)$ with image $\ak{p}\in\spec(A)$, if $\ak{p}$ is of height $1$, then so is $\ak{q}$, since any point of $\spec(B)$ is closed in its fibre over $\spec(A)$; conversely, if $\ak{q}$ is of height $1$, then so is $\ak{p}$, since $A\to B$ satisfies going down property (\cite[\href{https://stacks.math.columbia.edu/tag/00H8}{00H8}]{stacks-project}). This shows that the inverse image of $\ak{S}_{\pi}(A)$ via the surjection $\spec(B)\to \spec(A)$ is $\ak{S}_{\pi}(B)$.
\end{proof}

\begin{myprop}\label{prop:ht1-prime-inj}
	Let $A\to B$ be an injective and integral homomorphism of normal domains with $A$ Noetherian, $\pi$ a nonzero element of $A$. We assume that $B$ is the union of a directed system $(B_\lambda)_{\lambda\in\Lambda}$ of Noetherian normal $A$-subalgebras. 
	\begin{enumerate}
		\renewcommand{\labelenumi}{{\rm(\theenumi)}}
		\item We have $\ak{S}_\pi(B)=\lim_{\lambda\in\Lambda^{\oppo}} \ak{S}_\pi(B_\lambda)$, and for each $\ak{q}\in \ak{S}_\pi(B)$, if we denote by $\ak{q}_\lambda\in \ak{S}_\pi(B_\lambda)$ its image, then $(B_{\lambda,\ak{q}_\lambda})_{\lambda\in\Lambda}$ is a directed system of discrete valuation rings with faithfully flat transition maps, whose colimit is $B_{\ak{q}}$, a valuation ring of height $1$.\label{item:prop:ht1-prime-inj-1}
		\item For any integer $n>0$, the natural map 
		\begin{align}\label{eq:prop:ht1-prime-inj-1}
			B/\pi^n B\longrightarrow \prod_{\ak{q}\in\ak{S}_{\pi}(B)} B_{\ak{q}}/\pi^n B_{\ak{q}}
		\end{align}
		is injective, which thus induces an injective map of $\pi$-adic completions
		\begin{align}\label{eq:prop:ht1-prime-inj-2}
			\widehat{B}\longrightarrow \prod_{\ak{q}\in\ak{S}_{\pi}(B)} \widehat{B_{\ak{q}}}.
		\end{align}\label{item:prop:ht1-prime-inj-2}
	\end{enumerate}
\end{myprop}
\begin{proof}
	(\ref{item:prop:ht1-prime-inj-1}) Since $\spec(B)=\lim_{\lambda\in \Lambda^{\oppo}}\spec(B_{\lambda})$ by \cite[8.2.10]{ega4-3}, we have $\ak{S}_{\pi}(B)=\lim_{\lambda\in \Lambda^{\oppo}} \ak{S}_{\pi}(B_{\lambda})$ by \ref{lem:ht1-prime-map}. Since $B_\lambda$ is a Noetherian normal domain, its localization $B_{\lambda,\ak{q}_\lambda}$ is a discrete valuation ring, and the transition map $B_{\lambda,\ak{q}_\lambda}\to B_{\lambda',\ak{q}_{\lambda'}}$ for $\lambda\leq \lambda'$ is local and injective, thus an extension of discrete valuation rings, which completes the proof. 
	
	(\ref{item:prop:ht1-prime-inj-2}) For each $\lambda$, the map $B_\lambda/\pi^n B_\lambda\to \prod_{\ak{q}_\lambda\in\ak{S}_{\pi}(B_\lambda)} B_{\lambda,\ak{q}_\lambda}/\pi^n B_{\lambda,\ak{q}_\lambda}$ is injective by \cite[\href{https://stacks.math.columbia.edu/tag/031T}{031T}, \href{https://stacks.math.columbia.edu/tag/0311}{0311}]{stacks-project}. If we denote by $f_\lambda:\ak{S}_\pi(B)\to \ak{S}_\pi(B_\lambda)$ the natural surjection, then
	\begin{align}
		B_{\lambda,\ak{q}_\lambda}/\pi^n B_{\lambda,\ak{q}_\lambda}\longrightarrow \prod_{\ak{q}\in f_\lambda^{-1}(\ak{q}_{\lambda})} B_{\ak{q}}/\pi^n B_{\ak{q}}
	\end{align}
	is injective as $B_{\lambda,\ak{q}_\lambda}\to B_{\ak{q}}$ is faithfully flat. Thus, the composition of the two previous maps, $B_\lambda/\pi^n B_\lambda\to \prod_{\ak{q}\in\ak{S}_{\pi}(B)} B_{\ak{q}}/\pi^n B_{\ak{q}}$, is injective. The conclusion follows from taking filtered colimit on $\lambda\in\Lambda$.
\end{proof}

\begin{myrem}\label{rem:int-clos-fini}
	Let $A\to B$ be an injective and integral homomorphism of normal domains with $A$ Noetherian. We remark that if the fraction field of $B$ is a finite separable extension of that of $A$, then $B$ is finite over $A$ (\cite[\href{https://stacks.math.columbia.edu/tag/032L}{032L}]{stacks-project}). Thus, the assumption of \ref{prop:ht1-prime-inj} is satisfied if the fraction field of $B$ is a separable extension of that of $A$.
\end{myrem}

\begin{mydefn}\label{defn:tower}
	A \emph{tower of normal domains} is a directed system $(A_\lambda)_{\lambda\in \Lambda}$ of normal domains with injective and integral transition morphisms. We denote by $A_\infty$ the colimit of $(A_\lambda)_{\lambda\in \Lambda}$.
\end{mydefn}

We remark that if $(\ca{K}_\lambda)_{\lambda\in \Lambda}$ is the tower of the fraction fields of a tower of normal domains $(A_\lambda)_{\lambda\in \Lambda}$, then in fact $A_\lambda$ is the integral closure of $A_{\lambda_0}$ in $\ca{K}_\lambda$ for $\lambda,\lambda_0\in \Lambda$ with $\lambda\geq \lambda_0$. Moreover, for any element $\pi\in A_{\lambda_0}$, we obtain an inverse system of sets $(\ak{S}_\pi(A_\lambda))_{\lambda\in\Lambda_{\geq\lambda_0}^{\oppo}}$ with surjective transition maps by \ref{lem:ht1-prime-map}, and we have $\ak{S}_\pi(A_\infty)=\lim_{\lambda\in\Lambda_{\geq\lambda_0}^{\oppo}} \ak{S}_\pi(A_\lambda)$.

\begin{mylem}\label{lem:p-adic-separated1}
	Let $A$ be a ring, $\pi$ an element of $A$, $M\to N$ an injective homomorphism of $\pi$-torsion free $A$-modules. Assume that $M=M[1/\pi]\cap N\subseteq N[1/p]$. Then, for any integer $n>0$, the homomorphism 
	\begin{align}
		M/\pi^n M\longrightarrow N/\pi^n N
	\end{align}
	is injective. In particular, the homomorphism of the $\pi$-adic completions (endowed with the $\pi$-adic topology) $\widehat{M}\to \widehat{N}$ is a closed embedding.
\end{mylem}
\begin{proof}
	Firstly, we show that $\pi^nM=M\cap \pi^n N$ (i.e. $M/\pi^n M\to N/\pi^n N$ is injective). For $x\in M\cap \pi^n N$, $\pi^{-n}x\in M[1/\pi]$ lies in $N$. Hence, $\pi^{-n}x\in M[1/\pi]\cap N=M$, which proves the assertion. Then, we see that $\widehat{M}\to \widehat{N}$ is injective and that the $\pi$-adic topology on $\widehat{M}$ coincides with the topology induced from the $\pi$-adic topology of $\widehat{N}$. Since $\widehat{M}$ is complete and $\widehat{N}$ is separated, $\widehat{M}$ is closed in $\widehat{N}$ (\cite[\Luoma{2}.16, Proposition 8]{bourbaki1971top1-4}). Thus, $\widehat{M}$ identifies with a closed topological subgroup of $\widehat{N}$.
\end{proof}

\begin{mydefn}\label{defn:tower-completion}
	Let $K$ be a valuation field of height $1$ with a non-zero element $\pi$ in its maximal ideal, $(A_\lambda)_{\lambda\in \Lambda}$ a tower of normal domains flat over $\ca{O}_K$. The $\pi$-adic completions form a directed system $(\widehat{A}_\lambda)_{\lambda\in \Lambda}$ of flat $\ca{O}_K$-algebras, whose transition maps are closed embeddings with respect to the $\pi$-adic topology by \ref{lem:p-adic-separated1}. We set 
	\begin{align}
		\widetilde{A}_\infty=\colim_{\lambda\in \Lambda}\widehat{A}_\lambda.
	\end{align}
	As $\widehat{A}_\lambda/\pi^n \widehat{A}_\lambda\to \widetilde{A}_\infty/\pi^n\widetilde{A}_\infty$ is also injective for any $n\in \bb{N}$ by \ref{lem:p-adic-separated1}, we see that $\widehat{A}_\lambda\to \widetilde{A}_\infty$ is also a closed embedding with respect to the $\pi$-adic topology and that $\widetilde{A}_\infty$ is $\pi$-adically separated. Thus, we always regard $\widetilde{A}_\infty$ as a topological $\ca{O}_K$-subalgebra of the $\pi$-adic completion $\widehat{A}_\infty$ of $A_\infty$.
\end{mydefn}

\begin{mydefn}\label{defn:infinitesimal}
	Let $M$ be a separated topological $\bb{Q}_p$-module endowed with a continuous action of a pro-$p$ group $G$. For any $g\in G$ and $x\in M$, if the limit
	\begin{align}\label{eq:infinitesimal}
		\lim_{\bb{Z}_p\setminus\{0\}\ni a\to 0} a^{-1}(g^a-1)(x)
	\end{align} 
	exists in $M$, then we denote it by $\varphi_g(x)$ and call the assignment $x\mapsto \varphi_g(x)$ the \emph{infinitesimal action} of $g$ on $x$. 
\end{mydefn}

The following lemma follows directly from the definition.

\begin{mylem}\label{lem:infinitesimal}
	Let $M$ be a separated topological $\bb{Q}_p$-module endowed with a continuous action of a pro-$p$ group $G$.
	\begin{enumerate}
		\renewcommand{\labelenumi}{{\rm(\theenumi)}}
		\item Assume that the infinitesimal action of an element $g\in G$ exists for any $x\in M$. Then, $\varphi_g:x\mapsto \varphi_g(x)$ is a $\bb{Q}_p$-linear endomorphism of $M$, and we have $\varphi_{g^a}=a\varphi_g$ for any $a\in\bb{Z}_p$. Moreover, for any $g'\in G$, the infinitesimal action of $g'gg'^{-1}$ also exists for any $x\in M$, and we have $\varphi_{g'gg'^{-1}}=g'\circ\varphi_g\circ g'^{-1}$. \label{item:infinitesimal-1}
		\item If the infinitesimal action of two elements $g,g'\in G$ with $gg'=g'g$ exists for any $x\in M$, then $g'\circ \varphi_g=\varphi_g\circ g'$, $\varphi_g\circ\varphi_{g'}=\varphi_{g'}\circ\varphi_g$, and $\varphi_{gg'}=\varphi_g+\varphi_{g'}$.\label{item:infinitesimal-2}
		\item Let $G'\to G$ be a continuous homomorphism of pro-$p$ groups, $M'$ a separated topological $\bb{Q}_p$-module endowed with a continuous action of a pro-$p$ group $G'$, $f:M\to M'$ a continuous $\bb{Q}_p$-homomorphism compatible with the actions of $G$ and $G'$, $x\in M$, $g'\in G'$. Assume that the infinitesimal actions of $g'$ and its image $g\in G$ exist for $f(x)\in M'$ and $x$ respectively. Then, $f(\varphi_g(x))=\varphi_{g'}(f(x))$.\label{item:infinitesimal-3}
	\end{enumerate}
\end{mylem}

\begin{myprop}[{\cite[5.3]{tsuji2018localsimpson}}]\label{prop:derivative}
	Let $A$ be a topological $\bb{Q}_p$-algebra endowed with a continuous action of a topological group $G$. Assume that $G$ contains a pro-$p$ open subgroup $G_0$ of finite index and that there exists a tower $(A_\lambda)_{\lambda\in \Lambda}$ of normal domains flat over $\bb{Z}_p$ such that there is an isomorphism of topological rings $A\cong \widetilde{A}_\infty[1/p]$ and that for any $\lambda\in\Lambda$ the subalgebra $\widehat{A}_\lambda[1/p]$ is $G_0$-stable and invariant by an open subgroup of $G_0$ (via the isomorphism $A\cong \widetilde{A}_\infty[1/p]$). Then, for any $g\in G$ and any object $W$ of $\repnpr(G,A)$, there exists a unique $A$-linear endomorphism $\varphi_g|_W$ of $W$ satisfying the following conditions (we simply write $\varphi_g|_W$ by $\varphi_g$ if there is no ambiguity):
	\begin{enumerate}
		\renewcommand{\labelenumi}{{\rm(\theenumi)}}
		\item For any $g\in G_0$, $\varphi_g$ is the infinitesimal action \eqref{eq:infinitesimal} of $g$ on $W$.\label{item:prop:derivative-1}
		\item For any $g\in G$ and $n\in\bb{N}$, we have $\varphi_{g^n}=n\varphi_g$.\label{item:prop:derivative-2}
		\item For any $g\in G_0$ and $x\in W$, there exists $m_x\in \bb{N}$ such that for any $a\in p^{m_x}\bb{Z}_p$,
		\begin{align}\label{eq:prop:derivative1}
			g^a(x)=\exp(a\varphi_g)(x)=\sum_{k=0}^\infty \frac{a^k}{k!}(\underbrace{\varphi_g\circ\cdots \circ\varphi_g}_{k\trm{ copies}})(x).
		\end{align}\label{item:prop:derivative-3}
	\end{enumerate}
	In particular, $\varphi_g|_W$ does not depend on the choice of $G_0$ or the tower $(A_\lambda)_{\lambda\in \Lambda}$. Thus, we still call it the \emph{infinitesimal action} of $g\in G$ on $W$.
\end{myprop}
\begin{proof}
	Firstly, assume that $G=G_0$. Since $\widetilde{A}_\infty$ is a $p$-adically separated flat $\bb{Z}_p$-algebra, there is a canonical norm on $\widetilde{A}_\infty[1/p]$ which induces its $p$-adic topology (cf. \ref{para:normed-module-2}). Thus, we are in the situation of \cite[5.3]{tsuji2018localsimpson}, and the conclusion follows from it. In general, for any $g\in G$, we set $\varphi_g=r^{-1}\varphi_{g^r}$ where $r$ is the index of $G_0$ in $G$ and $\varphi_{g^r}$ is the infinitesimal action of $g^r\in G_0$ on finite projective $A$-representation $W$ of $G_0$ (defined by restricting the $G$-action of $W$). One can check easily by \ref{lem:infinitesimal} that this $\varphi_g$ satisfies all the required properties.
\end{proof}

\begin{myrem}\label{rem:derivative}
	Let $A'$ be a topological $\bb{Q}_p$-algebra endowed with a continuous action of a topological group $G'$ satisfying the assumptions in {\rm\ref{prop:derivative}}. Assume that there is a morphism of topological groups $G'\to G$ and a morphism of topological rings $A\to A'$ which is compatible with the actions of $G$ and $G'$. For any object $W$ of $\repnpr(G,A)$, the base change $W'=A'\otimes_A W$ is naturally an object of $\repnpr(G',A')$ by \ref{defn:repn}. Then, for any $g'\in G'$ with image $g\in G$, we deduce from \eqref{eq:infinitesimal} and \ref{prop:derivative} that
	\begin{align}
		\varphi_{g'}|_{W'}=\id_{A'}\otimes \varphi_g|_W.
	\end{align}
\end{myrem}

\begin{mylem}\label{lem:derivative-uniform-cont}
	Under the assumptions in {\rm\ref{prop:derivative}}, assume further that $G$ acts trivially on $A$. Then, there is a pro-$p$ open subgroup $G_1$ of $G$ such that the map
	\begin{align}\label{eq:rem:derivative-1}
		\phi:\bb{Z}_p\times G_1\times W\longrightarrow W,
	\end{align}
	sending $(0,g,x)$ to $\varphi_g(x)$ and sending $(a,g,x)$ to $a^{-1}(g^ax-x)$ for $a\neq 0$, is continuous.
\end{mylem}
\begin{proof}
	Since $G$ acts trivially on $A$ by assumption, for any pro-$p$ open subgroup $G_1$ of $G$, $\phi$ induces a map
	\begin{align}
		\widetilde{\phi}:\bb{Z}_p\times G_1\longrightarrow \mrm{End}_A(W),
	\end{align}
	sending $(0,g)$ to $\varphi_g$ and sending $(a,g)$ to $a^{-1}(g^a-1)$ for $a\neq 0$. We fix an $A$-linear surjection $A^r\to W$. As $W$ is finite projective over $A$, we get an $A$-linear surjection $\psi:\mrm{M}_r(A)\to \mrm{End}_A(W)$ whose quotient topology on $\mrm{End}_A(W)$ coincides with its canonical topology (\ref{para:canonical-top}). As the matrix multiplication $\mrm{M}_r(A)\times A^r\to A^r$ is continuous, the natural $A$-linear homomorphism $\mrm{End}_A(W)\times W\to W$ is also continuous (\ref{para:canonical-top}). Thus, it suffices to show that $\widetilde{\phi}$ is continuous for some $G_1$. Moreover, by the definition \eqref{eq:infinitesimal} of $\varphi_g$, $\widetilde{\phi}$ is uniformly continuous if and only if its restriction on $(\bb{Z}_p\setminus\{0\})\times G_1$ is uniformly continuous (\cite[\Luoma{2}.20, Th\'eor\`eme 2]{bourbaki1971top1-4}). We claim that the latter holds for some $G_1$. 
	
	We note that the topology on $A$ is defined by the $p$-adic topology of $\widetilde{A}_\infty$. As $G$ acts continuously and $A$-linearly on $W$ by assumption, the induced map $G\to \mrm{End}_A(W)$ is continuous. For any $k\geq 2$, we take pro-$p$ open subgroups $G_k$ of $G$ whose image in $\mrm{End}_A(W)$ lies in $\psi(\id+p^k\mrm{M}_r(\widetilde{A}_\infty))$. For $g_1,g_2\in G_2$ with $g_1g_2^{-1}\in G_k$, let $P_1, P_2\in \id+p^2\mrm{M}_r(\widehat{A}_\lambda)$ be some liftings of the images of $g_1,g_2\in G_2$ respectively with $P_1P_2^{-1}\in \id+p^k\mrm{M}_r(\widehat{A}_\lambda)$ for some $\lambda\in\Lambda$. In particular, $P_1-P_2\in p^k\mrm{M}_r(\widehat{A}_\lambda)$, and  
	\begin{align}
		\log(P_1)-\log(P_2)=\sum_{n=1}^{\infty}\frac{(-1)^{n-1}}{n}((P_2-\id+(P_1-P_2))^n-(P_2-\id)^n)\in p^k\mrm{M}_r(\widehat{A}_\lambda)
	\end{align}
	as $p^{2(n-1)}\in n!\cdot\bb{Z}_p$ for any $n\geq 1$. Similarly, for any $a_1,a_2\in \bb{Z}_p\setminus\{0\}$ with $a_1-a_2\in p^k\bb{Z}_p$, we have
	\begin{align}
		a_1^{-1}(P_1^{a_1}-\id)-a_1^{-1}(P_2^{a_1}-\id)&=\sum_{n=1}^{\infty} a_1^{n-1}\frac{\log(P_1)^n-\log(P_2)^n}{n!} \in p^k\mrm{M}_r(\widehat{A}_\lambda),\\
		a_1^{-1}(P_2^{a_1}-\id)-a_2^{-1}(P_2^{a_2}-\id)&=\sum_{n=2}^{\infty} (a_1^{n-1}-a_2^{n-1})\frac{\log(P_2)^n}{n!} \in p^k\mrm{M}_r(\widehat{A}_\lambda).
	\end{align}
	Thus, $\widetilde{\phi}(a_1,g_1)-\widetilde{\phi}(a_2,g_2)$ belongs to $\psi(p^k\mrm{M}_r(\widetilde{A}_\infty))$, which implies that $\widetilde{\phi}$ is uniformly continuous by taking $G_1=G_2$.
\end{proof}

\begin{mycor}\label{cor:operator}
	Under the assumptions in {\rm\ref{prop:derivative}}, assume further that $G$ is a compact $p$-adic analytic group. Let $\lie(G)$ be the Lie algebra of $G$ over $\bb{Q}_p$. Then, there is a unique morphism of Lie algebras over $\bb{Q}_p$,
	\begin{align}\label{eq:lie-action}
		\varphi:\lie(G)\longrightarrow \mrm{End}_A(W),
	\end{align}
	such that its composition with the logarithm map of $G$ \eqref{eq:log-lie} $\log_G:G\to \lie(G)$ is the map $\varphi:G\to \mrm{End}_A(W)$ sending $g$ to the infinitesimal action $\varphi_g$ of $g\in G$ on $W$.
	
	We call $\varphi$ the canonical Lie algebra action induced by the infinitesimal action of $G$ on $W$, or simply the \emph{infinitesimal Lie algebra action}.
\end{mycor}
\begin{proof}
	Recall that $\lie(G)=\colim_{G_0} L_{G_0}\otimes_{\bb{Z}_p}\bb{Q}_p$ where the colimit is taken over the system of uniform pro-$p$ open subgroups $G_0$ of $G$ (see \ref{defn:lie-alg}), and that the $\bb{Z}_p$-linear Lie algebra $L_{G_0}=(G_0,+_{G_0},[\ ,\ ]_{G_0})$ is defined in \ref{para:L_G}. As $\log:G_0\to L_{G_0}$ is a homeomorphism, the uniqueness is obvious. It remains to check that the map $L_{G_0}\to \mrm{End}_A(W)$ sending $g$ to $\varphi_g$ is compatible with addition and Lie bracket. 
	
	As $G_0$ is topologically finitely generated, there exists $\lambda\in\Lambda$ and a finite projective $\widehat{A}_\lambda[1/p]$-representation $W_\lambda$ of $G_0$ such that $W=A\otimes_{\widehat{A}_\lambda[1/p]}W_\lambda$ (\cite[5.2.(1)]{tsuji2018localsimpson}). By \ref{rem:derivative}, it suffices to check that the map $L_{G_0}\to \mrm{End}_{\widehat{A}_\lambda[1/p]}(W_\lambda)$ sending $g$ to $\varphi_g$ is compatible with addition and Lie bracket. We take a uniform pro-$p$ open subgroup $G_1$ of $G_0$ such that $\widehat{A}_\lambda[1/p]$ is $G_1$-invariant by the assumptions in \ref{prop:derivative}. After replacing $G_1$ by $G_1^{p^n}$ for some $n\in\bb{N}$, we may assume by \ref{lem:derivative-uniform-cont} that the map  
	\begin{align}\label{eq:cor:operator}
		\phi_\lambda:\bb{Z}_p\times G_1\times W_\lambda\longrightarrow W_\lambda,
	\end{align}
	sending $(0,g,x)$ to $\varphi_g(x)$ and sending $(a,g,x)$ to $a^{-1}(g^ax-x)$ for $a\neq 0$, is continuous. As the map $L_{G_0}\to \mrm{End}_{\widehat{A}_\lambda[1/p]}(W_\lambda)$ is compatible with multiplication by an integer by \ref{prop:derivative}.(\ref{item:prop:derivative-2}), it suffices to check that its restriction on $L_{G_1}$ is compatible with addition and Lie bracket. For any $g,g'\in G_1$ and $x\in W_\lambda$, applying the continuity of \eqref{eq:cor:operator} to the convergent sequence $\{(p^n,(g^{p^n}g'^{p^n})^{p^{-n}},x)\}_{n\in\bb{N}}\subseteq \bb{Z}_p\times G_1\times W_\lambda$ with limit $(0,g+_{G_1}g',x)$ by \eqref{eq:para:L_G-add}, we get
	\begin{align}
		\varphi_{g+_{G_1}g'}(x)=\lim_{n\to \infty} p^{-n}(g^{p^n}g'^{p^n}x-x).
	\end{align}
	On the other hand, 
	\begin{align}
		(\varphi_g+\varphi_{g'})(x)=\lim_{n\to \infty}p^{-n}(g^{p^n}x+g'^{p^n}x-2x).
	\end{align}
	Thus, $\varphi_{g+_{G_1}g'}(x)-(\varphi_g+\varphi_{g'})(x)=\lim_{n\to\infty}p^{-n}(g^{p^n}-1)(g'^{p^n}x-x)=0$, since the action $G_1\times W_\lambda\to W_\lambda$ is continuous and $\lim_{n\to\infty}p^{-n}(g'^{p^n}x-x)=\varphi_{g'}(x)$. Similarly, applying the continuity of \eqref{eq:cor:operator} to the convergent sequence $\{(p^{2n},(g^{-p^n}g'^{-p^n}g^{p^n}g'^{p^n})^{p^{-2n}},x)\}_{n\in\bb{N}}\subseteq \bb{Z}_p\times G_1\times W_\lambda$ with limit $(0,[g,g']_{G_1},x)$ by \eqref{eq:para:L_G-bracket}, we get
	\begin{align}
		\varphi_{[g,g']_{G_1}}(x)=\lim_{n\to\infty}p^{-2n}(g^{-p^n}g'^{-p^n}g^{p^n}g'^{p^n}x-x).
	\end{align} 
	On the other hand, applying the continuity of \eqref{eq:cor:operator} to the convergent sequence $\{(p^n,g,p^{-n}(g'^{p^n}x-x))\}_{n\in\bb{N}}\subseteq \bb{Z}_p\times G_1\times W_\lambda$ with limit $(0,g,\varphi_{g'}(x))$, we get
	\begin{align}\label{eq:cor:operator-2}
		(\varphi_g\circ\varphi_{g'}-\varphi_{g'}\circ\varphi_g)(x)=\lim_{n\to \infty}p^{-2n}(g^{p^n}g'^{p^n}x-g'^{p^n}g^{p^n}x).
	\end{align}
	Thus, $\varphi_{[g,g']_{G_1}}(x)-(\varphi_g\circ\varphi_{g'}-\varphi_{g'}\circ\varphi_g)(x)=\lim_{n\to \infty}p^{-2n}(g^{-p^n}g'^{-p^n}-1)(g^{p^n}g'^{p^n}x-g'^{p^n}g^{p^n}x)=0$ by the continuity of the action $G_1\times W_\lambda\to W_\lambda$ and \eqref{eq:cor:operator-2}.
\end{proof}

\begin{mydefn}[{cf. \cite[14.1]{tsuji2018localsimpson}}]\label{defn:analytic}
	Under the assumptions in {\rm\ref{prop:derivative}}, let $H$ be a subgroup of $G$. For an object $W$ in $\repnpr(G,A)$, we say that $W$ is \emph{$H$-analytic} if $\exp(\varphi_h)(x)=\sum_{k=0}^\infty \varphi_h^k(x)/k!$ converges to $h(x)$ for any $x\in W$ and any $h\in H$. We denote by $\repnan{H}(G,A)$ the full subcategory of $\repnpr(G,A)$ formed by $H$-analytic objects.
\end{mydefn}

\begin{mypara}\label{para:notation-kummer}
	Let $K$ be a field of characteristic not equal to $p$, $\overline{K}$ an algebraic closure of $K$, $(\zeta_{p^n})_{n\in \bb{N}}$ a compatible system of primitive $p$-power roots of unity in $\overline{K}$, $G=\gal(\overline{K}/K)$. For any $n\in \bb{N}\cup\{\infty\}$, we define a Galois extension of $K$ in $\overline{K}$ by
	\begin{align}
		K_n=K(\zeta_{p^k}\ |\ k\in \bb{N}_{\leq n}).
	\end{align} 
	Consider the cyclotomic character 
	\begin{align}\label{eq:cycl-char}
		\chi:G\longrightarrow \bb{Z}_p^\times
	\end{align}
	which is defined by $\sigma(\zeta_{p^n})=\zeta_{p^n}^{\chi(\sigma)}$ for any $\sigma\in G$ and $n\in\bb{N}$. It factors through an injection $\Sigma=\gal(K_\infty/K)\hookrightarrow \bb{Z}_p^\times$, and does not depend on the choice of the system $(\zeta_{p^n})_{n\in \bb{N}}$. In particular, $\Sigma$ is either finite cyclic or isomorphic to the direct product of a finite cyclic group with $\bb{Z}_p$.
	
	We fix $d\in \bb{N}$. Let $t_1,\dots,t_d$ be elements of $K$ with compatible systems of $p$-power roots $(t_{i,p^n})_{n\in \bb{N}}$ in $\overline{K}$ (where $1\leq i\leq d$). For consistency, sometimes we also denote $\zeta_{p^n}$ by $t_{0,p^n}$. For any $n\in \bb{N}\cup\{\infty\}$ and any $\underline{m}=(m_1,\dots,m_d)\in (\bb{N}\cup\{\infty\})^d$, we define an extension of $K$ in $\overline{K}$ by
	\begin{align}
		K_{n,\underline{m}}=K_n(t_{1,p^{k_1}}, \cdots,t_{d,p^{k_d}}\ |\ k_i\in \bb{N}_{\leq m_i},\ 1\leq i\leq d).
	\end{align}
	It is a Galois extension of $K$ if $n\geq \max\{m_1,\dots,m_d\}$. Consider the continuous map
	\begin{align}\label{eq:cont-cocycle}
		\xi=(\xi_1,\dots,\xi_d):G\longrightarrow \bb{Z}_p^d
	\end{align}
	defined by $\tau(t_{i,p^n})=\zeta_{p^n}^{\xi_i(\tau)}t_{i,p^n}$ for any $\tau\in G$, $n\in\bb{N}$ and $1\leq i\leq d$. Notice that for any $\sigma,\tau\in G$,
	\begin{align}\label{eq:cont-cocycle-cond}
		\xi(\sigma\tau)=\xi(\sigma)+\chi(\sigma)\xi(\tau)
	\end{align}
	Thus, $\xi$ is a continuous $1$-cocycle. It becomes a group homomorphism when restricted to $H=\gal(\overline{K}/K_\infty)$, which factors through an injection $\Delta=\gal(K_{\infty,\underline{\infty}}/K_\infty)\hookrightarrow \bb{Z}_p^d$. In particular, $\Delta$ is isomorphic to $\bb{Z}_p^r$ for some $0\leq r\leq d$. For any $\sigma\in\Gamma=\gal(K_{\infty,\underline{\infty}}/K)$ and $\tau\in\Delta$, we have
	\begin{align}\label{eq:para:notation-kummer-relation}
		\sigma\tau\sigma^{-1}=\tau^{\chi(\sigma)},
	\end{align}
	by the definition of $\chi$. We have named some Galois groups as indicated in the following diagram:
	\begin{align}
		\xymatrix{
			\overline{K}&\\
			K_{\infty,\underline{\infty}}\ar[u]&\\
			K_\infty\ar[u]^-{\Delta}\ar@/^2pc/[uu]^-{H}&K\ar[l]^-{\Sigma}\ar[lu]_-{\Gamma}\ar@/_1pc/[luu]_{G}
		}
	\end{align}
	We remark that $\Gamma$ is a compact $p$-adic analytic group as an extension of $\Sigma$ by $\Delta$, and there is a natural exact sequence of Lie algebras over $\bb{Q}_p$ by \ref{lem:lie-ext},
	\begin{align}\label{eq:lie-sequence-gamma}
		0\longrightarrow \lie(\Delta)\longrightarrow
		\lie(\Gamma)\longrightarrow \lie(\Sigma)\longrightarrow 0.
	\end{align}
	Notice that the group homomorphism $\log\circ \chi:\Gamma\to \bb{Z}_p$ induces a homomorphism of $\bb{Q}_p$-linear Lie algebras $\log\circ \chi:\lie(\Gamma)\to \lie(\bb{Z}_p)=\bb{Q}_p$ which factors through $\lie(\Sigma)$, where $\log:\bb{Z}_p^\times\to \bb{Z}_p$ is the $p$-adic logarithm map \eqref{eq:p-adic-log}. We deduce from \eqref{eq:para:notation-kummer-relation} and \eqref{eq:para:L_G-bracket} that for any $x\in \lie(\Gamma)$ and $y\in \lie(\Delta)$,
	\begin{align}\label{eq:lie-str-gamma}
		[x,y]=\log(\chi(x))\cdot y.
	\end{align}
\end{mypara}

\begin{mypara}\label{defn:kummer-tower}
	Let $A$ be a Noetherian normal domain flat over $\bb{Z}_p$ with fraction field $K$, $t_i \in K$ with a compatible system of $p$-power roots $(t_{i,p^n})_{n\in \bb{N}}$ in $\overline{K}$ for any $1\leq i\leq d$. With the notation in \ref{para:notation-kummer}, let $A_{n,\underline{m}}$ be the integral closure of $A$ in $K_{n,\underline{m}}$ for any $n\in\bb{N}\cup\{\infty\}$ and $\underline{m}\in(\bb{N}\cup\{\infty\})^d$. We remark that $A_{n,\underline{m}}$ is a Noetherian normal domain finite over $A$ if $n,\underline{m}$ are finite by \ref{rem:int-clos-fini}. Endowing $\bb{N}^{1+d}$ with the product order (cf. \ref{para:product}), we call the tower of Noetherian normal domains $(A_{n,\underline{m}})_{(n,\underline{m})\in\bb{N}^{1+d}}$ the \emph{Kummer tower of $A$} defined by $\zeta_{p^n},t_{1,p^n},\dots,t_{d,p^n}$.
	
	Notice that $\Gamma$ acts continuously on $\widetilde{A}_{\infty,\underline{\infty}}[1/p]$ (defined in \ref{defn:tower-completion}) which satisfies the assumptions in \ref{cor:operator}. Then, for any object $W$ of $\repnpr(\Gamma,\widetilde{A}_{\infty,\underline{\infty}}[1/p])$, there is a canonical morphism of Lie algebras over $\bb{Q}_p$ induced by the infinitesimal action of $\Gamma$ on $W$,
	\begin{align}\label{eq:gamma-lie-action}
		\varphi:\lie(\Gamma)\longrightarrow \mrm{End}_{\widetilde{A}_{\infty,\underline{\infty}}[1/p]}(W).
	\end{align}
\end{mypara}

\begin{mylem}[{cf. \cite[Propositions 5, 7]{brinon2003sen}}]\label{lem:operator}
	With the notation in {\rm\ref{defn:kummer-tower}}, for any $\sigma\in\Gamma$, $\tau\in \Delta$ and any object $W$ of $\repnpr(\Gamma,\widetilde{A}_{\infty,\underline{\infty}}[1/p])$, we have
	\begin{align}
		\sigma \circ \varphi_{\tau} \circ \sigma^{-1}&=\chi(\sigma)\cdot \varphi_{\tau}, \label{eq:lem:operator-1}\\
		\varphi_\sigma\circ \varphi_\tau-\varphi_\tau\circ \varphi_\sigma&=\log(\chi(\sigma))\cdot\varphi_\tau,\label{eq:lem:operator-2}
	\end{align}
	as $\widetilde{A}_{\infty,\underline{\infty}}[1/p]$-linear endomorphisms on $W$.
\end{mylem}
\begin{proof}
	As
	$\sigma\tau\sigma^{-1}=\tau^{\chi(\sigma)}$ \eqref{eq:para:notation-kummer-relation}, we have $\sigma \circ \varphi_{\tau} \circ \sigma^{-1}=\varphi_{\sigma\tau\sigma^{-1}}=\varphi_{\tau^{\chi(\sigma)}}=\chi(\sigma)\cdot \varphi_{\tau}$ (cf. \ref{lem:infinitesimal}). As $[\log_\Gamma(\sigma),\log_\Gamma(\tau)]=\log(\chi(\sigma))\cdot \log_\Gamma(\tau)$ by \eqref{eq:lie-str-gamma}, we have $\varphi_\sigma\circ \varphi_\tau-\varphi_\tau\circ \varphi_\sigma=\varphi_{[\log_\Gamma(\sigma),\log_\Gamma(\tau)]}=\log(\chi(\sigma))\cdot\varphi_\tau$.
\end{proof}

\begin{myprop}[{cf. \cite[Proposition 5]{brinon2003sen}, \cite[14.17]{tsuji2018localsimpson}}]\label{prop:operator-nilpotent}
	With the notation in {\rm\ref{defn:kummer-tower}}, assume that $K_{\infty}$ is an infinite extension of $K$ and that $\ak{S}_p(A_{\infty,\underline{\infty}})$ is finite. Then, the infinitesimal action $\varphi_\tau$ of any element $\tau\in \Delta$ on any object of $\repnpr(\Gamma,\widetilde{A}_{\infty,\underline{\infty}}[1/p])$ is nilpotent.
\end{myprop}
\begin{proof}
	We follow the proof of \cite[14.17]{tsuji2018localsimpson}. Notice that $\Sigma$ identifies to an open subgroup of $\bb{Z}_p^\times$ via the cyclotomic character \eqref{eq:cycl-char}. Thus, there exists $\sigma\in \Gamma$ such that $\log(\chi(\sigma))\neq 0$. For any $\ak{q}\in\ak{S}_p(A_{\infty,\underline{\infty}})$, the localization $A_{\infty,\underline{\infty},\ak{q}}$ is a valuation ring of height $1$ by \ref{prop:ht1-prime-inj}.(\ref{item:prop:ht1-prime-inj-1}), and we denote by $L_{\ak{q}}$ the fraction field of the $p$-adic completion of $A_{\infty,\underline{\infty},\ak{q}}$. Then, the natural map
	\begin{align}\label{eq:prop:operator-nilpotent-1}
		\widehat{A}_{\infty,\underline{\infty}}[1/p]\longrightarrow \prod_{\ak{q}\in \ak{S}_p(A_{\infty,\underline{\infty}})} L_{\ak{q}}
	\end{align}
	is injective by \ref{prop:ht1-prime-inj}.(\ref{item:prop:ht1-prime-inj-2}). 
	
	For an object $W$ of $\repnpr(\Gamma,\widetilde{A}_{\infty,\underline{\infty}}[1/p])$ and $g\in \Gamma$, we denote by $W_{\ak{q}}$ (resp. $\varphi_{g,\ak{q}}$) the scalar extension of $W$ (resp. $\varphi_g$) to $L_{\ak{q}}$. As $\ak{S}_p(A_{\infty,\underline{\infty}})$ is finite, it suffices to show that $\varphi_{\tau,\ak{q}}$ is nilpotent by the injection \eqref{eq:prop:operator-nilpotent-1} and \ref{rem:derivative}. Let $P_{\ak{q}}(T)\in L_{\ak{q}}[T]$ be the characteristic polynomial of $\varphi_{\sigma,\ak{q}}$. We take an integer $N>0$ large enough such that $P_{\ak{q}}(T-N\log(\chi(\sigma)))$ is prime to $P_{\ak{q}}(T)$. Thus, the endomorphism $P_{\ak{q}}(\varphi_{\sigma,\ak{q}}-N\log(\chi(\sigma)))$ on $W_{\ak{q}}$ is an automorphism. By \eqref{eq:lem:operator-2}, we have
	\begin{align}
		P_{\ak{q}}(\varphi_{\sigma,\ak{q}}-N\log(\chi(\sigma)))\circ (\varphi_{\tau,\ak{q}})^N=(\varphi_{\tau,\ak{q}})^N\circ P_{\ak{q}}(\varphi_{\sigma,\ak{q}})=0.
	\end{align}
	Hence, we have $(\varphi_{\tau,\ak{q}})^N=0$.
\end{proof}

\begin{myprop}\label{prop:kummer-tower-lem}
	With the notation in {\rm\ref{defn:kummer-tower}}, consider the following statements:
	\begin{enumerate}
		\renewcommand{\labelenumi}{{\rm(\theenumi)}}
		\item We have $A/pA\neq 0$ (thus $\ak{S}_p(A)$ is non-empty), and for any $\ak{p}\in \ak{S}_p(A)$, if we denote by $\widehat{K}^{\ak{p}}$ the completion of $K$ with respect to the discrete valuation ring $A_{\ak{p}}$ and consider the Kummer tower $(\widehat{K}^{\ak{p}}_{n,\underline{m}})_{(n,\underline{m})\in\bb{N}^{1+d}}$ of $\widehat{K}^{\ak{p}}$ defined by $\zeta_{p^n},t_{1,p^n},\dots,t_{d,p^n}$, then the image of the continuous homomorphism \eqref{eq:cont-cocycle} $\gal(\widehat{K}^{\ak{p}}_{\infty,\underline{\infty}}/\widehat{K}^{\ak{p}}_\infty)\to \bb{Z}_p^d$ is open. \label{item:prop:kummer-tower-lem-1}
		\item There exists $n_0\in \bb{N}$ such that for any $n\in \bb{N}_{\geq n_0}$, the cyclotomic character \eqref{eq:cycl-char} $\chi:G\to \bb{Z}_p^\times$ and the $p$-adic logarithm map \eqref{eq:p-adic-log} $\log:\bb{Z}_p^\times\to \bb{Z}_p$ induce an isomorphism
		\begin{align}\label{eq:prop:kummer-tower-lem-chi}
			\log\circ\chi: \gal(K_{\infty,\underline{\infty}}/K_{n,\underline{\infty}})\iso p^n\bb{Z}_p;
		\end{align}
		and there exists $\underline{m_0}\in \bb{N}^d$ such that for any $\underline{m}\in (\bb{N}^d)_{\geq \underline{m_0}}$ the continuous $1$-cocycle \eqref{eq:cont-cocycle} $\xi:G\to \bb{Z}_p^d$ induces an isomorphism 
			\begin{align}\label{eq:prop:kummer-tower-lem-xi}
				\xi:\gal(K_{\infty,\underline{\infty}}/K_{\infty,\underline{m}})\iso p^{m_1}\bb{Z}_p\times\cdots\times p^{m_d}\bb{Z}_p
			\end{align}
			where $\underline{m}=(m_1,\dots,m_d)$.\label{item:prop:kummer-tower-lem-2}
		\item There exists $(n_0,\underline{m_0})\in \bb{N}^{1+d}$ such that for any $(n,\underline{m})\in (\bb{N}^{1+d})_{\geq (n_0,\underline{m_0})}$, the natural map $\ak{S}_p(A_{n,\underline{m}})\to \ak{S}_p(A_{n_0,\underline{m_0}})$ is a bijection.\label{item:prop:kummer-tower-lem-3}
		\item The cardinality of $\ak{S}_p(A_{n,\underline{m}})$,  when $(n,  \underline{m})$ varies in $\bb{N}^{1+d}$, is bounded.\label{item:prop:kummer-tower-lem-4}
	\end{enumerate}
	Then, {\rm(\ref{item:prop:kummer-tower-lem-1})} implies {\rm(\ref{item:prop:kummer-tower-lem-2})}, {\rm(\ref{item:prop:kummer-tower-lem-3})} and {\rm(\ref{item:prop:kummer-tower-lem-4})}; and {\rm(\ref{item:prop:kummer-tower-lem-3})} is equivalent to {\rm(\ref{item:prop:kummer-tower-lem-4})}.
\end{myprop}
\begin{proof}
	Notice that for elements $(n',\underline{m'})\geq (n,\underline{m})$ in $\bb{N}^{1+d}$, the natural map $\ak{S}_p(A_{n',\underline{m'}})\to \ak{S}_p(A_{n,\underline{m}})$ is surjective by \ref{lem:ht1-prime-map}. Since $\bb{N}^{1+d}$ is directed, we see that (\ref{item:prop:kummer-tower-lem-3}) and (\ref{item:prop:kummer-tower-lem-4}) are equivalent. 
	
	Now we assume (\ref{item:prop:kummer-tower-lem-1}). We take an integer $m_0\in \bb{N}$ such that $(p^{m_0}\bb{Z}_p)^d$ lies in the image of the injective homomorphism $\xi:\gal(\widehat{K}^{\ak{p}}_{\infty,\underline{\infty}}/\widehat{K}^{\ak{p}}_\infty)\to \bb{Z}_p^d$ \eqref{eq:cont-cocycle}. We identify $(p^{m_0}\bb{Z}_p)^d$ with an open normal subgroup of $\gal(\widehat{K}^{\ak{p}}_{\infty,\underline{\infty}}/\widehat{K}^{\ak{p}}_\infty)$. We claim that the invariant subextension of $\widehat{K}^{\ak{p}}_{\infty,\underline{\infty}}$ by $p^{m_1}\bb{Z}_p\times\cdots\times p^{m_d}\bb{Z}_p \subseteq (p^{m_0}\bb{Z}_p)^d$ is $\widehat{K}^{\ak{p}}_{\infty,\underline{m}}$ for any $\underline{m}=(m_1,\dots,m_d)\in \bb{N}_{\geq m_0}^d$. Indeed, the invariant subextension contains $\widehat{K}^{\ak{p}}_{\infty,\underline{m}}$ by the definition of $\xi$. On the other hand, $\gal(\widehat{K}^{\ak{p}}_{\infty,\underline{\infty}}/\widehat{K}^{\ak{p}}_{\infty,\underline{m}})$ identifies with a closed subgroup of $p^{m_1}\bb{Z}_p\times\cdots\times p^{m_d}\bb{Z}_p$ via $\xi$. Thus, the claim follows from the Galois theory. In particular, $\xi$ induces a natural isomorphism 
	\begin{align}\label{eq:prop:kummer-tower-lem-1}
		\xi:\gal(\widehat{K}^{\ak{p}}_{\infty,\underline{\infty}}/\widehat{K}^{\ak{p}}_{\infty,\underline{m}})\iso p^{m_1}\bb{Z}_p\times\cdots\times p^{m_d}\bb{Z}_p.
	\end{align}
	We claim that $\widehat{K}^{\ak{p}}_{\infty,\underline{\infty}}$ is an infinite extension of $\widehat{K}^{\ak{p}}_{0,\underline{\infty}}$. Otherwise, $\widehat{K}^{\ak{p}}_{\infty,\underline{\infty}}$ is an extension of a finite extension of $\widehat{K}^{\ak{p}}$ by adding $t_{1,p^n},\dots,t_{d,p^n}$, so that the dimension of the $p$-adic analytic group $\gal(\widehat{K}^{\ak{p}}_{\infty,\underline{\infty}}/\widehat{K}^{\ak{p}})$ is no more than $d$ by \eqref{eq:cont-cocycle}. On the other hand, $\widehat{K}^{\ak{p}}$ is a complete discrete valuation field, while the valuation on $\widehat{K}^{\ak{p}}_{\infty}$ is non-discrete of height $1$. Thus, $\gal(\widehat{K}^{\ak{p}}_{\infty}/\widehat{K}^{\ak{p}})$ is an open subgroup of $\bb{Z}_p^\times$, which implies that the dimension of the $p$-adic analytic group $\gal(\widehat{K}^{\ak{p}}_{\infty,\underline{\infty}}/\widehat{K}^{\ak{p}})$ is $1+d$ under the assumption (\ref{item:prop:kummer-tower-lem-1}) by \eqref{eq:lie-sequence-gamma}. We get a contradiction, which proves the claim. Thus, the image of the cyclotomic character 
	\begin{align}\label{eq:prop:kummer-tower-lem-1-1}
		\chi:\gal(\widehat{K}^{\ak{p}}_{\infty,\underline{\infty}}/\widehat{K}^{\ak{p}}_{0,\underline{\infty}})\longrightarrow \bb{Z}_p^\times
	\end{align}
	is open. We take $n_0\in\bb{N}_{\geq 2}$ such that $1+p^{n_0}\bb{Z}_p$ lies in the image of \eqref{eq:prop:kummer-tower-lem-1-1}. Similarly as above, the invariant subextension of $\widehat{K}^{\ak{p}}_{\infty,\underline{\infty}}$ by $1+p^n\bb{Z}_p\subseteq 1+p^{n_0}\bb{Z}_p$ is $\widehat{K}^{\ak{p}}_{n,\underline{\infty}}$ by the definition of $\chi$, for any $n\in \bb{N}_{\geq n_0}$. In particular, $\chi$ and $\log$ induce an isomorphism
	\begin{align}\label{eq:prop:kummer-tower-lem-1-2}
		\log\circ\chi:\gal(\widehat{K}^{\ak{p}}_{\infty,\underline{\infty}}/\widehat{K}^{\ak{p}}_{n,\underline{\infty}})\iso p^n\bb{Z}_p.
	\end{align}	
	For any $(n,\underline{m})\in (\bb{N}^{1+d})_{\geq (n_0,\underline{m_0})}$, we have
	\begin{align}\label{eq:prop:kummer-tower-lem-2}
		[\widehat{K}^{\ak{p}}_{\infty,\underline{m}}:\widehat{K}^{\ak{p}}_{\infty,\underline{m_0}}]\leq[\widehat{K}^{\ak{p}}_{n,\underline{m}}:\widehat{K}^{\ak{p}}_{n,\underline{m_0}}]\leq[K_{n,\underline{m}}:K_{n,\underline{m_0}}]\leq p^{\sum_{i=1}^d(m_i-m_0)}.
	\end{align}
	By \eqref{eq:prop:kummer-tower-lem-1}, we see that the inequalities in \eqref{eq:prop:kummer-tower-lem-2} are equalities, which implies \eqref{eq:prop:kummer-tower-lem-xi}. In particular, each fibre of $\ak{S}_p(A_{n,\underline{m}})\to \ak{S}_p(A_{n,\underline{m_0}})$ consists of a single element (cf. \cite[\Luoma{6}.\textsection 8.5, Cor.3]{bourbaki2006commalg5-7}). Similarly,
	\begin{align}\label{eq:prop:kummer-tower-lem-3}
		[\widehat{K}^{\ak{p}}_{n,\underline{\infty}}:\widehat{K}^{\ak{p}}_{n_0,\underline{\infty}}]\leq[\widehat{K}^{\ak{p}}_{n,\underline{m_0}}:\widehat{K}^{\ak{p}}_{n_0,\underline{m_0}}]\leq[K_{n,\underline{m_0}}:K_{n_0,\underline{m_0}}]\leq p^{n-n_0}.
	\end{align}
	By \eqref{eq:prop:kummer-tower-lem-1-2}, we see that the inequalities in \eqref{eq:prop:kummer-tower-lem-3} are equalities, which implies \eqref{eq:prop:kummer-tower-lem-chi}. In particular, each fibre of $\ak{S}_p(A_{n,\underline{m_0}})\to \ak{S}_p(A_{n_0,\underline{m_0}})$ consists of a single element. Therefore, we obtain (\ref{item:prop:kummer-tower-lem-3}).
\end{proof}

We will give in \ref{prop:rank} and \ref{prop:B-rank} some differential criteria for checking the condition {\rm\ref{prop:kummer-tower-lem}.(\ref{item:prop:kummer-tower-lem-1})} for a Kummer tower.

\begin{mypara}\label{para:gamma-basis}
	With the notation in \ref{defn:kummer-tower}, for further computation (e.g. \ref{prop:special-functor}), we introduce a standard basis of $\lie(\Gamma)$ under the assumption that the Kummer tower $(A_{n,\underline{m}})_{(n,\underline{m})\in \bb{N}^{1+d}}$ satisfies the condition {\rm\ref{prop:kummer-tower-lem}.(\ref{item:prop:kummer-tower-lem-2})}. We name some Galois groups as indicated in the following diagram for any $n\in\bb{N}$:
	\begin{align}
		\xymatrix{
			K_{\infty,\underline{\infty}}&K_{n,\underline{\infty}}\ar[l]_-{\Sigma_{n,\underline{\infty}}}\\
			K_\infty\ar[u]^-{\Delta}&K_n\ar[l]^-{\Sigma_n}\ar[lu]|-{\Gamma_n}\ar[u]
		}
	\end{align}
	By the assumption {\rm\ref{prop:kummer-tower-lem}.(\ref{item:prop:kummer-tower-lem-2})}, there is an isomorphism for some $n_0\in\bb{N}$,
	\begin{align}
		\log\circ\chi: \Sigma_{n_0,\underline{\infty}}\iso\Sigma_{n_0}\iso p^{n_0}\bb{Z}_p,
	\end{align} 
	and $\xi$ identifies $\Delta$ with an open subgroup of $\bb{Z}_p^d$. The isomorphism $\Sigma_{n_0,\underline{\infty}}\iso\Sigma_{n_0}$ identifies $\Gamma_{n_0}$ with the semi-direct product $\Sigma_{n_0}\ltimes \Delta$ defined by $\sigma\in\Sigma_{n_0}$ acting on $\tau\in\Delta$ by $\tau\mapsto \tau^{\chi(\sigma)}$ (cf. \eqref{eq:para:notation-kummer-relation}). Moreover, there is an open embedding of topological groups
	\begin{align}
		(\log\circ\chi,\xi):\Gamma_{n_0}=\Sigma_{n_0}\ltimes \Delta\longrightarrow \bb{Z}_p\ltimes\bb{Z}_p^d 
	\end{align}
	where $\bb{Z}_p\ltimes\bb{Z}_p^d$ is the semi-direct product of $\bb{Z}_p$ acting on $\bb{Z}_p^d$ by multiplication. It induces an isomorphism of $\bb{Q}_p$-linear Lie algebras
	\begin{align}\label{eq:para:gamma-basis-1}
		\lie(\Gamma)\iso \lie(\bb{Z}_p\ltimes\bb{Z}_p^d).
	\end{align}
	Let $\partial_i \in\lie(\Gamma)$ be the image of $(0,\dots,1,\dots,0)\in\bb{Z}_p\ltimes\bb{Z}_p^d$ (where $1$ appears at the $i$-th component) via the logarithm map of $\bb{Z}_p\ltimes\bb{Z}_p^d$ and \eqref{eq:para:gamma-basis-1}. We deduce from \eqref{eq:lie-str-gamma} that for any $1\leq i,j\leq d$,
	\begin{align}
		[\partial_0,\partial_i]=\partial_i\quad\trm{ and }\quad[\partial_i,\partial_j]=0,
	\end{align}
	and we deduce from \eqref{eq:lie-sequence-gamma} that $\partial_0,\partial_1,\dots,\partial_d$ form a $\bb{Q}_p$-basis of $\lie(\Gamma)$, which we call the \emph{standard basis}. Moreover, if we extend $\xi_1,\dots,\xi_d$ to $\bb{Q}_p$-linear forms on $\lie(\Delta)$, then we see that they form a dual basis of $\partial_1,\dots,\partial_d$.

	Consider an object $W$ of $\repnpr(\Gamma,\widetilde{A}_{\infty,\underline{\infty}}[1/p])$ and the canonical Lie algebra action \eqref{eq:gamma-lie-action} $\varphi:\lie(\Gamma)\to \mrm{End}_{\widetilde{A}_{\infty,\underline{\infty}}[1/p]}(W)$ induced by the infinitesimal action of $\Gamma$ on $W$. For any $g\in G$, we set
	\begin{align}\label{eq:para:gamma-basis-chi}
		\varphi^{\chi}_g=\log(\chi(g))\varphi_{\partial_0},
	\end{align} 
	which defines a continuous group homomorphism $\varphi^{\chi}:G\to \mrm{End}_{\widetilde{A}_{\infty,\underline{\infty}}[1/p]}(W)$ factoring through $\Gamma$ with $\varphi^{\chi}|_{\Sigma_{0,\underline{\infty}}}=\varphi|_{\Sigma_{0,\underline{\infty}}}$ and $\varphi^{\chi}|_\Delta=0$. We also set
	\begin{align}\label{eq:para:gamma-basis-xi}
		\varphi^{\xi}_g=\sum_{i=1}^d\xi_i(g)\varphi_{\partial_i},
	\end{align}
	which defines a continuous $1$-cocycle $\varphi^{\xi}:G\to \mrm{End}_{\widetilde{A}_{\infty,\underline{\infty}}[1/p]}(W)$ factoring through $\Gamma$ with $\varphi^{\xi}|_\Delta=\varphi|_{\Delta}$ and $\varphi^{\xi}|_{\Sigma_{0,\underline{\infty}}}=0$.
\end{mypara}

\begin{mylem}\label{lem:varphi-decomposition}
	Under the assumptions in {\rm\ref{para:gamma-basis}}, for any $g\in \Gamma$ with $\log(\chi(g))\neq 0$, we have
	\begin{align}\label{eq:lem:varphi-decomposition-1}
		\log_\Gamma(g)=\log(\chi(g))\partial_0+\frac{\log(\chi(g))}{\chi(g)-1}\sum_{i=1}^d\xi_i(g)\partial_i \in \lie(\Gamma).
	\end{align}
	In particular, for any object $W$ of $\repnpr(\Gamma,\widetilde{A}_{\infty,\underline{\infty}}[1/p])$, we have
	\begin{align}\label{eq:lem:varphi-decomposition-2}
		\varphi_g=\varphi^{\chi}_g+\frac{\log(\chi(g))}{\chi(g)-1}\varphi^{\xi}_g
	\end{align}
	as $\widetilde{A}_{\infty,\underline{\infty}}[1/p]$-linear endomorphisms of $W$.
\end{mylem}
\begin{proof}
	Since for any $r\geq 1$ we have $\log(\chi(g^r))=r\log(\chi(g))$ and $\xi(g^r)/(\chi(g^r)-1)=\xi(g)/(\chi(g)-1)$ by \eqref{eq:cont-cocycle-cond}, it suffices to prove \eqref{eq:lem:varphi-decomposition-1} for $g^r$. Thus, we may assume that $g\in \Gamma_{n_0}$ and let $g=\tau\sigma$ be the unique decomposition for some $\tau\in\Delta$ and $\sigma\in \Sigma_{n_0,\underline{\infty}}$. Since $\log(\chi(g))=\log(\chi(\sigma))$ and $\xi(g)=\xi(\tau)$, we have $\log(\chi(g))\partial_0=\log_\Gamma(\sigma)$ and $\sum_{i=1}^d\xi_i(g)\partial_i=\log_\Gamma(\tau)$. It remains to check that in $\lie(\Gamma)$, we have
	\begin{align}
		\log_\Gamma(g)=\log_\Gamma(\sigma)+\frac{\log(\chi(g))}{\chi(g)-1}\log_\Gamma(\tau).
	\end{align}
	By iteratively using the identity $\sigma\tau\sigma^{-1}=\tau^{\chi(\sigma)}$, we get $g^{p^n}=\tau^{\frac{\chi(\sigma)^{p^n}-1}{\chi(\sigma)-1}}\sigma^{p^n}$. After enlarging $n_0$, we may assume that $\Gamma_{n_0}$ is contained in a uniform pro-$p$ open subgroup $\Gamma'$ of $\Gamma$. Thus, by \eqref{eq:para:L_G-add},
	\begin{align}
		g+_{\Gamma'}\sigma^{-1}=\lim_{n\to\infty}(g^{p^n}\sigma^{-p^n})^{p^{-n}}=\lim_{n\to\infty}\tau^{\frac{\chi(\sigma)^{p^n}-1}{p^n(\chi(\sigma)-1)}}=\tau^{\frac{\log(\chi(g))}{\chi(g)-1}},
	\end{align}
	which completes the proof.
\end{proof}

\section{Revisiting Brinon's Generalization of Sen's Theory after Tsuji}\label{sec:brinon}
In this section, we revisit Brinon's generalization \cite{brinon2003sen} of Sen's theory following Tsuji \cite[\textsection15]{tsuji2018localsimpson}. More precisely, we establish a $p$-adic Simpson correspondence over a complete discrete valuation field of mixed characteristic (cf. \ref{thm:simpson-K}). Then, we give a canonical definition of Sen operators, which does not depend on choosing a $p$-basis of the residue field (and its $p$-power roots) (cf. \ref{defn:sen-brinon-operator}).

\begin{mypara}\label{para:notation-K}
	We use the following notation in this section. Let $K$ be a complete discrete valuation field of characteristic $0$ whose residue field $\kappa$ is of characteristic $p>0$ such that $[\kappa:\kappa^p]=p^d<\infty$ (i.e. $\kappa$ admits a finite $p$-basis, cf. \cite[21.1.9]{ega4-1}). We fix an algebraic closure $\overline{K}$ of $K$, and denote by $\widehat{\overline{K}}$ its $p$-adic completion. Let $t_1,\dots, t_d$ be $d$ elements of $\ca{O}_K^\times$ with compatible systems of $p$-power roots $(t_{1,p^n})_{n\in \bb{N}}, \cdots,(t_{d,p^n})_{n\in \bb{N}}$ in $\ca{O}_{\overline{K}}^\times$ such that the images of $t_1,\dots,t_d$ in $\kappa$ form a $p$-basis. We consider the Kummer tower $(\ca{O}_{K_{n,\underline{m}}})_{(n,\underline{m})\in \bb{N}^{1+d}}$ of $\ca{O}_K$ defined by $\zeta_{p^n},t_{1,p^n}, \dots,t_{d,p^n}$ (\ref{defn:kummer-tower}). We take again the notation in \ref{para:notation-kummer}.
	\begin{align}
		\xymatrix{
			\overline{K}&\\
			K_{\infty,\underline{\infty}}\ar[u]&\\
			K_\infty\ar[u]^-{\Delta}\ar@/^2pc/[uu]^-{H}&K\ar[l]^-{\Sigma}\ar[lu]_-{\Gamma}\ar@/_1pc/[luu]_{G}
		}
	\end{align}
\end{mypara}

\begin{mylem}\label{lem:cohen}
	There exists a complete discrete valuation subfield $K'$ of $K$ with $\ca{O}_{K'}/p\ca{O}_{K'}=\kappa$ such that $K$ is a totally ramified finite extension of $K'$ and that $t_1,\dots, t_d\in \ca{O}_{K'}$.
\end{mylem}
\begin{proof}
	Let $\pi$ be a uniformizer of $K$. By Cohen structure theorem \cite[19.8.8]{ega4-1}, there exists a complete discrete valuation ring $R$ extension of $\bb{Z}_p$ with a local injective homomorphism $f:R\to \ca{O}_K$ which induces an isomorphism $f_1:R/pR\iso\ca{O}_K/\pi\ca{O}_K=\kappa$. We take $s_1,\dots,s_d\in R$ lifting the images of $t_1,\dots,t_d\in\ca{O}_K$ in $\kappa$ respectively. We claim that it suffices to find a series of homomorphisms $(f_n:R/p^nR\to \ca{O}_K/\pi^n\ca{O}_K)_{n\geq 2}$ such that $f_n(s_i)=t_i$ and $f_n$ lifts $f_{n-1}$. Indeed, this series defines a homomorphism $f_\infty:R\to \ca{O}_K$ by taking limit on $n$, which sends $s_i$ to $t_i$ and identifies the residue fields. Thus, $f_\infty$ is finite (\cite[\href{https://stacks.math.columbia.edu/tag/031D}{031D}]{stacks-project}) and thus a totally ramified extension of discrete valuation rings. The claim follows.
	
	We construct $(f_n)_{n\geq 2}$ inductively. Suppose that we have constructed $f_{n-1}$. We fix a lifting $\widetilde{f}_{n-1}:R/p^nR\to \ca{O}_K/\pi^n\ca{O}_K$ of $f_{n-1}$, and consider the commutative diagram
	\begin{align}
		\xymatrix{
			\ca{O}_K/\pi^n\ca{O}_K\ar[r]&\ca{O}_K/\pi^{n-1}\ca{O}_K\\
			\bb{Z}/p^n\bb{Z}\ar[r]\ar[u]&R/p^nR\ar[u]_-{f_{n-1}}\ar[ul]|-{\widetilde{f}_{n-1}}
		}
	\end{align}
	There is a map
	\begin{align}
		\ho_R(\Omega^1_{(R/p^nR)/(\bb{Z}/p^n\bb{Z})},\pi^{n-1}\ca{O}_K/\pi^n\ca{O}_K)\longrightarrow \ho_{\bb{Z}\trm{-Alg}}(R/p^nR,\ca{O}_K/\pi^n\ca{O}_K)
	\end{align}
	sending $D$ to $\widetilde{f}_{n-1}+D\circ \df_{R/p^nR}$. Recall that $\Omega^1_{(R/p^nR)/(\bb{Z}/p^n\bb{Z})}$ is a finite free $R/p^n$-module with basis $\df s_1,\dots,\df s_d$ (\cite[3.2]{he2021faltingsext}). We can take $D$ sending $\df s_i$ to $t_i-\widetilde{f}_{n-1}(s_i)$, as $t_i=f_{n-1}(s_i)$ by the induction hypothesis. Taking $f_n=\widetilde{f}_{n-1}+D\circ \df_{R/p^nR}$, we see that $t_i=f_{n}(s_i)$ and $f_n$ lifts $f_{n-1}$, which completes the induction.
\end{proof}

\begin{mylem}\label{lem:cohen-str}
	Let $K'$ be a subfield of $K$ as in {\rm\ref{lem:cohen}}, $(\ca{O}_{K'_{n,\underline{m}}})_{(n,\underline{m})\in \bb{N}^{1+d}}$ the Kummer tower of $\ca{O}_{K'}$ defined by $\zeta_{p^n},t_{1,p^n}, \dots,t_{d,p^n}$ ({\rm\ref{defn:kummer-tower}}).
	\begin{enumerate}
		\renewcommand{\labelenumi}{{\rm(\theenumi)}}
		\item The extension $K'_n$ over $K'$ is totally ramified, and $\ca{O}_{K'_n}= \ca{O}_{K'}[X_0]/(\frac{X_0^{p^n}-1}{X_0^{p^{n-1}}-1})=\ca{O}_{K'}[\zeta_{p^n}]$.
		\item The extension $K'_{n,\underline{m}}$ over $K'_n$ is weakly unramified, and $\ca{O}_{K'_{n,\underline{m}}}=\ca{O}_{K'_n}[X_1,\dots,X_d]/(X_1^{p^{m_1}}-t_1,\dots,X_d^{p^{m_d}}-t_d)=\ca{O}_{K'}[\zeta_{p^n},t_{1,p^{m_1}},\dots,t_{d,p^{m_d}}]$, where $\underline{m}=(m_1,\dots,m_d)$.
		\item The cyclotomic character \eqref{eq:cycl-char} $\chi:\gal(K'_{\infty,\underline{\infty}}/K'_{0,\underline{\infty}})\to \bb{Z}_p^\times$ is an isomorphism.
		\item The continuous homomorphism \eqref{eq:cont-cocycle} $\xi:\gal(K'_{\infty,\underline{\infty}}/K'_\infty)\to \bb{Z}_p^d$ is an isomorphism.\label{item:lem:cohen-str-4}
	\end{enumerate}
\end{mylem}
\begin{proof}
	The first two assertions follow from the fact that $p$ is a uniformizer of the complete discrete valuation field $K'$ and that $t_1,\dots,t_d$ form a $p$-basis of its residue field (cf. \cite[\Luoma{1}.\textsection6]{serre1979local}). One can deduce easily the last two assertions from the first two by the arguments of \ref{prop:kummer-tower-lem}.
\end{proof}

\begin{myprop}\label{prop:brinon-kummer-tower-adequate}
	The Kummer tower $(\ca{O}_{K_{n,\underline{m}}})_{(n,\underline{m})\in \bb{N}^{1+d}}$ satisfies the condition {\rm\ref{prop:kummer-tower-lem}.(\ref{item:prop:kummer-tower-lem-1})}.
\end{myprop}
\begin{proof}
	We take $K'$ as in \ref{lem:cohen}. Since $K$ is a finite extension of $K'$, the Galois group $\Delta$ identifies with an open subgroup of $\gal(K'_{\infty,\underline{\infty}}/K'_\infty)$. The conclusion follows from \ref{lem:cohen-str}.(\ref{item:lem:cohen-str-4}).
\end{proof}
\begin{myrem}\label{rem:brinon-kummer-tower-adequate}
	Let $n_0\in\bb{N}_{\geq 2}$ such that $K\cap K'_{\infty,\underline{\infty}}\subseteq K'_{n_0,\underline{n_0}}$. Then, for any $(n,\underline{m})\in\bb{N}_{\geq n_0}^{1+d}$, the natural map $\gal(K_{\infty,\underline{\infty}}/K_{n,\underline{m}})\to \gal(K'_{\infty,\underline{\infty}}/K'_{n,\underline{m}})$ is an isomorphism by Galois theory. In particular, the conclusion of {\rm\ref{prop:kummer-tower-lem}.(\ref{item:prop:kummer-tower-lem-2})} for $(\ca{O}_{K_{n,\underline{m}}})_{(n,\underline{m})\in \bb{N}^{1+d}}$ holds for any $(n,\underline{m})\in\bb{N}_{\geq n_0}^{1+d}$ by \ref{lem:cohen-str}.
\end{myrem}

\begin{mypara}
	Recall that the $\ca{O}_K$-module $\widehat{\Omega}^1_{\ca{O}_K}$ (defined in \ref{para:notation-Tate-mod}) is finitely generated whose free part has rank $d$, and that $\widehat{\Omega}^1_{\ca{O}_K}[1/p]$ admits a $K$-basis $\df\log(t_1),\dots,\df\log(t_d)$ (cf. \cite[3.3]{he2021faltingsext}). For simplicity, we set (cf. \ref{para:notation-Tate-mod})
	\begin{align}
		\scr{E}_{\ca{O}_K}=V_p(\Omega^1_{\ca{O}_{\overline{K}}/\ca{O}_K})=\plim_{x\mapsto px}\Omega^1_{\ca{O}_{\overline{K}}/\ca{O}_K}.
	\end{align}
	It is a $\widehat{\overline{K}}$-module as $\Omega^1_{\ca{O}_{\overline{K}}/\ca{O}_K}$ is $p$-primary torsion (\cite[4.2]{he2021faltingsext}), and endowed with the natural action of $G$. For any $(s_{p^n})_{n\in\bb{N}}\in V_p(\overline{K}^{\times})$, we take $k\in \bb{N}$ sufficiently large such that $p^ks_1,p^ks_1^{-1}\in \ca{O}_{\overline{K}}$ (thus $p^ks_{p^n}^{\pm 1}\in \ca{O}_{\overline{K}}$). The element $p^{-2k}(p^k s_{p^n}^{-1}\df (p^k s_{p^n}))_{n\in\bb{N}}\in \scr{E}_{\ca{O}_K}$ does not depend on the choice of $k$, which we denote by $(\df \log(s_{p^n}))_{n\in\bb{N}}$. Similarly, we define $\df\log(s)\in \widehat{\Omega}^1_{\ca{O}_K}[1/p]$ for any $s\in K^\times$.
\end{mypara}

\begin{mythm}[{\cite[4.4]{he2021faltingsext}}]\label{thm:fal-ext}
	There is a canonical $G$-equivariant exact sequence of $\widehat{\overline{K}}$-modules, called the \emph{Faltings extension} of $\ca{O}_K$,
	\begin{align}\label{eq:fal-ext}
		0\longrightarrow \widehat{\overline{K}}(1)\stackrel{\iota}{\longrightarrow} \scr{E}_{\ca{O}_K}\stackrel{\jmath}{\longrightarrow}\widehat{\overline{K}}\otimes_{\ca{O}_K}\widehat{\Omega}^1_{\ca{O}_K}\longrightarrow 0,
	\end{align}
	satisfying the following properties:
	\begin{enumerate}
		\renewcommand{\labelenumi}{{\rm(\theenumi)}}
		\item We have $\iota((\zeta_{p^n})_{n\in\bb{N}})=(\df\log(\zeta_{p^n}))_{n\in\bb{N}}$.\label{item:thm:fal-ext-1}
		\item For any $s\in K^\times$ and any compatible system of $p$-power roots $(s_{p^n})_{n\in\bb{N}}$ of $s$ in $\overline{K}$, $\jmath(\df \log(s_{p^n}))_{n\in\bb{N}})=\df \log(s)$.\label{item:thm:fal-ext-2}
		\item The $\widehat{\overline{K}}$-linear surjection $\jmath$ admits a section sending $\df\log(t_i)$ to $(\df\log(t_{i,p^n}))_{n\in\bb{N}}$ for any $1\leq i\leq d$.\label{item:thm:fal-ext-3}
	\end{enumerate}
	In particular, $\scr{E}_{\ca{O}_K}$ is a finite free $\widehat{\overline{K}}$-module with basis $\{(\df\log(t_{i,p^n}))_{n\in\bb{N}}\}_{0\leq i\leq d}$, where $t_{0,p^n}=\zeta_{p^n}$, on which $G$ acts continuously with respect to the canonical topology (where $\widehat{\overline{K}}$ is endowed with the $p$-adic topology defined by its valuation ring).
\end{mythm}
\begin{proof}
	The sequence \eqref{eq:fal-ext} is constructed in \cite[4.4]{he2021faltingsext} and (\ref{item:thm:fal-ext-1}), (\ref{item:thm:fal-ext-3}) are proved there. Notice that (\ref{item:thm:fal-ext-2}) follows from the constructing process \cite[(4.4.5)]{he2021faltingsext} (see also \ref{prop:B_nm-fal-ext} for a detailed proof). For the ``in particular'' part, it remains to check the continuity of the $G$-action. We set $\alpha_i=(\df\log(t_{i,p^n}))_{n\in\bb{N}}\in\scr{E}_{\ca{O}_K}$ for any $0\leq i\leq d$. For any $g\in G$ and $1\leq i\leq d$, we have
	\begin{align}\label{eq:fal-ext-action}
		g(\alpha_0)=\chi(g)\alpha_0\quad\trm{ and }\quad g(\alpha_i)=\xi_i(g)\alpha_0+\alpha_i,
	\end{align}
	where $\chi:G\to \bb{Z}_p^\times$ is the cyclotomic character \eqref{eq:cycl-char} and $\xi=(\xi_1,\dots,\xi_d):G\to \bb{Z}_p^d$ is the continuous $1$-cocycle \eqref{eq:cont-cocycle}. The elements $\alpha_0,\dots,\alpha_d$ generate a finite free $\ca{O}_{\widehat{\overline{K}}}$-submodule $\scr{E}_{\ca{O}_K}^+$ of $\scr{E}_{\ca{O}_K}$ which is $G$-stable. For any $r\in\bb{N}$, each element of $\scr{E}_{\ca{O}_K}/p^r\scr{E}_{\ca{O}_K}^+=\oplus_{i=0}^d(\overline{K}/p^r\ca{O}_{\overline{K}})\alpha_i$ is fixed by an open subgroup of $G$, which implies that the map $G\times(\scr{E}_{\ca{O}_K}/p^r\scr{E}_{\ca{O}_K}^+)\to \scr{E}_{\ca{O}_K}/p^r\scr{E}_{\ca{O}_K}^+$ (given by the action of $G$) is continuous with respect to the discrete topology on $\scr{E}_{\ca{O}_K}/p^r\scr{E}_{\ca{O}_K}^+$. Taking inverse limit on $r\in\bb{N}$, we see that $G\times \scr{E}_{\ca{O}_K}\to \scr{E}_{\ca{O}_K}$ is continuous with respect to the limit topology on $\scr{E}_{\ca{O}_K}$, which indeed coincides with the canonical topology (\ref{para:canonical-top}).
\end{proof}

\begin{myrem}\label{rem:fal-ext}
	The Faltings extension \eqref{eq:fal-ext} is functorial in the following sense: let $K'$ be a complete discrete valuation field extension of $K$ whose residue field admits a finite $p$-basis, $\overline{K}\to \overline{K'}$ a compatible embedding of the algebraic closures of $K$ and $K'$. It defines a natural map $\Omega^1_{\ca{O}_{\overline{K}}/\ca{O}_K}\to \Omega^1_{\ca{O}_{\overline{K'}}/\ca{O}_{K'}}$ by pullback and thus a natural morphism of exact sequences
	\begin{align}\label{diam:rem:fal-ext}
		\xymatrix{
			0\ar[r]& \widehat{\overline{K}}(1)\ar[r]^-{\iota}\ar[d]& \scr{E}_{\ca{O}_K}\ar[r]^-{\jmath}\ar[d]&\widehat{\overline{K}}\otimes_{\ca{O}_K}\widehat{\Omega}^1_{\ca{O}_K}\ar[r]\ar[d] &0\\
			0\ar[r]& \widehat{\overline{K'}}(1)\ar[r]^-{\iota}& \scr{E}_{\ca{O}_{K'}}\ar[r]^-{\jmath}&\widehat{\overline{K'}}\otimes_{\ca{O}_{K'}}\widehat{\Omega}^1_{\ca{O}_{K'}}\ar[r] &0
		}
	\end{align}
	Moreover, if $K'$ is a finite extension of $K$, then $K'\otimes_{\ca{O}_K}\widehat{\Omega}^1_{\ca{O}_K}\to K'\otimes_{\ca{O}_{K'}}\widehat{\Omega}^1_{\ca{O}_{K'}}$ is an isomorphism (cf. the proof of \cite[3.3]{he2021faltingsext}). Thus, the vertical maps in \eqref{diam:rem:fal-ext} are isomorphisms.
\end{myrem}

\begin{mycor}\label{cor:fal-ext-connect}
	The connecting map of the Faltings extension \eqref{eq:fal-ext} induces a canonical $\widehat{K_\infty}$-linear isomorphism
	\begin{align}\label{eq:fal-ext-connect-2}
		\widehat{K_\infty}\otimes_{\ca{O}_K}\widehat{\Omega}^1_{\ca{O}_K}\iso H^1(H,\widehat{\overline{K}}(1)),
	\end{align}
	sending $\df\log(t_i)$ to $\xi_i\otimes\zeta$, where $H^1$ denotes the continuous group cohomology, $\zeta=(\zeta_{p^n})_{n\in\bb{N}}\in \bb{Z}_p(1)$ and $\xi=(\xi_1,\dots,\xi_d):H\to \bb{Z}_p^d$ is the continuous $1$-cocycle \eqref{eq:cont-cocycle}.
\end{mycor}
\begin{proof}
	We remark that the Faltings extension \eqref{eq:fal-ext} is an exact sequence of finite projective $\widehat{\overline{K}}$-representations of $G$, which admits a continuous splitting (not $G$-equivariant), so that we obtain a long exact sequence of continuous group cohomologies (cf. \cite[\textsection 2]{tate1976galcoh}). The corollary follows from Hyodo's computation of $H^1(H,\widehat{\overline{K}}(1))$ (cf. \cite[2-1, 5-1]{hyodo1986hodge}). We will give a detailed proof in \ref{app:prop:fal-ext-connect}.
\end{proof}

\begin{myrem}\label{rem:fal-ext-coh}
	A similar result for $H^1(G,\widehat{\overline{K}}(1))$ is given in \cite[4.5]{he2021faltingsext}, relying on Hyodo's computation.
\end{myrem}

\begin{mypara}\label{para:hyodo-ring}
	We set
	\begin{align}\label{eq:hyodo-ring}
		\hyodoring=\colim_{n\in\bb{N}} \mrm{Sym}^n_{\widehat{\overline{K}}} (\scr{E}_{\ca{O}_K}(-1)),
	\end{align}
	where $\mrm{Sym}^n$ is taking the homogeneous part of degree $n$ of the symmetric algebra, and the transition map $\mrm{Sym}^n\to \mrm{Sym}^{n+1}$ is defined by sending $[x_1\otimes\cdots\otimes x_n]$ to $[1\otimes x_1\otimes\cdots\otimes x_n]$ (where $1$ denotes the image of $1\in\widehat{\overline{K}}$ via \eqref{eq:fal-ext}). It is a $\widehat{\overline{K}}$-module endowed with the natural action of $G$. There is a natural $G$-equivariant exact sequence of $\widehat{\overline{K}}$-modules induced by \eqref{eq:fal-ext},
	\begin{align}\label{eq:para:hyodo-ring}
		0\to \mrm{Sym}^{n-1}_{\widehat{\overline{K}}} (\scr{E}_{\ca{O}_K}(-1))\to \mrm{Sym}^n_{\widehat{\overline{K}}} (\scr{E}_{\ca{O}_K}(-1))\to \widehat{\overline{K}}\otimes_{\ca{O}_K}(\mrm{Sym}^n_{\ca{O}_K}\widehat{\Omega}^1_{\ca{O}_K})(-n)\to 0.
	\end{align}
	The $\widehat{\overline{K}}$-module $\hyodoring$ admits a natural $\widehat{\overline{K}}$-algebra structure induced by the multiplication morphisms $\mrm{Sym}^n\otimes \mrm{Sym}^m\to \mrm{Sym}^{n+m}$.
\end{mypara}

\begin{mycor}\label{cor:hyodo-ring}
	We set $\zeta=(\zeta_{p^n})_{n\in\bb{N}}\in \bb{Z}_p(1)$ and denote by $\zeta^{-1}\in \bb{Z}_p(-1)=\ho_{\bb{Z}_p}(\bb{Z}_p(1),\bb{Z}_p)$ the dual basis of $\zeta$.
	\begin{enumerate}
		\renewcommand{\labelenumi}{{\rm(\theenumi)}}
		\item  There is an isomorphism of $\widehat{\overline{K}}$-algebras, 
		\begin{align}\label{eq:hyodo-ring-polynomial}
			\widehat{\overline{K}}[T_1,\dots,T_d]\iso \hyodoring,
		\end{align}
		sending the variable $T_i$ to $(\df\log(t_{i,p^n}))_{n\in\bb{N}}\otimes\zeta^{-1}$ for any $1\leq i\leq d$. \label{item:cor:hyodo-ring-1}
		\item If we endow $\widehat{\overline{K}}[T_1,\dots,T_d]$ with the semi-linear $G$-action by transport of structure via \eqref{eq:hyodo-ring-polynomial}, then for any $g\in G$ and $1\leq i\leq d$, we have
		\begin{align}
			g(T_i)=\chi(g)^{-1}(\xi_i(g)+T_i),
		\end{align}
		where $\chi:G\to \bb{Z}_p^\times$ is the cyclotomic character \eqref{eq:cycl-char} and $\xi=(\xi_1,\dots,\xi_d):G\to \bb{Z}_p^d$ is the continuous $1$-cocycle \eqref{eq:cont-cocycle}. In particular, $G$ acts continuously on $\mrm{Sym}^n_{\widehat{\overline{K}}} (\scr{E}_{\ca{O}_K}(-1))$ with respect to the canonical topology for any $n\in\bb{N}$. \label{item:cor:hyodo-ring-2}
		\item The canonical map $\jmath$ in \eqref{eq:fal-ext} induces a canonical isomorphism of $\hyodoring$-modules,
		\begin{align}\label{eq:hyodo-ring-diff}
			\Omega^1_{\hyodoring/\widehat{\overline{K}}}\iso \hyodoring\otimes_{\ca{O}_K}\widehat{\Omega}^1_{\ca{O}_K}(-1),
		\end{align}
		and the universal differential map $\df_{\hyodoring}:\hyodoring\to \hyodoring\otimes_{\ca{O}_K}\widehat{\Omega}^1_{\ca{O}_K}(-1)$ sends $T_i$ to $\df\log(t_i)\otimes \zeta^{-1}$ for any $1\leq i\leq d$. \label{item:cor:hyodo-ring-3}
	\end{enumerate}
\end{mycor}
\begin{proof}
	It follows directly from \ref{thm:fal-ext} and its arguments.
\end{proof}

\begin{mydefn}[{cf. \cite[\textsection1]{hyodo1989variation}, \cite[\Luoma{2}.15]{abbes2016p}, \cite[\textsection15]{tsuji2018localsimpson}}]\label{defn:hyodo-ring}
	The $\widehat{\overline{K}}$-algebra $\hyodoring$ constructed in \eqref{eq:hyodo-ring} is called the \emph{Hyodo ring} of $\ca{O}_K$.
\end{mydefn}

\begin{mycor}\label{cor:hyodo-ring-coh}
	We have
	\begin{align}
		\colim_{n\in\bb{N}}H^q(H,\mrm{Sym}^n_{\widehat{\overline{K}}} (\scr{E}_{\ca{O}_K}(-1))) = \left\{ \begin{array}{ll}
			\widehat{K_\infty} & \textrm{if $q=0$,}\\
			0 & \textrm{otherwise,}
		\end{array} \right.
	\end{align}
	where $H^q$ denotes the continuous group cohomology, and $\mrm{Sym}^n_{\widehat{\overline{K}}} (\scr{E}_{\ca{O}_K}(-1))$ is endowed with the canonical topology as a finite-dimensional $\widehat{\overline{K}}$-module. In particular, $(\hyodoring)^H=\widehat{K_\infty}$.
\end{mycor}
\begin{proof}
	It follows from the argument of \cite[(1.2.2)]{hyodo1989variation}, which relies on the cohomological property \ref{cor:fal-ext-connect} of the Faltings extension. We will give a detailed proof in \ref{app:cor:hyodo-ring-coh}.
\end{proof}
\begin{myrem}\label{rem:hyodo-ring-coh}
	One can also obtain a similar result for $H^q(G,\mrm{Sym}^n_{\widehat{\overline{K}}} (\scr{E}_{\ca{O}_K}(-1)))$ by the argument of \cite[(1.2.1)]{hyodo1989variation}.
\end{myrem}

\begin{mypara}\label{para:fal-ext-dual}
	Taking a Tate twist of the dual of the Faltings extension \eqref{eq:fal-ext} of $\ca{O}_K$, we obtain a canonical exact sequence of finite projective $\widehat{\overline{K}}$-representations of $G$,
	\begin{align}\label{eq:fal-ext-dual}
		0\longrightarrow \ho_{\ca{O}_K}(\widehat{\Omega}^1_{\ca{O}_K}(-1),\widehat{\overline{K}})\stackrel{\jmath^*}{\longrightarrow} \scr{E}^*_{\ca{O}_K}(1)\stackrel{\iota^*}{\longrightarrow}\widehat{\overline{K}}\longrightarrow 0
	\end{align} 
	where $\scr{E}^*_{\ca{O}_K}=\ho_{\widehat{\overline{K}}}(\scr{E}_{\ca{O}_K},\widehat{\overline{K}})$. There is a canonical $G$-equivariant $\widehat{\overline{K}}$-linear Lie algebra structure on $\scr{E}^*_{\ca{O}_K}(1)$ associated to the linear form $\iota^*$, defined by the Lie bracket for any $f_1,f_2\in \scr{E}^*_{\ca{O}_K}(1)$,
	\begin{align}
		[f_1,f_2]=\iota^*(f_1)f_2-\iota^*(f_2)f_1.
	\end{align}
	Thus, $\ho_{\ca{O}_K}(\widehat{\Omega}^1_{\ca{O}_K}(-1),\widehat{\overline{K}})$ is a Lie ideal of $\scr{E}^*_{\ca{O}_K}(1)$, and $\widehat{\overline{K}}$ is the quotient, and the induced Lie algebra structures on them are trivial. Any $\widehat{\overline{K}}$-linear splitting of \eqref{eq:fal-ext-dual} identifies $\scr{E}^*_{\ca{O}_K}(1)$ with the semi-direct product of Lie algebras of $\widehat{\overline{K}}$ acting on $\ho_{\ca{O}_K}(\widehat{\Omega}^1_{\ca{O}_K}(-1),\widehat{\overline{K}})$ by multiplication. Let $\{T_i=(\df\log(t_{i,p^n}))_{n\in\bb{N}}\otimes\zeta^{-1}\}_{0\leq i\leq d}$ (where $t_{0,p^n}=\zeta_{p^n}$) denote the basis of $\scr{E}_{\ca{O}_K}(-1)$, and let $\{T_i^*\}_{0\leq i\leq d}$ be the dual basis of $\scr{E}^*_{\ca{O}_K}(1)$. Then, we see that the Lie bracket on $\scr{E}^*_{\ca{O}_K}(1)$ is determined by
	\begin{align}
		[T_0^*,T_i^*]=T_i^*\quad\trm{ and }\quad [T_i^*,T_j^*]=0,
	\end{align}
	for any $1\leq i,j\leq d$. Indeed, this dual basis induces an isomorphism of $\widehat{\overline{K}}$-linear Lie algebras
	\begin{align}
		\widehat{\overline{K}}\otimes_{\bb{Q}_p}\lie(\bb{Z}_p\ltimes\bb{Z}_p^d)\iso \scr{E}^*_{\ca{O}_K}(1),\ 1\otimes\partial_i\mapsto T_i^*,
	\end{align}
	where $\{\partial_i\}_{0\leq i\leq d}$ is the standard basis of $\lie(\bb{Z}_p\ltimes\bb{Z}_p^d)$ (cf. \ref{para:gamma-basis}).
\end{mypara}

\begin{mythm}[{\cite[Th\'eor\`eme 1, 2]{brinon2003sen}, \cite[\textsection9]{ohkubo2011note}}]\label{thm:sen-brinon}
	The functor
	\begin{align}
		\repnpr(\Gamma,K_{\infty,\underline{\infty}})\longrightarrow \repnpr(G,\widehat{\overline{K}}),\ V\mapsto  \widehat{\overline{K}}\otimes_{K_{\infty,\underline{\infty}}}V
	\end{align}
	is an equivalence of categories.
\end{mythm}

\begin{myprop}[{\cite[14.16]{tsuji2018localsimpson}}]\label{prop:gamma-analytic}
	The functor (cf. {\rm\ref{defn:analytic}})
	\begin{align}\label{eq:3.6.1}
		\repnan{\Delta}(\Gamma,K_\infty)\longrightarrow \repnpr(\Gamma,K_{\infty,\underline{\infty}}),\ V\mapsto K_{\infty,\underline{\infty}}\otimes_{K_\infty} V
	\end{align}
	is an equivalence. 
\end{myprop}
\begin{proof}
	It follows from the proof of \cite[14.16]{tsuji2018localsimpson}. We note that the lemma \cite[14.17]{tsuji2018localsimpson} used in the proof holds by \ref{prop:operator-nilpotent}, and the lemma \cite[14.18]{tsuji2018localsimpson} holds since any finite field extension of $K$ is still a complete discrete valuation field.
\end{proof}

\begin{mylem}[{\cite[15.3.(2)]{tsuji2018localsimpson}}]\label{lem:G-Kcycl-fixed}
	For any object $V$ of $\repnpr(\Sigma,K_\infty)$, the $(\Sigma,K_\infty)$-finite part of $\widehat{K_\infty}\otimes_{K_\infty}V$ (see {\rm\ref{defn:repn}}) is $V$.
\end{mylem}
\begin{proof}
	It follows from the argument of \cite[15.3.(2)]{tsuji2018localsimpson} (cf. \ref{rem:tate-sen}).
\end{proof}

\begin{mypara}\label{para:notation-K-tilde}
	We shall give an explicit way in \ref{prop:special-functor} to construct Higgs bundles from representations, which generalizes \cite[15.1.(4)]{tsuji2018localsimpson}. Firstly, we introduce another Kummer tower more general than the one considered in \ref{para:notation-K}. We fix $e\in \bb{N}$. Let $\widetilde{t}_1,\dots,\widetilde{t}_e$ be elements of $K$ with compatible systems of $p$-power roots $(\widetilde{t}_{1,p^n})_{n\in\bb{N}},\dots,(\widetilde{t}_{e,p^n})_{n\in\bb{N}}$ in $\overline{K}$. Consider the Kummer tower $(\ca{O}_{K_{n,\undertilde{l}}})_{(n,\undertilde{l})\in \bb{N}^{1+e}}$ of $\ca{O}_K$ defined by $\zeta_{p^n},\widetilde{t}_{1,p^n},\dots,\widetilde{t}_{e,p^n}$. We take the notation in \ref{para:notation-kummer} for this Kummer tower by adding tildes.
	\begin{align}
		\xymatrix{
			\overline{K}&\\
			K_{\infty,\undertilde{\infty}}\ar[u]&\\
			K_\infty\ar[u]^-{\widetilde{\Delta}}\ar@/^2pc/[uu]^-{H}&K\ar[l]^-{\Sigma}\ar[lu]_-{\widetilde{\Gamma}}\ar@/_1pc/[luu]_{G}
		}
	\end{align}
	We have the continuous $1$-cocycle
	\begin{align}
		\widetilde{\xi}=(\widetilde{\xi}_1,\dots,\widetilde{\xi}_e):G\longrightarrow\bb{Z}_p^e,
	\end{align}
	describing the action of $G$ on $\widetilde{t}_{1,p^n},\dots,\widetilde{t}_{e,p^n}$, cf.  \eqref{eq:cont-cocycle}. We define $1+e$ elements in $\scr{E}_{\ca{O}_K}(-1)\subseteq \hyodoring$ by
	\begin{align}
		\widetilde{T}_0=1,\ \widetilde{T}_1=(\df\log(\widetilde{t}_{1,p^n}))_n\otimes\zeta^{-1},\ \cdots,\ \widetilde{T}_e=(\df\log(\widetilde{t}_{e,p^n}))_n\otimes\zeta^{-1}.
	\end{align}
	Similarly to \ref{cor:hyodo-ring}.(\ref{item:cor:hyodo-ring-2}), for any $g\in G$ and $1\leq i\leq e$, we have
	\begin{align}\label{eq:para:special-functor}
		g(\widetilde{T}_i)=\chi(g)^{-1}(\widetilde{\xi}_i(g)+\widetilde{T}_i).
	\end{align}
	We remark that $\df_{\hyodoring}(\widetilde{T}_i)=\df\log(\widetilde{t}_i)\otimes\zeta^{-1}\in K\otimes_{\ca{O}_K}\widehat{\Omega}^1_{\ca{O}_K}(-1)$ by \ref{thm:fal-ext}.(\ref{item:thm:fal-ext-2}). 
\end{mypara}

\begin{mylem}\label{lem:diff-lie-alg}
	With the notation in {\rm\ref{para:notation-K-tilde}}, we write \begin{align}\label{eq:lem:diff-lie-alg}
		(\widetilde{T}_1,\dots,\widetilde{T}_e)=(T_1,\dots,T_d)A+B
	\end{align}
	as elements of $\scr{E}_{\ca{O}_K}(-1)\subseteq \hyodoring$, where $A=(a_{ij})\in\mrm{M}_{d\times e}(\widehat{\overline{K}}),B=(b_j)\in \mrm{M}_{1\times e}(\widehat{\overline{K}})$. Then, $A\in \mrm{M}_{d\times e}(K)$, and we have
	\begin{align}\label{eq:lem:diff-lie-alg-2}
		(\widetilde{\xi}_1,\dots,\widetilde{\xi}_e)=(\xi_1,\dots,\xi_d)A
	\end{align}
	as vectors with value in the continuous group cohomology group $H^1(H,\widehat{\overline{K}})$. In particular, we have
	\begin{align}\label{eq:rank}
		\dim(\widetilde{\Delta})\geq \mrm{rank}(A).
	\end{align}
\end{mylem}
\begin{proof}
	Notice that $(\df_{\hyodoring}(\widetilde{T}_1),\dots,\df_{\hyodoring}(\widetilde{T}_e))=(\df_{\hyodoring}(T_1),\dots,\df_{\hyodoring}(T_d))A$ and that $\df_{\hyodoring}(T_1),\dots,\df_{\hyodoring}(T_d)$ form a basis of $K\otimes_{\ca{O}_K}\widehat{\Omega}^1_{\ca{O}_K}(-1)$. Thus, $A\in \mrm{M}_{d\times e}(K)$ as  $\df_{\hyodoring}(\widetilde{T}_i)=\df\log(\widetilde{t}_i)\otimes\zeta^{-1}\in K\otimes_{\ca{O}_K}\widehat{\Omega}^1_{\ca{O}_K}(-1)$. We act on \eqref{eq:lem:diff-lie-alg} by $g\in G$, then by \eqref{eq:para:special-functor},
	\begin{align}
		(\widetilde{\xi}_1(g),\dots,\widetilde{\xi}_e(g))=(\xi_1(g),\dots,\xi_d(g))A+\chi(g)g(B)-B.
	\end{align}
	Thus, \eqref{eq:lem:diff-lie-alg-2} follows from the fact that $\chi(H)=1$ and the map $H\to \widehat{\overline{K}}$ sending $g$ to $g(b_j)-b_j$ is a $1$-coboundry. In particular, the image of the composition of the natural maps
	\begin{align}
		\ho(\widetilde{\Delta},\bb{Q}_p)\to \ho(H,\bb{Q}_p)\to H^1(H,\widehat{\overline{K}})
	\end{align}
	contains $(\xi_1,\dots,\xi_d)A$. Since $\xi_1,\dots,\xi_d$ form a $\widehat{K_\infty}$-basis of $H^1(H,\widehat{\overline{K}})$ by \eqref{eq:fal-ext-connect-2}, we see that 
	$\dim(\widetilde{\Delta})\geq \mrm{rank}(A)$.
\end{proof}

\begin{myprop}\label{prop:rank}
	With the notation in {\rm\ref{para:notation-K-tilde}}, the following conditions are equivalent:
	\begin{enumerate}
		\renewcommand{\labelenumi}{{\rm(\theenumi)}}
		\item The group $\widetilde{\Delta}$ has dimension $e$, i.e. the Kummer tower $(\ca{O}_{K_{n,\undertilde{l}}})_{(n,\undertilde{l})\in \bb{N}^{1+e}}$ satisfies the condition {\rm\ref{prop:kummer-tower-lem}.(\ref{item:prop:kummer-tower-lem-1})}.\label{item:prop:rank-1}
		\item The $1+e$ elements $\widetilde{T}_0, \widetilde{T}_1,\cdots,\widetilde{T}_e$ of $\scr{E}_{\ca{O}_K}(-1)\subseteq \hyodoring$ are linearly independent over $\widehat{\overline{K}}$.\label{item:prop:rank-2}
		\item The $e$ elements $\df \widetilde{t}_1,\dots,\df \widetilde{t}_e$ of $\widehat{\Omega}^1_{\ca{O}_K}[1/p]$ are linearly independent over $K$.\label{item:prop:rank-3}
	\end{enumerate}
\end{myprop}
\begin{proof}
	It is clear that (\ref{item:prop:rank-2}) and (\ref{item:prop:rank-3}) are equivalent by the splitting of the Faltings extension defined in \ref{thm:fal-ext}.(\ref{item:thm:fal-ext-3}). We see that (\ref{item:prop:rank-3}) implies (\ref{item:prop:rank-1}) by \eqref{eq:rank}. It remains to check (\ref{item:prop:rank-1}) $\Rightarrow$ (\ref{item:prop:rank-2}).
	
	We set $\alpha_i=\widetilde{T}_i\otimes\zeta\in \scr{E}_{\ca{O}_K}$ ($0\leq i\leq e$). Let $l\in \bb{N}$ be the smallest integer such that $\alpha_{i_1},\dots,\alpha_{i_l}$ are linearly dependent for some $0\leq i_1<\cdots <i_l\leq e$. Assume that $l\geq 1$. We write $x_{i_1}\alpha_{i_1}+\cdots+x_{i_l}\alpha_{i_l}=0$ for some $x_{i_1},\dots, x_{i_l}\in \widehat{\overline{K}}\setminus\{0\}$. By \eqref{eq:para:special-functor}, for any $g\in G$ and $1\leq i\leq e$, we have
	\begin{align}\label{eq:lem:fal-ext-2}
		g(\alpha_0)=\chi(g)\alpha_0\quad\trm{ and }\quad g(\alpha_i)=\widetilde{\xi}_i(g)\alpha_0+\alpha_i.
	\end{align}
	 We consider the cases where $i_1=0$ and $i_1>0$ separately. 
	
	If $i_1=0$, then we may assume that $(i_1,i_2,\dots,i_l)=(0,1,\dots,l-1)$. Since $\alpha_0\neq 0$, we have $l>1$ and we may assume that $x_{l-1}=1$. Thus,
	\begin{align}
		0&=\sum_{i=0}^{l-1} x_i\alpha_i-g(\sum_{i=0}^{l-1} x_i\alpha_i)\\
		&=\left(x_0-\chi(g)g(x_0)-\sum_{i=1}^{l-1}\widetilde{\xi}_i(g)g(x_i)\right)\alpha_0+(x_1-g(x_1))\alpha_1+\cdots+(x_{l-2}-g(x_{l-2}))\alpha_{l-2}.\nonumber
	\end{align}
	By the minimality of $l$, we have $x_0=\chi(g)g(x_0)+\sum_{i=1}^{l-1}\widetilde{\xi}_i(g)g(x_i)$, $x_1=g(x_1)$, $\dots$ , $x_{l-2}=g(x_{l-2})$. Thus, $x_0=\chi(g)g(x_0)+\sum_{i=1}^{l-1}\widetilde{\xi}_i(g)x_i$, which is a contradiction as we can take $(\chi(g),\widetilde{\xi}_1(g),\dots,\widetilde{\xi}_e(g))$ running through an open subgroup of $0\times \bb{Z}_p^e$ by varying $g$ in $\widetilde{\Delta}$.
	
	If $i_1>0$, then we may assume that $(i_1,i_2,\dots,i_l)=(1,2,\dots,l)$ and that $x_l=1$. Thus, 
	\begin{align}
		0&=\sum_{i=1}^l x_i\alpha_i-g(\sum_{i=1}^l x_i\alpha_i)\\
		&=\left(-\sum_{i=1}^{l-1}\widetilde{\xi}_i(g)g(x_i)\right)\alpha_0+(x_1-g(x_1))\alpha_1+\cdots+(x_{l-1}-g(x_{l-1}))\alpha_{l-1}.\nonumber
	\end{align}
	We get a contradiction in a similar way.
\end{proof}

\begin{mypara}\label{para:special-functor}
	Following \ref{para:notation-K-tilde}, we assume that the equivalent conditions in \ref{prop:rank} hold, and we take the notation in \ref{para:gamma-basis} by adding tildes. Recall that for any object $W$ of $\repnpr(\widetilde{\Gamma},K_{\infty,\undertilde{\infty}})$, there is a canonical Lie algebra action induced by the infinitesimal action of $\widetilde{\Gamma}$ on $W$ defined in \eqref{eq:gamma-lie-action},
	\begin{align}
		\varphi:\lie(\widetilde{\Gamma})\to \mrm{End}_{K_{\infty,\undertilde{\infty}}}(W).
	\end{align} 
	Let $\widetilde{\partial}_0\in\lie(\widetilde{\Sigma}_{0,\undertilde{\infty}})$ and $\widetilde{\partial}_1,\dots,\widetilde{\partial}_e\in\lie(\widetilde{\Delta})$ be the standard basis defined in \ref{para:gamma-basis}, and we put for any $g\in G$,
	\begin{align}\label{eq:para:special-functor-2}
		\varphi^{\widetilde{\chi}}_g=\log(\chi(g))\varphi_{\widetilde{\partial}_0},\quad\trm{ and }\quad\varphi^{\widetilde{\xi}}_g=\sum_{i=1}^d\widetilde{\xi}_i(g)\varphi_{\widetilde{\partial}_i}.
	\end{align}	
\end{mypara}

\begin{mylem}\label{lem:fal-ext-lie-alg-special}
	Under the assumption in {\rm\ref{para:special-functor}} and with the same notation, the map 
	\begin{align}\label{eq:lem:fal-ext-lie-alg-special}
		\psi:\scr{E}^*_{\ca{O}_K}(1)=\ho_{\widehat{\overline{K}}}(\scr{E}_{\ca{O}_K}(-1),\widehat{\overline{K}})\longrightarrow \widehat{\overline{K}}\otimes_{\bb{Q}_p}\lie(\widetilde{\Gamma}),
	\end{align}
	sending $f$ to $\sum_{i=0}^e f(\widetilde{T}_i)\otimes \widetilde{\partial}_i$, is surjective and induces a morphism of exact sequences of $\widehat{\overline{K}}$-linear Lie algebras
	\begin{align}\label{diam:lem:fal-ext-lie-alg-special}
		\xymatrix{
			0\ar[r]& \ho_{\ca{O}_K}(\widehat{\Omega}^1_{\ca{O}_K}(-1),\widehat{\overline{K}})\ar[r]^-{\jmath^*}\ar@{->>}[d]& \scr{E}^*_{\ca{O}_K}(1)\ar[r]^-{\iota^*}\ar@{->>}[d]^-{\psi}&\widehat{\overline{K}}\ar[r]\ar[d]^-{\wr}& 0\\
			0\ar[r]& \widehat{\overline{K}}\otimes_{\bb{Q}_p}\lie(\widetilde{\Delta})\ar[r]& \widehat{\overline{K}}\otimes_{\bb{Q}_p}\lie(\widetilde{\Gamma})\ar[r]&\widehat{\overline{K}}\otimes_{\bb{Q}_p}\lie(\Sigma)\ar[r]& 0
		}
	\end{align}
	where the first row is \eqref{eq:fal-ext-dual}.
\end{mylem}
\begin{proof}
	The surjectivity of $\psi$ follows from \ref{prop:rank}.(\ref{item:prop:rank-2}). It remains to check $\psi$ is compatible with Lie brackets. With the notation in \ref{lem:diff-lie-alg}, we have
	\begin{align}
		\psi(T_0^*)=\widetilde{\partial}_0+\sum_{j=1}^e b_j\widetilde{\partial}_j,\quad\trm{ and }\quad \psi(T_i^*)=\sum_{j=1}^e a_{ij}\widetilde{\partial}_j,\ \forall 1\leq i\leq d,
	\end{align}
	where $(T_0^*,\dots,T_d^*)$ is the dual basis of $(T_0,\dots,T_d)$ defined in \ref{para:fal-ext-dual}. Thus, $[\psi(T_0^*),\psi(T_i^*)]=\psi(T_i^*)$ and $[\psi(T_i^*),\psi(T_j^*)]=0$ for any $1\leq i,j\leq d$, which completes the proof.
\end{proof}

\begin{myrem}\label{rem:canonical-diff-lie-alg}
	The first vertical map in \eqref{diam:lem:fal-ext-lie-alg-special} can be defined without taking bases. Consider the canonical maps
	\begin{align}\label{eq:prop:canonical-diff-lie-alg-0}
		\ho(\widetilde{\Delta},\bb{Q}_p)\to \ho(H,\bb{Q}_p)\to H^1(H,\widehat{\overline{K}})\iso \widehat{K_\infty}\otimes_{\ca{O}_K}\widehat{\Omega}^1_{\ca{O}_K}(-1)
	\end{align}
	where the first arrow is induced by the surjection $H\to \widetilde{\Delta}$, and the last arrow is induced by the connecting map of the Faltings extension of $\ca{O}_K$ \eqref{eq:fal-ext-connect-2} which sends $\xi_i$ to $\df_{\hyodoring}(T_i)$ for any $1\leq i\leq d$. Thus, the composition of \eqref{eq:prop:canonical-diff-lie-alg-0} sends $\widetilde{\xi}_i$ to $\df_{\hyodoring}(\widetilde{T}_i)$ for any $1\leq i\leq e$ by \ref{lem:diff-lie-alg}. Since $\{\widetilde{\xi}_i\}_{1\leq i\leq e}$ is the dual basis of the basis $\{\widetilde{\partial}_i\}_{1\leq i\leq e}\subseteq \lie(\widetilde{\Delta})$ by construction, the composition of \eqref{eq:prop:canonical-diff-lie-alg-0} induces a natural injective $K$-linear map
	\begin{align}\label{eq:prop:canonical-diff-lie-alg-1}
			\ho_{\bb{Q}_p}(\lie(\widetilde{\Delta}),K)\longrightarrow K\otimes_{\ca{O}_K}\widehat{\Omega}^1_{\ca{O}_K}(-1)
	\end{align}
	which sends $f$ to $\sum_{i=1}^e f(\widetilde{\partial}_i)\otimes \df_{\hyodoring}(\widetilde{T}_i)$. Its dual is a natural surjective $K$-linear map
	\begin{align}\label{eq:prop:canonical-diff-lie-alg-2}
		\ho_{\ca{O}_K}(\widehat{\Omega}^1_{\ca{O}_K}(-1),K)\longrightarrow K\otimes_{\bb{Q}_p}\lie(\widetilde{\Delta})
	\end{align}
	which sends $f$ to $\sum_{i=1}^e f(\df_{\hyodoring}(\widetilde{T}_i))\otimes \widetilde{\partial}_i$, and induces the first vertical map in \eqref{diam:lem:fal-ext-lie-alg-special} by extending scalars.
\end{myrem}

\begin{mypara}\label{defn:higgs}
	Let $(C,\ca{O})$ be a ringed site, $\ca{M}$ an $\ca{O}$-module. A \emph{Higgs field} on an $\ca{O}$-module $\ca{F}$ \emph{with coefficients in $\ca{M}$} is an $\ca{O}$-linear morphism $\theta:\ca{F}\to \ca{F}\otimes_{\ca{O}}\ca{M}$ such that $\theta^{(1)}\circ\theta=0$, where $\theta^{(1)}$ is the $\ca{O}$-linear morphism $\ca{F}\otimes_{\ca{O}}\ca{M}\to \ca{F}\otimes_{\ca{O}}\wedge^2\ca{M}$ defined by $\theta^{(1)}(x\otimes \omega)=\theta(x)\wedge \omega$ for any local sections $x$ of $\ca{F}$ and $\omega$ of $\ca{M}$.
	
	If $\ca{M}$ is a finite free $\ca{O}$-module with basis $\omega_1,\dots,\omega_d$, then to give a Higgs field $\theta$ on $\ca{F}$ is equivalent to give $d$ endomorphisms $\theta_i$ ($1\leq i\leq d$) of the $\ca{O}$-module $\ca{F}$ which commute with each other. For any local section $x$ of $\ca{M}$, we have
	\begin{align}
		\theta(x)=\theta_1(x)\otimes \omega_1+\cdots \theta_d(x)\otimes \omega_d.
	\end{align}
	We call the $d$-tuple $(\theta_1,\dots,\theta_d)$ the \emph{coordinates of the Higgs field $\theta$} with respect to the $\ca{O}$-basis $\omega_1,\dots,\omega_d$ of $\ca{M}$.
	
	Assume that $\ca{M}$ is a finite projective $\ca{O}$-module. We say that a Higgs field $\theta$ on an $\ca{O}$-module $\ca{F}$ is \emph{nilpotent} if there is a finite decreasing filtration by $\ca{O}$-submodules $\ca{F}=\ca{F}^0\supseteq \ca{F}^1\supseteq \cdots \supseteq \ca{F}^n=0$ such that $\theta(\ca{F}^i)\subseteq  \ca{F}^{i+1}\otimes_{\ca{O}}\ca{M}$ for any $0\leq i<n$.
	
	One checks easily by \ref{cor:hyodo-ring} that the universal differential map
	\begin{align}
		\df_{\hyodoring}:\hyodoring\to \hyodoring\otimes_{\ca{O}_K}\widehat{\Omega}^1_{\ca{O}_K}(-1)
	\end{align}
	is a $G$-equivariant $\widehat{\overline{K}}$-linear Higgs field on the $\widehat{\overline{K}}$-module $\hyodoring$ with coefficients in $\widehat{\overline{K}}\otimes_{\ca{O}_K}\widehat{\Omega}^1_{\ca{O}_K}(-1)$.
\end{mypara}

\begin{mydefn}[{cf. \cite[page 872]{tsuji2018localsimpson}}]\label{defn:nil-higgs-bundle-K}
	We define a category $\higgs(\Sigma,K_\infty,\widehat{\Omega}^1_{\ca{O}_K}(-1))$ as follows: 
	\begin{enumerate}
		\renewcommand{\labelenumi}{{\rm(\theenumi)}}
		\item An object $(M,\rho,\theta)$ is a finite projective $K_\infty$-representation $(M,\rho)$ of $\Sigma$ (cf. \ref{defn:repn}) endowed with a $K_\infty$-linear nilpotent Higgs field $\theta:M\to M\otimes_{\ca{O}_K}\widehat{\Omega}^1_{\ca{O}_K}(-1)$ (with coefficients in $K_\infty\otimes_{\ca{O}_K}\widehat{\Omega}^1_{\ca{O}_K}(-1)$) which is $\Sigma$-equivariant (i.e. $\theta\circ\rho(\sigma)=(\rho(\sigma)\otimes\chi^{-1}(\sigma))\circ \theta$ for any $\sigma\in\Sigma$).
		\item A morphism $(M,\rho,\theta)\to(M',\rho',\theta')$ is a $\Sigma$-equivariant $K_\infty$-linear morphism $f:M\to M'$ which is compatible with the Higgs fields (i.e. $\theta'\circ f=(f\otimes 1)\circ \theta$).
	\end{enumerate}
	It is an additive tensor category, where the tensor product is given by $(M,\rho,\theta)\otimes(M',\rho',\theta')=(M\otimes_{K_\infty}M',\rho\otimes\rho',\theta\otimes 1+1\otimes \theta')$.
\end{mydefn}

\begin{myprop}[{cf. \cite[page 873]{tsuji2018localsimpson}}]\label{prop:special-functor}
	Under the assumption in {\rm\ref{para:special-functor}} and with the same notation, let $(V,\rho)$ be an object of $\repnan{\widetilde{\Delta}}(\widetilde{\Gamma},K_\infty)$. 
	\begin{enumerate}
		\renewcommand{\labelenumi}{{\rm(\theenumi)}}
		\item For any $g\in \widetilde{\Gamma}$, we set
		\begin{align}
			\overline{\rho}(g)=\exp(-\varphi^{\widetilde{\xi}}_g)\rho(g).
		\end{align}
		Then, $\overline{\rho}|_{\widetilde{\Delta}}=1$ and $\overline{\rho}|_{\widetilde{\Sigma}_{0,\undertilde{\infty}}}=\rho|_{\widetilde{\Sigma}_{0,\undertilde{\infty}}}$. Moreover, $(V,\overline{\rho})$ is an object of $\repnpr(\Sigma,K_\infty)$.\label{item:prop:special-functor-1}
		\item The $K_\infty$-linear homomorphism
		\begin{align}\label{eq:prop:special-functor-theta}
			\theta_V: V\longrightarrow V\otimes_{\ca{O}_K}\widehat{\Omega}^1_{\ca{O}_K}(-1),\ v\mapsto -\sum_{i=1}^e\varphi_{\widetilde{\partial}_i}(v)\otimes \df_{\hyodoring}(\widetilde{T}_i)
		\end{align}
		is a nilpotent Higgs field which is $\Sigma$-equivariant via $\overline{\rho}$.\label{item:prop:special-functor-2}
	\end{enumerate}
	Therefore, there is a functor
	\begin{align}\label{eq:special-functor}
		\repnan{\widetilde{\Delta}}(\widetilde{\Gamma},K_\infty)\longrightarrow \higgs(\Sigma,K_\infty,\widehat{\Omega}^1_{\ca{O}_K}(-1)),\ (V,\rho)\mapsto (V,\overline{\rho},\theta_V),
	\end{align}
	which relies on the choice of $\widetilde{t}_1,\dots,\widetilde{t}_e$.
\end{myprop}
\begin{proof}
	(\ref{item:prop:special-functor-1}) Notice that the infinitesimal Lie algebra action of $\lie(\widetilde{\Delta})$ on $V$ (\ref{cor:operator}), 
	\begin{align}\label{eq:prop:special-functor-lie}
		\lie(\widetilde{\Delta})\longrightarrow \mrm{End}_{K_\infty}(V),
	\end{align}
	is nilpotent by \ref{rem:derivative} and \ref{prop:operator-nilpotent} (whose assumptions are satisfied as $K$ is a complete discrete valuation field). Thus, $\exp(-\varphi^{\widetilde{\xi}}_g)$ and $\overline{\rho}(g)$ are well-defined endomorphisms of $V$. Since $(V,\rho)$ is $\widetilde{\Delta}$-analytic (\ref{defn:analytic}), if $g\in \widetilde{\Delta}$ then $\rho(g)=\exp(\varphi_g)=\exp(\varphi^{\widetilde{\xi}}_g)$ as $\varphi^{\widetilde{\xi}}|_{\widetilde{\Delta}}=\varphi|_{\widetilde{\Delta}}$. Thus, $\overline{\rho}(g)=\id_V$. It is clear that $\overline{\rho}$ is continuous and $K_\infty$-semi-linear. It remains to show that $\overline{\rho}(g_1g_2)=\overline{\rho}(g_1)\overline{\rho}(g_2)$ for any $g_1,g_2\in \widetilde{\Gamma}$ (so that $\overline{\rho}$ is a $K_\infty$-representation of $\widetilde{\Gamma}/\widetilde{\Delta}=\Sigma$).
	\begin{align}\label{eq:prop:special-functor-1}
		\overline{\rho}(g_1g_2)&=\exp(-\sum_{i=1}^e  \widetilde{\xi}_i(g_1g_2)\varphi_{\widetilde{\partial}_i})\rho(g_1g_2)\quad \trm{(by \eqref{eq:para:special-functor-2})}\\
		&=\exp(-\sum_{i=1}^e  (\widetilde{\xi}_i(g_1)+\chi(g_1)\widetilde{\xi}_i(g_2))\varphi_{\widetilde{\partial}_i})\rho(g_1g_2) \quad\trm{(by \eqref{eq:cont-cocycle-cond})}\nonumber\\
		&=\exp(-\sum_{i=1}^e \widetilde{\xi}_i(g_1)\varphi_{\widetilde{\partial}_i})\rho(g_1)\exp(-\sum_{i=1}^e \widetilde{\xi}_i(g_2)\varphi_{\widetilde{\partial}_i})\rho(g_2)\quad \trm{(by \eqref{eq:lem:operator-1})}\nonumber\\
		&=\overline{\rho}(g_1)\overline{\rho}(g_2).\nonumber
	\end{align}
	
	(\ref{item:prop:special-functor-2}) Since the endomorphisms $\varphi_{\widetilde{\partial}_i}$ on $V$ are nilpotent and commute with each other, $\theta_V$ is a nilpotent Higgs field. It remains to check the $\Sigma$-equivariance of $\theta_V$. For $v\in V$ and $g\in \widetilde{\Gamma}$, we have
	\begin{align}
		\overline{\rho}(g)(\theta_V(v))&=\exp(-\sum_{i=1}^e  \widetilde{\xi}_i(g)\varphi_{\widetilde{\partial}_i})\rho(g)(-\sum_{j=1}^e\varphi_{\widetilde{\partial}_j}(v)\otimes \df_{\hyodoring}(\widetilde{T}_j))\quad\trm{(by \eqref{eq:para:special-functor-2})}\\
		&=\exp(-\sum_{i=1}^e  \widetilde{\xi}_i(g)\varphi_{\widetilde{\partial}_i})(-\sum_{j=1}^e\chi(g)\varphi_{\widetilde{\partial}_j}(\rho(g)(v))\otimes \chi(g)^{-1}\df_{\hyodoring}(\widetilde{T}_j))\quad \trm{(by \eqref{eq:lem:operator-1})}\nonumber\\
		&=-\sum_{j=1}^e\varphi_{\widetilde{\partial}_j}(\exp(-\sum_{i=1}^e  \widetilde{\xi}_i(g)\varphi_{\widetilde{\partial}_i})\rho(g)(v))\otimes \df_{\hyodoring}(\widetilde{T}_j)\quad \trm{(as $\varphi_{\widetilde{\partial}_i}\varphi_{\widetilde{\partial}_j}=\varphi_{\widetilde{\partial}_j}\varphi_{\widetilde{\partial}_i}$)}\nonumber\\
		&=\theta_V(\overline{\rho}(g)(v)).\nonumber
	\end{align}
\end{proof}

\begin{myrem}\label{rem:special-functor}
	\begin{enumerate}
		\renewcommand{\labelenumi}{{\rm(\theenumi)}}
		\item The minus sign in the definition of $\theta_V$ is designed for \ref{thm:simpson-K}.(\ref{item:simpson-K4}) (cf. \ref{rem:simpson-K-2}).
		\item There is another definition for the Higgs field $\theta_V$ \eqref{eq:prop:special-functor-theta}. Consider the canonical maps 
		\begin{align}
			V\longrightarrow\ho_{\bb{Q}_p}(\lie(\widetilde{\Delta}),V)\longrightarrow V\otimes_{\ca{O}_K}\widehat{\Omega}^1_{\ca{O}_K}(-1),
		\end{align}
		where the first map is induced by the infinitesimal Lie algebra algebra action of $\lie(\widetilde{\Delta})$ on $V$ which sends $v\in V$ to $(\widetilde{\partial}\mapsto \varphi_{\widetilde{\partial}}(v))$, and the second map is induced by extending scalars from the canonical map \eqref{eq:prop:canonical-diff-lie-alg-1} which sends $f$ to $\sum_{i=1}^e f(\widetilde{\partial}_i)\otimes \df_{\hyodoring}(\widetilde{T}_i)$. It is clear that the composition is $-\theta_V$.
	\end{enumerate}
\end{myrem}

\begin{mypara}
	Let $M=(M,\rho,\theta)$ be an object of $\higgs(\Sigma,K_\infty,\widehat{\Omega}^1_{\ca{O}_K}(-1))$. On the $\hyodoring$-module $\hyodoring\otimes_{K_\infty}M$, we define a semi-linear action of $G$ by the diagonal action $g\otimes \overline{g}$ for any $g\in G$ with image $\overline{g}\in \Sigma$, and a $G$-equivariant $\widehat{\overline{K}}$-linear Higgs field $\df_{\hyodoring}\otimes 1+1\otimes \theta$ with value in $\widehat{\overline{K}}\otimes_{\ca{O}_K}\widehat{\Omega}^1_{\ca{O}_K}(-1)$. In particular, its $\widehat{\overline{K}}$-submodule
	\begin{align}
		\bb{V}(M)=(\hyodoring\otimes_{K_\infty}M)^{\df_{\hyodoring}\otimes 1+1\otimes \theta=0}
	\end{align}
	is endowed with the induced semi-linear action of $G$.
	
	Let $W=(W,\rho)$ be an object of $\repnpr(G,\widehat{\overline{K}})$. On the $\hyodoring$-module $\hyodoring\otimes_{\widehat{\overline{K}}}W$, we define a semi-linear action of $G$ by the diagonal action $g\otimes g$ for any $g\in G$, and a $G$-equivariant $\widehat{\overline{K}}$-linear Higgs field $\df_{\hyodoring}\otimes 1$. In particular, its $K_\infty$-submodule 
	\begin{align}
		\bb{D}(W)=((\hyodoring\otimes_{\widehat{\overline{K}}}W)^{H})^{(\Sigma,K_\infty)\trm{-fini}}
	\end{align}
	where $(-)^{(\Sigma,K_\infty)\trm{-fini}}$ is taking the $(\Sigma,K_\infty)$-finite part (cf. \ref{defn:repn}), is endowed with the induced semi-linear action of $\Sigma$ and the induced $\Sigma$-equivariant $K_\infty$-linear Higgs field with value in $K_\infty\otimes_{\ca{O}_K}\widehat{\Omega}^1_{\ca{O}_K}(-1)$ (cf. \ref{cor:hyodo-ring-coh} and \ref{lem:G-Kcycl-fixed}).
	
	We remark that the definitions of $\bb{V}$ and $\bb{D}$ do not depend on the choice of $t_1,\dots, t_d$.
\end{mypara}

\begin{mythm}[{cf. \cite[15.1]{tsuji2018localsimpson}}]\label{thm:simpson-K}
	We keep the notation in {\rm\ref{para:notation-K}}.
	\begin{enumerate}
		\renewcommand{\labelenumi}{{\rm(\theenumi)}}
		\item For any object $M$ of $\higgs(\Sigma,K_\infty,\widehat{\Omega}^1_{\ca{O}_K}(-1))$, $\bb{V}(M)$ is a finite-dimensional $\widehat{\overline{K}}$-module on which $G$ acts continuously with respect to the canonical topology (thus $\bb{V}(M)$ is an object of $\repnpr(G,\widehat{\overline{K}})$). Moreover, the canonical $\hyodoring$-linear morphism (which is $G$-equivariant and compatible with Higgs fields by definition)
		\begin{align}\label{eq:thm:simpson-K-1}
			\hyodoring\otimes_{\widehat{\overline{K}}}\bb{V}(M)\longrightarrow \hyodoring\otimes_{K_\infty} M
		\end{align}
		is an isomorphism.\label{item:simpson-K1}
		\item For any object $W$ of $\repnpr(G,\widehat{\overline{K}})$, $\bb{D}(W)$ is a finite-dimensional $\widehat{\overline{K}}$-module on which $\Sigma$ acts continuously with respect to the canonical topology (thus $\bb{D}(W)$ is an object of $\higgs(\Sigma,K_\infty,\widehat{\Omega}^1_{\ca{O}_K}(-1))$). Moreover, the canonical $\hyodoring$-linear morphism (which is $G$-equivariant and compatible with Higgs fields by definition)
		\begin{align}\label{eq:thm:simpson-K-2}
			\hyodoring\otimes_{K_\infty}\bb{D}(W)\longrightarrow \hyodoring\otimes_{\widehat{\overline{K}}} W
		\end{align}
		is an isomorphism.\label{item:simpson-K2}
		\item The functors
		\begin{align}\label{eq:thm:simpson-K-3}
			\xymatrix{
				\repnpr(G,\widehat{\overline{K}})\ar@<0.2pc>[r]^-{\bb{D}}&\higgs(\Sigma,K_\infty,\widehat{\Omega}^1_{\ca{O}_K}(-1))\ar@<0.2pc>[l]^-{\bb{V}}
			}
		\end{align}
		are equivalences of additive tensor categories, quasi-inverse to each other.\label{item:simpson-K3}
		\item Under the assumption in {\rm\ref{para:special-functor}} and with the same notation, let $(V,\rho)$ be an object of $\repnan{\widetilde{\Delta}}(\widetilde{\Gamma},K_\infty)$, $M=(V,\overline{\rho},\theta_V)$ the object of $\higgs(\Sigma,K_\infty,\widehat{\Omega}^1_{\ca{O}_K}(-1))$ defined by the functor \eqref{eq:special-functor}. Then, there exists a natural $G$-equivariant  $\widehat{\overline{K}}$-linear isomorphism
		\begin{align}\label{eq:thm:simpson-K-4}
			\bb{V}(M)\iso \widehat{\overline{K}}\otimes_{K_\infty}V
		\end{align}\label{item:simpson-K4}
		which depends on the choice of $\widetilde{t}_1,\dots,\widetilde{t}_e$.
	\end{enumerate}
\end{mythm}
\begin{proof}
	We follow the proof of 	\cite[15.1]{tsuji2018localsimpson}. We identify $\hyodoring=\widehat{\overline{K}}[T_1,\dots,T_d]$ by \eqref{eq:hyodo-ring-polynomial}.
	
	(\ref{item:simpson-K1}) Let $\theta:M\to M\otimes_{\ca{O}_K}\widehat{\Omega}^1_{\ca{O}_K}(-1)$ denote the $K_\infty$-linear Higgs field of $M$. We write
	\begin{align}
		\theta(x)&=\sum_{i=1}^d\theta_i(x)\otimes \df_{\hyodoring}(T_i),\ \forall x\in M,\\
		\df_{\hyodoring}(f)&=\sum_{i=1}^d\frac{\partial f}{\partial T_i}\otimes \df_{\hyodoring}(T_i),\ \forall f\in\hyodoring,
	\end{align}
	where $\theta_i$ (resp. $\frac{\partial}{\partial T_i}$) are $K_\infty$-linear (resp. $\widehat{\overline{K}}$-linear) endomorphisms of $M$ (resp. $\hyodoring$), which commute with each other. Since $\theta_i$ are nilpotent by definition, we can define a $\hyodoring$-linear isomorphism
	\begin{align}
		\iota=\exp(-\sum_{i=1}^d T_i\otimes \theta_i):\hyodoring\otimes_{K_\infty}M\iso \hyodoring\otimes_{K_\infty}M,
	\end{align}
	whose inverse is given by $\iota^{-1}=\exp(\sum_{i=1}^dT_i\otimes \theta_i)$. We claim that the following diagram is commutative. 
	\begin{align}
		\xymatrix{
			\hyodoring\otimes_{K_\infty}M\ar[d]_-{\iota}\ar[rr]^-{\df_{\hyodoring}\otimes 1}&&\hyodoring\otimes_{K_\infty}M\otimes_{\ca{O}_K}\widehat{\Omega}^1_{\ca{O}_K}(-1)\ar[d]^-{\iota\otimes 1}\\
			\hyodoring\otimes_{K_\infty}M\ar[rr]^-{\df_{\hyodoring}\otimes 1+1\otimes\theta}&&\hyodoring\otimes_{K_\infty}M\otimes_{\ca{O}_K}\widehat{\Omega}^1_{\ca{O}_K}(-1)
		}
	\end{align}
	Indeed, we have
	\begin{align}\label{eq:iota-higgs}
		&(\df_{\hyodoring}\otimes 1)\circ \exp(\sum_{i=1}^dT_i\otimes \theta_i)\\
		=&\exp(\sum_{i=1}^dT_i\otimes \theta_i)\circ (\df_{\hyodoring}\otimes 1)+\sum_{j=1}^d\left(\exp(\sum_{i=1}^dT_i\otimes \theta_i)\circ(1\otimes \theta_j)\right)\otimes\df_{\hyodoring}(T_j)\nonumber\\
		=&(\exp(\sum_{i=1}^dT_i\otimes \theta_i)\otimes 1)\circ(\df_{\hyodoring}\otimes 1+1\otimes\theta).\nonumber
	\end{align}
	Thus, the restriction of $\iota$ induces a $\widehat{\overline{K}}$-linear isomorphism
	\begin{align}\label{eq:3.14.10}
		\iota_0:\widehat{\overline{K}}\otimes_{K_\infty}M=(\hyodoring\otimes_{K_\infty}M)^{\df_{\hyodoring}\otimes 1=0}\iso \bb{V}(M)=(\hyodoring\otimes_{K_\infty}M)^{\df_{\hyodoring}\otimes 1+1\otimes \theta=0},
	\end{align}
	from which we see that $\bb{V}(M)$ is finite-dimensional over $\widehat{\overline{K}}$ and thus contained in the finite-dimensional $\widehat{\overline{K}}$-submodule $\mrm{Sym}^n_{\widehat{\overline{K}}} (\scr{E}_{\ca{O}_K}(-1))\otimes_{K_\infty} M$ of $\hyodoring\otimes_{K_\infty}M$ for some integer $n>0$. Since $\bb{V}(M)$ is a direct summand of $\mrm{Sym}^n_{\widehat{\overline{K}}} (\scr{E}_{\ca{O}_K}(-1))\otimes_{K_\infty} M$, the topology on $\bb{V}(M)$ induced from $\mrm{Sym}^n_{\widehat{\overline{K}}} (\scr{E}_{\ca{O}_K}(-1))\otimes_{K_\infty} M$ coincides with the canonical topology as a finite-dimensional $\widehat{\overline{K}}$-module. Since $G$ acts continuously on $\mrm{Sym}^n_{\widehat{\overline{K}}} (\scr{E}_{\ca{O}_K}(-1))\otimes_{K_\infty} M$ by \ref{cor:hyodo-ring}.(\ref{item:cor:hyodo-ring-2}), it acts also continuously on $\bb{V}(M)$ with respect to the canonical topology, which means that $\bb{V}(M)$ is an object of $\repnpr(G,\widehat{\overline{K}})$. Finally, notice that the composition of the $\hyodoring$-linear maps
	\begin{align}
		\xymatrix@C=4pc{
			\hyodoring\otimes_{K_\infty}M\ar[r]^-{\id_{\hyodoring}\otimes \iota_0}&
			\hyodoring\otimes_{\widehat{\overline{K}}}\bb{V}(M)\ar[r]^-{\eqref{eq:thm:simpson-K-1}}& \hyodoring\otimes_{K_\infty} M
		}
	\end{align}
	is the isomorphism $\iota$. Thus, \eqref{eq:thm:simpson-K-1} is an isomorphism, which completes the proof of (\ref{item:simpson-K1}).
	
	(\ref{item:simpson-K4}) Since the $K_\infty$-endomorphisms $\varphi_{\widetilde{\partial}_i}$ ($1\leq i\leq e$) on $V$ are nilpotent and commute with each other, we can define a $\hyodoring$-linear isomorphism
	\begin{align}\label{eq:thm:simpson-K-jmath}
		\jmath=\exp(-\sum_{i=1}^e\widetilde{T}_i\otimes \varphi_{\widetilde{\partial}_i}):\hyodoring\otimes_{K_\infty}M\iso \hyodoring\otimes_{K_\infty} V,
	\end{align}
	whose inverse is given by $\jmath^{-1}=\exp(\sum_{i=1}^e\widetilde{T}_i\otimes \varphi_{\widetilde{\partial}_i})$. We claim that $\jmath$ is $G$-equivariant. Indeed, for any $g\in G$ and $x\in M$, we have
	\begin{align}
		g(\jmath(1\otimes x))&=\exp(-\sum_{i=1}^e \chi(g)^{-1}(\widetilde{\xi}_i(g)+\widetilde{T}_i)\otimes \chi(g)\varphi_{\widetilde{\partial}_i})(1\otimes \rho(g)(x))\quad(\trm{by \eqref{eq:para:special-functor}, \eqref{eq:lem:operator-1}})\\
		&=\exp(-\sum_{i=1}^e\widetilde{T}_i\otimes \varphi_{\widetilde{\partial}_i})\exp(-\sum_{i=1}^e1\otimes \widetilde{\xi}_i(g)\varphi_{\widetilde{\partial}_i})(1\otimes \rho(g)(x))\nonumber\\
		&=\jmath(1\otimes\overline{\rho}(g)(x))\quad \trm{(by \eqref{eq:para:special-functor-2})}.\nonumber
	\end{align}
	On the other hand, by the same argument as \eqref{eq:iota-higgs}, we see that $\jmath$ is compatible with Higgs fields, i.e. the following diagram is commutative.
	\begin{align}
		\xymatrix{
			\hyodoring\otimes_{K_\infty}M\ar[d]_-{\jmath}\ar[rr]^-{\df_{\hyodoring}\otimes 1+1\otimes \theta_V}&&\hyodoring\otimes_{K_\infty}M\otimes_{\ca{O}_K}\widehat{\Omega}^1_{\ca{O}_K}(-1)\ar[d]^-{\jmath\otimes 1}\\ 
			\hyodoring\otimes_{K_\infty} V\ar[rr]^-{\df_{\hyodoring}\otimes 1}&&\hyodoring\otimes_{K_\infty} V\otimes_{\ca{O}_K}\widehat{\Omega}^1_{\ca{O}_K}(-1)
		}
	\end{align}
	Thus, the restriction of $\jmath$ induces a $G$-equivariant $\widehat{\overline{K}}$-linear isomorphism
	\begin{align}\label{eq:3.18.11}
		\bb{V}(M)=(\hyodoring\otimes_{K_\infty}M)^{\df_{\hyodoring}\otimes 1+1\otimes \theta_V=0}\iso \widehat{\overline{K}}\otimes_{K_\infty}V=(\hyodoring\otimes_{K_\infty} V)^{\df_{\hyodoring}\otimes 1=0}.
	\end{align}
	
	(\ref{item:simpson-K2}) We apply (\ref{item:simpson-K4}) to the case where $e=d$ and $(\widetilde{t}_1,\dots,\widetilde{t}_e)=(t_1,\dots,t_d)$ (with the same $p$-power roots). By \ref{thm:sen-brinon}, \ref{prop:gamma-analytic} and (\ref{item:simpson-K4}), we see that the functor $\bb{V}$ is essentially surjective. Thus, we may assume that $W=\bb{V}(M)$ for some object $M$ of $\higgs(\Sigma,K_\infty,\widehat{\Omega}^1_{\ca{O}_K}(-1))$. Taking the $H$-invariant part of the isomorphism \eqref{eq:thm:simpson-K-1}, by \ref{cor:hyodo-ring-coh} we get a canonical isomorphism
	\begin{align}
		(\hyodoring\otimes_{\widehat{\overline{K}}}W)^{H}\iso \widehat{K_\infty}\otimes_{K_\infty}M.
	\end{align}
	Taking the $(\Sigma,K_\infty)$-finite part, we get  by \ref{lem:G-Kcycl-fixed} a canonical $\Sigma$-equivariant $K_\infty$-linear isomorphism $\bb{D}(W)\iso M$ compatible with Higgs fields, which completes the proof of (\ref{item:simpson-K2}).
	
	(\ref{item:simpson-K3}) The proof of (\ref{item:simpson-K2}) shows that the canonical morphism $\bb{D}\circ\bb{V}\to \id$ is an isomorphism. Taking the Higgs field zero part of the isomorphism \eqref{eq:thm:simpson-K-2}, we see that $\bb{V}\circ\bb{D}\to \id$ is an isomorphism. Using the isomorphisms \eqref{eq:thm:simpson-K-1} and \eqref{eq:thm:simpson-K-2}, we see that $\bb{V}$ and $\bb{D}$ are compatible with tensor products. This completes the proof.
\end{proof}

\begin{myrem}\label{rem:simpson-K-1}
	Let $L$ be a complete discrete valuation field of characteristic $0$ with perfect residue field of characteristic $p$. If $\ca{O}_K$ is the $p$-adic completion of the localization of an adequate $\ca{O}_L$-algebra $A$ at some $\ak{p}\in\ak{S}_p(A)$ (cf. \ref{defn:quasi-adequate-alg}), then \ref{thm:simpson-K} (except (\ref{item:simpson-K4})) is a special case of \cite[15.2]{tsuji2018localsimpson}.
\end{myrem}

\begin{myrem}\label{rem:simpson-K-2}
Our construction of \eqref{eq:prop:special-functor-theta} has a sign difference with Tsuji's. The essential reason is that the $G$-action defined by Tsuji on the Hyodo ring $\hyodoring=\widehat{\overline{K}}[T_1,\dots,T_d]$ is given by $g(T_i)=\chi(g)^{-1}(-\xi_i(g)+T_i)$ (cf. \cite[page 872]{tsuji2018localsimpson} and \ref{prop:fal-ext-comparison}).
\end{myrem}

\begin{mylem}\label{lem:sen-brinon-operator}
	Under the assumption in {\rm\ref{para:special-functor}} and with the same notation, let $(V,\rho)$ (resp. $(\widetilde{V},\rho'))$ be an object of $\repnan{\Delta}(\Gamma,K_\infty)$ (resp. $\repnan{\widetilde{\Delta}}(\widetilde{\Gamma},K_\infty)$). Consider the $K_\infty$-linear endomorphisms $\{\varphi_{\partial_i}|_V\}_{0\leq i\leq d}$ on $V$ (resp. $\{\varphi_{\widetilde{\partial}_i}\}_{0\leq i\leq d}$ on $\widetilde{V}$) defined by the infinitesimal Lie algebra action of $\lie(\Gamma)$ (resp. $\lie(\widetilde{\Gamma})$). We write 
	\begin{align}\label{eq:lem:sen-brinon-operator-base-change}
		(\widetilde{T}_1,\dots,\widetilde{T}_e)=(T_1,\dots,T_d)A+B
	\end{align}
	as elements of $\scr{E}_{\ca{O}_K}(-1)\subseteq \hyodoring$, where $A=(a_{ij})\in\mrm{M}_{d\times e}(K),B=(b_j)\in \mrm{M}_{1\times e}(\widehat{\overline{K}})$ (cf. {\rm\ref{lem:diff-lie-alg}}). Assume that there is an isomorphism $\beta:\widehat{\overline{K}}\otimes_{K_\infty}\widetilde{V}\iso \widehat{\overline{K}}\otimes_{K_\infty}V$ in $\repnpr(G,\widehat{\overline{K}})$. Then, there are identities of $\widehat{\overline{K}}$-linear endomorphisms
	\begin{align}
		\beta^{-1}\circ(1\otimes \varphi_{\partial_1}|_V,\dots,1\otimes \varphi_{\partial_d}|_V)\circ \beta&=(1\otimes \varphi_{\widetilde{\partial}_1}|_{\widetilde{V}},\dots,1\otimes \varphi_{\widetilde{\partial}_e}|_{\widetilde{V}})A^{\mrm{T}},\label{eq:sen-geo-bc}\\
		\beta^{-1}\circ (1\otimes \varphi_{\partial_0}|_V)\circ \beta&=1\otimes \varphi_{\widetilde{\partial}_0}|_{\widetilde{V}}+(1\otimes \varphi_{\widetilde{\partial}_1}|_{\widetilde{V}},\dots,1\otimes \varphi_{\widetilde{\partial}_e}|_{\widetilde{V}})B^\mrm{T},\label{eq:sen-ari-bc}
	\end{align}
	where $A^{\mrm{T}}$ and $B^\mrm{T}$ are the transposes of $A$ and $B$.
\end{mylem}
\begin{proof}
	Let $M=(V,\overline{\rho},\theta_V)$ (resp. $\widetilde{M}=(\widetilde{V},\overline{\rho'},\theta_{\widetilde{V}}))$ be the object of $\higgs(\Sigma,K_\infty,\widehat{\Omega}^1_{\ca{O}_K}(-1))$ defined by the functor {\rm\eqref{eq:special-functor}}. Consider the commutative diagram
	\begin{align}
		\xymatrix{
			\hyodoring\otimes_{K_\infty}\widetilde{M}\ar[r]^-{\widetilde{\jmath}}\ar[d]_-{\jmath^{-1}\circ(1\otimes \beta)\circ\widetilde{\jmath}}& \hyodoring\otimes_{K_\infty} \widetilde{V}\ar[d]^-{1\otimes \beta}\\
			\hyodoring\otimes_{K_\infty}M\ar[r]^-{\jmath}& \hyodoring\otimes_{K_\infty} V
		}
	\end{align}
	where $\jmath=\exp(-\sum_{i=1}^d T_i\otimes \varphi_{\partial_i}|_V)$ (resp. $\widetilde{\jmath}=\exp(-\sum_{i=1}^e\widetilde{T}_i\otimes \varphi_{\widetilde{\partial}_i}|_{\widetilde{V}})$) is the $\hyodoring$-linear isomorphism defined in \eqref{eq:thm:simpson-K-jmath}. Notice that  
	\begin{align}
		((\hyodoring\otimes_{K_\infty}\widetilde{M})^H)^{(\Sigma,K_\infty)\trm{-fini}}=(\widehat{K_\infty}\otimes_{K_\infty}\widetilde{M})^{(\Sigma,K_\infty)\trm{-fini}}=\widetilde{M}
	\end{align}
	where the first equality follows from \ref{cor:hyodo-ring-coh}, and the second equality follows from \ref{lem:G-Kcycl-fixed}.
	Since $1\otimes \beta$, $\jmath$ and $\widetilde{\jmath}$ are $G$-equivariant and compatible with Higgs fields by the proof of \ref{thm:simpson-K}.(\ref{item:simpson-K4}), $\jmath^{-1}\circ(1\otimes \beta)\circ\widetilde{\jmath}$ induces a  $\Sigma$-equivariant $K_\infty$-linear map compatible with Higgs fields,
	\begin{align}
		h:\widetilde{M}=((\hyodoring\otimes_{K_\infty}\widetilde{M})^H)^{(\Sigma,K_\infty)\trm{-fini}}\longrightarrow ((\hyodoring\otimes_{K_\infty}M)^H)^{(\Sigma,K_\infty)\trm{-fini}}=M,
	\end{align}
	so that we actually have $\jmath^{-1}\circ(1\otimes \beta)\circ\widetilde{\jmath}=1\otimes h$ (we remark that $h=\bb{D}(1\otimes \beta)$). 
	
	Notice that $\varphi_{\partial_0}|_V\circ \varphi_{\partial_i}^n|_V=\varphi_{\partial_i}^n|_V\circ(\varphi_{\partial_0}|_V+n)$ for any $n\in\bb{N}$ and $1\leq i\leq d$ by \eqref{eq:lem:operator-2}. Thus, we have
	\begin{align}\label{eq:lem:sen-brinon-operator-1}
		&\jmath^{-1}\circ (1\otimes \varphi_{\partial_0}|_V) \circ \jmath\\
		=& \exp(\sum_{i=1}^dT_i\otimes \varphi_{\partial_i}|_V)\circ (1\otimes \varphi_{\partial_0}|_V) \circ\exp(-\sum_{i=1}^dT_i\otimes \varphi_{\partial_i}|_V)\nonumber\\
		=&\exp(\sum_{i=1}^dT_i\otimes \varphi_{\partial_i}|_V)\circ\exp(-\sum_{i=1}^dT_i\otimes \varphi_{\partial_i}|_V)\circ (1\otimes \varphi_{\partial_0}|_V-\sum_{i=1}^dT_i\otimes \varphi_{\partial_i}|_V)\nonumber\\
		=&1\otimes \varphi_{\partial_0}|_V-\sum_{i=1}^dT_i\otimes \varphi_{\partial_i}|_V.\nonumber
	\end{align}
	Since $h$ is compatible with Higgs fields, we have $\sum_{j=1}^e (h\circ \varphi_{\widetilde{\partial}_j}|_{\widetilde{V}})\otimes \df_{\hyodoring}(\widetilde{T}_j)=\sum_{i=1}^d  (\varphi_{\partial_i}|_V\circ h)\otimes \df_{\hyodoring}(T_i)$ by the definition \eqref{eq:prop:special-functor-theta}. Notice that $(\df_{\hyodoring}(\widetilde{T}_1),\dots,\df_{\hyodoring}(\widetilde{T}_e))=(\df_{\hyodoring}(T_1),\dots,\df_{\hyodoring}(T_d))A$ and that $\df_{\hyodoring}(T_1),\dots,\df_{\hyodoring}(T_d)$ are linearly independent. Thus, we have
	\begin{align}
		(\varphi_{\partial_1}|_V\circ h,\dots,\varphi_{\partial_d}|_V\circ h)=(h\circ \varphi_{\widetilde{\partial}_1}|_{\widetilde{V}},\dots,h\circ \varphi_{\widetilde{\partial}_e}|_{\widetilde{V}})A^{\mrm{T}},
	\end{align}
	which implies \eqref{eq:sen-geo-bc}, since $\jmath$ and $\widetilde{\jmath}$ commute with $1\otimes\varphi_{\partial_i}|_V$ and $1\otimes\varphi_{\widetilde{\partial}_i}|_{\widetilde{V}}$ respectively. Since $\overline{\rho'}|_{\widetilde{\Sigma}_{0,\undertilde{\infty}}}=\rho'|_{\widetilde{\Sigma}_{0,\undertilde{\infty}}}$ by \ref{prop:special-functor}.(\ref{item:prop:special-functor-1}), we have an identification
	\begin{align}
		\varphi_{\widetilde{\partial}_0}|_{\widetilde{V}}=\varphi_{\widetilde{\partial}_0}|_{\widetilde{M}}\in \mrm{End}_{K_\infty}(\widetilde{V})=\mrm{End}_{K_\infty}(\widetilde{M}),
	\end{align}
	where $\varphi_{\widetilde{\partial}_0}|_{\widetilde{V}}$ is given by the infinitesimal action of $\widetilde{\partial}_0\in \lie(\widetilde{\Sigma}_{0,\undertilde{\infty}})$ via $\rho'$ on $V$, and $\varphi_{\widetilde{\partial}_0}|_{\widetilde{M}}$ is given by the infinitesimal action of $\widetilde{\partial}_0\in \lie(\widetilde{\Sigma}_{0,\undertilde{\infty}})=\lie(\Sigma)$ via $\overline{\rho'}$ on $M$. Similarly, we have $\varphi_{\partial_0}|_V=\varphi_{\partial_0}|_M$. Since $h:\widetilde{M}\to M$ is $\Sigma$-equivariant, we deduce that $h\circ \varphi_{\widetilde{\partial}_0}|_{\widetilde{V}}=\varphi_{\partial_0}|_{V}\circ h$ by \ref{lem:infinitesimal}.(\ref{item:infinitesimal-3}). Using these properties, we have
	\begin{align}
		&(1\otimes h)^{-1}\circ(1\otimes \varphi_{\partial_0}|_V-\sum_{i=1}^dT_i\otimes \varphi_{\partial_i}|_V)\circ (1\otimes h)\\
		=&1\otimes \varphi_{\widetilde{\partial}_0}|_{\widetilde{V}}- \sum_{i=1}^dT_i\otimes (h^{-1}\circ\varphi_{\partial_i}|_V\circ h)\nonumber\\
		=&1\otimes \varphi_{\widetilde{\partial}_0}|_{\widetilde{V}}- \sum_{i=1}^d\sum_{j=1}^e T_ia_{ij}\otimes \varphi_{\widetilde{\partial}_i}|_{\widetilde{V}}\nonumber\\
		=&1\otimes \varphi_{\widetilde{\partial}_0}|_{\widetilde{V}}- \sum_{j=1}^e (\widetilde{T}_j-b_j)\otimes \varphi_{\widetilde{\partial}_i}|_{\widetilde{V}}.\nonumber
	\end{align}
	By the argument of \eqref{eq:lem:sen-brinon-operator-1}, we see that
	\begin{align}
		\widetilde{\jmath}\circ (1\otimes \varphi_{\widetilde{\partial}_0}|_{\widetilde{V}}- \sum_{j=1}^e (\widetilde{T}_j-b_j)\otimes \varphi_{\widetilde{\partial}_i}|_{\widetilde{V}})\circ \widetilde{\jmath}^{-1}=1\otimes \varphi_{\widetilde{\partial}_0}|_{\widetilde{V}}+\sum_{j=1}^e b_j\otimes \varphi_{\widetilde{\partial}_j}|_{\widetilde{V}}
	\end{align}
	which completes the proof.
\end{proof}

\begin{mythm}\label{thm:sen-brinon-operator}
	Let $K$ be a complete discrete valuation field extension of $\bb{Q}_p$ whose residue field admits a finite $p$-basis, $G=\gal(\overline{K}/K)$. Then, for any object $W$ of $\repnpr(G,\widehat{\overline{K}})$, there is a canonical homomorphism of $\widehat{\overline{K}}$-linear Lie algebras (see {\rm\ref{para:fal-ext-dual}}) 
	\begin{align}\label{eq:sen-brinon-operator}
		\varphi_{\sen}|_W:\scr{E}_{\ca{O}_K}^*(1)\longrightarrow \mrm{End}_{\widehat{\overline{K}}}(W),
	\end{align}
	which is $G$-equivariant  with respect to the canonical action on $\scr{E}_{\ca{O}_K}^*(1)$ defined in {\rm\ref{para:fal-ext-dual}} and the adjoint action on $\mrm{End}_{\widehat{\overline{K}}}(W)$ (i.e. $g\in G$ sends an endomorphism $\phi$ to $g\circ \phi \circ g^{-1}$), and functorial in $W$, i.e. it defines a canonical functor 
	\begin{align}\label{eq:sen-functor}
		\varphi_{\sen}:\repnpr(G,\widehat{\overline{K}})\longrightarrow \mbf{Rep}^{\mrm{proj}}(\scr{E}_{\ca{O}_K}^*(1),\widehat{\overline{K}}),
	\end{align}
	from the category of finite projective (continuous semi-linear) $\widehat{\overline{K}}$-representations of the profinite group $G$ to the category of finite projective $\widehat{\overline{K}}$-linear representations of the Lie algebra $\scr{E}_{\ca{O}_K}^*(1)$. 
	
	Moreover, under the assumption in {\rm\ref{para:special-functor}} and with the same notation, assume that there is an object $\widetilde{V}$ of $\repnan{\widetilde{\Delta}}(\widetilde{\Gamma},K_\infty)$ such that $W=\widehat{\overline{K}}\otimes_{K_\infty}\widetilde{V}$. Then, for any $f\in \scr{E}_{\ca{O}_K}^*(1)=\ho_{\widehat{\overline{K}}}(\scr{E}_{\ca{O}_K}(-1),\widehat{\overline{K}})$,
	\begin{align}\label{eq:thm:sen-brinon-operator}
		\varphi_{\sen}|_W(f)=\sum_{i=0}^e f(\widetilde{T}_i)\otimes \varphi_{\widetilde{\partial}_i}|_{\widetilde{V}}.
	\end{align}
\end{mythm}
\begin{proof}
	Recall that the base change functor
	\begin{align}\label{eq:thm:sen-brinon-operator-1}
		\repnan{\Delta}(\Gamma,K_\infty)\longrightarrow \repnpr(G,\widehat{\overline{K}}),\ V\mapsto \widehat{\overline{K}}\otimes_{K_\infty}V,
	\end{align}
	is an equivalence by \ref{thm:sen-brinon} and \ref{prop:gamma-analytic}. Thus, there is an essentially unique object $V$ of $\repnan{\Delta}(\Gamma,K_\infty)$ such that $W=\widehat{\overline{K}}\otimes_{K_\infty}V$. We claim that $\varphi_{\sen}|_W$ defined by the formula \eqref{eq:thm:sen-brinon-operator} does not depend on the choice of $\widetilde{V}$ and $\widetilde{t}_i$ (so that $\varphi_{\sen}|_W$ is canonically defined by the essential surjectivity of \eqref{eq:thm:sen-brinon-operator-1}, and functorial in $W$ by the fully faithfulness of \eqref{eq:thm:sen-brinon-operator-1}). With the notation in \ref{lem:sen-brinon-operator}, for any $f\in\ho_{\widehat{\overline{K}}}(\scr{E}_{\ca{O}_K}(-1),\widehat{\overline{K}})$, we have
	\begin{align}
		&(f(\widetilde{T}_0),f(\widetilde{T}_1),\dots,f(\widetilde{T}_e))\otimes (\varphi_{\widetilde{\partial}_0},\varphi_{\widetilde{\partial}_1},\dots,\varphi_{\widetilde{\partial}_e})^{\mrm{T}}\\
		=&(f(T_0),f(T_1),\dots,f(T_d))\left(\begin{array}{cc}
			1 & B  \\
			0 & A
		\end{array}\right)\otimes (\varphi_{\widetilde{\partial}_0},\varphi_{\widetilde{\partial}_1},\dots,\varphi_{\widetilde{\partial}_e})^{\mrm{T}}\nonumber\\
		=&(f(T_0),f(T_1),\dots,f(T_d))\otimes (\varphi_{\partial_0},\varphi_{\partial_1},\dots,\varphi_{\partial_d})^{\mrm{T}}\nonumber
	\end{align}
	where the last equality follows from \ref{lem:sen-brinon-operator}, which proves the claim. 
	
	Notice that the map $\varphi_{\sen}|_W$ defined by \eqref{eq:thm:sen-brinon-operator} fits into the following commutative diagram
	\begin{align}
		\xymatrix{
			\scr{E}_{\ca{O}_K}^*(1)\ar[rr]^-{	\varphi_{\sen}|_W}\ar[dr]_-{\psi}&& \mrm{End}_{\widehat{\overline{K}}}(W)\\
			& \widehat{\overline{K}}\otimes_{\bb{Q}_p}\lie(\widetilde{\Gamma})\ar[ur]_-{\id_{\widehat{\overline{K}}}\otimes \varphi|_{\widetilde{V}}}&
		}
	\end{align}
	where $\psi$ is the surjection \eqref{eq:lem:fal-ext-lie-alg-special}, $\varphi|_{\widetilde{V}}:\lie(\widetilde{\Gamma})\to \mrm{End}_{K_\infty}(\widetilde{V})$ is the  infinitesimal Lie algebra algebra action \eqref{eq:gamma-lie-action}. This shows that $\varphi_{\sen}|_W$ is a morphism of  $\widehat{\overline{K}}$-linear Lie algebras.
	
	It remains to check the $G$-equivariance of $\varphi_{\sen}|_W$. For any $g\in G$, we have 
	\begin{align}
		g\circ \varphi_{\widetilde{\partial}_i}\circ g^{-1}&=\chi(g)\varphi_{\widetilde{\partial}_i}, \ \forall 1\leq i\leq e,\\
		g\circ \varphi_{\widetilde{\partial}_0}\circ g^{-1}&=\varphi_{\widetilde{\partial}_0}-\sum_{i=1}^e\widetilde{\xi}_i(g)\varphi_{\widetilde{\partial}_i},
	\end{align}
	where the first follows from \eqref{eq:lem:operator-1}, and the second follows from \ref{lem:varphi-decomposition} and the identity $\widetilde{\xi}(g\sigma_0g^{-1})=\widetilde{\xi}(g)(1-\chi(\sigma_0))$ for any $\sigma_0\in \widetilde{\Sigma}_{0,\undertilde{\infty}}$ by \eqref{eq:cont-cocycle-cond}. Therefore,
		\begin{align}
			\varphi_{\sen}|_W(g\cdot f)=&\sum_{i=0}^e g(f(g^{-1}\widetilde{T}_i))\otimes \varphi_{\widetilde{\partial}_i}\\
			=&(g\otimes 1)\left(f(\widetilde{T}_0)\otimes \varphi_{\widetilde{\partial}_0}+\sum_{i=1}^e \chi(g)(f(\widetilde{T}_i)+\widetilde{\xi}_i(g^{-1})f(\widetilde{T}_0))\otimes \varphi_{\widetilde{\partial}_i}\right) \quad\trm{(by \eqref{eq:para:special-functor})}\nonumber\\
			=&(g\otimes 1)\left(f(\widetilde{T}_0)\otimes (\varphi_{\widetilde{\partial}_0}-\sum_{i=1}^e\widetilde{\xi}_i(g)\varphi_{\widetilde{\partial}_i})+\chi(g)\sum_{i=1}^e f(\widetilde{T}_i)\otimes \varphi_{\widetilde{\partial}_i}\right) \quad\trm{(by \eqref{eq:cont-cocycle-cond})}\nonumber\\
			=&\sum_{i=0}^e g(f(\widetilde{T}_i))\otimes (g\circ\varphi_{\widetilde{\partial}_i}\circ g^{-1})=g\circ \varphi_{\sen}|_W(f)\circ g^{-1}\nonumber
		\end{align}
	which shows the $G$-equivariance.
\end{proof}

\begin{myrem}\label{rem:sen-brinon-operator-field}
	The same argument also shows that the $\widehat{\overline{K}}$-linear map
	\begin{align}
		W\longrightarrow W\otimes_{\widehat{\overline{K}}}\scr{E}_{\ca{O}_K}(-1)
	\end{align}
	sending $x$ to $\sum_{i=0}^e (\id_{\widehat{\overline{K}}}\otimes \varphi_{\widetilde{\partial}_i}|_{\widetilde{V}})(x)\otimes \widetilde{T}_i$, is $G$-equivariant and does not depend on the choice of $\widetilde{V}$ or $\widetilde{t}_i$. It naturally induces the map $\varphi_{\sen}|_W$ \eqref{eq:sen-brinon-operator}. We note that it is not a Higgs field.
\end{myrem}

\begin{mydefn}\label{defn:sen-brinon-operator}
	Let $W$ be an object of $\repnpr(G,\widehat{\overline{K}})$. We denote by $\Phi(W)$ the image of $\varphi_{\sen}|_W$, and by $\Phi^{\geo}(W)$ the image of $\ho_{\ca{O}_K}(\widehat{\Omega}^1_{\ca{O}_K}(-1),\widehat{\overline{K}})$ under $\varphi_{\sen}|_W$. We call an element of $\Phi(W)\subseteq \mrm{End}_{\widehat{\overline{K}}}(W)$ a \emph{Sen operator} of $W$. We call an element of $\Phi^{\geo}(W)\subseteq \mrm{End}_{\widehat{\overline{K}}}(W)$ a \emph{geometric Sen operator} of $W$. And we call the image of $1\in\widehat{\overline{K}}$ in $\Phi^{\ari}(W)=\Phi(W)/\Phi^{\geo}(W)$ the \emph{arithmetic Sen operator} of $W$.
\end{mydefn}

	Namely, we defined a canonical morphism of exact sequences of $\widehat{\overline{K}}$-linear Lie algebras
	\begin{align}
	\xymatrix{
		0\ar[r]&\ho_{\ca{O}_K}(\widehat{\Omega}^1_{\ca{O}_K}(-1),\widehat{\overline{K}})\ar[r]^-{\jmath^*}\ar@{->>}[d]&\scr{E}^*_{\ca{O}_K}(1)\ar[r]^-{\iota^*}\ar@{->>}[d]&\widehat{\overline{K}}\ar@{->>}[d]\ar[r]&0\\
		0\ar[r]&\Phi^{\geo}(W)\ar[r]&\Phi(W)\ar[r]&\Phi^{\ari}(W)\ar[r]&0
		}
	\end{align}
	which factors through \eqref{diam:lem:fal-ext-lie-alg-special} under the assumption of \eqref{eq:thm:sen-brinon-operator}.
	
\begin{myprop}\label{prop:sen-brinon-operator-func}
	Let $K'$ be a complete discrete valuation field extension of $K$ whose residue field admits a finite $p$-basis, $\overline{K'}$ an algebraic closure of $K'$ containing $\overline{K}$, $G'=\gal(\overline{K'}/K')$, $W$ an object of $\repnpr(G,\widehat{\overline{K}})$, $W'=\widehat{\overline{K'}}\otimes_{\widehat{\overline{K}}}W$ the associated object of $\repnpr(G',\widehat{\overline{K'}})$. Assume that $K'\otimes_{\ca{O}_K}\widehat{\Omega}^1_{\ca{O}_K}\to K'\otimes_{\ca{O}_{K'}}\widehat{\Omega}^1_{\ca{O}_{K'}}$ is injective. Then, there is a natural commutative diagram
	\begin{align}\label{diam:sen-brinon-operator-func}
		\xymatrix{
			\scr{E}^*_{\ca{O}_{K'}}(1)\ar[rr]^-{\varphi_{\sen}|_{W'}}\ar@{->>}[d]&&\mrm{End}_{\widehat{\overline{K'}}}(W')\\
			\widehat{\overline{K'}}\otimes_{\widehat{\overline{K}}}\scr{E}^*_{\ca{O}_K}(1)
			\ar[rr]^-{\id_{\widehat{\overline{K'}}}\otimes\varphi_{\sen}|_W}&&\widehat{\overline{K'}}\otimes_{\widehat{\overline{K}}}\mrm{End}_{\widehat{\overline{K}}}(W)
			\ar[u]_-{\wr}
		}
	\end{align}
	where $\varphi_{\sen}$ are the canonical Lie algebra actions defined in {\rm\ref{thm:sen-brinon-operator}}, the left vertical arrow is the surjection induced by taking dual of the natural injection $\widehat{\overline{K'}}\otimes_{\widehat{\overline{K}}}\scr{E}_{\ca{O}_K}(-1)\to \scr{E}_{\ca{O}_{K'}}(-1)$ (cf. {\rm\ref{rem:fal-ext}}), and the right vertical arrow is the canonical isomorphism. In particular, the inverse of the right vertical arrow induces a natural isomorphism 
	\begin{align}\label{eq:prop:sen-brinon-operator-func}
		 \Phi(W')\iso\widehat{\overline{K'}}\otimes_{\widehat{\overline{K}}}\Phi(W)
	\end{align}
	which is compatible with geometric and arithmetic Sen operators.
\end{myprop}
\begin{proof}
	Let $t'_{1,p^n},\dots,t'_{d,p^n}\in \overline{K'}$ be the images of $t_{1,p^n},\dots,t_{d,p^n}\in \overline{K}$. Then, there is a commutative diagram
	\begin{align}
		\xymatrix{
			K'\ar[r]_-{\Sigma'}\ar@/^1pc/[rr]|-{\Gamma'}\ar@/^2pc/[rrr]|-{G'}&K'_{\infty}\ar[r]_-{\Delta'}&K'_{\infty,\underline{\infty}}\ar[r]&\overline{K'}\\
			K\ar[u]\ar[r]^-{\Sigma}\ar@/_1pc/[rr]|-{\Gamma}\ar@/_2pc/[rrr]|-{G}&K_{\infty}\ar[r]^-{\Delta}\ar[u]&K_{\infty,\underline{\infty}}\ar[r]\ar[u]&\overline{K}\ar[u]
		}
	\end{align}
	Since $\df\log(t_1'),\dots,\df\log(t_d')$ are $K'$-linearly independent in $K'\otimes_{\ca{O}_{K'}}\widehat{\Omega}^1_{\ca{O}_{K'}}$ by assumption, $\Delta'$ is also of dimension $d$ by \ref{prop:rank}. In particular, we have a natural isomorphism $\lie(\Gamma')\iso\lie(\Gamma)$ which identifies their standard bases $\{\partial'_i\}_{1\leq i\leq d}$ and $\{\partial_i\}_{1\leq i\leq d}$ defined in \ref{para:gamma-basis}. Let $V$ be an object of $\repnan{\Delta}(\Gamma,K_\infty)$ such that $W=\widehat{\overline{K}}\otimes_{K_\infty}V$. Then, the object $V'=K'_{\infty}\otimes_{K_\infty}V$ of $\repnan{\Delta'}(\Gamma',K'_\infty)$ satisfies that $W'=\widehat{\overline{K'}}\otimes_{K'_\infty}V'$. By \ref{rem:derivative}, the natural identification $\mrm{End}_{\widehat{\overline{K'}}}(W')=\widehat{\overline{K'}}\otimes_{\widehat{\overline{K}}}\mrm{End}_{\widehat{\overline{K}}}(W)$ identifies $\id_{\widehat{\overline{K'}}}\otimes \varphi_{\partial'_i}|_{V'}$ with $\id_{\widehat{\overline{K'}}}\otimes \varphi_{\partial_i}|V$. This shows that the diagram \eqref{diam:sen-brinon-operator-func} is commutative which induces an isomorphism \eqref{eq:prop:sen-brinon-operator-func}.
\end{proof}

\begin{mylem}\label{lem:sen-brinon-operator-basic}
	Let $W$ be an object of $\repnpr(G,\widehat{\overline{K}})$. 
	\begin{enumerate}
		\renewcommand{\labelenumi}{{\rm(\theenumi)}}
		\item Geometric Sen operators of $W$ are nilpotent and commute with each other.\label{item:lem:sen-brinon-operator-1}
		\item If $\phi\in\Phi(W)$ is a lifting of the arithmetic Sen operator and $\theta\in\Phi^{\geo}(W)$ is a geometric Sen operator, then $[\phi,\theta]=\theta$.\label{item:lem:sen-brinon-operator-2}
		\item Moreover, we have $\phi(W^G)=0$ and $\theta(W^H)=0$.\label{item:lem:sen-brinon-operator-3}
	\end{enumerate}
\end{mylem}
\begin{proof}
	(\ref{item:lem:sen-brinon-operator-1}) follows from \ref{prop:special-functor}.(\ref{item:prop:special-functor-2}) and the definition of geometric Sen operator \eqref{eq:thm:sen-brinon-operator}. (\ref{item:lem:sen-brinon-operator-2}) follows from the Lie algebra structure of $\scr{E}_{\ca{O}_K}^*(1)$ defined in \ref{para:fal-ext-dual}. (\ref{item:lem:sen-brinon-operator-3}) follows from the definition of Sen operators, and we will give a detailed proof in \ref{prop:sen-fixed}. 
\end{proof}

\begin{mylem}\label{lem:sen-brinon-filtration}
	Any object $W$ of $\repnpr(G,\widehat{\overline{K}})$ admits a canonical and functorial finite ascending filtration $\{\mrm{F}_n\}_{n\in\bb{N}}$ stable under the Lie algebra action $\varphi_{\sen}|_W$ \eqref{eq:sen-brinon-operator} such that any geometric Sen operator sends $\mrm{F}_{n+1}W$ into $\mrm{F}_nW$. In particular, the arithmetic Sen operator of $W$ acts naturally on the graded object $\mrm{Gr}^{\mrm{F}}W=\oplus_{n\in\bb{N}}\mrm{Gr}^{\mrm{F}}_nW$, where $\mrm{Gr}^{\mrm{F}}_nW=\mrm{F}_nW/\mrm{F}_{n-1}W$.
\end{mylem}
\begin{proof}
	We set $\mrm{F}_0W=0$, and
	\begin{align}
		\mrm{F}_1W=\bigcap_{\theta\in\Phi^{\geo}(W)}\ke(\theta).
	\end{align}
	By \ref{lem:sen-brinon-operator-basic}, one checks easily that $\mrm{F}_1W$ is stable by any Sen operator of $W$. If $W\neq 0$, then $\mrm{F}_1W$ is non-zero, as $\theta$ is nilpotent. It is also functorial in $W$ by \ref{thm:sen-brinon-operator}. Then, for any $n\in\bb{N}_{>0}$, $\mrm{F}_nW$ is defined inductively by $\mrm{Gr}^{\mrm{F}}_nW=\mrm{F}_1(W/\mrm{F}_{n-1}W)=\bigcap_{\theta\in\Phi^{\geo}(W)}\ke(\theta|_{W/\mrm{F}_{n-1}W})$.
\end{proof}

\begin{myprop}\label{prop:sen-char-poly}
	Let $W$ be an object of $\repnpr(G,\widehat{\overline{K}})$. Any lifting of the arithmetic Sen operator $\phi\in\Phi(W)$ of $W$ has the same characteristic polynomial, whose coefficients are in $K$.
\end{myprop}
\begin{proof}
	By \ref{lem:sen-brinon-filtration}, there exists a $\widehat{\overline{K}}$-basis of $W$ with respect to which the matrix of $\phi$ is upper triangular and the matrix of any geometric operator $\theta$ of $W$ is strictly upper triangular. We see that $\phi$ and $\phi+\theta$ have the same characteristic polynomial. The coefficients of this polynomial lie in $K$ by \cite[Proposition 5.(a)]{brinon2003sen}.
\end{proof}

\begin{mylem}\label{lem:smallest}
	Let $F$ be a field, $V,W$ two $F$-linear spaces, $\Phi\subseteq W\otimes_F V$ a subset, $\ak{F}_{\Phi}=\{f\in \ho_{F}(V,F)\ |\ f_W(\Phi)=0\}$, where $f_W=\id_W\otimes f\in \ho_{F}(W\otimes_FV,W)$. Then, $V_{\Phi}=\bigcap_{f\in \ak{F}_{\Phi}}\ke(f)$ is the smallest $F$-linear subspace $V'$ of $V$ such that $\Phi\subseteq W\otimes_F V'$. Moreover, an $F$-linear subspace $V'$ of $V$ is equal to $V_\Phi$ if and only if $\ak{F}_{\Phi}$ is equal to $\ak{F}_{V'}=\{f\in \ho_{F}(V,F)\ |\ f(V')=0\}$.
\end{mylem}
\begin{proof}
	Firstly, we claim that $\Phi\subseteq W\otimes_F V_{\Phi}$. Consider the exact sequence
	\begin{align}
		\xymatrix{
			0\ar[r]& V_\Phi\ar[r]& V\ar[rr]^-{(f)_{f\in\ak{F}_{\Phi}}}&& \prod_{f\in\ak{F}_{\Phi}}F.
		}
	\end{align}
	Since $W$ is flat over $F$, we have an exact sequence
	\begin{align}\label{eq:lem:smallest}
		\xymatrix{
			0\ar[r]& W\otimes_F V_\Phi\ar[r]& W\otimes_F V \ar[rr]^-{(\id_W\otimes f)_{f\in\ak{F}_{\Phi}}}&& W\otimes_F (\prod_{f\in\ak{F}_{\Phi}}F) .
		}
	\end{align}
	Since $W\otimes_F (\prod_{f\in\ak{F}_{\Phi}}F)\subseteq \prod_{f\in\ak{F}_{\Phi}}W$, the subset $\Phi\subseteq W\otimes_F V$ is mapped to zero in \eqref{eq:lem:smallest}, which proves the claim.
	
	Secondly, for any $F$-linear subspace $V'$ of $V$, it is clear that $V'\subseteq \bigcap_{f\in\ak{F}_{V'}}\ke(f)$. This is actually an equality, since for any element $v\in V\setminus V'$ there exists $f\in \ak{F}_{V'}$ such that $f(v)\neq 0$. 
	
	Assume that $\Phi\subseteq W\otimes_F V'$. Then, $\ak{F}_{V'}\subseteq \ak{F}_{\Phi}$ so that $V_{\Phi}=\bigcap_{f\in\ak{F}_\Phi}\ke(f)\subseteq \bigcap_{f\in\ak{F}_{V'}}\ke(f)=V'$. It shows that $V_\Phi$ is the smallest $F$-linear subspace $V'$ of $V$ such that $\Phi\subseteq W\otimes_F V'$. In particular, we have $\ak{F}_{V_\Phi}\subseteq \ak{F}_\Phi$. On the other hand, the definition of $V_\Phi$ implies that $\ak{F}_\Phi\subseteq \ak{F}_{V_\Phi}$. Thus, $\ak{F}_{V_\Phi}= \ak{F}_\Phi$ and the final assertion follows.
\end{proof}

\begin{mythm}[{\cite[Theorem 11]{sen1980sen}, \cite[3.1]{ohkubo2014sen}}]\label{thm:sen-lie}
	Let $I$ be the inertial subgroup of $G$, $(V,\rho)$ an object of $\repnpr(G,\bb{Q}_p)$, $W=\widehat{\overline{K}}\otimes_{\bb{Q}_p}V$ the associated object of $\repnpr(G,\widehat{\overline{K}})$. Then, $\lie(\rho(I))$ is the smallest $\bb{Q}_p$-subspace $S$ of $\mrm{End}_{\bb{Q}_p}(V)$ such that the space of Sen operators $\Phi(W)$ is contained in $\widehat{\overline{K}}\otimes_{\bb{Q}_p} S$.
\end{mythm}
\begin{myrem}\label{rem:sen-lie}
	We don't know whether or not $\lie(\rho(I\cap H))$ is the smallest $\bb{Q}_p$-subspace $S$ of $\mrm{End}_{\bb{Q}_p}(V)$ such that the space of geometric Sen operators $\Phi^{\geo}(W)$ is contained in $\widehat{\overline{K}}\otimes_{\bb{Q}_p} S$, where $H=\gal(\overline{K}/K_\infty)$. Recall that Sen-Ohkubo's proof of \ref{thm:sen-lie} relies on Sen's ramification theorem on a Galois extension of a complete discrete valuation field whose Galois group is a $p$-adic analytic group (\cite[Lemma 3]{sen1973lie}, \cite[1.3]{ohkubo2014sen}). Thus, it seems that we couldn't apply their techniques directly to this question. Nevertheless, we have the following weaker result.
\end{myrem}

\begin{mycor}\label{cor:sen-lie}
	With the notation in {\rm\ref{thm:sen-lie}}, the space of geometric Sen operators $\Phi^{\geo}(W)$ is contained in $\widehat{\overline{K}}\otimes_{\bb{Q}_p}\lie (\rho(I\cap H))$.
\end{mycor}
\begin{proof}
	Note that $\Phi^{\geo}(W)=[\Phi(W),\Phi(W)]$ by the Lie algebra structure on $\scr{E}^*_{\ca{O}_K}(1)$ (cf. \ref{para:fal-ext-dual}), and that $[\lie(\rho(I)),\lie(\rho(I))]\subseteq \lie (\rho(I\cap H))$ as $I/(I\cap H)\subseteq G/H=\Sigma$ is abelian. The conclusion follows directly from the fact that $\Phi(W)\subseteq \widehat{\overline{K}}\otimes_{\bb{Q}_p}\lie(\rho(I))$ by \ref{thm:sen-lie}.
\end{proof}

\section{Extending Sen operators to Infinite-Dimensional Representations}\label{sec:inf}
We extend Sen operators to certain infinite-dimensional representations and their completions over a complete discrete valuation ring. The goal is to show that geometric Sen operators still annihilate the Galois invariant part of the representation.

\begin{mypara}\label{para:continuation}
	By saying that a topological abelian group $M$ is ``complete'', we always mean that it is separated and every Cauchy net of $M$ admits a limit point (\cite[8.2.6]{gabber2004foundations}). The forgetful functor from the category of complete topological abelian groups to the category of topological abelian groups admits a left adjoint, called the \emph{completion} and denoted by $M\mapsto\widehat{M}$. The canonical map $M\to\widehat{M}$ has dense image and induces the topology on $M$ from that of $\widehat{M}$ (\cite[8.2.8]{gabber2004foundations}). The adjoint property implies that for any continuous group homomorphism of topological abelian groups $f:M\to N$, there is a unique continuous group homomorphism of the completions $\widehat{f}:\widehat{M}\to \widehat{N}$ making the following diagram commutative
	\begin{align}
		\xymatrix{
			M\ar[d]\ar[r]^-{f}&N\ar[d]\\
			\widehat{M}\ar[r]^-{\widehat{f}}&\widehat{N}
		}
	\end{align}
	where the vertical arrows are the canonical maps. We call $\widehat{f}$ the \emph{continuation} (or \emph{completion}) of $f$.
\end{mypara}

\begin{mypara}\label{para:normed-module}
	We briefly review the definition of normed modules mainly following \cite[\textsection 5]{tsuji2018localsimpson}, and we also refer to \cite{bgr1984analysis} for a systematic development. A (non-Archimedean) \emph{norm} on an abelian group $M$ is a map $|\ |:M\to\bb{R}_{\geq 0}$ such that $|x|=0$ if and only if $x=0$, and that $|x-y|\leq \max\{|x|,|y|\}$ for any $x,y\in M$. For any $r\in\bb{R}_{> 0}$, we denote the closed ball of radius $r$ by
	\begin{align}
		M^{\leq r}=\{x\in M\ | \ |x|\leq r\}.
	\end{align}
	The metric topology makes $M$ into a separated topological abelian group, where $\{M^{\leq r}\}_{r\in\bb{R}_{>0}}$ forms a fundamental system of closed neighbourhoods of $0$. The norm map extends uniquely over the completion $\widehat{M}$ of $M$, and $M$ naturally identifies with a dense normed subgroup of $\widehat{M}$.
	
	A \emph{normed ring} $R$ is a ring endowed with a norm such that $|xy|\leq |x||y|$ for any $x,y\in R$. We remark that there is a natural normed ring structure on $\widehat{R}$ induced by that of $R$. Given a normed ring $R$, a \emph{normed $R$-module} is an $R$-module $M$ endowed with a norm such that $|ax|\leq |a||x|$ for any $a\in R$ and $x\in M$. Moreover, we call $R$ a \emph{Banach ring} if it is complete, and we call $M$ an \emph{$R$-Banach module} if $M$ is complete.
\end{mypara}

\begin{mypara}\label{para:normed-module-2}
	For any valuation field $K$ of height $1$ extension of $\bb{Q}_p$, we fix a valuation map $v_K:K^\times\to \bb{Q}$ normalized by $v_K(p)=1$, and we endow $K$ with the norm $|\ |_K$ defined by $|x|_K=p^{-v_K(x)}$ for any $x\in K^\times$. Let $V$ be a normed $K$-module. Since $p^nV^{\leq 1}=V^{\leq p^{-n}}$ for any $n\in\bb{N}$, $V^{\leq 1}$ is a $p$-adically separated flat $\ca{O}_K$-module, and the induced metric topology on $V^{\leq 1}$ coincides with its $p$-adic topology. We have $V=V^{\leq 1}[1/p]$. 
	
	Conversely, given a $p$-adically separated flat $\ca{O}_K$-module $M$, we can define a norm on $M[1/p]$ by setting $|x|=p^{-v_M(x)}$ for any $x\in M[1/p]\setminus\{0\}$ where $v_M(x)$ is the biggest integer such that $x\in p^{v_M(x)}M$. The metric topology defined by this norm on $M[1/p]$ coincides with its $p$-adic topology defined by $M$, and makes $M[1/p]$ into a normed $K$-module with $M[1/p]^{\leq 1}=M$. We remark that $V$ (resp. $M[1/p]$) is complete if and only if $V^{\leq 1}$ (resp. $M$) is $p$-adically complete.
\end{mypara}

\begin{mylem}\label{lem:banach-sub}
	Let $K$ be a complete valuation field of height $1$ extension of $\bb{Q}_p$, $V$ a normed $K$-module. Then, the induced topology on any finite-dimensional $K$-subspace $V_0$ of $V$ coincides with its canonical topology (cf. {\rm\ref{para:canonical-top}}).
\end{mylem}
\begin{proof}
	The norm on $V$ induces a norm on $V_0$, whose associated metric topology defines the induced topology on $V_0$. On the other hand, any $K$-linear isomorphism $V_0\cong K^n$ defines a norm on $V_0$ which induces the canonical topology on $V_0$. The conclusion follows from the fact that any two norms on a finite-dimensional $K$-space are equivalent (\cite[2.3.3.5]{bgr1984analysis}).
\end{proof}

\begin{myprop}\label{prop:sen-brinon-operator-inf}
	With the notation in {\rm\ref{para:notation-K}}, let $W$ be a normed $\widehat{\overline{K}}$-module endowed with a continuous semi-linear action of $G$. Assume that the $(G,\widehat{\overline{K}})$-finite part of $W$ is equal to $W$ itself (cf. {\rm\ref{defn:repn}}). Then, there exists a unique $\widehat{\overline{K}}$-linear Lie algebra action of $\scr{E}_{\ca{O}_K}^*(1)$ on $W$,
	\begin{align}\label{eq:sen-brinon-operator-inf}
		\varphi_{\sen}|_W:\scr{E}_{\ca{O}_K}^*(1)\longrightarrow \mrm{End}_{\widehat{\overline{K}}}(W),
	\end{align}
	such that for any $G$-equivariant continuous $\widehat{\overline{K}}$-linear homomorphism $W_0\to W$ from an object $W_0$ of $\repnpr(G,\widehat{\overline{K}})$, $\varphi_{\sen}|_W$ is compatible with the Lie algebra action $\varphi_{\sen}|_{W_0}$ of $\scr{E}_{\ca{O}_K}^*(1)$ defined in {\rm\ref{thm:sen-brinon-operator}}.
\end{myprop}
\begin{proof}
	The maps $W_0\to W$ form a category $\ca{I}$. Indeed, it is the localization of the category $\repnpr(G,\widehat{\overline{K}})$ at the presheaf given by the restriction of the presheaf $\ho(-,W)$ on $\repn(G,\widehat{\overline{K}})$ represented by $W$ (cf. \cite[\Luoma{1}.3.4.0]{sga4-1}). We claim that $\ca{I}$ is filtered. Indeed, for any object $W_0\to W$ of $\ca{I}$, its image $W_1\subseteq W$ is $G$-stable. Since the topology on $W_1$ induced from $W$ coincides with the canonical topology as a finite-dimensional $\widehat{\overline{K}}$-space by \ref{lem:banach-sub}, we see that $W_1$ is a finite projective $\widehat{\overline{K}}$-representation of $G$ and a subrepresentation of $W$ and that $W_0\to W_1$ is a morphism in $\repnpr(G,\widehat{\overline{K}})$. As direct sums exist in $\ca{I}$, one checks easily that $\ca{I}$ is filtered. By assumption and the previous argument, $W=\colim_{\ca{I}} W_0$ as $\widehat{\overline{K}}$-modules. Since the Lie algebra $\scr{E}_{\ca{O}_K}^*(1)$ acts functorially on each $W_0$, it defines a unique action on $W$ compatible with that on each $W_0$.
\end{proof}
\begin{myrem}\label{rem:sen-brinon-operator-inf}
	Let $W$ be a $\widehat{\overline{K}}$-Banach space endowed with a continuous semi-linear action of $G$ such that the $(G,\widehat{\overline{K}})$-finite part $W^{\mrm{f}}$ of $W$ is dense in $W$. If we endow $W^{\mrm{f}}$ with the induced topology, then its completion coincides with $W$ (\cite[8.2.8.(\luoma{3})]{gabber2004foundations}). By \ref{prop:sen-brinon-operator-inf}, we obtain a canonical Lie algebra action $\varphi_{\sen}|_{W^{\mrm{f}}}$ on $W^{\mrm{f}}$. If the operators on $W^{\mrm{f}}$ defined by $\varphi_{\sen}|_{W^{\mrm{f}}}$ are continuous, then we can extend this action uniquely to a Lie algebra action $\varphi_{\sen}|_W$ on $W$ by continuation (cf. \ref{para:continuation}). 
	
	However, in this work we haven't found a simple condition to guarantee the continuity of the Sen operators on $W^{\mrm{f}}$. Instead, we consider two types of dense subrepresentations of $W^{\mrm{f}}$ and discuss the continuity of Sen operators on them. Roughly speaking, the first type (considered in the rest of this section) is the union of representations with ``small lattices'',  which is ad hoc but suitable for doing descent and decompletion (so that nice properties are preserved after continuation, cf. \ref{thm:sen-operator-fix-infty}). The other type (considered in the end of section \ref{sec:sen-B}) is the union of representations defined over $\bb{Q}_p$, which is more canonical but we need to reduce to the first type for proving properties (cf. \ref{thm:sen-operator-B-fix}).
\end{myrem}

\begin{mypara}\label{para:small}
	Let $K$ be a valuation field of height $1$ extension of $\bb{Q}_p$ with a valuation map $v_K:K^\times\to \bb{Q}$ normalized by $v_K(p)=1$, $A$ a $p$-adically complete flat $\ca{O}_K$-algebra, $M$ a $p$-torsion free $p$-adically complete $A$-module. Consider an $A$-linear endomorphism $\phi$ on $M$ such that $\phi(M)\subseteq \alpha M$ for some element $\alpha\in\ak{m}_K$ with $v_K(\alpha)>\frac{1}{p-1}$. 
	As $v_K(n!)\leq \frac{n-1}{p-1}$ for any $n\in\bb{N}_{>0}$, the series for any $x\in M$,
	\begin{align}
		\exp(\phi)(x)&=\sum_{n=0}^{\infty}\frac{1}{n!}\phi^n(x),\\
		\log(1+\phi)(x)&=\sum_{n=1}^{\infty}\frac{(-1)^{n-1}}{n}\phi^n(x)
	\end{align}
	are well-defined and converge in $M$ with respect to the $p$-adic topology. They define two $A$-linear endomorphisms $\exp(\phi)\in \id+\alpha\mrm{End}_A(M)$ and $\log(1+\phi)\in \alpha\mrm{End}_A(M)$ of $M$ such that $\exp(\log(1+\phi))=1+\phi$ and $\log(\exp(\phi))=\phi$. Thus, we deduce easily that for any $n\in\bb{N}$, $p^{-n}((1+\phi)^{p^n}-1)=p^{-n}(\exp(p^n\log(1+\phi))-1)\in \alpha\mrm{End}_A(M)$, and that 
	\begin{align}\label{eq:para:small-derivative}
		\log(1+\phi)(x)=\lim_{n\to \infty}p^{-n}((1+\phi)^{p^n}-1)(x),
	\end{align}
	namely, $\log(1+\phi)$ is the infinitesimal action of $1+\phi$ on $M$ (cf. \ref{defn:infinitesimal}).
\end{mypara}

\begin{mydefn}[{cf. \cite[\Luoma{2}.13.1, \Luoma{2}.13.2]{abbes2016p}}]\label{defn:small}
	Let $K$ be a valuation field of height $1$ extension of $\bb{Q}_p$, $\ak{a}$ an ideal of $\ca{O}_K$, $A$ a $p$-adically complete flat $\ca{O}_K$-algebra endowed with a continuous action of a topological group $G$ by homomorphisms of $\ca{O}_K$-algebras. 
	\begin{enumerate}
		\renewcommand{\labelenumi}{{\rm(\theenumi)}}
		\item For any object $M$ of $\repnpr(G,A)$, we say that $M$ is \emph{$\ak{a}$-small} if $M$ is a finite free $A$-module admitting a basis consisting of elements that are $G$-invariant modulo $\alpha M$ for some $\alpha\in\ak{a}$. We denote by $\repnsm{\ak{a}}(G,A)$ the full subcategory of $\repnpr(G,A)$ consisting of $\ak{a}$-small objects.\label{item-defn:small-1}
		\item For any object $W$ of $\repnpr(G,A[1/p])$, we say that $W$ is \emph{$\ak{a}$-small} if there exists a $G$-stable $A$-submodule $W^+$ of $W$ generated by finitely many elements that are $G$-invariant modulo $\alpha W^+$ for some $\alpha\in\ak{a}$. We denote by $\repnsm{\ak{a}}(G,A[1/p])$ the full subcategory of $\repnpr(G,A[1/p])$ consisting of $\ak{a}$-small objects.\label{item-defn:small-2}
	\end{enumerate}
\end{mydefn}

\begin{myrem}\label{rem:small}
	In \ref{defn:small}.(\ref{item-defn:small-1}), the $G$-action on $M/\alpha M$ may not be trivial, since this action is $A$-semi-linear not $A$-linear.
	In \ref{defn:small}.(\ref{item-defn:small-2}), if $W$ is $\ak{a}$-small, then the canonical topology on $W=W^+[1/p]$ is induced by the $p$-adic topology of $W^+$ (cf. \ref{para:canonical-top}). In particular, there is a natural faithful functor
		\begin{align}
			\repnsm{\ak{a}}(G,A)\longrightarrow \repnsm{\ak{a}}(G,A[1/p]),\ M\mapsto M[1/p].
		\end{align}
	We remark that even if $A=\ca{O}_K$ (so that $W^+$ is a finite free $\ca{O}_K$-module), $W^+$ may not admit a basis consisting of elements that are $G$-invariant modulo $\alpha W^+$, i.e. $W^+$ may not be $\ak{a}$-small.
\end{myrem}

\begin{mypara}\label{para:notation-K-good}
	In the rest of this section, we take again the assumptions and notation in \ref{para:notation-K}.
	\begin{align}
		\xymatrix{
			\overline{K}&\\
			K_{\infty,\underline{\infty}}\ar[u]&\\
			K_\infty\ar[u]^-{\Delta}\ar@/^2pc/[uu]^-{H}&K\ar[l]^-{\Sigma}\ar[lu]_-{\Gamma}\ar@/_1pc/[luu]_{G}
		}
	\end{align}
	Consider the following assumption on $K$: 
	\begin{enumerate}
		\renewcommand{\labelenumi}{{\rm(\theenumi)}}
		\item[($\ast$)] Let $\kappa$ be the residue field of $K$, $\cap_{n\geq 0}\kappa^{p^n}$ the maximal perfect subfield of $\kappa$, $K_{\mrm{can}}$ the algebraic closure in $K$ of the fraction field of the Witt ring $W(\cap_{n\geq 0}\kappa^{p^n})$. Then, $K_{\mrm{can}}\to K$ is a weakly unramified extension of complete discrete valuation fields, i.e. a uniformizer of $K_{\mrm{can}}$ is still a uniformizer of $K$.
	\end{enumerate}
	This assumption is considered by Hyodo \cite[(0-5)]{hyodo1986hodge} when computing the cohomology $H^q(G,\widehat{\overline{K}})$. We remark that there exists a finite Galois extension of $K$ which satisfies this assumption by Epp's theorem on eliminating ramification \cite[1.9, 2.0.(1)]{epp1973ramifi}. The assumption implies that for any finite field extension $K'_{\mrm{can}}$ of $K_{\mrm{can}}$, we have $\ca{O}_{K'}=\ca{O}_{K'_{\mrm{can}}}\otimes_{\ca{O}_{K_{\mrm{can}}}}\ca{O}_K$ where $K'=K'_{\mrm{can}}K$ (\cite[\href{https://stacks.math.columbia.edu/tag/09E7}{09E7}, \href{https://stacks.math.columbia.edu/tag/09EQ}{09EQ}]{stacks-project}). In particular, the residue field of $K_n$ is separable over that of $K$ for any $n\in\bb{N}$, and thus $t_1,\dots,t_d$ still form a $p$-basis of the residue field of $K_n$. Hence, there is an isomorphism of $\ca{O}_{K_n}$-algebras for any $\underline{m}=(m_1,\dots,m_d)\in\bb{N}^d$,
	\begin{align}\label{eq:para:notation-K-good}
		\ca{O}_{K_n}[T_1,\dots,T_d]/(T_1^{p^{m_1}}-t_1,\dots,T_d^{p^{m_d}}-t_d)\iso \ca{O}_{K_{n,\underline{m}}},
	\end{align}
	sending $T_i$ to $t_{i,p^m}$ (\cite[1-2]{hyodo1986hodge}). In particular, the continuous homomorphism \eqref{eq:cont-cocycle} $\xi:\Delta\to \bb{Z}_p^d$ is an isomorphism, which also implies that $K_{\infty,\underline{\infty}}=K_\infty\otimes_K K_{0,\underline{\infty}}$.
\end{mypara}

\begin{mythm}[{Faltings, \cite{faltings2005simpson}, cf. \cite[\Luoma{2}.14.4]{abbes2016p}}]\label{thm:small-descent}
	Under the assumption {\rm($\ast$)} in {\rm\ref{para:notation-K-good}} on $K$, let $\ak{a}$ be the ideal of $\ca{O}_{K_\infty}$ consisting of elements $\alpha$ with normalized valuation $v_{K_\infty}(\alpha)>\frac{2}{p-1}$. Then, the functor
	\begin{align}
		\repnsm{\ak{a}}(\Delta,\ca{O}_{\widehat{K_\infty}})\longrightarrow \repnsm{\ak{a}}(H,\ca{O}_{\widehat{\overline{K}}})
	\end{align}
	is an equivalence of categories.
\end{mythm}
\begin{proof}
	It follows from the same arguments of \cite[\Luoma{2}.14.4]{abbes2016p}.
\end{proof}

\begin{mythm}[{\cite[11.2, 12.4]{tsuji2018localsimpson}}]\label{thm:small-generic-descent}
	Under the assumption {\rm($\ast$)} in {\rm\ref{para:notation-K-good}} on $K$, let $\ak{a}$ be the ideal of $\ca{O}_{K_\infty}$ consisting of elements $\alpha$ with normalized valuation $v_{K_\infty}(\alpha)>\frac{2}{p-1}$. Then, the functor
	\begin{align}\label{eq:thm:small-generic-descent}
		\repnsm{\ak{a}}(\Delta,\widehat{K_\infty})\longrightarrow \repnsm{\ak{a}}(H,\widehat{\overline{K}})
	\end{align}
	is an equivalence of categories.
\end{mythm}
\begin{proof}
	It follows from the same arguments of \cite[11.2, 12.4]{tsuji2018localsimpson}.
\end{proof}

\begin{mylem}\label{lem:nil-higgs-small}
	Let $V$ be an object of $\repnan{\Delta}(\Gamma,K_\infty)$. Then, the associated $\widehat{K_\infty}$-representation $\widehat{V}=\widehat{K_\infty}\otimes_{K_\infty}V$ of $\Delta$ is $\ak{a}$-small for any nonzero ideal $\ak{a}$ of $\ca{O}_{K_\infty}$.
\end{mylem}
\begin{proof}
	Since $\Delta$ is commutative and by \ref{rem:derivative} and \ref{prop:operator-nilpotent} the infinitesimal Lie algebra action $\varphi:\lie(\Delta)\to \mrm{End}_{K_\infty}(V)$ is nilpotent, we can take a basis $v_1,\dots,v_n$ of $V$ as in the argument of \ref{prop:sen-char-poly} such that $\varphi_\tau(v_i)\in\sum_{j>i}K_\infty v_j$ for any $\tau\in\Delta$ and $1\leq i\leq n$. Since $V$ is $\Delta$-analytic, we have $\tau(v_i)-v_i=\exp(\varphi_\tau)(v_i)-v_i\in \sum_{j>i}K_\infty v_j$. For any $k_0\in\bb{N}$, after replacing $v_1,\dots,v_n$ by $p^{-k_1}v_1,\dots,p^{-k_n}v_n$ for some integers $k_n\gg \cdots\gg k_1>k_0$, we may assume that $\tau(v_i)-v_i\in \sum_{j>i}p^{k_0}\ca{O}_{K_\infty} v_j$. Thus, the finite free $\ca{O}_{K_\infty}$-submodule $V^+$ of $V$ generated by $v_1,\dots,v_n$ is $\Delta$-stable and $v_i$ is $\Delta$-invariant modulo $p^{k_0}$. In particular, $\widehat{V}$ is $p^{k_0}\ca{O}_{K_\infty}$-small.
\end{proof}

\begin{mypara}\label{para:small-diam}
	Under the assumption {\rm($\ast$)} in {\rm\ref{para:notation-K-good}} on $K$, let $\ak{a}$ be the ideal of $\ca{O}_{K_\infty}$ consisting of elements $\alpha$ with normalized valuation $v_{K_\infty}(\alpha)>\frac{2}{p-1}$. There is a canonical commutative diagram
	\begin{align}\label{diam:small}
		\xymatrix{
			\repnpr(G,\widehat{\overline{K}})\ar[r]& \repnsm{\ak{a}}(H,\widehat{\overline{K}})&\repnsm{\ak{a}}(H,\ca{O}_{\widehat{\overline{K}}})\ar[l]\\
			\repnan{\Delta}(\Gamma,K_\infty)\ar[r]\ar[u]&\repnsm{\ak{a}}(\Delta,\widehat{K_\infty})\ar[u]&\repnsm{\ak{a}}(\Delta,\ca{O}_{\widehat{K_\infty}})\ar[l]\ar[u] 
		}
	\end{align}
	where the vertical arrows are equivalences of categories by \ref{thm:sen-brinon}, \ref{prop:gamma-analytic}, \ref{thm:small-generic-descent} and \ref{thm:small-descent}, and where the horizontal arrows of the left square are induced by these equivalences and \ref{lem:nil-higgs-small}. It allows us to calculate geometric Sen operators using ``small lattices''.	 
\end{mypara}

\begin{mylem}\label{lem:compare-derivative}
	Under the assumption {\rm($\ast$)} in {\rm\ref{para:notation-K-good}} on $K$, let $W$ be an object of $\repnpr(G,\widehat{\overline{K}})$ such that there exists an object $W^+$ of $\repnsm{\ak{a}}(H,\ca{O}_{\widehat{\overline{K}}})$ with $W=W^+[1/p]$ in $\repnsm{\ak{a}}(H,\widehat{\overline{K}})$, where $\ak{a}$ is the ideal of $\ca{O}_{K_\infty}$ consisting of elements $\alpha$ with normalized valuation $v_{K_\infty}(\alpha)>\frac{2}{p-1}$. Then, there is a commutative diagram
	\begin{align}\label{diam:lem:sen-operator-small}
		\xymatrix{
			\scr{E}_{\ca{O}_K}^*(1)\ar[r]^-{	\varphi_{\sen}|_W}\ar[d]_-{\psi}^-{\wr}& \mrm{End}_{\widehat{\overline{K}}}(W)\\
			\widehat{\overline{K}}\otimes_{\bb{Q}_p}\lie(\Gamma)&\Delta\ar[l]^-{1\otimes\log_{\Delta}}\ar[u]_-{\id_{\widehat{\overline{K}}}\otimes \log(-)|_{V^+}}
		}
	\end{align}
	where $\varphi_{\sen}|_W$ is the canonical Lie algebra action defined in {\rm\ref{thm:sen-brinon-operator}}, $\psi$ is the isomorphism \eqref{eq:lem:fal-ext-lie-alg-special}, $V^+$ is the essentially unique object of $\repnsm{\ak{a}}(\Delta,\ca{O}_{\widehat{K_\infty}})$ such that $W^+=\ca{O}_{\widehat{\overline{K}}}\otimes_{\ca{O}_{\widehat{K_\infty}}}V^+$, and $\log(\tau)|_{V^+}$ is the $\ca{O}_{\widehat{K_\infty}}$-linear endomorphism on $V^+$ defined in {\rm\ref{para:small}} for $\tau\in\Delta$.
\end{mylem}
\begin{proof}
	Let $V$ be the essentially unique object of $\repnan{\Delta}(\Gamma,K_\infty)$ such that $W=\widehat{\overline{K}}\otimes_{K_\infty} V$. Then, we conclude by the equivalence \eqref{eq:thm:small-generic-descent} that $\widehat{V}=V^+[1/p]$ in $\repnsm{\ak{a}}(\Delta,\widehat{K_\infty})$, where $\widehat{V}=\widehat{K_\infty}\otimes_{K_\infty}V$.
	\begin{align}
		\xymatrix{
			W\ar@{|->}[r]& W=W^+[\frac{1}{p}]&W^+\ar@{|->}[l]\\
			V\ar@{|->}[r]\ar@{|->}[u]&\widehat{V}=V^+[\frac{1}{p}]\ar@{|->}[u]&V^+\ar@{|->}[l]\ar@{|->}[u] 
		}
	\end{align}
	Consider the infinitesimal Lie algebra action $\varphi|_V:\lie(\Gamma)\to \mrm{End}_{K_\infty}(V)$.
	\begin{align}
		\xymatrix{
			\scr{E}_{\ca{O}_K}^*(1)\ar[r]^-{	\varphi_{\sen}|_W}\ar[d]_-{\psi}^-{\wr}& \mrm{End}_{\widehat{\overline{K}}}(W)\\
			\widehat{\overline{K}}\otimes_{\bb{Q}_p}\lie(\Gamma)\ar[ur]_-{\id_{\widehat{\overline{K}}}\otimes \varphi|_V}&\Delta\ar[l]^-{1\otimes\log_{\Delta}}\ar[u]_-{\id_{\widehat{\overline{K}}}\otimes \log(-)|_{V^+}}
		}
	\end{align}
	The upper triangle commutes by \ref{thm:sen-brinon-operator}. It remains to check that the lower triangle commutes, i.e. $\id_{\widehat{K_\infty}}\otimes\varphi_{\tau}|_V=\id_{\widehat{K_\infty}}\otimes\log(\tau)|_{V^+}$ as $\widehat{K_\infty}$-linear endomorphisms of $\widehat{V}=V^+[1/p]$ for any $\tau\in\Delta$. For any $x\in V$ whose image $y$ in $V^+[1/p]$ lies in $V^+$,
	\begin{align}
		(\id_{\widehat{K_\infty}}\otimes\varphi_{\tau}|_V)(1\otimes x)=&1\otimes \lim_{n\to \infty}p^{-n}(\tau^{p^n}-1)(x) \quad\trm{(by \eqref{eq:infinitesimal})}\\
		=&1\otimes\lim_{n\to \infty}p^{-n}(\tau^{p^n}-1)(y)\nonumber\\
		=&(\id_{\widehat{K_\infty}}\otimes\log(\tau)|_{V^+})(1\otimes y)\quad\trm{(by \eqref{eq:para:small-derivative})},\nonumber
	\end{align}
	which completes the proof.
\end{proof}

On the other hand, one can control the Galois invariant part of the completion of a filtered colimit of ``small lattices''.

\begin{mylem}[{cf. \cite[\Luoma{2}.8.23]{abbes2016p}}]\label{lem:small-invariant}
	Under the assumption {\rm($\ast$)} in {\rm\ref{para:notation-K-good}} on $K$, let $(W_\lambda^+)_{\lambda\in\Lambda}$ be a directed system of objects in $\repnsm{\ak{a}}(H,\ca{O}_{\widehat{\overline{K}}})$, where $\ak{a}$ is the ideal of $\ca{O}_{K_\infty}$ consisting of elements $\alpha$ with normalized valuation $v_{K_\infty}(\alpha)>\frac{2}{p-1}$. Let $(V_\lambda^+)_{\lambda\in\Lambda}$ be the essentially unique directed system of objects in $\repnsm{\ak{a}}(\Delta,\ca{O}_{\widehat{K_\infty}})$ such that $W_\lambda^+=\ca{O}_{\widehat{\overline{K}}}\otimes_{\ca{O}_{\widehat{K_\infty}}}V_\lambda^+$ by {\rm\ref{thm:small-descent}}. We put $W_\infty^+=\colim_{\lambda\in\Lambda}W_\lambda^+$ and $V_\infty^+=\colim_{\lambda\in\Lambda}V_\lambda^+$. Then, the natural map
	\begin{align}
		(\widehat{V_\infty^+}[1/p])^\Delta\longrightarrow (\widehat{W_\infty^+}[1/p])^H
	\end{align}
	is a bijection, where the completions are $p$-adic. 
\end{mylem}
\begin{proof}
	By virtue of \cite[\Luoma{2}.8.23]{abbes2016p} (cf. \eqref{eq:para:notation-K-good}), the map $(V_\lambda^+/p^nV_\lambda^+)^\Delta\to (\ca{O}_{\widehat{\overline{K}}}\otimes_{\ca{O}_{\widehat{K_\infty}}}V_\lambda^+/p^nV_\lambda^+)^H$ is almost injective whose cokernel is killed by any element $\alpha\in\ca{O}_{K_\infty}$ with $v_{K_\infty}(\alpha)>\frac{1}{p-1}$, and thus so is the map 
	\begin{align}\label{eq:thm:sen-operator-fix-infty-2}
		(\widehat{V_\infty^+})^\Delta=\lim_{n\to\infty}\colim_{\lambda\in\Lambda} (V_\lambda^+/p^nV_\lambda^+)^\Delta\longrightarrow\lim_{n\to\infty}\colim_{\lambda\in\Lambda}(W_\lambda^+/p^n W_\lambda^+)^H=(\widehat{W_\infty^+})^H.
	\end{align}
	Inverting $p$, we get $(\widehat{V_\infty^+}[1/p])^\Delta=(\widehat{V_\infty^+})^\Delta[1/p]=(\widehat{W_\infty^+})^H[1/p]=(\widehat{W_\infty^+}[1/p])^H$.
\end{proof}

\begin{mypara}\label{para:system-repn}
	Let $(W_\lambda^+)_{\lambda\in\Lambda}$ be a directed system of objects in $\repnpr(G,\ca{O}_{\widehat{\overline{K}}})$. We put $W_\infty^+=\colim_{\lambda\in\Lambda} W_\lambda^+$ as $\ca{O}_{\widehat{\overline{K}}}$-modules, and denote its $p$-adic completion by $\widehat{W_\infty^+}$. We set
	\begin{align}
		W_\lambda=W_\lambda^+[\frac{1}{p}],\quad W_\infty=W_\infty^+[\frac{1}{p}],\quad \widehat{W_\infty}=\widehat{W_\infty^+}[\frac{1}{p}],
	\end{align}
	endowed with the $p$-adic topology defined by $W_\lambda^+$, $W_\infty^+$ and $\widehat{W_\infty^+}$ respectively. We remark that $\widehat{W_\infty}$ is the completion of $W_\infty$ as topological abelian group (by the canonical isomorphism $W_\infty^+/p^nW_\infty^+\iso \widehat{W_\infty^+}/p^n\widehat{W_\infty^+}$ for any $n\in\bb{N}$, cf. \cite[8.2.8.(\luoma{3})]{gabber2004foundations}). As $(W_\lambda)_{\lambda\in\Lambda}$ forms a directed system of objects in $\repnpr(G,\widehat{\overline{K}})$, the canonical Lie algebra actions $\varphi_{\sen}|_{W_\lambda}$ defined in \ref{thm:sen-brinon-operator} induces a homomorphism of $\widehat{\overline{K}}$-linear Lie algebras
	\begin{align}\label{eq:system-repn-sen}
		\varphi_{\sen}|_{W_\infty}:\scr{E}_{\ca{O}_K}^*(1)\longrightarrow \mrm{End}_{\widehat{\overline{K}}}(W_\infty).
	\end{align}
	We denote by $\Phi(W_\infty)$ its image, and by $\Phi^{\geo}(W_\infty)$ the image of $\ho_{\ca{O}_K}(\widehat{\Omega}^1_{\ca{O}_K}(-1),\widehat{\overline{K}})$. On the other hand, the compatible $G$-actions $G\times W_\lambda^+\to W_\lambda^+$ induces an action $G\times W_\infty^+\to W_\infty^+$ by taking colimit, and thus induces an action $G\times \widehat{W_\infty^+}\to \widehat{W_\infty^+}$ by taking $p$-adic completion.
\end{mypara}
\begin{mylem}\label{lem:system-repn}
	We keep the notation in {\rm\ref{para:system-repn}}.
	\begin{enumerate}
		\renewcommand{\labelenumi}{{\rm(\theenumi)}}
		\item The $G$-actions on $W_\infty^+$ and $\widehat{W_\infty^+}$ are continuous with respect to the $p$-adic topology.\label{item:lem:system-repn-1}
		\item If $W_\infty^+$ is $p$-adically separated, then the Lie algebra action $\varphi_{\sen}|_{W_\infty}$ of $\scr{E}_{\ca{O}_K}^*(1)$ on $W_\infty$ coincides with the canonical Lie algebra action defined in {\rm\ref{prop:sen-brinon-operator-inf}}.\label{item:lem:system-repn-2} 
	\end{enumerate}
\end{mylem}
\begin{proof}
	(\ref{item:lem:system-repn-1}) For any $x\in W_\infty^+$, there exists $x_\lambda\in W_\lambda^+$ for some $\lambda\in \Lambda$ whose image in $W_\infty^+$ is $x$. For any $g\in G$ and $n\in \bb{N}$, as $G\times W_\lambda^+\to W_\lambda^+$ is continuous, there exists an open subgroup $G_0$ of $G$ such that $G_0g\cdot x_\lambda\subseteq gx_\lambda+p^nW_\lambda^+$. Thus, $G_0g\cdot(x+p^nW_\infty^+)\subseteq gx+p^nW_\infty^+$, which shows that $G\times W_\infty^+\to W_\infty^+$ and $G\times (W_\infty^+/p^nW_\infty^+)\to W_\infty^+/p^nW_\infty^+$ are continuous. Taking limit on $n$, we see that $G\times \widehat{W_\infty^+}\to \widehat{W_\infty^+}$ is also continuous.
	
	 (\ref{item:lem:system-repn-2}) If $W_\infty^+$ is $p$-adically separated, then it defines a norm on $W_\infty$ as in \ref{para:normed-module-2} so that we can apply \ref{prop:sen-brinon-operator-inf} to $W_\infty$. The conclusion follows directly from \ref{prop:sen-brinon-operator-inf} and the definition of \eqref{eq:system-repn-sen}.
\end{proof}

\begin{mythm}\label{thm:sen-operator-fix-infty}
	With the notation in {\rm\ref{para:system-repn}}, assume that $W_\lambda^+$ is $p^3\bb{Z}_p$-small (cf. {\rm\ref{defn:small}}) for each $\lambda\in\Lambda$. Then, any element of $\Phi^{\geo}(W_\infty)$ acts continuously on $W_\infty$, and it induces a $\widehat{\overline{K}}$-linear homomorphism by continuation,
	\begin{align}\label{eq:thm:sen-operator-fix-infty}
		\varphi_{\sen}^{\geo}|_{\widehat{W_\infty}}:\ho_{\ca{O}_K}(\widehat{\Omega}^1_{\ca{O}_K}(-1),\widehat{\overline{K}})\longrightarrow \mrm{End}_{\widehat{\overline{K}}}(\widehat{W_\infty}),
	\end{align}
	whose image $\Phi^{\geo}(\widehat{W_\infty})$ acts trivially on $(\widehat{W_\infty})^H$.
\end{mythm}
\begin{proof}
	We may replace $K$ by a finite extension and restrict the actions of $G$ to an open subgroup, as this does not change the action \eqref{eq:system-repn-sen}
	$\varphi_{\sen}|_{W_\infty}$ by \ref{prop:sen-brinon-operator-func}. Thus, we may assume that $K$ satisfies the assumption {\rm($\ast$)} in {\rm\ref{para:notation-K-good}} by Epp's theorem on eliminating ramification \cite[1.9, 2.0.(1)]{epp1973ramifi}. By \ref{thm:small-descent}, there is an essentially unique directed system $(V_\lambda^+)_{\lambda\in\Lambda}$ of objects in $\repnsm{\ak{a}}(\Delta,\ca{O}_{\widehat{K_\infty}})$ with $W_\lambda^+=\ca{O}_{\widehat{\overline{K}}}\otimes_{\ca{O}_{\widehat{K_\infty}}}V_\lambda^+$ in $\repnsm{\ak{a}}(H,\ca{O}_{\widehat{\overline{K}}})$. We obtain $\ca{O}_{\widehat{K_\infty}}$-linear group actions $\Delta\times V_\infty^+\to V_\infty^+$ and $\Delta\times \widehat{V_\infty^+}\to \widehat{V_\infty^+}$ by taking colimit and $p$-adic completion. 
	
	For each $\tau\in\Delta$, the compatible endomorphisms $\log(\tau)|_{V_\lambda^+}$ define an $\ca{O}_{\widehat{K_\infty}}$-endomorphism $\log(\tau)|_{V_\infty^+}$ on $V_\infty^+$ (cf. \ref{para:small}). As $\id_{\widehat{\overline{K}}}\otimes\log(\tau)|_{V_\infty^+}$ sends $W_\infty^+$ to itself, we see that any element of $\Phi^{\geo}(W_\infty)$ acts continuously on $W_\infty$ by \ref{lem:compare-derivative}. Thus, the Lie algebra action $\varphi_{\sen}|_{W_\infty}$ induces a canonical $\widehat{\overline{K}}$-linear homomorphism \eqref{eq:thm:sen-operator-fix-infty} by continuation.
	
	Notice that $\Delta$ acts trivially on $\widehat{V_\infty^+}/\alpha \widehat{V_\infty^+}$ for any $\alpha\in\ca{O}_{K_\infty}$ with $\frac{1}{p-1}<v_{K_\infty}(\alpha)\leq \frac{2}{p-1}$. Thus, $\log(\tau)$ is a well-defined $\ca{O}_{\widehat{K_\infty}}$-endomorphism on $\widehat{V_\infty^+}$ by \ref{para:small} for any $\tau\in\Delta$. Since $\log(\tau)|_{\widehat{V_\infty^+}}$ is compatible with $\log(\tau)|_{V_\infty^+}$ by the formula \eqref{eq:para:small-derivative}, the uniqueness of continuation implies that the following diagram is commutative as the continuation of \eqref{diam:lem:sen-operator-small},
	\begin{align}\label{diam:thm:sen-operator-fix-infty}
		\xymatrix{
			\ho_{\ca{O}_K}(\widehat{\Omega}^1_{\ca{O}_K}(-1),\widehat{\overline{K}})\ar[rr]^-{	\varphi_{\sen}^{\geo}|_{\widehat{W_\infty}}}\ar[d]_-{\psi}^-{\wr}&& \mrm{End}_{\widehat{\overline{K}}}(\widehat{W_\infty})\\
			\widehat{\overline{K}}\otimes_{\bb{Q}_p}\lie(\Delta)&&\Delta\ar[ll]^-{1\otimes\log_{\Delta}}\ar[u]
		}
	\end{align}
	where the right vertical arrow is induced by the map $\Delta\to \mrm{End}_{\ca{O}_{\widehat{\overline{K}}}}(\widehat{W_\infty^+})$ sending $\tau$ to the $p$-adic completion of the endomorphism $\id_{\ca{O}_{\widehat{\overline{K}}}}\otimes \log(\tau)|_{\widehat{V_\infty^+}}$ of $\ca{O}_{\widehat{\overline{K}}}\otimes_{\ca{O}_{\widehat{K_\infty}}}\widehat{V_\infty^+}$. Since the endomorphism $\log(\tau)|_{\widehat{V_\infty^+}}$ is the infinitesimal action of $\tau$ by \eqref{eq:para:small-derivative}, it acts trivially on $(\widehat{V_\infty^+})^\Delta$. Therefore, $\Phi^{\geo}(\widehat{W_\infty})$ acts trivially on $(\widehat{V_\infty})^\Delta=(\widehat{W_\infty})^H$ by \ref{lem:small-invariant} and \eqref{diam:thm:sen-operator-fix-infty}.
\end{proof}

\begin{myrem}\label{rem:thm:sen-operator-fix-infty}
	Even if any element of $\Phi(W_\infty)$ acts continuously on $W_\infty$ (which holds in many cases, cf. \ref{lem:inf-repn-B-basic}), we don't know whether the induced Lie algebra action $\varphi_{\sen}|_{\widehat{W_\infty}}$ by continuation is compatible with the canonical Lie algebra action $\varphi_{\sen}|_{(\widehat{W_\infty})^{\mrm{f}}}$ on the $(G,K)$-finite part $(\widehat{W_\infty})^{\mrm{f}}$ (endowed with the topology induced from $\widehat{W_\infty}$) defined in \ref{prop:sen-brinon-operator-inf}, since we don't know the continuity of the latter (cf. \ref{rem:sen-brinon-operator-inf}). Thus, we couldn't conclude easily that $\Phi(\widehat{W_\infty})$ annihilates $(\widehat{W_\infty})^G$. To see whether it is true or not, we need to study descent and decompletion of $\ca{O}_{\widehat{\overline{K}}}$-representations of $G$ and also compare the Galois invariant part as in \ref{lem:small-invariant}. We plan to investigate this in the future.
\end{myrem}

\section{Some Boundedness Conditions on a Ring Map}\label{sec:bound}

\begin{mydefn}\label{defn:pi-iso}
	Let $A$ be a ring, $\pi$ an element of $A$. 
	\begin{enumerate}
		\renewcommand{\labelenumi}{{\rm(\theenumi)}}
		\item We say that an $A$-module $M$ is $\pi$-zero if it is killed by $\pi$. We say that a morphism of $A$-modules $f:M\to N$ is a \emph{$\pi$-isomorphism} if its kernel and cokernel are $\pi$-zero.
		\item We say that a chain complex of $A$-modules $M_\bullet$ is \emph{$\pi$-exact} if the homology group $H_n(M_\bullet)$ is $\pi$-zero for any $n\in\bb{Z}$. We say that a morphism of chain complexes of $A$-modules $f:M_\bullet\to N_\bullet$ is a \emph{$\pi$-quasi-isomorphism} if it induces a $\pi$-isomorphism on the homology groups $H_n(M_\bullet)\to H_n(N_\bullet)$ for any $n\in\bb{Z}$.
	\end{enumerate} 
\end{mydefn}

\begin{mylem}[{\cite[2.6.3]{abbes2020suite}}]\label{lem:pi-iso-retract}
	Let $A$ be a ring, $\pi$ an element of $A$, $f:M\to N$ a morphism of $A$-modules.
	\begin{enumerate}
		\renewcommand{\labelenumi}{{\rm(\theenumi)}}
		\item If there exists an $A$-linear homomorphism $g:N\to M$ such that $g\circ f=\pi\id_M$ and $f\circ g=\pi\id_N$, then $f$ is a $\pi$-isomorphism.\label{item:lem:pi-iso-retract-2}
		\item If $f$ is a $\pi$-isomorphism, then there is a unique $A$-linear homomorphism $g:N\to M$ sending $y\in N$ to $\pi x\in M$ where $x\in f^{-1}(\pi y)$. In particular, $g\circ f=\pi^2\id_M$ and $f\circ g=\pi^2\id_N$.\label{item:lem:pi-iso-retract-1}
	\end{enumerate}
\end{mylem}
\begin{proof}
	(\ref{item:lem:pi-iso-retract-2}) is clear. For (\ref{item:lem:pi-iso-retract-1}), for any $y\in N$, $f^{-1}(\pi y)$ is not empty as $\pi\cok(f)=0$. The element $\pi x\in M$ does not depend on the choice of $x\in f^{-1}(\pi y)$ as $\pi\ke(f)=0$. Thus, the map $g:N\to M$ is well-defined. It is clearly unique and satisfy the relations $g\circ f=\pi^2\id_M$ and $f\circ g=\pi^2\id_N$.
\end{proof}

\begin{myrem}\label{rem:pi-iso-retract}
	Let $A$ be a ring, $\pi$ an element of $A$.
	\begin{enumerate}
		\renewcommand{\labelenumi}{{\rm(\theenumi)}}
		\item Let $f_\bullet:M_\bullet\to N_\bullet$ be a morphism of chain complexes of $A$-modules such that $f_n$ is a $\pi$-isomorphism for any $n\in \bb{Z}$. Then, there is a morphism of chain complexes of $A$-modules $g_\bullet:N_\bullet\to M_\bullet$ defined by \ref{lem:pi-iso-retract}.(\ref{item:lem:pi-iso-retract-1}) such that $g_\bullet\circ f_\bullet=\pi^2\id_{M_\bullet}$ and $f_\bullet\circ g_\bullet=\pi^2\id_{N_\bullet}$. We see that $f_\bullet$ is a $\pi^2$-quasi-isomorphism. Moreover, for any $A$-linear endofunctor $F$ of the category of $A$-modules, $F(f_\bullet):F(M_\bullet)\to F(N_\bullet)$ is a $\pi^2$-quasi-isomorphism.\label{item:rem:pi-iso-retract-1}
		\item Let $(f_n)_{n\in\bb{N}}:(M_n)_{n\in\bb{N}}\to (N_n)_{n\in\bb{N}}$ be a morphism of inverse systems of $A$-modules such that $f_n$ is a $\pi$-isomorphism for any $n\in \bb{N}$. Then, there is a morphism of inverse systems of $A$-modules $(g_n)_{n\in\bb{N}}:(N_n)_{n\in\bb{N}}\to (M_n)_{n\in\bb{N}}$ defined by \ref{lem:pi-iso-retract}.(\ref{item:lem:pi-iso-retract-1}) such that $g_n\circ f_n=\pi^2\id_{M_n}$ and $f_n\circ g_n=\pi^2\id_{N_n}$. We see that the induced morphism $\rr^q\lim (M_n)_{n\in\bb{N}}\to \rr^q\lim (N_n)_{n\in\bb{N}}$ is a $\pi^2$-isomorphism for any $q\in\bb{Z}$.\label{item:rem:pi-iso-retract-2}
	\end{enumerate} 
\end{myrem}
\begin{mylem}\label{lem:diff-inj-pi}
	Let $A$ be a ring, $\pi$ an element of $A$, $B\to C$ a $\pi$-isomorphism of $A$-algebras. Then, the canonical morphism $C\otimes_{B}\Omega^1_{B/A}\to \Omega^1_{C/A}$ is a $\pi^{17}$-isomorphism.
\end{mylem}
\begin{proof}
	Let $I$ (resp. $J$) be the kernel of the multiplication map $B\otimes_A B\to B$ (resp. $C\otimes_A C\to C$). Recall the $\Omega^1_{B/A}$ (resp. $\Omega^1_{C/A}$) is canonically isomorphic to $I/I^2$ (resp. $J/J^2$). Consider the morphism of exact sequences of $A$-modules
	\begin{align}
		\xymatrix{
			0\ar[r]& I\ar[d]\ar[r]& B\otimes_A B\ar[r]\ar[d]& B\ar[r]\ar[d]&0\\
			0\ar[r]& J\ar[r]& C\otimes_A C\ar[r]& C\ar[r]&0
		}
	\end{align}
	Since $B\to C$ is a $\pi$-isomorphism, $B\otimes_A B\to C\otimes_A C$ is a $\pi^4$-isomorphism by \ref{lem:pi-iso-retract}. By the snake lemma, we see that $I\to J$ is a $\pi^5$-isomorphism, and thus $I^2\to J^2$ is a $\pi^{10}$-isomorphism. The snake lemma shows that $I/I^2\to J/J^2$ is a $\pi^{15}$-isomorphism. On the other hand, $I/I^2\to C\otimes_B I/I^2$ is a $\pi^2$-isomorphism by \ref{lem:pi-iso-retract}. We conclude that $C\otimes_B I/I^2\to J/J^2$ is a $\pi^{17}$-isomorphism.
\end{proof}

\begin{myprop}\label{prop:cotangent-bc}
	Let $A$ be a ring, $A'$ and $B$ two $A$-algebras, $B'=A'\otimes_A B$. Then, the cone of the canonical morphism
	\begin{align}\label{eq:prop:cotangent-bc-1}
		\tau_{\leq 1}(B'\otimes_{B}^{\dl}\dl_{B/A})\longrightarrow \tau_{\leq 1}\dl_{B'/A'}
	\end{align}
	is concentrated in homological degree $2$, where $\tau_{\leq 1}$ is the canonical truncation of chain complexes ({\rm\cite[\href{https://stacks.math.columbia.edu/tag/0118}{0118}]{stacks-project}}), and $\dl_{B/A}$ denotes the cotangent complex of $B$ over $A$. Moreover, if $\tor_1^{A}(A',B)$ is $\pi$-zero for some $\pi\in A$, then \eqref{eq:prop:cotangent-bc-1} is a $\pi$-quasi-isomorphism.
\end{myprop}
\begin{proof}
	We take a surjective homomorphism from a polynomial $A$-algebra $P$ to $B$, and denote its kernel by $I$. Recall that $\tau_{\leq 1}\dl_{B/A}$ is quasi-isomorphic to the complex $I/I^2\to B\otimes_P \Omega^1_{P/A}$ (\cite[\href{https://stacks.math.columbia.edu/tag/08RB}{08RB}]{stacks-project}). Thus, in the derived category, we have
	\begin{align}\label{eq:prop:cotangent-bc-2}
		\tau_{\leq 1}(B'\otimes_B^{\dl}\dl_{B/A})=\tau_{\leq 1}(B'\otimes_B^{\dl}(\tau_{\leq 1}\dl_{B/A}))=(B'\otimes_B I/I^2\to B'\otimes_P \Omega^1_{P/A}),
	\end{align}
	where the first equality follows from the distinguished triangle $\tau_{\geq 2}\dl_{B/A}\to \dl_{B/A}\to \tau_{\leq 1}\dl_{B/A}\to$, and the second equality can be deduced from replacing $I/I^2$ by a flat resolution (note that $B\otimes_P \Omega^1_{P/A}$ is a free $B$-module).
	
	We set $P'=A'\otimes_A P$ and $I'=\ke(P'\to B')$. Then, $\tau_{\leq 1}\dl_{B'/A'}$ is quasi-isomorphic to the complex $I'/I'^2\to B'\otimes_{P'} \Omega^1_{P'/A'}$. Applying the functor $P'\otimes_P-$ to the exact sequence $0\to I\to P\to B\to 0$, we obtain an exact sequence
	\begin{align}
		\tor_1^P(P',B)\longrightarrow P'\otimes_P I\longrightarrow I'\longrightarrow 0.
	\end{align}
	Applying the functor $B'\otimes_{P'}-$, we get an exact sequence
	\begin{align}
		B'\otimes_{P'}\tor_1^P(P',B)\longrightarrow B'\otimes_B I/I^2\longrightarrow I'/I'^2\longrightarrow 0.
	\end{align}
	Let $N$ be the image of the first arrow. Then, by \eqref{eq:prop:cotangent-bc-2}, the cone of \eqref{eq:prop:cotangent-bc-1} is quasi-isomorphic to the complex $N[-2]$. Since $\tor_1^P(P',B)=\tor_1^A(A',B)$ as $P$ is flat over $A$, we see that $N$ is $\pi$-zero if $\pi\tor_1^A(A',B)=0$.
\end{proof}

\begin{mycor}\label{cor:cotangent-bc}
	Let $A$ be a ring, $A'$ and $B$ two $A$-algebras, $B'=A'\otimes_A B$. Assume that $\tau_{\leq 1}\dl_{B/A}$ is $\pi$-exact for some $\pi\in A$. Then, $\tau_{\leq 1}\dl_{B'/A'}$ is $\pi^2$-exact.
\end{mycor}
\begin{proof}
	Consider the convergent spectral sequence \cite[6.3.2.2]{ega3-2}
	\begin{align}
		E^2_{i,j}=\tor^B_i(B',H_j(\dl_{B/A}))\Rightarrow \tor^B_{i+j}(B',\dl_{B/A}),
	\end{align}
	with $\df^2_{i,j}:E^2_{i,j}\to E^2_{i-2,j+1}$. Then, there is an exact sequence
	\begin{align}\label{eq:cor:cotangent-bc}
		\tor^B_0(B',H_1(\dl_{B/A}))\longrightarrow H_1(B'\otimes_B^{\dl}\dl_{B/A})\longrightarrow \tor^B_1(B',H_0(\dl_{B/A}))\longrightarrow 0.
	\end{align}
	Since $H_1(\dl_{B/A})$ and $H_0(\dl_{B/A})$ are killed by $\pi$ as $\tau_{\leq 1}\dl_{B/A}$ is $\pi$-exact, we see that $\pi^2H_1(B'\otimes_B^{\dl}\dl_{B/A})=0$. Notice that $H_1(B'\otimes_B^{\dl}\dl_{B/A})\to H_1(\dl_{B'/A'})$ is surjective by \ref{prop:cotangent-bc}. We see that $\pi^2H_1(\dl_{B'/A'})=0$ and $\pi H_0(\dl_{B'/A'})=\pi(B'\otimes_B H_0(\dl_{B/A}))=0$.
\end{proof}

\begin{mylem}\label{lem:cotangent-bound-seq}
	Let $A$ be a ring, $\pi$ an element of $A$, $B\to C$ a homomorphism of $A$-algebras. Assume that $\tau_{\leq 1}\dl_{C/B}$ and $\tau_{\leq 1}\dl_{B/A}$ are both $\pi$-exact. Then, $\tau_{\leq 1}\dl_{C/A}$ is $\pi^3$-exact.
\end{mylem}
\begin{proof}
	We see that $\pi^2 H_1(C\otimes_B^{\dl}\dl_{B/A})=0$ by \eqref{eq:cor:cotangent-bc}, and that $\pi H_0(C\otimes_B^{\dl}\dl_{B/A})=\pi (C\otimes_B H_0(\dl_{B/A}))=0$. By the fundamental distinguished triangle of cotangent complexes (\cite[\Luoma{2}.2.1.5.6]{illusie1971cot1}), we obtain an exact sequence
	\begin{align}\label{eq:prop:cotangent-bound-seq-long}
		H_1(C\otimes_B^{\dl}\dl_{B/A})\to H_1(\dl_{C/A})\to H_1(\dl_{C/B})\to H_0(C\otimes_B^{\dl}\dl_{B/A})\to H_0(\dl_{C/A})\to H_0(\dl_{C/B})\to 0,
	\end{align}
	which shows that $H_1(\dl_{C/A})$ and $H_0(\dl_{C/A})$ are both killed by $\pi^3$.
\end{proof}

\begin{myprop}\label{prop:cotangent-bound}
	Let $A$ be a ring, $\pi$ an element of $A$, $B$ an $A$-algebra such that $A\to B$ is a $\pi$-isomorphism. Then, $\tau_{\leq 1}\dl_{B/A}$ is $\pi^{102}$-exact.
\end{myprop}
\begin{proof}
	Let $C$ be the image of $A$ in $B$. We take a surjective homomorphism from a polynomial $C$-algebra $P$ to $B$, and denote its kernel by $I$. Let $Q$ be the preimage of $C$ via the surjection $P\to B$. It is a $C$-subalgebra of $P$ such that $\pi P\subseteq Q$ and that $I$ is an ideal of $Q$. We remark that $C=Q/I$ and $B=P/I=C\otimes_Q P$. Consider the canonical exact sequence
	\begin{align}
		H_1(\dl_{P/C})\longrightarrow H_1(\dl_{P/Q})\longrightarrow P\otimes_{Q}\Omega^1_{Q/C}\longrightarrow \Omega^1_{P/C}\longrightarrow \Omega^1_{P/Q}\longrightarrow 0.
	\end{align}
	Since $H_1(\dl_{P/C})=0$ and $P\otimes_{Q}\Omega^1_{Q/C}\to \Omega^1_{P/C}$ is $\pi^{17}$-injective by \ref{lem:diff-inj-pi}, $H_1(\dl_{P/Q})$ is killed by $\pi^{17}$. It is clear that $\pi H_0(\dl_{P/Q})=\pi\Omega^1_{P/Q}=0$ as $\pi P\subseteq Q$. It follows from \ref{cor:cotangent-bc} that $\tau_{\leq 1}\dl_{B/C}$ is $\pi^{34}$-exact as $B=C\otimes_Q P$. On the other hand, let $J$ be the kernel of $A\to C$ which is killed by $\pi$. Then, $H_1(\dl_{C/A})=J/J^2$ and $H_0(\dl_{C/A})=0$ (\cite[\Luoma{3}.1.2.8.1]{illusie1971cot1}). It follows from \ref{lem:cotangent-bound-seq} that $\tau_{\leq 1}\dl_{B/A}$ is $\pi^{102}$-exact.
\end{proof}

\begin{mypara}\label{para:notation-bound}
	Let $A\to B$ be an injective homomorphism of normal domains flat over $\bb{Z}_p$. We fix an algebraic closure $\overline{\ca{L}}$ of the fraction field $\ca{L}$ of $B$, and let $\overline{\ca{K}}$ be the algebraic closure of the fraction field $\ca{K}$ of $A$ in $\overline{\ca{L}}$. Consider an algebraic extension $\ca{K}'$ of $\ca{K}$ in $\overline{\ca{K}}$ and the integral closure $A'$ of $A$ in $\ca{K}'$. Let $B'$ be the integral closure of $B$ in the composite $\ca{L}'=\ca{L}\ca{K}'\subseteq\overline{\ca{L}}$.
	\begin{align}
		\xymatrix{
			\ca{L}'& B'\ar[l] & A'\ar[l]\ar[r] & \ca{K}'\\
			\ca{L}\ar[u]& B\ar[u]\ar[l] & A\ar[l]\ar[r]\ar[u] & \ca{K}\ar[u]
		}
	\end{align}
	Let $\scr{P}_A$ be the family of algebraic extensions $\ca{K}'$ of $\ca{K}$ in $\overline{\ca{K}}$ such that
	\begin{enumerate}
		\renewcommand{\labelenumi}{{\rm(\theenumi)}}
		\item there exists a valuation ring $\ca{O}_{K'}$ extension of $\bb{Z}_p$ contained in $A'$ such that its fraction field $K'$ is a pre-perfectoid field in the sense of \cite[5.1.(1)]{he2021coh}, and
		\item the $\ca{O}_{K'}$-algebra $A'$ is almost pre-perfectoid in the sense of \cite[5.19]{he2021coh}.
	\end{enumerate}
	In particular, $\overline{\ca{K}}\in\scr{P}_A$.
\end{mypara}

\begin{mydefn}\label{defn:bound}
	With the notation in {\rm\ref{para:notation-bound}}, for any algebraic extension $\ca{K}'$ of $\ca{K}$ in $\overline{\ca{K}}$, we say that the map $A\to B$ is \emph{bounded at $\ca{K}'$} if there exists $k\in\bb{N}$ such that 
	\begin{align}\label{eq:defn:bound}
		p^k\cok(B\otimes_A A'\to B')=0.
	\end{align}
	For any $\ca{K}'\in\scr{P}_A$, we say that $A\to B$ is \emph{pre-perfectoid at $\ca{K}'$} if $\ca{L}'=\ca{L}\ca{K}'\in\scr{P}_B$ and if $A\to B$ is bounded at $\ca{K}'$. We say that $A\to B$ is \emph{pre-perfectoid} if it is pre-perfectoid at any $\ca{K}'\in\scr{P}_A$.
\end{mydefn}

\begin{mylem}\label{lem:pre-perfd-basic}
	We keep the notation in {\rm\ref{para:notation-bound}}.
	\begin{enumerate}
		\renewcommand{\labelenumi}{{\rm(\theenumi)}}
		\item If $A\to B$ is \'etale, then it is pre-perfectoid.\label{item:lem:pre-perfd-basic-1}
		\item If $A\to B$ is pre-perfectoid, then for any algebraic extension $\ca{K}'$ of $\ca{K}$ in $\overline{\ca{K}}$, the map $A'\to B'$ is also pre-perfectoid.\label{item:lem:pre-perfd-basic-2}
		\item If $B\to C$ is another injective homomorphism of normal domains flat over $\bb{Z}_p$ and if $A\to B$, $B\to C$ are pre-perfectoid, then so is $A\to C$.\label{item:lem:pre-perfd-basic-3}
	\end{enumerate} 
\end{mylem}
\begin{proof}
	(\ref{item:lem:pre-perfd-basic-1}) Note that $B\otimes_A A'$ is a finite product of normal domains as it is \'etale over the normal domain $A'$ (\cite[\Luoma{3}.3.3]{abbes2016p}). Thus, $B'$ identifies with one of the components. If  $\ca{K}'\in\scr{P}_A$, then $B\otimes_A A'$ is almost pre-perfectoid by \cite[5.37]{he2021coh} and so is $B'$ (i.e. $\ca{L}'\in\scr{P}_B$). This shows that $A\to B$ is pre-perfectoid.
	
	(\ref{item:lem:pre-perfd-basic-2}) We only need to unwind the definition. For any algebraic extension $\ca{K}''$ of $\ca{K}'$ in $\overline{\ca{K}}$, let $A''$ be the integral closure of $A$ in $\ca{K}''$, $B''$ the integral closure of $B$ in $\ca{L}''=\ca{L}\ca{K}''\subseteq\overline{\ca{L}}$. Assume that $\ca{K}''\in \scr{P}_{A'}$. By definition, we also have $\ca{K}''\in\scr{P}_A$, and thus $\ca{L}''\in\scr{P}_B$ so that $\ca{L}''\in\scr{P}_{B'}$. Since $A\to B$ is bounded at $\ca{K}''$, there exists $k\in\bb{N}$ such that $p^k\cok(B\otimes_A A''\to B'')=0$. Thus, $p^k\cok(B'\otimes_{A'} A''\to B'')=0$, which means that $A'\to B'$ is bounded at $\ca{K}''$. This completes the proof.
	
	(\ref{item:lem:pre-perfd-basic-3}) It also follows directly from unwinding the definition.
\end{proof}

\begin{mythm}[Almost purity, {\cite[7.9]{scholze2012perfectoid}}]\label{thm:almost-purity}
	Let $K$ be a pre-perfectoid field with a non-zero element $\pi$ in its maximal ideal, $R$ a flat $\ca{O}_K$-algebra which is almost pre-perfectoid, $R'$ the integral closure of $R$ in a finite \'etale $R[1/\pi]$-algebra. Then, $R'$ is almost pre-perfectoid and almost finite \'etale over $R$.
\end{mythm}
\begin{proof}
	Let $(S,\pi S)$ be the henselization of the pair $(R,\pi R)$. Then, $S'=S\otimes_R R'$ is the integral closure of $S$ in a finite \'etale $S[1/\pi]$-algebra. By \cite[5.41]{he2021coh}, we see that $R'$ is almost pre-perfectoid and that $S'$ is almost finite \'etale over $S$. Notice that $R\to R[1/\pi]\times S$ is faithfully flat. By almost faithfully flat descent \cite[\Luoma{5}.8.10]{abbes2016p}, we see that $R'$ is almost finite \'etale over $R$.
\end{proof}

\begin{mycor}\label{cor:bound-fet}
	Let $A$ be a normal domain flat over $\bb{Z}_p$, $B$ the integral closure of $A$ in a domain finite \'etale over $A[1/p]$. Then, the map $A\to B$ is pre-perfectoid.
\end{mycor}
\begin{proof}
	With the notation in {\rm\ref{para:notation-bound}}, for any $\ca{K}'\in\scr{P}_A$, $B'$ is almost pre-perfectoid and almost finite \'etale over $A'$ by almost purity \ref{thm:almost-purity}. As $B[1/p]$ is finite \'etale over $A[1/p]$, $B\otimes_A A'[1/p]$ is the integral closure of $A'[1/p]$ in $\ca{L}\otimes_{\ca{K}}\ca{K}'$, which is a finite product of normal domains and one of its component identifies with $B'[1/p]$. In particular, $B\otimes_AA'[1/p]\to B'[1/p]$ is surjective. We see that there exists $k\in\bb{N}$ such that $p^k\cok(B\otimes_A A'\to B')=0$, since the $B\otimes_A A'$-module $B'$ is almost finitely generated. It follows from the definition that $A\to B$ is pre-perfectoid.
\end{proof}

\begin{mypara}\label{para:blowup}
	Let $A$ be a ring, $I$ an ideal of $A$, $a$ an element of $I$. The affine blowup algebra $A[I/a]$ is the $A$-subalgebra of $A[1/a]$ generated by the subset $\{x/a\}_{x\in I}$ (\cite[\href{https://stacks.math.columbia.edu/tag/052P}{052P}]{stacks-project}). As the ideal $IA[I/a]$ is generated by $a$, there is a unique morphism $\spec(A[I/a])\to X$ over $\spec(A)$, where $X$ is the blowup of $\spec(A)$ in $I$. Moreover, if $I$ is generated by a subset $S$, then $\{\spec(A[I/a])\to X\}_{a\in S}$ forms a Zariski open covering (\cite[\href{https://stacks.math.columbia.edu/tag/0804}{0804}]{stacks-project}).
\end{mypara}

\begin{mylem}\label{lem:blowup-completion}
	Let $A$ be a ring, $\pi$ an element of $A$, $I$ an ideal of $A$ containing a power of $\pi$, $a$ an element of $I$, $\widehat{A}$ the $\pi$-adic completion of $A$, $I'=I\widehat{A}$, $a'$ the image of $a$ in $\widehat{A}$. Then, the natural morphism of affine blowup algebras $A[I/a]\to \widehat{A}[I'/a']$ induces an isomorphism of their $\pi$-adic completions
	\begin{align}
		(A[I/a])^\wedge\iso (\widehat{A}[I'/a'])^\wedge.
	\end{align}
\end{mylem}
\begin{proof}
	We denote by $\varphi: A\to \widehat{A}$ and $\phi:A[1/a]\to \widehat{A}[1/a]$ the natural morphisms. We need to show that for each integer $n>0$ the natural morphism 
	\begin{align}\label{eq:lem-blowup-completion-1}
		A[I/a]/\pi^nA[I/a]\longrightarrow \widehat{A}[I'/a']/\pi^n\widehat{A}[I'/a']
	\end{align}
	is an isomorphism. We claim that $\widehat{I}\to \widehat{I'} $ is an isomorphism. Indeed, since $A/I$ is killed by a power of $\pi$, we get from the short exact sequence $0\to I\to A\to A/I\to 0$ a short exact sequence $0\to \widehat{I}\to \widehat{A}\to A/I\to 0$ by $\pi$-adic completion (\cite[\href{https://stacks.math.columbia.edu/tag/0BNG}{0BNG}]{stacks-project}). Similarly, we get from the short exact sequence $0\to I'\to \widehat{A}\to \widehat{A}/I'\to 0$ a short exact sequence $0\to \widehat{I'}\to\widehat{A}\to \widehat{A}/I'\to 0$. Moreover, as $\pi^nA\subseteq I$ for $n$ large enough, we deduce from the canonical isomorphism $A/\pi^n A\iso \widehat{A}/\pi^n\widehat{A}$ that $A/I\to \widehat{A}/I'$ is an isomorphism. Combining with the previous short exact sequences, we see that $\widehat{I}\to \widehat{I'} $ is an isomorphism.
	
	For the surjectivity of \eqref{eq:lem-blowup-completion-1}, recall that the $A$-algebra $A[I/a]$ is generated by the elements $\{x/a\}_{x\in I}$. For $x'\in I'$, as $I/\pi^n I\to I'/\pi^n I'$ is surjective, there exists $x\in I$ and $y\in I'$ such that $x'=\varphi(x)+\pi^n y\in I'$. Thus, $x'/a'\equiv \phi(x/a) \mod \pi^n \widehat{A}[I'/a']$, which shows that \eqref{eq:lem-blowup-completion-1} is surjective. 
	
	For the injectivity of \eqref{eq:lem-blowup-completion-1}, as any element of $A[I/a]$ is of the form $x/a^k$ for some $x\in I$ and $k>0$, we suppose that $\phi(x/a^k)=\pi^n(x'/a'^{k'})$ for some $x'\in I'$ and $k'>0$. Thus, there exists $N>0$ such that $a'^N(a'^{k}\pi^nx'-a'^{k'}\varphi(x))=0$ in $I'$. In particular, $\varphi(a^{N+k'}x)\in\pi^nI'$. Since $I/\pi^n I\to I'/\pi^n I'$ is injective, there exists $y\in I$ such that $a^{N+k'}x=\pi^n y$. Thus, $x/a^k=\pi^n(y/a^{N+k+k'})\in \pi^n A[I/a]$, which shows that \eqref{eq:lem-blowup-completion-1} is injective.
\end{proof}

\begin{mylem}\label{lem:pi-bound-tech}
	Let $A$ be a ring, $\pi$ an element of $A$, $B$ an $A$-subalgebra of $A[1/\pi]$. Assume that the morphism of $\pi$-adic completions $\widehat{A}\to \widehat{B}$ is a $\pi^n$-isomorphism for some $n\in\bb{N}$. Then, $\pi^n\cok(A\to B)=0$.
\end{mylem}
\begin{proof}
	Consider the commutative diagram 
	\begin{align}
		\xymatrix{
			B\ar[r]^-{f'}&\widehat{B}\\
			A\ar[r]^-{f}\ar[u]^-{g}&\widehat{A}\ar[u]_-{g'}
		}
	\end{align}
	For $b\in B\subseteq A[1/\pi]$, we write $\pi^n b=g(x)/\pi^m$ in $A[1/\pi]$, where $x\in A$ and $m\in\bb{N}$. We have $\pi^{m+n}f'(b)=  g'(f(x))$. On the other hand, by the assumption $\pi^n\cok(g')=0$, there exists $y\in \widehat{A}$ such that $\pi^n f'(b)=g'(y)$. Thus, $g'(f(x)-\pi^my)=0$, which implies that $\pi^n(f(x)-\pi^m y)=0$ by the assumption $\pi^n\ke(g')=0$. By the isomorphism $A/\pi^{m+n} A\iso \widehat{A}/\pi^{m+n}\widehat{A}$, we have $ \pi^n x\in \pi^{m+n} A$. Therefore, $\pi^n b=g(\pi^n x)/\pi^{m+n}\in \im(g)$.
\end{proof}

\begin{mythm}[{\cite[6.3]{scholze2012perfectoid}}]\label{thm:blow-up-perfd}
	Let $K$ be a pre-perfectoid field with a non-zero element $\pi$ of its maximal ideal, $A$ an $\ca{O}_K$-algebra which is almost pre-perfectoid, $I$ a finitely generated ideal of $A$ containing a power of $\pi$, $a$ an element of $I$. 
	\begin{enumerate}
		\renewcommand{\labelenumi}{{\rm(\theenumi)}}
		\item The integral closure $B$ of the affine blowup algebra $A[I/a]$ in $A[I/a][1/\pi]=A[1/\pi a]$ is almost pre-perfectoid.\label{item:lem-blow-up-perfd-1}
		\item There exists $n\in\bb{N}$ such that $\pi^n\cok(A[I/a]\to B)=0$.\label{item:lem-blow-up-perfd-2}
	\end{enumerate}
\end{mythm}
\begin{proof}
	As $\widehat{A}$ is almost flat over $\ca{O}_{\widehat{K}}$ by definition, $\widehat{A}\to (A/A[\pi^\infty])^\wedge$ is surjective and is an almost isomorphism (\cite[5.27]{he2021coh}). Thus, after replacing $A$ by $A/A[\pi^\infty]\subseteq A[1/\pi]$ and $I$ by its image (which does not change $B$ and $\cok(A[I/a]\to B)$), we may assume that $A$ is flat over $\ca{O}_K$. Let $B'$ (resp. $B''$) be the integral closure of $\widehat{A}[I'/a']$ (resp. $(A[I/a])^\wedge$) in $\widehat{A}[I'/a'][1/\pi]$ (resp. $(A[I/a])^\wedge[1/\pi]$), where the completions are $\pi$-adic, $I'=I\widehat{A}$ and $a'$ is the image of $a$ in $\widehat{A}$. By \ref{lem:blowup-completion}, we have $(A[I/a])^\wedge=(\widehat{A}[I'/a'])^\wedge$. Thus, there exists a canonical morphism $B'\to B''$ and a commutative diagram
	\begin{align}
		\xymatrix{
			B\ar[r]&B'\ar[r]&B''\\
			A[\frac{I}{a}]\ar[r]\ar[u]&\widehat{A}[\frac{I'}{a'}]\ar[r]\ar[u]& (A[\frac{I}{a}])^\wedge\ar[u]
		}
	\end{align}
	Since the three $\ca{O}_K$-algebras in the second row are flat (\cite[5.20]{he2021coh}) and have the same $\pi$-adic completion flat over $\ca{O}_{\widehat{K}}$, the $\pi$-adic completions of the three $\ca{O}_K$-algebras in the first row are almost isomorphic by \cite[5.29]{he2021coh}.
	
	(\ref{item:lem-blow-up-perfd-1}) By definition, the $\ca{O}_{\widehat{K}}$-algebra $\widehat{A}$ is almost perfectoid. We endow $\widehat{A}[1/\pi]$ with the $\pi$-adic topology defined by $\widehat{A}$ so that it becomes a Tate $\widehat{K}$-algebra in the sense of \cite[2.6]{scholze2012perfectoid}. If $\widehat{A}^+$ denotes the integral closure of $\widehat{A}$ in $\widehat{A}[1/\pi]$ (which is almost isomorphic to $\widehat{A}$), then $(\widehat{A}[1/\pi],\widehat{A}^+)$ forms a perfectoid affinoid $\widehat{K}$-algebra in the sense of \cite[6.1]{scholze2012perfectoid}. Similarly, we endow $\widehat{A}[I'/a'][1/\pi]$ with the $\pi$-adic topology defined by $\widehat{A}[I'/a']$ so that it becomes a Tate $\widehat{K}$-algebra. Then, $(\widehat{A}[I'/a'][1/\pi],B')$ is an affinoid $\widehat{K}$-algebra. Its completion is $((A[I/a])^\wedge [1/\pi],B'')$, which is the completed affinoid $\widehat{K}$-algebra associated to the rational subset of $\mrm{Spa}(\widehat{A}[1/\pi],\widehat{A}^+)$ defined by $I$ and $a$ by the definition \cite[2.13]{scholze2012perfectoid}). By virtue of \cite[6.3.(\luoma{2})]{scholze2012perfectoid}, $((A[I/a])^\wedge [1/\pi],B'')$ is a perfectoid affinoid $\widehat{K}$-algebra. Thus, the $\ca{O}_{\widehat{K}}$-algebra $B''$ is almost perfectoid and bounded in $(A[I/a])^\wedge [1/\pi]$ with respect to the $\pi$-adic topology defined by $(A[I/a])^\wedge$. In particular, $B''\to \widehat{B''}$ is an almost isomorphism (\cite[5.5]{scholze2012perfectoid}). Thus, $B$ is almost pre-perfectoid, since $\widehat{B}\to \widehat{B''}$ is an almost isomorphism by the discussion in the beginning.
	
	(\ref{item:lem-blow-up-perfd-2}) Since $(A[I/a])^\wedge\to B''$ is injective as $(A[I/a])^\wedge$ is flat over $\ca{O}_K$, the map $(A[I/a])^\wedge \to \widehat{B}$ is almost injective by the almost isomorphisms $\widehat{B}\to \widehat{B''}\leftarrow B''$. Since $B''$ is bounded in $(A[I/a])^\wedge [1/\pi]$ with respect to the $\pi$-adic topology defined by $(A[I/a])^\wedge$, there exists $n\in\bb{N}$ such that $\pi^n$ annihilates the kernel and cokernel of $(A[I/a])^\wedge\to \widehat{B}$. The conclusion follows from \ref{lem:pi-bound-tech}.
\end{proof}

\begin{mycor}\label{cor:bound-blowup}
	Let $A$ be a normal domain flat over $\bb{Z}_p$, $I$ a finitely generated ideal of $A$ containing a power of $p$, $a$ an element of $I$, $B$ the integral closure of the affine blowup algebra $A[I/a]$ in $A[I/a][1/p]=A[1/pa]$. Then, the map $A\to B$ is pre-perfectoid.
\end{mycor}
\begin{proof}
	With the notation in {\rm\ref{para:notation-bound}}, for any $\ca{K}'\in\scr{P}_A$, notice that the affine blowup algebra $A'[I/a]$ is the image of $A[I/a]\otimes_A A'\to A'[1/a]$. Thus, $B'$ is the integral closure of $A'[I/a]$ in $A'[I/a][1/p]=A'[1/pa]$. Thus, it is almost pre-perfectoid and $p^k\cok(A'[I/a]\to B')=0$ for some $k\in \bb{N}$ by \ref{thm:blow-up-perfd}, which completes the proof.
\end{proof}

\section{Brief Review on Adequate Charts of Logarithmic Schemes}\label{sec:adequate}
The main geometric object of this article, quasi-adequate algebras, stems from logarithmic geometry. In this section, we firstly review basic notions of logarithmic geometry. We refer to \cite{kato1989log,kato1994toric,gabber2004foundations,ogus2018log} for a systematic development of logarithmic geometry, and to \cite[\Luoma{2}.5]{abbes2016p} and \cite[\textsection9]{he2021coh} for a brief summary of the theory. Then, we review adequate charts of a logarithmic scheme and the induced coverings following Tsuji \cite[\textsection4]{tsuji2018localsimpson}.

\begin{mypara}\label{para:monoid}
	All monoids considered in this article are unitary and commutative, and we denote the monoid structures additively. The category of monoids admits arbitrary colimits (cf. \cite[\Luoma{1}.1.1]{ogus2018log}), and we denote the colimit of a diagram $M_1\leftarrow M_0 \to M_2$ by $M_1\oplus_{M_0}M_2$. The forgetful functor from the category of groups (resp. rings) to the category of monoids (resp. with respect to the multiplicative structure) admits a left adjoint sending $M$ to $M^\mrm{gp}$ (resp. $\bb{Z}[M]$). For any monoid $M$, we denote by $\exp_M:M\to \bb{Z}[M]$ the canonical homomorphism of monoids. The forgetful functor from the category of finitely generated monoids to the category of fs (i.e. fine and saturated) monoids admits a left adjoint sending $M$ to the saturation $M^{\mrm{fs}}$ of its image in $M^{\mrm{gp}}$ (\cite[\Luoma{1}.1.3.5, \Luoma{1}.2.2.5]{ogus2018log}). 
\end{mypara}

\begin{mypara}
	A log scheme $X$ is a pair $(\underline{X},\alpha_X:\ca{M}_X\to \ca{O}_X)$ consisting of a scheme $\underline{X}$ and a homomorphism from a sheaf of monoids to the structural sheaf on the \'etale site of $\underline{X}$ (equivalent to the strictly \'etale site of $X$, see below) which induces an isomorphism $\alpha_X^{-1}(\ca{O}_X^\times)\iso\ca{O}_X^\times$, where $\ca{O}_X=\ca{O}_{\underline{X}_\et}$ and $\ca{O}_X^\times$ is the subsheaf of units. A morphism of log schemes $Y\to X$ is a pair $(f,f^\flat)$ consisting of a morphism of the underlying schemes $f:\underline{Y}\to \underline{X}$ and a homomorphism of sheaves of monoids $f^{\flat}:f^{-1}(\ca{M}_X)\to \ca{M}_Y$ compatible with the natural homomorphism $f^{-1}(\ca{O}_X)\to \ca{O}_Y$ via $\alpha_Y$ and $f^{-1}(\alpha_X)$. A morphism of log schemes $Y\to X$ is \emph{strict} if the log structure of $Y$ is the inverse image of that of $X$ (\cite[\Luoma{2}.5.11]{abbes2016p}). For an open immersion of schemes $j:U\to X$, let $\ca{M}_{U\to X}$ be the preimage of $j_{\et*}\ca{O}_{U_\et}^\times$ via the natural map $\ca{O}_{X_\et}\to j_{\et*}\ca{O}_{U_\et}$. Then, the log structure $\alpha_{U\to X}:\ca{M}_{U\to X}\to \ca{O}_{X_\et}$ on $X$ is called the \emph{compactifying log structure} associated to the open immersion $j$ (\cite[\Luoma{3}.1.6.1]{ogus2018log}).
\end{mypara}

\begin{mypara}\label{para:log-ring}
	A log ring is a homomorphism $M\to A$ from a monoid $M$ to the multiplicative monoid of a ring $A$. We denote by $\spec(M\to A)$ the log scheme with underlying scheme $\spec(A)$ endowed with the log structure associated to the pre-log structure $M\to \ca{O}_{\spec(A)_\et}$ induced by $M\to A$, and we set $\bb{A}_M=\spec(\exp_M:M\to \bb{Z}[M])$ (\cite[\Luoma{3}.1.2.3]{ogus2018log}). A \emph{chart} of a log scheme $X$ is a homomorphism $M\to \Gamma(X,\ca{M}_X)$ from a monoid $M$ to the monoid of global sections of $\ca{M}_X$ such that the induced morphism of log schemes $X\to\bb{A}_M$ is strict (\cite[\Luoma{2}.5.13]{abbes2016p}). We say that a log scheme $X$ is \emph{coherent} (resp. \emph{fs}) if strictly \'etale locally on $X$ it admits a chart from a finitely generated (resp. fs) monoid $M$ (\cite[\Luoma{2}.5.15]{abbes2016p}).
\end{mypara}

\begin{mypara}
	The inclusion functor from the category of schemes to the category of coherent log schemes (by endowing with trivial log structures) admits a left adjoint sending $X$ to its underlying scheme $\underline{X}$, and admits a right adjoint sending $X$ to the maximal open subscheme $X^{\triv}$ of $\underline{X}$ on which the log structure is trivial (\cite[\Luoma{3}.1.2.8]{ogus2018log}). 
	The inclusion functor from the category of fs log schemes to the category of coherent log schemes admits a right adjoint $X\mapsto X^{\mrm{fs}}$, and we remark that the canonical morphism of underlying schemes $\underline{X}^{\mrm{fs}}\to \underline{X}$ is finite with $(X^{\mrm{fs}})^{\triv}=X^{\triv}\times_{\underline{X}} \underline{X}^{\mrm{fs}}=X^{\triv}$ (\cite[\Luoma{3}.2.1.5]{ogus2018log}). The category of log schemes admits finite limits, which commute with taking underlying schemes and preserve coherence (\cite[\Luoma{3}.2.1.2]{ogus2018log}). By the universal property of the functor $X\mapsto X^{\mrm{fs}}$, the category of fs log schemes also admits finite limits (\cite[\Luoma{3}.2.1.6]{ogus2018log}).
\end{mypara}

\begin{mypara}\label{para:log-reg}
	Let $X$ be a regular fs log scheme (\cite[2.1]{kato1994toric}, \cite[2.3]{niziol2006toric}). Its underlying scheme $\underline{X}$ is locally Noetherian and normal, and $X^{\triv}$ is regular and dense in $X$ (\cite[4.1]{kato1994toric}). Moreover, the log structure on $X$ is the compactifying log structure associated to the open immersion $X^{\triv}\to \underline{X}$ (\cite[11.6]{kato1994toric}, \cite[2.6]{niziol2006toric}). A typical example is that given a regular scheme $X$ with a strict normal crossings divisor $D$, then $(X,\alpha_{X\setminus D\to X})$ is a regular fs log scheme whose open subset of triviality of log structure is $X\setminus D$ (\cite[\Luoma{3}.1.11.9]{ogus2018log}).
\end{mypara}

\begin{mypara}\label{para:log-diff}
	To any morphism of log schemes $Y\to X$, one can associate the $\ca{O}_Y$-module of log differentials $\Omega^1_{Y/X}$ with natural maps $\df:\ca{O}_Y\to \Omega^1_{Y/X}$ and $\df\log:\ca{M}_Y\to \Omega^1_{Y/X}$ (\cite[\Luoma{2}.5.21]{abbes2016p}). If $Y\to X$ is strict, then $\Omega^1_{Y/X}=\Omega^1_{\underline{Y}/\underline{X}}$. If $X$ and $Y$ are coherent, then $\Omega^1_{Y/X}$ is quasi-coherent. If $Y\to X$ is a smooth morphism of coherent log schemes, then $\Omega^1_{Y/X}$ is locally finite free (\cite[\Luoma{4}.3.2.1]{ogus2018log}). For any morphism of log rings $(\alpha:M\to A)\to(\beta:N\to B)$, if we denote by $X=\spec(M\to A)$ and $Y=\spec(N\to B)$, then $\Omega^1_{Y/X}$ is the quasi-coherent $\ca{O}_Y$-module associated to the $B$-module
	\begin{align}
		\Omega^1_{(N\to B)/(M\to A)}=\frac{\Omega^1_{B/A}\oplus (B\otimes_{\bb{Z}}(N^{\mrm{gp}}/M^{\mrm{gp}}))}{F}
	\end{align} 
	where $F$ is the $B$-submodule of $\Omega^1_{B/A}\oplus (B\otimes_{\bb{Z}}(N^{\mrm{gp}}/M^{\mrm{gp}}))$ generated by the elements $(\df (\beta(x)),-\beta(x)\otimes x)$ for any $x\in N$ (\cite[\Luoma{4}.1.2.6]{ogus2018log}). We remark that for any Cartesian diagram in the category of (resp. fs) log schemes
	\begin{align}
		\xymatrix{
			Y'\ar[r]^-{g'}\ar[d]_-{f'}&Y\ar[d]^-{f}\\
			X'\ar[r]^-{g}&X
		}
	\end{align}
	the canonical morphism $g'^*\Omega^1_{Y/X}\to \Omega^1_{Y'/X'}$ is an isomorphism (\cite[\Luoma{4}.1.2.15]{ogus2018log}).
\end{mypara}

\begin{mypara}\label{para:notation-adequate-chart}
	Let $K$ be a complete discrete valuation field of characteristic $0$ with perfect residue field of characteristic $p>0$, $S$ the log scheme with underlying scheme $\spec(\ca{O}_K)$ endowed with  the compactifying log structure associated to the open immersion $\spec(K)\to\spec(\ca{O}_K)$ (in particular, $S$ is a regular fs log scheme, cf. \ref{para:log-reg}), $f:X\to S$ a morphism of fs log schemes. We remark that if $f$ is smooth, then $X$ is also regular (\cite[\Luoma{4}.3.5.3]{ogus2018log}), and thus the log structure on $X$ is the compactifying log structure associated to $X^{\triv}\to \underline{X}$, cf. \ref{para:log-reg}.
\end{mypara}

\begin{mydefn}\label{defn:adequate-chart}
	With the notation in {\rm\ref{para:notation-adequate-chart}}, an \emph{adequate chart} of $f$ is a triple of homomorphisms of monoids $(\alpha:\bb{N}\to \Gamma(S,\ca{M}_S),\ \beta:P\to \Gamma(X,\ca{M}_X),\ \gamma:\bb{N}\to P)$ satisfying the following conditions:
	\begin{enumerate}
		\renewcommand{\labelenumi}{{\rm(\theenumi)}}
		\item The following diagram is commutative
		\begin{align}
			\xymatrix{
				\Gamma(X,\ca{M}_X)& P\ar[l]_-{\beta}\\
				\Gamma(S,\ca{M}_S)\ar[u]^-{f^{\flat}}& \bb{N}\ar[l]_-{\alpha}\ar[u]_-{\gamma}
			}
		\end{align}\label{item:defn:adequate-chart-1}
		\item The element $\alpha(1)\in \Gamma(S,\ca{M}_S)=\ca{O}_K\setminus\{0\}$ is a uniformizer of $\ca{O}_K$ (in particular, $S\to \bb{A}_{\bb{N}}$ is strict).\label{item:defn:adequate-chart-2}
		\item The homomorphism $\beta$ induces a strict and \'etale morphism $X\to S\times_{\bb{A}_{\bb{N}}}\bb{A}_P$ (in particular, $X\to \bb{A}_P$ is also strict).\label{item:defn:adequate-chart-3}
		\item The monoid $P$ is fs, and if we denote by $\gamma_\eta:\bb{Z}\to P_\eta=\bb{Z}\oplus_{\bb{N}}P$ the pushout of $\gamma$ by the inclusion $\bb{N}\to \bb{Z}$, then there exists an isomorphism for some $c, d\in \bb{N}$ with $c\leq d$,
		\begin{align}\label{eq:adequate-chart-P}
			P_\eta\cong \bb{Z}\oplus \bb{Z}^c\oplus \bb{N}^{d-c}
		\end{align}
		identifying $\gamma_\eta$ with the inclusion of $\bb{Z}$ into the first component of right hand side. \label{item:defn:adequate-chart-4}
	\end{enumerate}
\end{mydefn}

\begin{myrem}\label{rem:adequate-chart}
	In \ref{defn:adequate-chart}, the morphism of fs log schemes $S\times_{\bb{A}_{\bb{N}}}\bb{A}_P\to S$ is smooth (\cite[\Luoma{2}.5.25]{abbes2016p}), and thus so is $X\to S$. If we set $A=\ca{O}_K\otimes_{\bb{Z}[\bb{N}]}\bb{Z}[P]$ and $A_{\triv}=\ca{O}_K\otimes_{\bb{Z}[\bb{N}]}\bb{Z}[P^{\mrm{gp}}]$, then the underlying scheme of $S\times_{\bb{A}_{\bb{N}}}\bb{A}_P$ is $\spec(A)$, and $\spec(A_{\triv})$ is the maximal open subscheme on which the log structure is trivial (cf. \cite[\Luoma{3}.1.2.10]{ogus2018log}). As $X\to S\times_{\bb{A}_{\bb{N}}}\bb{A}_P$ is strictly \'etale, $X^{\triv}=\spec(A_{\triv})\times_{\spec(A)}\underline{X}$ and the log structure on $X$ is the compactifying log structure associated to the open immersion $X^{\triv}\to \underline{X}$ (cf. \ref{para:notation-adequate-chart}). Moreover, \eqref{eq:adequate-chart-P} induces an isomorphism (cf. \ref{para:monoid})
	\begin{align}\label{eq:rem:adequate-chart}
		A[\frac{1}{p}]=K\otimes_{\bb{Z}[\bb{Z}]}\bb{Z}[\bb{Z}]\otimes_{\bb{Z}[\bb{N}]}\bb{Z}[P]=K\otimes_{\bb{Z}[\bb{Z}]}\bb{Z}[P_\eta]\cong K[\bb{Z}^c\oplus \bb{N}^{d-c}],
	\end{align}
	and $A$ is a Noetherian normal domain (cf. \ref{para:log-reg}, \ref{para:notation-adequate-chart}).  
\end{myrem}

\begin{myrem}\label{rem:adequate-chart-compare}
	Our definition of adequate charts is slightly different from Abbes-Gros' definition \cite[\Luoma{3}.4.4]{abbes2016p}, where they require moreover that $\gamma$ is saturated (cf. \cite[4.2.2]{abbes2020suite}). Nevertheless, our adequate charts describe log smooth schemes over $S$ by the following proposition due to Abbes-Gros and Tsuji.
\end{myrem}

\begin{myprop}[{\cite[3.14, 3.16]{tsuji2018localsimpson}, cf. \cite[\Luoma{3}.4.6]{abbes2016p}}]\label{prop:adequate-chart}
	With the notation in {\rm\ref{para:notation-adequate-chart}}, the following conditions are equivalent:
	\begin{enumerate}
		\renewcommand{\labelenumi}{{\rm(\theenumi)}}
		\item The morphism of fs log schemes $f:X\to S$ is smooth and the underlying generic fibre $\underline{X}_K$ is regular.\label{item:prop:adequate-chart-1}
		\item Every geometric point of $\underline{X}$ admits a strictly \'etale neighbourhood $U$ in $X$ such that $U\to S$ factors through $S'=(\spec(\ca{O}_{K'}),\alpha_{\spec(K')\to\spec(\ca{O}_{K'})})$ for some tamely ramified finite extension $K'$ of $K$ and that the induced morphism of fs log schemes $U\to S'$ admits an adequate chart.\label{item:prop:adequate-chart-2}
	\end{enumerate}
\end{myprop}
\begin{proof}
	(\ref{item:prop:adequate-chart-2}) $\Rightarrow$ (\ref{item:prop:adequate-chart-1}): By the conditions in \ref{defn:adequate-chart}, $\underline{U}_{K'}$ is \'etale over $\spec(K'[\bb{Z}^c\oplus\bb{N}^{d-c}])$, thus regular. Since $U\to S'$ is smooth by \ref{rem:adequate-chart}, so is $U\to S$ as $S'$ is \'etale over $S$.
	
	(\ref{item:prop:adequate-chart-1}) $\Rightarrow$ (\ref{item:prop:adequate-chart-2}): For a geometric point of the generic fibre $\underline{X}_K$, the conclusion follows directly from \cite[3.14]{tsuji2018localsimpson} (where we take $S'=S$). For a geometric point of the special fibre of $\underline{X}$, the conclusion follows directly from \cite[3.16]{tsuji2018localsimpson}.
\end{proof}

\begin{mypara}\label{para:X}
	We follow Tsuji \cite[\textsection4]{tsuji2018localsimpson} to construct coverings of adequate log schemes from an adequate chart. Let $\gamma:\bb{N}\to P$ be an injective homomorphism of fs monoids such that there exists an isomorphism for some $c, d\in \bb{N}$ with $c\leq d$,
	\begin{align}\label{eq:monoid-str-P}
		P_\eta\cong \bb{Z}\oplus \bb{Z}^c\oplus \bb{N}^{d-c}
	\end{align}
	identifying $\gamma_\eta:\bb{Z}\to P_\eta=\bb{Z}\oplus_{\bb{N}}P$ with the inclusion of $\bb{Z}$ into the first component of right hand side as in \ref{defn:adequate-chart}.(\ref{item:defn:adequate-chart-4}). We identify $P^{\mrm{gp}}$ with $\bb{Z}^{1+d}$ and $\bb{N}^{\mrm{gp}}$ with the first component of $\bb{Z}^{1+d}$. For any $e\in \bb{N}_{>0}$ and $\underline{r}=(r_1,\dots,r_d)\in \bb{N}_{>0}^d$, we define a submonoid of $\bb{Q}^{1+d}$ by
	\begin{align}
		P_{e,\underline{r}}=\{x\in e^{-1}\bb{Z}\times r_1^{-1}\bb{Z}\times\cdots \times r_d^{-1}\bb{Z}\ |\ \exists k\in\bb{N}_{>0}\trm{ s.t. }kx\in P\}.
	\end{align}
	It is an fs monoid (\cite[3.2]{tsuji2018localsimpson}), and if we denote by $P_{e,\underline{r},\eta}$ the pushout $e^{-1}\bb{Z}\oplus_{e^{-1}\bb{N}}P_{e,\underline{r}}$, then there is an isomorphism
	\begin{align}\label{eq:monoid-tower-str-P}
		P_{e,\underline{r},\eta}\cong e^{-1}\bb{Z}\oplus r_1^{-1}\bb{Z}\oplus\cdots\oplus r_c^{-1}\bb{Z}\oplus r_{c+1}^{-1}\bb{N}\oplus \cdots\oplus r_d^{-1}\bb{N}
	\end{align}
	induced by \eqref{eq:monoid-str-P} (cf. \cite[page 810, equation (2)]{tsuji2018localsimpson}). 
	
	Let $K$ be a complete discrete valuation field of characteristic $0$ with perfect residue field of characteristic $p>0$, $L$ a finite field extension of $K$, $S$ (resp. $S^L$) the log scheme with underlying scheme $\spec(\ca{O}_K)$ (resp. $\spec(\ca{O}_L)$) endowed with the compactifying log structure defined by the closed point. We fix a homomorphism of monoids $\alpha:\bb{N}\to\ca{O}_K\setminus \{0\}$ sending $1$ to a uniformizer $\pi$ of $K$. For any $\underline{r}\in\bb{N}_{>0}^d$, consider the fibred product in the category of fs log schemes
	\begin{align}\label{eq:X-defn}
		X^L_{\underline{r}}=S^L\times_{\bb{A}_{\bb{N}}}^{\mrm{fs}}\bb{A}_{P_{1,\underline{r}}}, 
	\end{align}
	where the map $S^L\to\bb{A}_{\bb{N}}$ is induced by $\alpha$ and the inclusion $\ca{O}_K\to \ca{O}_L$. We omit the index $L$ or $\underline{r}$ if $L=K$ or $\underline{r}=\underline{1}$ respectively. Let $\beta:P_{1,\underline{r}}\to\Gamma(X^L_{\underline{r}},\ca{M}_{X^L_{\underline{r}}})$ be the induced homomorphism of monoids, and for any $1\leq i\leq d$, we denote by $T_{i,r_i}$ the image of $r_i^{-1}\cdot\underline{1}_i=(0,\dots,r_i^{-1},\dots,0)\in P_{1,\underline{r}}^{\mrm{gp}}$ in $\Gamma(X^L_{\underline{r}},\ca{M}_{X^L_{\underline{r}}}^{\mrm{gp}})$.
	\begin{align}\label{diam:X-1}
		\xymatrix{
			\Gamma(X^L_{\underline{r}},\ca{M}_{X^L_{\underline{r}}})& P_{1,\underline{r}}\ar[l]_-{\beta}\\
			\Gamma(S^L,\ca{M}_{S^L})\ar[u]& \bb{N}\ar[l]_-{\alpha}\ar[u]_-{\gamma}
		}
	\end{align} 
	In section \ref{sec:descent}, we will produce a Kummer tower from $X^L_{\underline{r}}\to X$ by varying $L$ and $\underline{r}$. We need an adequate chart of $X^L_{\underline{r}}$ over $S^L$.
\end{mypara}

\begin{mylem}[{\cite[page 812, equation (6)]{tsuji2018localsimpson}}]\label{lem:X-adequate-chart}
	With the notation in {\rm\ref{para:X}}, the morphism of fs log schemes $X^L_{\underline{r}}\to S^L$ admits an adequate chart
	\begin{align}\label{eq:lem:X-adequate-chart}
		(\alpha': e^{-1}\bb{N}\to \Gamma(S^L,\ca{M}_{S^L}),\ \beta':P_{e,\underline{r}}\to \Gamma(X^L_{\underline{r}},\ca{M}_{X^L_{\underline{r}}}),\ \gamma':e^{-1}\bb{N}\to P_{e,\underline{r}})
	\end{align}
	where $e$ is the ramification index of $L/K$, $\alpha'$ is a homomorphism of monoids sending $e^{-1}$ to a uniformizer $\pi'$ of $L$, $\beta'(k_0/e,k_1/r_1,\dots,k_d/r_d)=\pi'^{k_0}\cdot T_{1,r_1}^{k_1}\cdots T_{d,r_d}^{k_d}$, and $\gamma'$ is induced by the inclusion of the first component of \eqref{eq:monoid-tower-str-P}. Moreover, the canonical morphism of log schemes induced by this chart
	\begin{align}\label{eq:lem:X-adequate-chart-1}
		X^L_{\underline{r}}\longrightarrow S^L\times_{\bb{A}_{e^{-1}\bb{N}}}\bb{A}_{P_{e,\underline{r}}} 
	\end{align}
	is an isomorphism.
\end{mylem}
\begin{proof}
	For the convenience of the readers, we briefly recall Tsuji's proof. Consider the commutative diagram of fs monoids
	\begin{align}\label{diam:monoid-tower-str-P}
		\xymatrix{
			P_{e,\underline{r}}\ar[r]^-{(\iota,0)}& P_{e,\underline{r}}\oplus \bb{Z}&P_{1,\underline{r}}\ar[l]_-{(\iota,s)}\\
			e^{-1}\bb{N}\ar[r]^-{(\iota,0)}\ar[u]_-{\gamma'}& e^{-1}\bb{N}\oplus \bb{Z}\ar[u]_-{(\gamma',\id_{\bb{Z}})}& \bb{N}\ar[l]_-{(\iota,\iota)}\ar[u]_-{\gamma}
		}
	\end{align}
	where we use $\iota$ to denote the inclusions, $s:P_{1,\underline{r}}\to \bb{Z}$ is the projection to the first component, and $\gamma'$ is induced by the inclusion of the first component of \eqref{eq:monoid-tower-str-P}. The squares in \eqref{diam:monoid-tower-str-P} are cocartesian in the category of fs monoids (using \cite[4.2]{tsuji2018localsimpson} for the right). Notice that the map $\alpha:\bb{N}\to \Gamma(S^L,\ca{M}_{S^L})$ is the composition of 
	\begin{align}
		\bb{N}\stackrel{(\iota,\iota)}{\longrightarrow}e^{-1}\bb{N}\oplus \bb{Z}\stackrel{(\alpha',s')}{\longrightarrow} \Gamma(S^L,\ca{M}_{S^L}),
	\end{align}
	where $\alpha'$ is a homomorphism of monoids sending $e^{-1}$ to a uniformizer $\pi'$ of $L$, and $s':\bb{Z}\to \Gamma(S^L,\ca{M}_{S^L})=\ca{O}_L\setminus\{0\}$ sends $1$ to the unit $\pi/\pi'^e$. Thus, the cocartesian square in \eqref{diam:monoid-tower-str-P} on the right induces the following commutative diagram of monoids by the definition $X^L_{\underline{r}}=S^L\times_{\bb{A}_{\bb{N}}}^{\mrm{fs}}\bb{A}_{P_{1,\underline{r}}}$ \eqref{eq:X-defn}.
	\begin{align}\label{diam:X-2}
		\xymatrix{
			\Gamma(X^L_{\underline{r}},\ca{M}_{X^L_{\underline{r}}})& P_{e,\underline{r}}\oplus \bb{Z}\ar[l]_-{(\beta',s')}\\
			\Gamma(S^L,\ca{M}_{S^L})\ar[u]& e^{-1}\bb{N}\oplus \bb{Z}\ar[l]_-{(\alpha',s')}\ar[u]_-{(\gamma',\id_{\bb{Z}})}
		}
	\end{align}
	It is clear that the induced map $\beta'$ sends an element $(k_0/e,k_1/r_1,\dots,k_d/r_d)\in P_{e,\underline{r}}$ to $\pi'^{k_0}\cdot T_{1,r_1}^{k_1}\cdots T_{d,r_d}^{k_d}$. Since the morphism \eqref{eq:lem:X-adequate-chart-1} is induced by the composition of \eqref{diam:X-2} with the cocartesian square in \eqref{diam:monoid-tower-str-P} on the left, it is an isomorphism by the definition $X^L_{\underline{r}}=S^L\times_{\bb{A}_{\bb{N}}}^{\mrm{fs}}\bb{A}_{P_{1,\underline{r}}}$ \eqref{eq:X-defn}.
\end{proof}

\begin{myprop}[{\cite[4.3, 4.5]{tsuji2018localsimpson}}]\label{prop:X-tower-str}
	With the notation in {\rm\ref{para:X}}, there exists $k_0\in \bb{N}$ such that for any finite field extensions $L\subseteq L'$ of $K$ and any elements $\underline{r}|\underline{r}'$ of $\bb{N}_{>0}^d$, we have 
	\begin{align}
		p^{k_0} A^{L'}_{\underline{r}'}\subseteq \bigoplus_{ \underline{k}\in I}\ca{O}_{L'}\otimes_{\ca{O}_L} A^L_{\underline{r}}\cdot\prod_{i=1}^d T_{i,r_i'}^{k_i}\subseteq p^{-k_0} A^{L'}_{\underline{r}'}
	\end{align}
	where $A^L_{\underline{r}}=\Gamma(X^L_{\underline{r}},\ca{O}_{X^L_{\underline{r}}})$, and $I=\{(k_1,\dots,k_d)\in\bb{N}^d\ |\ 0\leq k_i<r_i'/r_i,\ 1\leq i\leq d\}$.
\end{myprop}
\begin{proof}
	Firstly, we note that
	\begin{align}
		A^L_{\underline{r}}[\frac{1}{p}]=L[r_1^{-1}\bb{Z}\oplus\cdots\oplus r_c^{-1}\bb{Z}\oplus r_{c+1}^{-1}\bb{N}\oplus\cdots\oplus r_d^{-1}\bb{N}]=L[T_{1,r_1}^{\pm 1},\dots,T_{c,r_c}^{\pm 1},T_{c+1,r_{c+1}},\dots,T_{d,r_d}]
	\end{align}
	by definition \eqref{eq:X-defn} (cf. \cite[page 812, equation (5)]{tsuji2018localsimpson}). Thus, $\bigoplus_{ \underline{k}\in I}\ca{O}_{L'}\otimes_{\ca{O}_L} A^L_{\underline{r}}\cdot\prod_{i=1}^d T_{i,r_i'}^{k_i}$ is a finite free $\ca{O}_{L'}\otimes_{\ca{O}_L} A^L_{\underline{r}}$-subalgebra of $A^{L'}_{\underline{r}'}[1/p]$. By \cite[4.5.(2)]{tsuji2018localsimpson}, we have
	\begin{align}\label{eq:prop:X-tower-str-1}
		pA^{L'}_{\underline{r}}\subseteq \ca{O}_{L'}\otimes_{\ca{O}_L} A^L_{\underline{r}}\subseteq A^{L'}_{\underline{r}}.
	\end{align}
	By \cite[4.3]{tsuji2018localsimpson} and the isomorphism \eqref{eq:lem:X-adequate-chart-1}, there exists $k_0\in \bb{N}_{>0}$ independent of $L,L',\underline{r},\underline{r}'$ such that (cf. \cite[4.9]{tsuji2018localsimpson})
	\begin{align}\label{eq:prop:X-tower-str-2}
		p^{k_0-1} A^{L'}_{\underline{r}'}\subseteq \bigoplus_{ \underline{k}\in I} A^{L'}_{\underline{r}}\cdot\prod_{i=1}^d T_{i,r_i'}^{k_i} \subseteq p^{1-k_0} A^{L'}_{\underline{r}'}.
	\end{align}
	The conclusion follows from combining \eqref{eq:prop:X-tower-str-1} and \eqref{eq:prop:X-tower-str-2}.
\end{proof}

\section{Quasi-adequate Algebras and Faltings Extension}\label{sec:quasi}
In this section, we fix a complete discrete valuation field $K$ of characteristic $0$ with perfect residue field of characteristic $p>0$, an algebraic closure $\overline{K}$ of $K$, and a compatible system of primitive $n$-th roots of unity $(\zeta_n)_{n\in\bb{N}}$ in $\overline{K}$. Sometimes we denote $\zeta_n$ by $t_{0,n}$.

\begin{mypara}\label{para:notation-Abar}
	Let $A$ be a Noetherian normal domain flat over $\bb{Z}_p$, $A_{\triv}$ a localization of $A$ with respect to a nonzero element of $pA$, $\ca{K}$ the fraction field of $A$, $\overline{\ca{K}}$ an algebraic closure of $\ca{K}$. The \emph{maximal unramified extension} $\ca{K}_{\mrm{ur}}$ of $\ca{K}$ with respect to $(A_{\triv},A)$ is the union of all finite extensions $\ca{K}'$ of $\ca{K}$ contained in $\overline{\ca{K}}$ such that the integral closure of $A_{\triv}$ in $\ca{K}'$ is \'etale over $A_{\triv}$. It is a Galois extension over $\ca{K}$, whose Galois group $\gal(\ca{K}_{\mrm{ur}}/\ca{K})$ is denoted by $G_A$. We call the integral closure $\overline{A}$ of $A$ in $\ca{K}_{\mrm{ur}}$ the \emph{maximal unramified extension} of $A$ with respect to $A_{\triv}$. It is $G_A$-stable under the natural action of $G_A$ on $\ca{K}_{\mrm{ur}}$. We remark that the integral closure of $A$ in any finite field extension $\ca{K}'$ of $\ca{K}$ is a Noetherian normal domain finite over $A$ by \ref{rem:int-clos-fini}. 
\end{mypara}

\begin{mydefn}\label{defn:triple}
	A \emph{$(K,\ca{O}_K,\ca{O}_{\overline{K}})$-triple} is a triple $(A_{\triv},A,\overline{A})$ consisting of a Noetherian normal domain $A$ flat over $\ca{O}_K$ with $A/pA\neq 0$, a localization $A_\triv$ of $A$ with respect to a nonzero element of $pA$, and the maximal unramified extension $\overline{A}$ of $A$ with respect to $A_{\triv}$ contained in an algebraic closure $\overline{\ca{K}}$ of the fraction field $\ca{K}$ of $A$ containing $\overline{K}$.
	
	A morphism of $(K,\ca{O}_K,\ca{O}_{\overline{K}})$-triples $(A_{\triv},A,\overline{A})\to (A'_{\triv},A',\overline{A'})$ is a homomorphism of $\ca{O}_{\overline{K}}$-algebras $f:\overline{A}\to \overline{A'}$ such that $f(A)\subseteq A'$ and $f(A_{\triv})\subseteq A'_{\triv}$. If $f$ is injective, then it induces an extension of the fraction fields $\ca{K}_{\mrm{ur}}\to \ca{K}'_{\mrm{ur}}$ and thus a natural homomorphism of Galois groups $G_{A'}\to G_A$.
\end{mydefn}

We will use this definition to describe the functoriality of the Faltings extensions (cf. \ref{rem:B-fal-ext}).

\begin{mypara}\label{para:notation-A-inj}
	Let $(A_{\triv},A,\overline{A})$ be a $(K,\ca{O}_K,\ca{O}_{\overline{K}})$-triple. We denote by $\ak{E}(A)$ the set of morphisms of $(K,\ca{O}_K,\ca{O}_{\overline{K}})$-triples $(A_{\triv},A,\overline{A})\to (E,\ca{O}_E,\ca{O}_{\overline{E}})$, where $E$ is a complete discrete valuation field extension of $K$ whose residue field admits a finite $p$-basis. There is a natural right action of the Galois group $G_A$ on $\ak{E}(A)$ defined by sending $v\in\ak{E}(A)$ to $v\circ g\in \ak{E}(A)$, where $g\in G_A$ is regarded as an automorphism of the $(K,\ca{O}_K,\ca{O}_{\overline{K}})$-triple $(A_{\triv},A,\overline{A})$ by the natural action of $G_A$ on $\overline{A}$.
	
	We fix an injection 
	\begin{align}\label{eq:notation-A-inj}
		\ak{S}_p(\overline{A})\longrightarrow \ak{E}(A),\ \ak{q}\mapsto ((A_{\triv},A,\overline{A})\to (E_{\ak{p}},\ca{O}_{E_{\ak{p}}},\ca{O}_{\overline{E}_{\ak{q}}})),
	\end{align}
	where $\ak{p}\in\ak{S}_p(A)$ is the image of $\ak{q}$, $\ca{O}_{E_{\ak{p}}}$ is the $p$-adic completion of the localization $A_{\ak{p}}$, $\overline{E}_{\ak{q}}$ is an algebraic closure of $E_{\ak{p}}$, and $\overline{A}\to\ca{O}_{\overline{E}_{\ak{q}}}$ is the injection induced by an extension of valuation rings $\overline{A}_{\ak{q}}\to \ca{O}_{\overline{E}_{\ak{q}}}$ over $A_{\ak{p}}\to \ca{O}_{E_{\ak{p}}}$ (cf. \ref{prop:ht1-prime-inj}). In particular, there is a natural homomorphism of Galois groups $\gal(\overline{E}_{\ak{q}}/E_{\ak{p}})\to \gal(\ca{K}_{\mrm{ur}}/\ca{K})=G_A$ by \ref{defn:triple}.
\end{mypara}

\begin{mylem}\label{lem:A-inj}
	Let $(A_{\triv},A,\overline{A})$ be a $(K,\ca{O}_K,\ca{O}_{\overline{K}})$-triple. The natural map of $p$-adic completions $f:\widehat{\overline{A}}\to \prod_{\ak{E}(A)}\ca{O}_{\widehat{\overline{E}}}$ is $G_A$-equivariant and injective.
\end{mylem}
\begin{proof}
	 We note that the natural $G_A$-action on $\prod_{\ak{E}(A)}\ca{O}_{\widehat{\overline{E}}}$ is given by  
	 \begin{align}\label{eq:lem:A-inj}
	 	g(x_v)_{v\in \ak{E}(A)}=(x_{v\circ g})_{v\in \ak{E}(A)}
	 \end{align}
	 for any $g\in G_A$. Thus, for any $x\in \widehat{\overline{A}}$, the $v$-component of $g (f(x))$ is the $(v\circ g)$-component of $f(x)$, which is equal to the image of $x$ under the composition of $\widehat{\overline{A}}\stackrel{g}{\longrightarrow}\widehat{\overline{A}}\stackrel{f_v}{\longrightarrow}\ca{O}_{\widehat{\overline{E}}}$, i.e. the $v$-component of $f(gx)$. This shows the $G_A$-equivariance of $f$. With the notation in \ref{para:notation-A-inj}, the natural maps
	\begin{align}
		\widehat{\overline{A}}\longrightarrow \prod_{\ak{q}\in\ak{S}_p(\overline{A})}(\overline{A}_{\ak{q}})^\wedge\longrightarrow \prod_{\ak{q}\in\ak{S}_p(\overline{A})}\ca{O}_{\widehat{\overline{E}}_{\ak{q}}}
	\end{align}
	are injective by \ref{prop:ht1-prime-inj}. Their composition is also the composition of the natural maps
	\begin{align}
		\widehat{\overline{A}}\longrightarrow \prod_{\ak{E}(A)}\ca{O}_{\widehat{\overline{E}}}\longrightarrow \prod_{\ak{q}\in\ak{S}_p(\overline{A})}\ca{O}_{\widehat{\overline{E}}_{\ak{q}}}
	\end{align}
	where the second map is induced by \eqref{eq:notation-A-inj}, which completes the proof.
\end{proof}

\begin{mydefn}\label{defn:quasi-adequate-alg}
	A $(K,\ca{O}_K,\ca{O}_{\overline{K}})$-triple $(B_{\triv},B,\overline{B})$ is called \emph{quasi-adequate} if there exists a commutative diagram of monoids
	\begin{align}
		\xymatrix{
			B& P\ar[l]_-{\beta}\\
			\ca{O}_K\ar[u]& \bb{N}\ar[l]_-{\alpha}\ar[u]_-{\gamma}
		}
	\end{align} 
	satisfying the following conditions:
	\begin{enumerate}
		\renewcommand{\labelenumi}{{\rm(\theenumi)}}
		\item The element $\alpha(1)$ is a uniformizer of $\ca{O}_K$.
		\item The monoid $P$ is fs, and if we denote by $\gamma_\eta:\bb{Z}\to P_\eta=\bb{Z}\oplus_{\bb{N}}P$ the pushout of $\gamma$ by the inclusion $\bb{N}\to \bb{Z}$, then there exists an isomorphism for some $c, d\in \bb{N}$ with $c\leq d$,
		\begin{align}\label{eq:monoid-str}
			P_\eta\cong \bb{Z}\oplus \bb{Z}^c\oplus \bb{N}^{d-c}
		\end{align}
		identifying $\gamma_\eta$ with the inclusion of $\bb{Z}$ into the first component of the right hand side.
		\item The homomorphism $\beta$ induces an injective ring homomorphism of finite type $A=\ca{O}_K\otimes_{\bb{Z}[\bb{N}]}\bb{Z}[P]\to B$ which is pre-perfectoid in the sense of \ref{defn:bound} such that $B\otimes_{\bb{Z}[P]}\bb{Z}[P^\mrm{gp}]=B_\triv$ and $A[1/p]\to B[1/p]$ is \'etale.
	\end{enumerate}
	We usually denote $(B_{\triv},B,\overline{B})$ by $B$, and call it a \emph{quasi-adequate $\ca{O}_K$-algebra} for simplicity. The triple $(\alpha:\bb{N}\to \ca{O}_K,\ \beta:P\to B,\ \gamma:\bb{N}\to P)$ is called a \emph{quasi-adequate chart} of $B$. If we fix an isomorphism \eqref{eq:monoid-str}, then we call the images $t_1,\dots,t_d\in B[1/p]$ of the standard basis of $\bb{Z}^c\oplus \bb{N}^{d-c}$ a \emph{system of coordinates} of the chart. We call $d$ the \emph{relative dimension} of $B$ over $\ca{O}_K$ (i.e. the Krull dimension of $B_{\triv}$). If $A\to B$ is \'etale, then we say that $B$ is an \emph{adequate $\ca{O}_K$-algebra} and call $(\alpha,\beta,\gamma)$ an \emph{adequate chart}.
\end{mydefn}

\begin{myrem}\label{rem:quasi-adequate}
	In \ref{defn:quasi-adequate-alg}, the condition $B/pB\neq 0$ imposes that $A/pA\neq 0$. As $A$ is a Noetherian normal domain flat over $\ca{O}_K$ by \ref{rem:adequate-chart}, if we set $A_{\triv}=\ca{O}_K\otimes_{\bb{Z}[P]}\bb{Z}[P^\mrm{gp}]$ and denote by $\overline{A}$ the maximal unramified extension of $A$ contained in $\overline{B}$, then $(A_{\triv},A,\overline{A})$ is an adequate $(K,\ca{O}_K,\ca{O}_{\overline{K}})$-triple. The inclusion $\overline{A}\subseteq \overline{B}$ induces an injective morphism of quasi-adequate $(K,\ca{O}_K,\ca{O}_{\overline{K}})$-triples $(A_{\triv},A,\overline{A})\to (B_{\triv},B,\overline{B})$. If $A\to B$ is \'etale (so that $B$ is adequate) and if we endow $\spec(B)$ with the compactifying log structure associated to the open immersion $\spec(B_{\triv})\to \spec(B)$, then it becomes a log scheme over $S=(\spec(\ca{O}_K),\alpha_{\spec(K)\to\spec(\ca{O}_K)})$ with an adequate chart in the sense of \ref{defn:adequate-chart} induced by $(\alpha,\beta,\gamma)$ (cf. \ref{rem:adequate-chart}).
\end{myrem}

\begin{myrem}\label{rem:quasi-adequate-exmp}
	Let $B'$ be a $B$-algebra which is a Noetherian normal domain flat over $\ca{O}_K$ with $B\to B'$ injective and $B'/pB'\neq 0$. We set $B'_{\triv}=B_{\triv}\otimes_B B'$ and take a maximal unramified extension $\overline{B'}$ of $(B'_{\triv},B')$ containing $\overline{B}$. Then, $B'$ is a quasi-adequate $\ca{O}_K$-algebra with the same chart of $B$ if $B[1/p]\to B'[1/p]$ is \'etale and if the ring homomorphism $B\to B'$ is of finite type and pre-perfectoid. This is satisfied in each of the following cases: 
	\begin{enumerate}
		\renewcommand{\labelenumi}{{\rm(\theenumi)}}
		\item The ring homomorphism $B\to B'$ is \'etale.
		\item The $B$-algebra $B'$ is the integral closure of $B$ in a finite \'etale $B[1/p]$-algebra (cf. \ref{cor:bound-fet}).
		\item The $B$-algebra $B'$ is the normalization of an affine blowup algebra $B[I/a]$, where $I$ is a finitely generated ideal of $B$ containing a power of $p$, and $a$ is a non-zero element of $I$ (cf. \ref{cor:bound-blowup}).
	\end{enumerate}
\end{myrem}

\begin{mylem}\label{lem:quasi-adequate-regular}
	Let $B$ be a quasi-adequate $\ca{O}_K$-algebra.
	\begin{enumerate}
		\renewcommand{\labelenumi}{{\rm(\theenumi)}}
		\item A system of coordinates $t_1,\dots, t_d\in B[1/p]$ defines a strict normal crossings divisor on the regular scheme $\spec(B[1/p])$, i.e. in the localization of $B[1/p]$ at any point, those elements $t_i$ contained in the maximal ideal form a subset of a regular system of parameters.\label{item:lem:quasi-adequate-regular-1}
		\item Let $(\alpha:\bb{N}\to \ca{O}_K,\ \beta:P\to B,\ \gamma:\bb{N}\to P)$ be a quasi-adequate chart of $B$, $Y=\spec(\beta:P\to B)$. Then, the generic fibre $Y_K$ is regular whose log structure is the compactifying log structure associated to the open immersion $\spec(B_{\triv})\to\spec(B[1/p])$.\label{item:lem:quasi-adequate-regular-2}
	\end{enumerate}
\end{mylem}
\begin{proof}
	(\ref{item:lem:quasi-adequate-regular-1}) follows from the fact that the system of coordinates $t_1,\dots, t_d\in B[1/p]$ identifies $B[1/p]$ with an \'etale $K[\bb{Z}^c\oplus \bb{N}^{d-c}]$-algebra (cf. \ref{rem:adequate-chart}). (\ref{item:lem:quasi-adequate-regular-2}) follows from the observation that  $Y_K=\spec(\bb{Z}^c\oplus \bb{N}^{d-c}\to B[1/p])$ (cf. \ref{para:log-reg}).
\end{proof}

\begin{mylem}\label{lem:quasi-adequate-diff}
	Let $B$ be a quasi-adequate $\ca{O}_K$-algebra and we fix a chart of $B$ as in {\rm\ref{defn:quasi-adequate-alg}}. Then, there exists $k\in \bb{N}$ such that the canonical truncation of the cotangent complex $\tau_{\leq 1}\dl_{B/A}$ is $p^k$-exact in the sense of {\rm\ref{defn:pi-iso}}.
\end{mylem}
\begin{proof}
	Since $A\to B$ is a ring homomorphism of finite type between Noetherian rings, the homology groups of $\mrm{L}_{B/A}$ are finitely generated $B$-modules (\cite[\Luoma{2}.2.3.7]{illusie1971cot1}). The conclusion follows from the fact that for any integer $n$, $H_n(\mrm{L}_{B/A})[1/p]=H_n(\mrm{L}_{B[1/p]/A[1/p]})=0$ as $B[1/p]$ is \'etale over $A[1/p]$.
\end{proof}

\begin{mylem}\label{lem:diff-abs}
	There exists $k\in \bb{N}$ such that for any $\ca{O}_K$-algebra $R$, the canonical map of $p$-adic completions $\widehat{\Omega}^1_R\to (\Omega^1_{R/\ca{O}_K})^\wedge$ is a $p^k$-isomorphism.
\end{mylem}
\begin{proof}
	Since $K$ is a complete discrete valuation field extension of $\bb{Q}_p$ with perfect residue field, $\widehat{\Omega}^1_{\ca{O}_K}$ is killed by $p^k$ for some $k\in \bb{N}$ (\cite[3.3]{he2021faltingsext}). For any $r\in\bb{N}$, we see that $\Omega^1_R/p^r\to \Omega^1_{R/\ca{O}_K}/p^r$ is surjective whose kernel $N_r$ is killed by $p^k$. Taking limit over $r\in\bb{N}$, since $\lim N_r$ and $\rr^1\lim N_r$ are also killed by $p^k$, the conclusion follows.
\end{proof}

\begin{myprop}\label{prop:quasi-adequate-diff}
	Let $B$ be a quasi-adequate $\ca{O}_K$-algebra, $S$ a multiplicative subset of $B$. Then, $\widehat{\Omega}^1_{S^{-1}B}[1/p]$ is a finite free $(S^{-1}B)^\wedge[1/p]$-module, where the completions are $p$-adic. Moreover, it admits a basis $\df t_1,\dots,\df t_d$ for any system of coordinates $t_1,\dots,t_d\in B[1/p]$.
\end{myprop}
\begin{proof}
	We note that $\widehat{\Omega}^1_{S^{-1}B}[1/p]=(\Omega^{1}_{S^{-1}B/\ca{O}_K})^\wedge[1/p]$ by \ref{lem:diff-abs}. We take a quasi-adequate chart of $B$ as in \ref{defn:quasi-adequate-alg}. By the fundamental distinguished triangle of cotangent complexes (\cite[\Luoma{2}.2.1.5.6]{illusie1971cot1}), we obtain an exact sequence,
	\begin{align}
		H_1(\mrm{L}_{S^{-1}B/A})\to S^{-1}B\otimes_A \Omega^{1}_{A/\ca{O}_K}\stackrel{\alpha}{\longrightarrow} \Omega^{1}_{S^{-1}B/\ca{O}_K} \to \Omega^{1}_{S^{-1}B/A}\to 0.
	\end{align}
	By \ref{lem:quasi-adequate-diff}, $\alpha$ is a $p^k$-isomorphism for some $k\in\bb{N}$. Taking $p$-adic completion, we obtain an isomorphism  $(S^{-1}B\otimes_A \Omega^{1}_{A/\ca{O}_K})^\wedge[1/p]\iso (\Omega^{1}_{S^{-1}B/\ca{O}_K})^\wedge[1/p]$ by \ref{rem:pi-iso-retract}.(\ref{item:rem:pi-iso-retract-1}). As $A$ is of finite type over $\ca{O}_K$, $\Omega^{1}_{A/\ca{O}_K}$ is a finitely generated $A$-module. Since $S^{-1}B$ is Noetherian, we have $(S^{-1}B\otimes_A \Omega^{1}_{A/\ca{O}_K})^\wedge[1/p]=(S^{-1}B)^\wedge\otimes_A \Omega^{1}_{A/\ca{O}_K}[1/p]$. The conclusion follows from the fact that $A[1/p]=K[\bb{Z}^c\oplus\bb{N}^{d-c}]$ (cf. \ref{rem:adequate-chart}).
\end{proof}

\begin{mycor}\label{cor:quasi-adequate-rank}
	Let $B$ be a quasi-adequate $\ca{O}_K$-algebra with relative dimension $d$. Then, for any $\ak{p}\in\ak{S}_p(B)$ (cf. {\rm\ref{defn:ht1-prime}}), we have $[\kappa(\ak{p}):\kappa(\ak{p})^p]=p^d$, where $\kappa(\ak{p})$ is the residue field of $B$ at $\ak{p}$.
\end{mycor}
\begin{proof}
	Notice that $B_{\ak{p}}$ is a discrete valuation ring extension of $\bb{Z}_p$. Thus, the rank of the $\widehat{B_{\ak{p}}}[1/p]$-module $\widehat{\Omega}^1_{B_{\ak{p}}}[1/p]$ is equal to $\log_{p}[\kappa(\ak{p}),\kappa(\ak{p})^p]$, where the completions are $p$-adic (\cite[3.3]{he2021faltingsext}). The conclusion follows from \ref{prop:quasi-adequate-diff}.
\end{proof}

\begin{mypara}\label{para:notation-quasi-adequate-tower}
	Let $B$ be a quasi-adequate $\ca{O}_K$-algebra, $\ca{L}$ its fraction field, $\ca{L}_{\mrm{ur}}$ the fraction field of $\overline{B}$, $G=\gal(\ca{L}_{\mrm{ur}}/\ca{L})$. As in \ref{para:X}, we construct coverings of $B$ by quasi-adequate algebras. We fix $B$ in the rest of this section, as well as the following notation.
	
	We fix a quasi-adequate chart $(\alpha:\bb{N}\to \ca{O}_K,\ \beta:P\to B,\ \gamma:\bb{N}\to P)$ of $B$, $(A_{\triv},A,\overline{A})$ the associated adequate $\ca{O}_K$-algebra defined in \ref{rem:quasi-adequate}, and a system of coordinates $t_1,\dots,t_d\in A[1/p]$. Let $\ca{K}$ (resp. $\ca{K}_{\mrm{ur}}$) be the fraction field of $A$ (resp. $\overline{A}$). For $1\leq i\leq d$, we fix a compatible system of $k$-th roots $(t_{i,k})_{k\in \bb{N}}$ of $t_i$ in $\overline{A}[1/p]$. For any field extension $E'/E$, let $\scr{F}_{E'/E}$ (resp. $\scr{F}^{\mrm{fini}}_{E'/E}$) be the set of algebraic (resp. finite) field extensions of $E$ contained in $E'$, and we endow it with the partial order defined by the inclusion relation. For any $L\in \ff{K}$ and any $\underline{r}=(r_1,\dots,r_d)\in\bb{N}^d_{>0}$, we set
	\begin{align}
		\ca{K}^L_{\underline{r}}=L\ca{K}(t_{i,r_i}\ |\ 1\leq i\leq d)\quad\trm{ and }\quad \ca{L}^L_{\underline{r}}=\ca{L}\ca{K}^L_{\underline{r}}
	\end{align}
	where the composites of fields are taken in $\ca{L}_{\mrm{ur}}$ (which contains $\ca{K}_{\mrm{ur}}$). It is clear that $\ca{K}^L_{\underline{r}}$ (resp. $\ca{L}^L_{\underline{r}}$) forms a system of fields over the directed partially ordered set $\ff{K}\times \bb{N}^d_{>0}$ (cf. \ref{para:product}). Let $A^L_{\underline{r}}$ (resp. $B^L_{\underline{r}}$) be the integral closure of $A$ (resp. $B$) in $\ca{K}^L_{\underline{r}}$ (resp. $\ca{L}^L_{\underline{r}}$). We note that there is an isomorphism of $L$-algebras
	\begin{align}\label{eq:A-str}
		L[T_1^{\pm 1},\dots,T_c^{\pm 1},T_{c+1},\dots,T_d]\iso A^L_{\underline{r}}[1/p]
	\end{align}
	sending $T_i$ to $t_{i,r_i}$ as $A[1/p]=K[\bb{Z}^c\oplus\bb{N}^{d-c}]$ (cf. \ref{rem:adequate-chart}). 
	
	Let $S^L$ (resp. $X^L_{\underline{r}}$) be the log scheme with underlying scheme $\spec(\ca{O}_L)$ (resp. $\spec(A^L_{\underline{r}})$) endowed with the compactifying log structure associated to the open immersion $\spec(L)\to \spec(\ca{O}_L)$ (resp. $\spec(A^L_{\underline{r},\triv})\to \spec(A^L_{\underline{r}})$ where $A^L_{\underline{r},\triv}=A_{\triv}\otimes_A A^L_{\underline{r}}$). The following lemma \ref{lem:A-X} guarantees the consistency of this notation with the notation in \ref{para:X}. Let $Y^L_{\underline{r}}$ be the log scheme with underlying scheme $\spec(B^L_{\underline{r}})$ whose log structure is the inverse image of that of $X^L_{\underline{r}}$ via the map $\spec(B^L_{\underline{r}})\to \spec(A^L_{\underline{r}})$ (i.e. $Y^L_{\underline{r}}\to X^L_{\underline{r}}$ is strict). We extend the notation above to any $(L,\underline{r})\in\scr{F}_{\overline{K}/K}\times (\bb{N}\cup\{\infty\})^d_{>0}$ by taking filtered colimits.
\end{mypara}

\begin{mylem}\label{lem:A-X}
	With the notation in {\rm\ref{para:X}} for the chart $(\alpha:\bb{N}\to \ca{O}_K,\ \beta:P\to B,\ \gamma:\bb{N}\to P)$ of $B$, for any $L\in\ff{K}$ and $\underline{r}=(r_1,\dots,r_d)\in\bb{N}^d_{>0}$, the homomorphism of monoids $P_{1,\underline{r}}\to A^L_{\underline{r}}$ sending $(k_0,k_1/r_1,\dots,k_d/r_d)$ to $\alpha(1)^{k_0}\cdot t_{1,r_1}^{k_1}\cdots t_{d,r_d}^{k_d}$ induces an isomorphism between  $\spec(A^L_{\underline{r}})$ with the underlying scheme of $S^L\times_{\bb{A}_{\bb{N}}}^{\mrm{fs}}\bb{A}_{P_{1,\underline{r}}}$. In particular, $X^L_{\underline{r}}=S^L\times_{\bb{A}_{\bb{N}}}^{\mrm{fs}}\bb{A}_{P_{1,\underline{r}}}$, where the left hand side is defined in {\rm\ref{para:notation-quasi-adequate-tower}}.
\end{mylem}
\begin{proof}
	Let $\spec(A')$ denote the underlying scheme of $S^L\times_{\bb{A}_{\bb{N}}}^{\mrm{fs}}\bb{A}_{P_{1,\underline{r}}}$. By \ref{lem:X-adequate-chart} and \ref{rem:adequate-chart}, $A'$ is a Noetherian normal domain finite over $A$ such that $A'[1/p]=L[r_1^{-1}\bb{N}\oplus\cdots r_c^{-1}\bb{N}\oplus r_{c+1}^{-1}\bb{Z}\oplus\cdots\oplus r_d^{-1}\bb{Z}]$. Thus, $A'[1/p]=A^L_{\underline{r}}[1/p]$ by \eqref{eq:A-str}, and we obtain $A'=A^L_{\underline{r}}$. The ``in particular'' part follows from the fact that the log structures on both sides are the compactifying log structure associated to the open immersion $\spec(A^L_{\underline{r},\triv})\to \spec(A^L_{\underline{r}})$ by definition, \ref{lem:X-adequate-chart} and \ref{rem:adequate-chart}.
\end{proof}

\begin{myprop}\label{prop:B-adequate}
	Let $L\in\ff{K}$ and $\underline{r}=(r_1,\dots,r_d)\in\bb{N}^d_{>0}$.
	\begin{enumerate}
		\renewcommand{\labelenumi}{{\rm(\theenumi)}}
		\item The $(L,\ca{O}_L,\ca{O}_{\overline{K}})$-algebra $(B^L_{\underline{r},\triv},B^L_{\underline{r}},\overline{B})$ is quasi-adequate with a chart (with the notation in {\rm\ref{para:X}})
		\begin{align}\label{eq:B-chart}
			(\alpha^L: e_L^{-1}\bb{N}\to \ca{O}_L,\ \beta^L_{\underline{r}}:P_{e_L,\underline{r}}\to B^L_{\underline{r}},\ \gamma^L_{\underline{r}}:e_L^{-1}\bb{N}\to P_{e_L,\underline{r}})
		\end{align}
		where $e_L$ is the ramification index of $L/K$, $\alpha^L$ is a homomorphism of monoids sending $e_L^{-1}$ to a uniformizer $\pi_L$ of $L$, $\beta^L_{\underline{r}}(k_0/e_L,k_1/r_1,\dots,k_d/r_d)=\pi_L^{k_0}\cdot t_{1,r_1}^{k_1}\cdots t_{d,r_d}^{k_d}$, and $\gamma^L_{\underline{r}}$ is induced by the inclusion of the first component of \eqref{eq:monoid-tower-str-P}. Moreover, $A^L_{\underline{r}}=\ca{O}_L\otimes_{\bb{Z}[e_L^{-1}\bb{N}]}\bb{Z}[P_{e_L,\underline{r}}]$. \label{item:prop:B-adequate-1}
		\item The scheme $\spec(B^L_{\underline{r}}[1/p])$ is an open and closed subscheme of $\spec(B\otimes_A A^L_{\underline{r}}[1/p])$. The two schemes are equal if and only if $\ca{L}^L_{\underline{r}}=\ca{L}\otimes_{\ca{K}}\ca{K}^L_{\underline{r}}$ (i.e. $[\ca{L}^L_{\underline{r}}:\ca{L}]=[\ca{K}^L_{\underline{r}}:\ca{K}]$). \label{item:prop:B-adequate-2}
		\item For any $L'\in\scr{F}^{\mrm{fini}}_{\overline{K}/L}$ and any element $\underline{r}'\in\bb{N}^d_{>0}\cdot\underline{r}$, the morphism of generic fibres $Y^{L'}_{\underline{r}',K}\to Y^L_{\underline{r},K}$ is \'etale.\label{item:prop:B-adequate-3}
	\end{enumerate}
\end{myprop}
\begin{proof}
	As $X^L_{\underline{r}}=S^L\times_{\bb{A}_{\bb{N}}}^{\mrm{fs}}\bb{A}_{P_{1,\underline{r}}}$ by \ref{lem:A-X}, we see that $X^{L'}_{\underline{r}',K}\to X^L_{\underline{r},K}$ is \'etale. In particular, $Y_K\times^{\mrm{fs}}_X X^L_{\underline{r}}$ is \'etale over $Y_K$ with underlying scheme $\spec(B\otimes_A A^L_{\underline{r}}[1/p])$ (as $Y\to X$ is strict). Since $Y_K$ is regular by \ref{lem:quasi-adequate-regular}.(\ref{item:lem:quasi-adequate-regular-2}), so is $Y_K\times^{\mrm{fs}}_X X^L_{\underline{r}}$. We see that $\spec(B\otimes_A A^L_{\underline{r}}[1/p])$ is the disjoint union of finitely many normal integral schemes whose set of generic points identifies with $\spec(\ca{L}\otimes_{\ca{K}}\ca{K}^L_{\underline{r}})$ (cf. \ref{para:log-reg}), and thus by definition $\spec(B^L_{\underline{r}}[1/p])$ is one of these components corresponding to the generic point $\spec(\ca{L}^L_{\underline{r}})$, so that we obtain (\ref{item:prop:B-adequate-2}). This implies that $Y^L_{\underline{r},K}\to Y_K\times^{\mrm{fs}}_X X^L_{\underline{r}}$ is an open and closed immersion. Thus, (\ref{item:prop:B-adequate-3}) follows as $Y_K\to X_K$ is strict and \'etale. 
	
	For (\ref{item:prop:B-adequate-1}), we take $\alpha^L,\beta^L_{\underline{r}},\gamma^L_{\underline{r}}$ as in the statement. Note that $A^L_{\underline{r}}=\ca{O}_L\otimes_{\bb{Z}[e_L^{-1}\bb{N}]}\bb{Z}[P_{e_L,\underline{r}}]$ by \ref{lem:A-X} and \ref{lem:X-adequate-chart}, and moreover it defines an adequate $\ca{O}_L$-algebra. It remains to check that the ring homomorphism $A^L_{\underline{r}}\to B^L_{\underline{r}}$ satisfies the conditions in the definition \ref{defn:quasi-adequate-alg} of quasi-adequate algebras. We have seen above that $A^L_{\underline{r}}[1/p]\to B^L_{\underline{r}}[1/p]$ is \'etale. As $B^L_{\underline{r}}$ is finite over $B$ by \ref{rem:int-clos-fini}, it is of finite type over $A^L_{\underline{r}}$. Finally, it follows from \ref{lem:pre-perfd-basic}.(\ref{item:lem:pre-perfd-basic-2}) that $A\to B$ being pre-perfectoid implies that so is $A^L_{\underline{r}}\to B^L_{\underline{r}}$.
\end{proof}

\begin{mylem}\label{lem:B-perfd}
	Let $F\in\scr{F}_{\overline{K}/K}$ be a pre-perfectoid field. Then, the $\ca{O}_F$-algebra $B^F_{\underline{\infty}}$ is almost pre-perfectoid.
\end{mylem}
\begin{proof}
	Recall that $B^F_{\underline{\infty}}$ is the filtered colimit of $B^L_{\underline{r}}$ over $(L,\underline{r})\in \scr{F}^{\mrm{fini}}_{F/K}\times \bb{N}^d_{>0}$ (cf. \ref{para:notation-quasi-adequate-tower}). We claim that the Frobenius on $A^F_{\underline{\infty}}/pA^F_{\underline{\infty}}$ is surjective.	Recall that for any $L\in\scr{F}^{\mrm{fini}}_{F/K}$ and $\underline{r}\in\bb{N}^d_{>0}$, there is an isomorphism $f^L_{\underline{r}}:\ca{O}_L\otimes_{\bb{Z}[e_L^{-1}\bb{N}]}\bb{Z}[P_{e_L,\underline{r}}]\iso A^L_{\underline{r}}$ sending $1\otimes \exp(k_0/e_L,k_1/r_1,\dots,k_d/r_d)$ to $\pi_L^{k_0}\cdot t_{1,r_1}^{k_1}\cdots t_{d,r_d}^{k_d}$ by \ref{lem:A-X} and \ref{lem:X-adequate-chart}, where $\pi_L$ is a uniformizer of $L$ and $x=(k_0/e_L,k_1/r_1,\dots,k_d/r_d)\in P_{e_L,\underline{r}}$. Since $F$ is a non-discrete valuation field such that the Frobenius map on $\ca{O}_F/p\ca{O}_F$ is surjective, there exists $L'\in \scr{F}^{\mrm{fini}}_{F/L}$ such that $e_{L'}/e_L$ is divisible by $p$ (\cite[5.4]{he2021coh}). We take $\underline{r}'=p\underline{r}\in\bb{N}^d_{>0}$. By definition, we have $x'=(k_0/pe_L,k_1/r_1',\dots,k_d/r_d')\in P_{e_{L'},\underline{r}'}$ as $px'=x$. Thus, there is a unit $u\in\ca{O}_{L'}^\times$ such that $f^L_{\underline{r}}(x)=u\cdot f^{L'}_{\underline{r}'}(x')^p$. Since the Frobenius map on $\ca{O}_F/p\ca{O}_F$ is surjective, $u$ admits a $p$-th root in $\ca{O}_F/p\ca{O}_F$, and thus $f^L_{\underline{r}}(x)$ admits a $p$-th root in $A^F_{\underline{\infty}}/pA^F_{\underline{\infty}}$, which proves the claim.
	
	Since $A^F_{\underline{\infty}}$ is normal and flat over $\ca{O}_F$, the Frobenius map induces an injection $\phi:A^F_{\underline{\infty}}/p_1A^F_{\underline{\infty}}\to A^F_{\underline{\infty}}/pA^F_{\underline{\infty}}$ (\cite[5.21]{he2021coh}), where $p_1$ is a $p$-th root of $p$ up to a unit. By the claim above, $\phi$ is an isomorphism, which means that the $\ca{O}_F$-algebra $A^F_{\underline{\infty}}$ is almost pre-perfectoid (\cite[5.19]{he2021coh}). The conclusion follows from the fact that $A\to B$ is pre-perfectoid.
\end{proof}

\begin{mylem}\label{lem:B-structure}
	Let $L\subseteq L'$ be elements of $\ff{K}$, $\underline{r}|\underline{r}'$ elements of $\bb{N}^d_{>0}$. Assume that $\ca{L}^{L'}_{\underline{r}'}=\ca{L}^L_{\underline{r}}\otimes_{\ca{K}^L_{\underline{r}}}\ca{K}^{L'}_{\underline{r}'}$ (i.e. $[\ca{L}^{L'}_{\underline{r}'}:\ca{L}^L_{\underline{r}}]=[\ca{K}^{L'}_{\underline{r}'}:\ca{K}^L_{\underline{r}}]$). Then, the finite free $B^L_{\underline{r}}$-module
	\begin{align}\label{eq:lem:B-structure}
		D^{L'/L}_{\underline{r}'/\underline{r}}=\bigoplus_{ \underline{k}\in I}\ca{O}_{L'}\otimes_{\ca{O}_L} B^L_{\underline{r}}\cdot\prod_{i=1}^d t_{i,r_i'}^{k_i},
	\end{align}
	identifies naturally with a finite free $B^L_{\underline{r}}$-subalgebra of $B^{L'}_{\underline{r}'}[1/p]$, where $I=\{(k_1,\dots,k_d)\in\bb{N}^d\ |\ 0\leq k_i<r_i'/r_i,\ 1\leq i\leq d\}$. Moreover, for any finite field extension $L''$ of $L'$ and any element $\underline{r}''\in\bb{N}^d_{>0}$ divisible by $\underline{r}'$ such that $\ca{L}^{L''}_{\underline{r}''}=\ca{L}^L_{\underline{r}}\otimes_{\ca{K}^L_{\underline{r}}}\ca{K}^{L''}_{\underline{r}''}$, we have
	\begin{align}
		D^{L'/L}_{\underline{r}'/\underline{r}}=B^{L'}_{\underline{r}'}[1/p]\cap D^{L''/L}_{\underline{r}''/\underline{r}}.
	\end{align} 
\end{mylem}
\begin{proof}
	We put $C^{L'/L}_{\underline{r}'/\underline{r}}=\bigoplus_{ \underline{k}\in I}\ca{O}_{L'}\otimes_{\ca{O}_L} A^L_{\underline{r}}\cdot\prod_{i=1}^d t_{i,r_i'}^{k_i}$. Thus, $D^{L'/L}_{\underline{r}'/\underline{r}}=B^L_{\underline{r}}\otimes_{A^L_{\underline{r}}}C^{L'/L}_{\underline{r}'/\underline{r}}$. By assumption, we have $B^{L'}_{\underline{r}'}[1/p]=B^L_{\underline{r}}\otimes_{A^L_{\underline{r}}}A^{L'}_{\underline{r}'}[1/p]$ (cf. \ref{prop:B-adequate}.(\ref{item:prop:B-adequate-2})), which is equal to $D^{L'/L}_{\underline{r}'/\underline{r}}[1/p]$ as $A^{L'}_{\underline{r}'}[1/p]= C^{L'/L}_{\underline{r}'/\underline{r}}[1/p]$ (cf. \ref{prop:X-tower-str}). This proves the first assertion. For the ``moreover'' part, notice that $\ca{O}_{L''}\otimes_{\ca{O}_{L'}}D^{L'/L}_{\underline{r}'/\underline{r}}$ is a direct summand of $D^{L''/L}_{\underline{r}''/\underline{r}}$, both regarded as $\ca{O}_{L''}\otimes_{\ca{O}_L}B^L_{\underline{r}}$-submodules of $B^{L''}_{\underline{r}''}[1/p]$. Therefore, 
	\begin{align}
		\ca{O}_{L''}\otimes_{\ca{O}_{L'}}D^{L'/L}_{\underline{r}'/\underline{r}}=(L''\otimes_{\ca{O}_{L'}}D^{L'/L}_{\underline{r}'/\underline{r}})\cap D^{L''/L}_{\underline{r}''/\underline{r}}\supseteq B^{L'}_{\underline{r}'}[1/p]\cap D^{L''/L}_{\underline{r}''/\underline{r}}.
	\end{align}
	It remains to check that $B^{L'}_{\underline{r}'}[1/p]\cap (\ca{O}_{L''}\otimes_{\ca{O}_{L'}}D^{L'/L}_{\underline{r}'/\underline{r}})=D^{L'/L}_{\underline{r}'/\underline{r}}$. As $B^{L'}_{\underline{r}'}[1/p]=D^{L'/L}_{\underline{r}'/\underline{r}}[1/p]$, restricting to the coefficient of each $\prod_{i=1}^d t_{i,r_i'}^{k_i}$ by \eqref{eq:lem:B-structure}, we reduce to show that $(L'\otimes_{\ca{O}_L}B^L_{\underline{r}})\cap (\ca{O}_{L''}\otimes_{\ca{O}_L}B^L_{\underline{r}})=\ca{O}_{L'}\otimes_{\ca{O}_L}B^L_{\underline{r}}$ as submodules of $L''\otimes_{\ca{O}_L}B^L_{\underline{r}}$. This follows from the identity $L'\cap \ca{O}_{L''}=\ca{O}_{L'}$ and the flatness of $B^L_{\underline{r}}$ over $\ca{O}_L$.
\end{proof}

\begin{myprop}\label{prop:B-p-iso}
	Let $F\in\scr{F}_{\overline{K}/K}$ be a pre-perfectoid field. Then, there exists $k_0\in\bb{N}$ such that for any $L\in\scr{F}^{\mrm{fini}}_{F/K}$ and $\underline{r}\in\bb{N}^d_{>0}$, if $\ca{L}^F_{\underline{\infty}}=\ca{L}^L_{\underline{r}}\otimes_{\ca{K}^L_{\underline{r}}}\ca{K}^F_{\underline{\infty}}$, then the natural map
	\begin{align}\label{eq:prop:B-p-iso}
		B^L_{\underline{r}}\otimes_{A^L_{\underline{r}}}A^{L'}_{\underline{r}'}\longrightarrow B^{L'}_{\underline{r}'}
	\end{align}
	is a $p^{k_0}$-isomorphism for any $L'\in \scr{F}^{\mrm{fini}}_{F/L}$ and $\underline{r}'\in\bb{N}^d_{>0}\cdot\underline{r}$. 
\end{myprop}
\begin{proof}
	The assumption $\ca{L}^F_{\underline{\infty}}=\ca{L}^L_{\underline{r}}\otimes_{\ca{K}^L_{\underline{r}}}\ca{K}^F_{\underline{\infty}}$ implies that $\ca{L}^{L'}_{\underline{r}'}=\ca{L}^L_{\underline{r}}\otimes_{\ca{K}^L_{\underline{r}}}\ca{K}^{L'}_{\underline{r}'}$ (i.e. $[\ca{L}^{L'}_{\underline{r}'}:\ca{L}^L_{\underline{r}}]=[\ca{K}^{L'}_{\underline{r}'}:\ca{K}^L_{\underline{r}}]$) for any $L'\in \scr{F}^{\mrm{fini}}_{F/L}$ and $\underline{r}'\in\bb{N}^d_{>0}\cdot\underline{r}$. We take again the notation in \ref{lem:B-structure} and its proof, and we put $D^{F/L}_{\underline{\infty}/\underline{r}}=\colim_{L',\underline{r}'} D^{L'/L}_{\underline{r}'/\underline{r}}$ as $B^L_{\underline{r}}$-subalgebra of $B^F_{\underline{\infty}}[1/p]$. By \ref{lem:B-structure}, we have
	\begin{align}\label{eq:prop:B-p-iso-1}
		D^{L'/L}_{\underline{r}'/\underline{r}}=B^{L'}_{\underline{r}'}[1/p]\cap D^{F/L}_{\underline{\infty}/\underline{r}}.
	\end{align}
	On the other hand, since $A\to B$ is pre-perfectoid and $\ca{K}^F_{\underline{\infty}}\in\scr{P}_A$ (cf. \ref{defn:bound}) by \ref{lem:B-perfd}, there exists $k_0\in\bb{N}$ such that 
	\begin{align}\label{eq:prop:B-p-iso-2}
		p^{k_0}\cok(B\otimes_A A^F_{\underline{\infty}}\to B^F_{\underline{\infty}})=0.
	\end{align}
	After enlarging $k_0$ (independently of $L,L'$) by \ref{lem:A-X} and \ref{prop:X-tower-str}, we may also assume that 
	\begin{align}\label{eq:prop:B-p-iso-3}
		p^{k_0}A^{L'}_{\underline{r}'}\subseteq C^{L'/L}_{\underline{r}'/\underline{r}}\subseteq p^{-k_0}A^{L'}_{\underline{r}'}.
	\end{align}
	Applying the functor $B^L_{\underline{r}}\otimes_{A^L_{\underline{r}}}-$, we see that the natural maps
	\begin{align}\label{eq:prop:B-p-iso-4}
		B^L_{\underline{r}}\otimes_{A^L_{\underline{r}}}A^{L'}_{\underline{r}'}\longleftarrow B^L_{\underline{r}}\otimes_{A^L_{\underline{r}}}p^{k_0}A^{L'}_{\underline{r}'}\longrightarrow D^{L'/L}_{\underline{r}'/\underline{r}}
	\end{align}
	are $p^{4k_0}$-isomorphisms by \ref{rem:pi-iso-retract}.(\ref{item:rem:pi-iso-retract-1}). As $D^{L'/L}_{\underline{r}'/\underline{r}}\subseteq B^{L'}_{\underline{r}'}[1/p]$, we deduce that the kernel of \eqref{eq:prop:B-p-iso} is $p^{k_0}$-zero after replacing $k_0$ by a constant multiple. On the other hand, as $k_0$ is independent of $L,L'$, after replacing it by a constant multiple, we deduce from \eqref{eq:prop:B-p-iso-2} and \eqref{eq:prop:B-p-iso-4} that $p^{k_0}B^F_{\underline{\infty}}\subseteq D^{F/L}_{\underline{\infty}/\underline{r}}$. Thus, $p^{k_0}B^{L'}_{\underline{r}'}\subseteq B^{L'}_{\underline{r}'}[1/p]\cap D^{F/L}_{\underline{\infty}/\underline{r}}=D^{L'/L}_{\underline{r}'/\underline{r}}$ by \eqref{eq:prop:B-p-iso-1}, which implies by \eqref{eq:prop:B-p-iso-4} that the cokernel of \eqref{eq:prop:B-p-iso} is $p^{k_0}$-zero after replacing $k_0$ by a constant multiple.
\end{proof}

\begin{mycor}\label{cor:B-cotangent-bound}
	Let $F\in\scr{F}_{\overline{K}/K}$ be a pre-perfectoid field, $L\in\scr{F}^{\mrm{fini}}_{F/K}$, $\underline{r}\in\bb{N}^d_{>0}$. Assume that $\ca{L}^F_{\underline{\infty}}=\ca{L}^L_{\underline{r}}\otimes_{\ca{K}^L_{\underline{r}}}\ca{K}^F_{\underline{\infty}}$. Then, there exists $k_0\in\bb{N}$ such that the truncated cotangent complex $\tau_{\leq 1}\mrm{L}_{B^{L'}_{\underline{r}'}/A^{L'}_{\underline{r}'}}$ is $p^{k_0}$-exact for any $L'\in \scr{F}^{\mrm{fini}}_{F/L}$ and $\underline{r}'\in\bb{N}^d_{>0}\cdot\underline{r}$.
\end{mycor}
\begin{proof}
	We set $C=B^L_{\underline{r}}\otimes_{A^L_{\underline{r}}}A^{L'}_{\underline{r}'}$.
	\begin{align}
		\xymatrix{
			B^{L'}_{\underline{r}'}&C\ar[l]&A^{L'}_{\underline{r}'}\ar[l]\\
			&B^L_{\underline{r}}\ar[u]&A^L_{\underline{r}}\ar[l]\ar[u]
		}
	\end{align}
	We take $k_1\in\bb{N}$ such that $\tau_{\leq 1}\mrm{L}_{B^L_{\underline{r}}/A^L_{\underline{r}}}$ is $p^{k_1}$-exact by \ref{lem:quasi-adequate-diff} and \ref{prop:B-adequate}.(\ref{item:prop:B-adequate-1}). Thus, $\tau_{\leq 1}\mrm{L}_{C/A^{L'}_{\underline{r}'}}$ is $p^{2k_1}$-exact by \ref{cor:cotangent-bc}. Since $C\to B^{L'}_{\underline{r}'}$ is a $p^{k_0}$-isomorphism by \ref{prop:B-p-iso}, $\tau_{\leq 1}\mrm{L}_{B^{L'}_{\underline{r}'}/C}$ is $p^{102k_0}$-exact by \ref{prop:cotangent-bound}. After replacing $k_0$ by $3\max\{2k_1,102k_0\}$, we see that $\tau_{\leq 1}\mrm{L}_{B^{L'}_{\underline{r}'}/A^{L'}_{\underline{r}'}}$ is $p^{k_0}$-exact for any $L'\in \scr{F}^{\mrm{fini}}_{F/L}$ and $\underline{r}'\in\bb{N}^d_{>0}\cdot\underline{r}$ by \ref{lem:cotangent-bound-seq}.
\end{proof}

\begin{mycor}\label{cor:B-diff-compare}
	Under the assumptions in {\rm\ref{cor:B-cotangent-bound}} and with the same notation, let $L'\in \scr{F}^{\mrm{fini}}_{F/L}$ and $\underline{r}'\in\bb{N}^d_{>0}\cdot\underline{r}$.
	\begin{enumerate}
		\renewcommand{\labelenumi}{{\rm(\theenumi)}}
		\item The canonical morphism
		\begin{align}
			B^{L'}_{\underline{r}'}\otimes_{A^{L'}_{\underline{r}'}}\Omega^1_{X^{L'}_{\underline{r}'}/S^{L'}}&\longrightarrow \Omega^1_{Y^{L'}_{\underline{r}'}/S^{L'}}\label{eq:cor:B-diff-compare-1}
		\end{align}
		is a $p^{k_0}$-isomorphism.\label{item:cor:B-diff-compare-1}
		\item Let $L''\in \scr{F}^{\mrm{fini}}_{F/L'}$ and $\underline{r}''\in\bb{N}^d_{>0}\cdot\underline{r}'$. Then, the canonical morphism
		\begin{align}
			B^{L''}_{\underline{r}''}\otimes_{A^{L''}_{\underline{r}''}}\Omega^1_{X^{L''}_{\underline{r}''}/X^{L'}_{\underline{r}'}}&\longrightarrow \Omega^1_{Y^{L''}_{\underline{r}''}/Y^{L'}_{\underline{r}'}}\label{eq:cor:B-diff-compare-2}
		\end{align}
		is a $p^{2k_0}$-isomorphism.\label{item:cor:B-diff-compare-2}
	\end{enumerate}
\end{mycor}
\begin{proof}
	(\ref{item:cor:B-diff-compare-1}) The fundamental distinguished triangle of logarithmic cotangent complexes defined by Gabber (\cite[8.29]{olsson2005logcot}) gives a canonical exact sequence
	\begin{align}
		H_1(\dl_{Y^{L'}_{\underline{r}'}/X^{L'}_{\underline{r}'}})\longrightarrow B^{L'}_{\underline{r}'}\otimes_{A^{L'}_{\underline{r}'}}\Omega^1_{X^{L'}_{\underline{r}'}/S^{L'}}\stackrel{\alpha}{\longrightarrow} \Omega^1_{Y^{L'}_{\underline{r}'}/S^{L'}}\longrightarrow \Omega^1_{Y^{L'}_{\underline{r}'}/X^{L'}_{\underline{r}'}}\longrightarrow 0.
	\end{align}
	Since $Y^{L'}_{\underline{r}'}\to X^{L'}_{\underline{r}'}$ is strict, its logarithmic cotangent complex is quasi-isomorphic to the cotangent complex $\dl_{B^{L'}_{\underline{r}'}/A^{L'}_{\underline{r}'}}$ (\cite[8.22]{olsson2005logcot}). Thus, $\alpha$ is a $p^{k_0}$-isomorphism by \ref{cor:B-cotangent-bound}.
	
	(\ref{item:cor:B-diff-compare-2}) The proof is similar to that of (\ref{item:cor:B-diff-compare-1}). There are canonical exact sequences
	\begin{align}
		B^{L''}_{\underline{r}''}\otimes_{B^{L'}_{\underline{r}'}}\Omega^1_{Y^{L'}_{\underline{r}'}/X^{L'}_{\underline{r}'}}\longrightarrow &\Omega^1_{Y^{L''}_{\underline{r}''}/X^{L'}_{\underline{r}'}}\stackrel{\alpha}{\longrightarrow} \Omega^1_{Y^{L''}_{\underline{r}''}/Y^{L'}_{\underline{r}'}}\longrightarrow 0,\\
		H_1(\dl_{Y^{L''}_{\underline{r}''}/X^{L''}_{\underline{r}''}})\longrightarrow B^{L''}_{\underline{r}''}\otimes_{A^{L''}_{\underline{r}''}}\Omega^1_{X^{L''}_{\underline{r}''}/X^{L'}_{\underline{r}'}}\stackrel{\beta}{\longrightarrow} &\Omega^1_{Y^{L''}_{\underline{r}''}/X^{L'}_{\underline{r}'}}\longrightarrow \Omega^1_{Y^{L''}_{\underline{r}''}/X^{L''}_{\underline{r}''}}\longrightarrow 0.
	\end{align}
	By \ref{cor:B-cotangent-bound}, we see that $\alpha$ and $\beta$ are $p^{k_0}$-isomorphisms, and thus $\alpha\circ\beta$ is a $p^{2k_0}$-isomorphism.
\end{proof}

\begin{mypara}\label{para:notation-Ybar}
	We start to compute some modules of log differentials and construct the Faltings extension of $B$. Since $\ca{L}$ is a finite extension of $\ca{K}$, there exists $L_0\in\ff{K}$ and $r_0\in\bb{N}_{>0}$ such that $\ca{L}\cap \ca{K}^{\overline{K}}_{\underline{\infty}}\subseteq \ca{K}^{L_0}_{\underline{r_0}}$. As $\ca{K}^{\overline{K}}_{\underline{\infty}}$ is a Galois extension of $\ca{K}$, we see by Galois theory that for any $L\in\scr{F}^{\mrm{fini}}_{\overline{K}/L_0}$ and $r\in\bb{N}_{>0}\cdot r_0$, 
	\begin{align}
		\ca{L}^{L}_{\underline{r}}=\ca{L}^{L_0}_{\underline{r_0}}\otimes_{\ca{K}^{L_0}_{\underline{r_0}}}\ca{K}^{L}_{\underline{r}}\quad\trm{(i.e. $[\ca{L}^{L}_{\underline{r}}:\ca{L}^{L_0}_{\underline{r_0}}]=[\ca{K}^{L}_{\underline{r}}:\ca{K}^{L_0}_{\underline{r_0}}]$)}.
	\end{align}
	Recall that the log structure that we put on $Y^{\overline{K}}_{\underline{\infty}}$ is the colimit of the inverse images of log structures on $Y^L_{\underline{r}}$ over $(L,\underline{r})\in \scr{F}^{\mrm{fini}}_{\overline{K}/K}\times \bb{N}^d_{>0}$ (cf. \cite[12.2.10]{gabber2004foundations}). In the rest of this section, we denote by $\overline{Y}$ the log scheme with underlying scheme $\spec(\overline{B})$ whose log structure is the inverse image of that of $Y^{\overline{K}}_{\underline{\infty}}$ via the map $\spec(\overline{B})\to \spec(B^{\overline{K}}_{\underline{\infty}})$. We denote by $\mbf{e}_0,\dots,\mbf{e}_d$ the standard basis of $\bb{Z}^{1+d}$. Recall that for any $1\leq i\leq d$ and $r\in\bb{N}_{>0}$, the element $t_{i,r}\in \overline{B}[1/p]$ is the image of $r^{-1}\mbf{e}_i\in P_{1,r,\eta}=\bb{Z}\oplus (r^{-1}\bb{Z})^c\oplus (r^{-1}\bb{N})^{d-c}$ via the chart \eqref{eq:B-chart}. For any morphism of log schemes $\overline{Y}\to Z$ over $S$, we denote by $\df\log(t_{i,r})\in \Gamma(\spec(\overline{B}),\Omega^1_{\overline{Y}/ Z})$ the image of the global section $r^{-1}\mbf{e}_i$ via the canonical map $\df\log: \ca{M}_{\overline{Y}}^{\mrm{gp}}\to \Omega^1_{\overline{Y}/ Z}$ (cf. \ref{para:log-diff}).
\end{mypara}

\begin{mylem}[Abhyankar's lemma]\label{lem:abhyankar}
	Let $\ca{L}'\in\scr{F}^{\mrm{fini}}_{\ca{L}_{\mrm{ur}}/\ca{L}}$, $B'$ the integral closure of $B$ in $\ca{L}'$. Then, there exists $r\in\bb{N}_{>0}$ such that $B'_{\underline{r}}[1/p]$ is finite \'etale over $B_{\underline{r}}[1/p]$, where $B'_{\underline{r}}$ is the integral closure of $B'$ in $\ca{L}'_{\underline{r}}=\ca{L}'\ca{L}_{\underline{r}}$. In particular, if we set $B'_{\underline{r},\triv}=B_{\triv}\otimes_BB'_{\underline{r}}$, then $(B'_{\underline{r},\triv},B'_{\underline{r}},\overline{B})$ is a quasi-adequate $\ca{O}_K$-algebra with a chart induced by the chart \eqref{eq:B-chart} of $(B_{\underline{r},\triv},B_{\underline{r}},\overline{B})$.
\end{mylem}
\begin{proof}
	Recall that $t_1,\dots,t_d$ defines a strict normal crossings divisor on the regular scheme $\spec(B[1/p])$ by \ref{lem:quasi-adequate-regular}.(\ref{item:lem:quasi-adequate-regular-1}), and that for any $r\in\bb{N}_{>0}$, we have $A_{\underline{r}}[1/p]=A[1/p][T_1,\dots,T_d]/(T_1^r-t_1,\dots,T_d^r-t_d)$ by \eqref{eq:A-str}. By Abhyankar's lemma \cite[\Luoma{13}.5.2]{sga1}, there exists $r\in \bb{N}_{>0}$ such that the integral closure $C$ of $B\otimes_A A_{\underline{r}}[1/p]$ in $B'_{\triv}\otimes_A A_{\underline{r}}[1/p]$ is finite \'etale over $B\otimes_A A_{\underline{r}}[1/p]$. Notice that $\spec(B[1/p])$ is an open and closed subscheme of $\spec(B\otimes_A A_{\underline{r}}[1/p])$ by \ref{prop:B-adequate}.(\ref{item:prop:B-adequate-2}). By the same argument, we also see that $\spec(B'_{\underline{r}}[1/p])$ is an open and closed subscheme of $\spec(C)$. Thus, $B'_{\underline{r}}[1/p]$ is finite \'etale over $B_{\underline{r}}[1/p]$. The ``in particular'' part follows from \ref{rem:quasi-adequate-exmp}.
\end{proof}

\begin{myprop}\label{prop:B_nm-almost-purity}
	Let $F\in\scr{F}_{\overline{K}/K}$ be a pre-perfectoid field (e.g. $F=\overline{K}$). Then, for any $\ca{L}'\in\scr{F}^{\mrm{fini}}_{\ca{L}_{\mrm{ur}}/\ca{L}^F_{\underline{\infty}}}$, the integral closure $B'$ of $B$ in $\ca{L}'$ is almost finite \'etale over $B^F_{\underline{\infty}}$. In particular, the cotangent complex $\dl_{\overline{B}/B^F_{\underline{\infty}}}$ is almost zero.
\end{myprop}
\begin{proof}
	By Abhyankar's lemma \ref{lem:abhyankar} and a limit argument (cf. \cite[8.21]{he2021coh}), $B'[1/p]$ is finite \'etale over $B^F_{\underline{\infty}}[1/p]$. Since $B^F_{\underline{\infty}}$ is almost pre-perfectoid (\ref{lem:B-perfd}), $B'$ is almost finite \'etale over $B^F_{\underline{\infty}}$ by almost purity \ref{thm:almost-purity}. Thus, the cotangent complex $\dl_{B'/B^F_{\underline{\infty}}}$ is almost zero (\cite[2.5.37]{gabber2003almost}), and so is $\dl_{\overline{B}/B^F_{\underline{\infty}}}$ by taking filtered colimit.
\end{proof}

\begin{mylem}\label{lem:cycl-diff}
	Let $E'/E$ be an extension of discrete valuation fields, $Z$ (resp. $Z'$) the log scheme with underlying scheme $\spec(\ca{O}_{E})$ (resp. $\spec(\ca{O}_{E'})$) with log structure defined by the closed point. Then, the kernel of $\Omega^1_{\ca{O}_{E'}/\ca{O}_E}\to \Omega^1_{Z'/Z}$ is killed by a uniformizer of $E$, and its cokernel is killed by a uniformizer of $E'$.
\end{mylem}
\begin{proof}
	Let $\pi'$ (resp. $\pi$) be a uniformizer of $E'$ (resp. $E$). We write $\pi=u\pi'^e$ for some $u\in \ca{O}_{E'}^\times$ and $e\in\bb{N}$. Then, $Z'\to Z$ admits a chart $(\alpha:\bb{N}\to \ca{O}_E,\ \beta:e^{-1}\bb{N}\oplus \bb{Z}\to \ca{O}_{E'},\ \gamma:\bb{N}\to e^{-1}\bb{N}\oplus \bb{Z})$ where $\alpha(1)=\pi$, $\beta(e^{-1},0)=\pi'$, $\beta(0,1)=u$ and $\gamma(1)=(1,1)$. By \ref{para:log-diff}, there is an $\ca{O}_{E'}$-linear surjection
	\begin{align}\label{eq:lem:cycl-diff}
		\Omega^1_{\ca{O}_{E'}/\ca{O}_E}\oplus \ca{O}_{E'}\otimes (e^{-1}\bb{Z}\oplus \bb{Z})/\bb{Z}\longrightarrow \Omega^1_{Z'/Z}
	\end{align}
	whose kernel $M$ is generated by $(\df u,-u\otimes(0,1))$ and $(\df\pi',-\pi'\otimes(e^{-1},0))$. Thus, the cokernel of $\Omega^1_{\ca{O}_{E'}/\ca{O}_E}\to \Omega^1_{Z'/Z}$ is killed by $\pi'$. Let $\omega\in \Omega^1_{\ca{O}_{E'}/\ca{O}_E}\cap M$. We have $\omega=a\df u+b\df\pi'$ for some $a,b\in \ca{O}_{E'}$ such that $b\pi'=eau$. Since $0=\df\pi=\pi'^e\df u+eu\pi'^{e-1}\df\pi'$, we see that $\pi\omega=au(\pi'^e\df u+eu\pi'^{e-1}\df\pi')=0$.
\end{proof}

\begin{myprop}[{\cite[Th\'eor\`eme 1']{fontaine1982formes}}]\label{prop:cycl-diff}
	The $\ca{O}_{\overline{K}}$-linear homomorphism
	\begin{align}
		\overline{K}/\ca{O}_{\overline{K}}\longrightarrow \Omega^1_{S^{\overline{K}}/S},
	\end{align}
	sending $p^{-n}$ to $\df\log(\zeta_{p^n})$ for any $n\in \bb{N}$, is a $p^k$-isomorphism for some $k\in \bb{N}$.
\end{myprop}
\begin{proof}
	Recall that there is a fractional ideal $\ak{a}$ of $\overline{K}$ and an $\ca{O}_{\overline{K}}$-linear isomorphism
	\begin{align}
		\overline{K}/\ak{a}\iso \Omega^1_{\ca{O}_{\overline{K}}/\ca{O}_K},
	\end{align}
	sending $p^{-n}$ to $\df\log(\zeta_{p^n})$ for any $n\in \bb{N}$ (\cite[Th\'eor\`eme 1']{fontaine1982formes}). The conclusion follows from the fact that $\Omega^1_{\ca{O}_L/\ca{O}_K}\to \Omega^1_{S^L/S}$ is a $p$-isomorphism for any $L\in\ff{K}$ by \ref{lem:cycl-diff}.
\end{proof}

\begin{mylem}\label{lem:B-diff-S}
	With the notation in {\rm\ref{para:notation-Ybar}}, there exists $k_0\in\bb{N}$ such that for any $L\in\scr{F}^{\mrm{fini}}_{\overline{K}/L_0}$, the $B^L_{\underline{\infty}}$-linear map
	\begin{align}\label{eq:lem:B-diff-S}
		B^L_{\underline{\infty}}[\frac{1}{p}]^d\longrightarrow \Omega^1_{Y^L_{\underline{\infty}}/S^L},
	\end{align}
	sending $p^{-n}\mbf{e}_i$ to $\df\log (t_{i,p^n})$ for any $1\leq i\leq d$ and $n\in\bb{N}$, is a $p^{k_0}$-isomorphism.
\end{mylem}
\begin{proof}
	Since $X^L_{\underline{r}}=S^L\times_{\bb{A}_{\bb{N}}}^{\mrm{fs}}\bb{A}_{P_{1,\underline{r}}}$ for any $r\in\bb{N}_{>0}$ by \ref{lem:A-X}, we have $\Omega^1_{X^L_{\underline{r}}/S^L}=A^L_{\underline{r}}\otimes_{\bb{Z}}(P_{1,\underline{r}}^{\trm{gp}}/\bb{N}^{\trm{gp}})=A^L_{\underline{r}}\otimes_{\bb{Z}}(r^{-1}\bb{Z})^d$, which identifies $\df\log(t_{i,p^n})$ with $p^{-n}\mbf{e}_i$ if $p^n|r$ (cf. \ref{para:notation-Ybar}). Taking $r\to \infty$, we obtain the conclusion by \ref{cor:B-diff-compare}.(\ref{item:cor:B-diff-compare-1}).
\end{proof}

\begin{mylem}\label{lem:B-diff-ari}
	With the notation in {\rm\ref{para:notation-Ybar}}, there exists $k\in\bb{N}$ such that for any $r\in\bb{N}_{>0}\cdot r_0$, the $B^{\overline{K}}_{\underline{r}}$-linear map
	\begin{align}\label{eq:lem:B-diff-ari}
		B^{\overline{K}}_{\underline{r}}[\frac{1}{p}]/B^{\overline{K}}_{\underline{r}}\longrightarrow \Omega^1_{Y^{\overline{K}}_{\underline{r}}/Y^{L_0}_{\underline{r}}},
	\end{align}
	sending $p^{-n}$ to $\df\log(\zeta_{p^n})$ for any $n\in\bb{N}$, is a $p^k$-isomorphism.
\end{mylem}
\begin{proof}
	Since $X^L_{\underline{r}}= S^L\times_{S^{L_0}}^{\mrm{fs}}X^{L_0}_{\underline{r}}$ for any $L\in\scr{F}^{\mrm{fini}}_{\overline{K}/L_0}$ by \ref{lem:A-X}, we have $\Omega^1_{X^L_{\underline{r}}/X^{L_0}_{\underline{r}}}=A^L_{\underline{r}}\otimes_{\ca{O}_L}\Omega^1_{S^L/S^{L_0}}$. Taking $L$ running through $\scr{F}^{\mrm{fini}}_{\overline{K}/L_0}$, we see that there exists $k\in \bb{N}$ by \ref{prop:cycl-diff} such that the map $A^{\overline{K}}_{\underline{r}}\otimes_{\ca{O}_{\overline{K}}}\overline{K}/\ca{O}_{\overline{K}}\to \Omega^1_{X^{\overline{K}}_{\underline{r}}/X^{L_0}_{\underline{r}}}$ is a $p^k$-isomorphism for any $r\in\bb{N}_{>0}$. The conclusion follows from \ref{cor:B-diff-compare}.(\ref{item:cor:B-diff-compare-2}).
\end{proof}

\begin{myprop}[{cf. \cite[\Luoma{2}.7.9]{abbes2016p}, \cite[3.6]{he2021faltingsext}}]\label{prop:YbarS-diff}
	The $\overline{B}$-linear homomorphism 
	\begin{align}
		(\overline{B}[\frac{1}{p}]/\overline{B})\oplus \overline{B}[\frac{1}{p}]^d\longrightarrow \Omega^1_{\overline{Y}/S}
	\end{align}
	sending $p^{-n}\mbf{e}_i$ to $\df\log(t_{i,p^n})$ for any $n\in\bb{N}$ and $0\leq i\leq d$ (where $t_{0,p^n}=\zeta_{p^n}$), is a $p^k$-isomorphism for some $k\in \bb{N}$.
\end{myprop}
\begin{proof}
	With the notation in {\rm\ref{para:notation-Ybar}}, consider the canonical exact sequences
	\begin{align}
		B^{\overline{K}}_{\underline{\infty}}\otimes_{B^{L_0}_{\underline{\infty}}}\Omega^1_{Y^{L_0}_{\underline{\infty}}/S^{L_0}}\stackrel{\alpha_1}{\longrightarrow} &\Omega^1_{Y^{\overline{K}}_{\underline{\infty}}/S^{L_0}}\stackrel{\beta_1}{\longrightarrow} \Omega^1_{Y^{\overline{K}}_{\underline{\infty}}/Y^{L_0}_{\underline{\infty}}}\longrightarrow 0,\\
		B^{\overline{K}}_{\underline{\infty}}\otimes_{\ca{O}_{\overline{K}}}\Omega^1_{S^{\overline{K}}/S^{L_0}}\stackrel{\alpha_2}{\longrightarrow} &\Omega^1_{Y^{\overline{K}}_{\underline{\infty}}/S^{L_0}}\stackrel{\beta_2}{\longrightarrow} \Omega^1_{Y^{\overline{K}}_{\underline{\infty}}/S^{\overline{K}}}\longrightarrow 0.
	\end{align}
	By \ref{lem:B-diff-ari} and \ref{prop:cycl-diff}, there are $p^k$-isomorphisms
	\begin{align}
		B^{\overline{K}}_{\underline{\infty}}\otimes_{\ca{O}_{\overline{K}}}\overline{K}/\ca{O}_{\overline{K}}&\longrightarrow \Omega^1_{Y^{\overline{K}}_{\underline{\infty}}/Y^{L_0}_{\underline{\infty}}},\\
		\overline{K}/\ca{O}_{\overline{K}}&\longrightarrow \Omega^1_{S^{\overline{K}}/S^{L_0}}.
	\end{align}
	On the other hand, by \ref{lem:B-diff-S}, there are $p^{k_0}$-isomorphisms
	\begin{align}
		B^{L_0}_{\underline{\infty}}[\frac{1}{p}]^d&\longrightarrow \Omega^1_{Y^{L_0}_{\underline{\infty}}/S^{L_0}},\\
		B^{\overline{K}}_{\underline{\infty}}[\frac{1}{p}]^d&\longrightarrow\Omega^1_{Y^{\overline{K}}_{\underline{\infty}}/S^{\overline{K}}}.
	\end{align}
	By considering the compositions $\beta_2\circ\alpha_1$ and $\beta_1\circ\alpha_2$, one checks easily that the $B^{\overline{K}}_{\underline{\infty}}$-linear homomorphism
	\begin{align}
		(B^{\overline{K}}_{\underline{\infty}}[\frac{1}{p}]/B^{\overline{K}}_{\underline{\infty}})\oplus B^{\overline{K}}_{\underline{\infty}}[\frac{1}{p}]^d\longrightarrow \Omega^1_{Y^{\overline{K}}_{\underline{\infty}}/S^{L_0}}
	\end{align}
	sending $p^{-n}\mbf{e}_i$ to $\df\log(t_{i,p^n})$ for any $n\in\bb{N}$ and $0\leq i\leq d$, is a $p^l$-isomorphism for some $l\in \bb{N}$. Since $S^{L_0}\to S$ is a morphism of Noetherian fs log schemes which induces an \'etale morphism of the generic fibres and induces a finite morphism of the underlying schemes, the homology groups of the logarithmic cotangent complex $\dl_{S^{L_0}/S}$ defined by Gabber are $p$-primary torsion, finitely generated $\ca{O}_{L_0}$-modules (\cite[8.30]{olsson2005logcot}). After enlarging $l$ (depending only on $L_0$), we may assume that the map $(B^{\overline{K}}_{\underline{\infty}}[1/p]/B^{\overline{K}}_{\underline{\infty}})\oplus B^{\overline{K}}_{\underline{\infty}}[1/p]^d\to \Omega^1_{Y^{\overline{K}}_{\underline{\infty}}/S}$ is a $p^l$-isomorphism. The conclusion follows from \ref{prop:B_nm-almost-purity}.
\end{proof}

\begin{mylem}\label{lem:B-diff-geo}
	With the notation in {\rm\ref{para:notation-Ybar}}, there exists $k_0\in\bb{N}$ such that for any $L\in\scr{F}^{\mrm{fini}}_{\overline{K}/L_0}$, the $B^L_{\underline{\infty}}$-linear map
	\begin{align}\label{eq:lem:B-diff-geo}
		(B^L_{\underline{\infty}}[\frac{1}{p}]/r_0^{-1}B^L_{\underline{\infty}})^d\longrightarrow \Omega^1_{Y^L_{\underline{\infty}}/Y^L_{\underline{r_0}}},
	\end{align}
	sending $p^{-n}\mbf{e}_i$ to $\df\log (t_{i,p^n})$ for any $n\in\bb{N}$ and $1\leq i\leq d$, is a $p^{k_0}$-isomorphism.
\end{mylem}
\begin{proof}
	Since $X^L_{\underline{r}'}= X^L_{\underline{r}}\times_{\bb{A}_{P_{1,\underline{r}}}}^{\mrm{fs}}\bb{A}_{P_{1,\underline{r}'}}$ for any elements $r|r'$ in $\bb{N}_{>0}$ by \ref{lem:A-X}, we have $\Omega^1_{X^L_{\underline{r}'}/ X^L_{\underline{r}}}=A^L_{\underline{r}'}\otimes_{\bb{Z}}(P_{1,\underline{r}'}^{\trm{gp}}/P_{1,\underline{r}}^{\trm{gp}})=A^L_{\underline{r}'}\otimes_{\bb{Z}}(r'^{-1}\bb{Z}/r^{-1}\bb{Z})^d$, which identifies $\df\log(t_{i,p^n})$ with $p^{-n}\mbf{e}_i$ if $p^n|r'$ (cf. \ref{para:notation-Ybar}). Taking $r=r_0$ and $r'\to \infty$, we obtain the conclusion by \ref{cor:B-diff-compare}.(\ref{item:cor:B-diff-compare-2}).
\end{proof}

\begin{myprop}[{cf. \cite[\Luoma{2}.7.13]{abbes2016p}, \cite[4.2]{he2021faltingsext}}]\label{prop:YbarY-diff}
	The $\overline{B}$-linear homomorphism 
	\begin{align}
		(\overline{B}[\frac{1}{p}]/\overline{B})^{1+d}\longrightarrow \Omega^1_{\overline{Y}/Y}
	\end{align}
	sending $p^{-n}\mbf{e}_i$ to $\df\log(t_{i,p^n})$ for any $n\in\bb{N}$ and $0\leq i\leq d$ (where $t_{0,p^n}=\zeta_{p^n}$), is a $p^k$-isomorphism for some $k\in \bb{N}$.
\end{myprop}
\begin{proof}
	The proof is similar to that of \ref{prop:YbarS-diff}. With the notation in {\rm\ref{para:notation-Ybar}}, consider the canonical exact sequences
	\begin{align}
		B^{\overline{K}}_{\underline{\infty}}\otimes_{B^{L_0}_{\underline{\infty}}}\Omega^1_{Y^{L_0}_{\underline{\infty}}/Y^{L_0}_{\underline{r_0}}}\stackrel{\alpha_1}{\longrightarrow} &\Omega^1_{Y^{\overline{K}}_{\underline{\infty}}/Y^{L_0}_{\underline{r_0}}}\stackrel{\beta_1}{\longrightarrow} \Omega^1_{Y^{\overline{K}}_{\underline{\infty}}/Y^{L_0}_{\underline{\infty}}}\longrightarrow 0,\\
		B^{\overline{K}}_{\underline{\infty}}\otimes_{B^{\overline{K}}_{\underline{r_0}}}\Omega^1_{Y^{\overline{K}}_{\underline{r_0}}/Y^{L_0}_{\underline{r_0}}}\stackrel{\alpha_2}{\longrightarrow} &\Omega^1_{Y^{\overline{K}}_{\underline{\infty}}/Y^{L_0}_{\underline{r_0}}}\stackrel{\beta_2}{\longrightarrow} \Omega^1_{Y^{\overline{K}}_{\underline{\infty}}/Y^{\overline{K}}_{\underline{r_0}}}\longrightarrow 0.
	\end{align}
	By \ref{lem:B-diff-ari}, there are $p^{k}$-isomorphisms
	\begin{align}
		B^{\overline{K}}_{\underline{\infty}}\otimes_{\ca{O}_{\overline{K}}}\overline{K}/\ca{O}_{\overline{K}}&\longrightarrow \Omega^1_{Y^{\overline{K}}_{\underline{\infty}}/Y^{L_0}_{\underline{\infty}}},\\
		B^{\overline{K}}_{\underline{r_0}}\otimes_{\ca{O}_{\overline{K}}}\overline{K}/\ca{O}_{\overline{K}}&\longrightarrow \Omega^1_{Y^{\overline{K}}_{\underline{r_0}}/Y^{L_0}_{\underline{r_0}}}.
	\end{align}
	On the other hand, by \ref{lem:B-diff-geo}, there are $p^{k_0}$-isomorphisms
	\begin{align}
		(B^{L_0}_{\underline{\infty}}[\frac{1}{p}]/r_0^{-1}B^{L_0}_{\underline{\infty}})^d&\longrightarrow \Omega^1_{Y^{L_0}_{\underline{\infty}}/Y^{L_0}_{\underline{r_0}}},\\
		(B^{\overline{K}}_{\underline{\infty}}[\frac{1}{p}]/r_0^{-1}B^{\overline{K}}_{\underline{\infty}})^d&\longrightarrow\Omega^1_{Y^{\overline{K}}_{\underline{\infty}}/Y^{\overline{K}}_{\underline{r_0}}}.
	\end{align}
	By considering the compositions $\beta_2\circ\alpha_1$ and $\beta_1\circ\alpha_2$, one checks easily that the $B^{\overline{K}}_{\underline{\infty}}$-linear homomorphism
	\begin{align}
		(B^{\overline{K}}_{\underline{\infty}}[\frac{1}{p}]/B^{\overline{K}}_{\underline{\infty}})^{1+d}\longrightarrow \Omega^1_{Y^{\overline{K}}_{\underline{\infty}}/Y^{L_0}_{\underline{r_0}}}
	\end{align}
	sending $p^{-n}\mbf{e}_i$ to $\df\log(t_{i,p^n})$ for any $n\in\bb{N}$ and $0\leq i\leq d$, is a $p^l$-isomorphism for some $l\in \bb{N}$. Since $Y^{L_0}_{\underline{r_0}}\to Y$ is a morphism between Noetherian fs log schemes which induces an \'etale morphism of the generic fibres (\ref{prop:B-adequate}.(\ref{item:prop:B-adequate-3})) and induces a finite morphism of the underlying schemes (\ref{rem:int-clos-fini}), the homology groups of the logarithmic cotangent complex $\dl_{Y^{L_0}_{\underline{r_0}}/Y}$ defined by Gabber are $p$-primary torsion, finitely generated $B^{L_0}_{\underline{r_0}}$-modules (\cite[8.30]{olsson2005logcot}). After enlarging $l$ (depending only on $L_0$ and $r_0$), we may assume that the map $(B^{\overline{K}}_{\underline{\infty}}[\frac{1}{p}]/B^{\overline{K}}_{\underline{\infty}})^{1+d}\to \Omega^1_{Y^{\overline{K}}_{\underline{\infty}}/Y}$ is a $p^l$-isomorphism. The conclusion follows from \ref{prop:B_nm-almost-purity}.
\end{proof}

\begin{mypara}\label{para:ds}
	The $\widehat{\overline{B}}[1/p]$-module $V_p(\Omega^1_{\overline{Y}/Y})=\plim_{x\mapsto px}\Omega^1_{\overline{Y}/Y}$ (cf. \ref{para:notation-Tate-mod}) is endowed with a natural action of $G$. For any element $(s_{p^n})_{n\in \bb{N}}$ of $\plim_{x\mapsto x^p}\overline{B}[1/p]\cap \overline{B}_{\triv}^\times$, we take $l\in\bb{N}$ such that $p^ls_1\in \overline{B}$. Thus, $p^ls_{p^n}^r\in\overline{B}$ for any $n\in\bb{N}$ and $0\leq r\leq p^n$. Notice that the element $p^{-2l}((p^ls_{p^n}^{p^n-1})\df (p^ls_{p^n}))_{n\in\bb{N}}$ of $V_p(\Omega^1_{\overline{Y}/Y})$ does not depend on the choice of $l$. Thus, we denote this element by $(s_{p^n}^{p^n-1}\df s_{p^n})_{n\in\bb{N}}$.
\end{mypara}

\begin{myprop}[{cf. \cite[\Luoma{2}.7.22]{abbes2016p}, \cite[4.4]{he2021faltingsext}}]\label{prop:B_nm-fal-ext}
	There is a canonical $G$-equivariant exact sequence of $\widehat{\overline{B}}[1/p]$-modules,
	\begin{align}\label{eq:prop:B_nm-fal-ext}
		0\longrightarrow \widehat{\overline{B}}[\frac{1}{p}](1)\stackrel{\iota}{\longrightarrow}V_p(\Omega^1_{\overline{Y}/Y})\stackrel{\jmath}{\longrightarrow} \widehat{\overline{B}}[\frac{1}{p}]\otimes_B\Omega^1_{Y/S}\longrightarrow 0,
	\end{align}
	satisfying the following properties:
	\begin{enumerate}
		\renewcommand{\labelenumi}{{\rm(\theenumi)}}
		\item We have $\iota(1\otimes (\zeta_{p^n})_{n\in \bb{N}})=(\df\log(\zeta_{p^n}))_{n\in \bb{N}}$.\label{item:prop:B_nm-fal-ext-1}
		\item For any element $s\in B[1/p]\cap B_{\triv}^\times$ and any compatible system of $p$-power roots $(s_{p^n})_{n\in \bb{N}}$ of $s$ in $\overline{B}[1/p]$, $\jmath((s_{p^n}^{p^n-1}\df s_{p^n})_{n\in\bb{N}})=\df s$ (cf. {\rm\ref{para:ds}}).\label{item:prop:B_nm-fal-ext-2}
		\item The $\widehat{\overline{B}}[1/p]$-linear surjection $\jmath$ admits a section sending $\df\log(t_i)$ to $(\df\log(t_{i,p^n}))_{n\in\bb{N}}$ for any $1\leq i\leq d$.\label{item:prop:B_nm-fal-ext-3}
	\end{enumerate} 
	In particular, $V_p(\Omega^1_{\overline{Y}/Y})$ is a finite free $\widehat{\overline{B}}[\frac{1}{p}]$-module with basis $\{(\df\log(t_{i,p^n}))_{n\in\bb{N}}\}_{0\leq i\leq d}$, where $t_{0,p^n}=\zeta_{p^n}$.
\end{myprop}
\begin{proof}
	Consider the commutative diagram
	\begin{align}\label{diam:prop:B_nm-fal-ext-1}
		\xymatrix{
			0\ar[r]&0\oplus\overline{B}^d\ar[r]\ar[d]& (\overline{B}[\frac{1}{p}]/\overline{B})\oplus \overline{B}[\frac{1}{p}]^d\ar[r]\ar[d]& (\overline{B}[\frac{1}{p}]/\overline{B})\oplus (\overline{B}[\frac{1}{p}]/\overline{B})^d\ar[d]\ar[r]&0\\
			&\overline{B}\otimes_B \Omega^1_{Y/S}\ar[r]&\Omega^1_{\overline{Y}/S}\ar[r]&\Omega^1_{\overline{Y}/Y}\ar[r]&0
		}
	\end{align}
	where the vertical maps send $p^{-n}\mbf{e}_i$ to $\df\log(t_{i,p^n})$ for any $n\in\bb{N}$ and $0\leq i\leq d$. There exists $k\in \bb{N}$ such that the vertical maps are $p^k$-isomorphisms by \ref{lem:quasi-adequate-regular}.(\ref{item:lem:quasi-adequate-regular-2}), \ref{prop:YbarS-diff} and \ref{prop:YbarY-diff}. It is clear that the first row is exact and splits. Applying $\ho_{\bb{Z}_p}(\bb{Z}_p/p^n\bb{Z}_p,-)$ to \eqref{diam:prop:B_nm-fal-ext-1}, we get a commutative diagram
	\begin{align}\label{diam:prop:B_nm-fal-ext-2}
		\xymatrix@C=1pc{
			0\ar[r]\ar[d]& (p^{-n}\overline{B}/\overline{B})\oplus 0\ar[r]\ar[d]& (p^{-n}\overline{B}/\overline{B})\oplus (p^{-n}\overline{B}/\overline{B})^d\ar[d]\ar[r]&0\oplus(\overline{B}/p^n\overline{B})^d \ar[d]\ar[r]&0\\
			(\overline{B}\otimes_B \Omega^1_{Y/S})[p^n]\ar[r]&(\Omega^1_{\overline{Y}/S})[p^n]\ar[r]&\Omega^1_{\overline{Y}/Y}[p^n]\ar[r]^-{\jmath}&(\overline{B}\otimes_B \Omega^1_{Y/S})/p^n\ar[r]&0
		}
	\end{align}
	where the connecting map $\jmath$ sends $(p^ls_{p^n}^{p^n-1})\df (p^ls_{p^n})$ to $\df (p^{2l}s)$ (with the notation in \ref{para:ds}). The vertical maps are 
	$p^{2k}$-isomorphisms by \ref{rem:pi-iso-retract}.(\ref{item:rem:pi-iso-retract-1}), and the first row is exact and splits. Taking inverse limit on $n\in \bb{N}$ and inverting $p$, by \ref{rem:pi-iso-retract}.(\ref{item:rem:pi-iso-retract-2}), we get a canonical $G$-equivariant exact sequence, which admits a splitting (not $G$-equivariant),
	\begin{align}
		0\longrightarrow T_p(\Omega^1_{\overline{Y}/S})[\frac{1}{p}]\longrightarrow V_p(\Omega^1_{\overline{Y}/Y})\stackrel{\jmath}{\longrightarrow} \widehat{\overline{B}}[\frac{1}{p}]\otimes_B\Omega^1_{Y/S}\longrightarrow 0,
	\end{align}
	where we used the fact that $T_p(\Omega^1_{\overline{Y}/Y})[\frac{1}{p}]=V_p(\Omega^1_{\overline{Y}/Y})$ as $\Omega^1_{\overline{Y}/Y}$ is $p$-primary torsion, and that $\Omega^1_{Y/S}[1/p]$ is finite free over $B[1/p]$ with basis $\df\log(t_1),\dots,\df\log(t_d)$. Notice that \eqref{diam:prop:B_nm-fal-ext-2} also implies that $\iota:\widehat{\overline{B}}[\frac{1}{p}](1)\to T_p(\Omega^1_{\overline{Y}/S})[\frac{1}{p}]$ sending $1\otimes (\zeta_{p^n})_{n\in\bb{N}}$ to $(\df\log(\zeta_{p^n}))_{n\in\bb{N}}$ is an isomorphism. The conclusion follows.
\end{proof}

\begin{mypara}\label{para:B-fal-ext-compare}
	A priori, the sequence \eqref{eq:prop:B_nm-fal-ext} relies on the choice of the quasi-adequate chart of $B$. In the rest of this section, we show that it can be canonically defined, independently of the choice of a chart. Firstly, we check the compatibility of \eqref{eq:prop:B_nm-fal-ext} with the Faltings extension \eqref{eq:fal-ext} of a complete discrete valuation ring. Let  $(B_{\triv},B,\overline{B})\to (E,\ca{O}_E,\ca{O}_{\overline{E}})$ be an element of $\ak{E}(B)$ (cf. \ref{para:notation-A-inj}). Let $Z$ (resp. $\overline{Z}$) be the log scheme of underlying scheme $\spec(\ca{O}_E)$ (resp. $\spec(\ca{O}_{\overline{E}})$) with the compactifying log structure defined by the closed point. Notice that the map $\Omega^1_{\ca{O}_{\overline{E}}/\ca{O}_E}\to \Omega^1_{\overline{Z}/Z}$ is a $p$-isomorphism by \ref{lem:cycl-diff}, which thus induces a natural isomorphism of $\widehat{\overline{E}}$-modules $\scr{E}_{\ca{O}_E}=V_p(\Omega^1_{\ca{O}_{\overline{E}}/\ca{O}_E})\iso V_p(\Omega^1_{\overline{Z}/Z})$. The map $\Omega^1_{\ca{O}_E/\ca{O}_K}\to \Omega^1_{Z/S}$ is also a $p$-isomorphism by \ref{lem:cycl-diff}, which thus induces a natural isomorphism of $E$-modules $\widehat{\Omega}^1_{\ca{O}_E}[\frac{1}{p}]=(\Omega^1_{\ca{O}_E/\ca{O}_K})^\wedge[\frac{1}{p}]\iso (\Omega^1_{Z/S})^\wedge[\frac{1}{p}]$. By the explicit descriptions of $\iota$ and $\jmath$, the exact sequence \eqref{eq:prop:B_nm-fal-ext} fits into the following natural commutative diagram
	\begin{align}\label{eq:B-fal-ext-compare}
		\xymatrix{
			0\ar[r]& \widehat{\overline{B}}[\frac{1}{p}](1)\ar[r]^-{\iota}\ar[d]&V_p(\Omega^1_{\overline{Y}/Y})\ar[r]^-{\jmath}\ar[d]& \widehat{\overline{B}}[\frac{1}{p}]\otimes_{B}\Omega^1_{Y/S}\ar[r]\ar[d]& 0\\
			0\ar[r]& \prod_{\ak{E}(B)}\widehat{\overline{E}}(1)\ar[r]^-{\iota}&\prod_{\ak{E}(B)}\scr{E}_{\ca{O}_E}\ar[r]^-{\jmath}& \prod_{\ak{E}(B)}\widehat{\overline{E}}\otimes_{\ca{O}_E}\widehat{\Omega}^1_{\ca{O}_E}\ar[r]& 0
		}
	\end{align}
	where the second sequence is the product of the Faltings extensions of $\ca{O}_E$ defined in \ref{thm:fal-ext}. 
	
	For any $\ak{q}\in\ak{S}_p(\overline{B})$ with image $\ak{p}\in\ak{S}_p(B)$, consider the element $(B_{\triv},B,\overline{B})\to (E_{\ak{p}},\ca{O}_{E_{\ak{p}}},\ca{O}_{\overline{E}_{\ak{q}}})$ of $\ak{E}(B)$ defined in \ref{para:notation-A-inj}. Notice that $\{\df\log(t_i)\}_{1\leq i\leq d}$ forms an $E_{\ak{p}}$-basis of $\widehat{\Omega}^1_{\ca{O}_{E_{\ak{p}}}}[1/p]$ by \ref{prop:quasi-adequate-diff} (as $\ca{O}_{E_{\ak{p}}}$ is the $p$-adic completion of the localization $B_{\ak{p}}$). Thus, $\{(\df\log(t_{i,p^n}))_{n\in\bb{N}}\}_{0\leq i\leq d}$ forms an $\widehat{\overline{E}}_{\ak{q}}$-basis of $\scr{E}_{\ca{O}_{E_{\ak{p}}}}$. With respect to this basis, we can identify the projection of each vertical map in \eqref{eq:B-fal-ext-compare} to the components corresponding to $\ak{S}_p(\overline{B})$ with a direct sum of the natural map $\widehat{\overline{B}}[\frac{1}{p}]\to \prod_{\ak{q}\in\ak{S}_p(\overline{B})}  \widehat{\overline{E}}_{\ak{q}}$, which is injective by \ref{lem:A-inj}. In conclusion, the vertical maps in \eqref{eq:B-fal-ext-compare} are injective.
\end{mypara}

\begin{mylem}\label{lem:dlogs-unique}
	For any element $s\in B[1/p]\cap B_{\triv}^\times$ and any compatible system of $p$-power roots $(s_{p^n})_{n\in \bb{N}}$ of $s$ in $\overline{B}[1/p]$, there is at most one element $\omega\in V_p(\Omega^1_{\overline{Y}/Y})$ such that $s\omega=(s_{p^n}^{p^n-1}\df s_{p^n})_{n\in\bb{N}}$.
\end{mylem}
\begin{proof}
	By the discussion in \ref{para:B-fal-ext-compare}, we have $V_p(\Omega^1_{\overline{Y}/Y})\subseteq \prod_{\ak{q}\in\ak{S}_p(\overline{B})}\scr{E}_{\ca{O}_{E_{\ak{p}}}}$. Since $s$ acts invertibly on the latter, the unicity of $\omega$ follows.
\end{proof}

\begin{myprop}\label{prop:dlogs}
	For any element $s\in B[1/p]\cap B_{\triv}^\times$ and any compatible system of $p$-power roots $(s_{p^n})_{n\in \bb{N}}$ of $s$ in $\overline{B}[1/p]$, there is a unique element $\omega\in V_p(\Omega^1_{\overline{Y}/Y})$ such that $s\omega=(s_{p^n}^{p^n-1}\df s_{p^n})_{n\in\bb{N}}$, which we denote by $(\df\log (s_{p^n}))_{n\in\bb{N}}$.
\end{myprop}
\begin{proof}
	We note that for any $\ak{p}\in\spec(B[1/p])$, the non-units of $\{t_1,\dots,t_d\}$ in $B_{\ak{p}}$ form a subset of a regular system of parameters of the regular local ring $B_{\ak{p}}$ by \ref{lem:quasi-adequate-regular}.(\ref{item:lem:quasi-adequate-regular-1}). Since the divisor $\mrm{div}(s)$ on $\spec(B_{\ak{p}})$ defined by $s$ is set-theoretically contained in the union of the integral effective Cartier divisors $\mrm{div}(t_1),\dots, \mrm{div}(t_d)$, we can write $s=ut_1^{a_1}\cdots t_d^{a_d}$ for some $u\in B_{\ak{p}}^\times$ and $a_1,\dots,a_d\in \bb{N}$ (\cite[\href{https://stacks.math.columbia.edu/tag/0BCP}{0BCP}]{stacks-project}). Thus, we can take $f_1,\dots,f_l\in B$ such that $\{\spec(B[1/pf_i])\}_{1\leq i\leq l}$ forms an open covering of $\spec(B[1/p])$ and that $s$ admits an expression as before in each $B[1/pf_i]$. Consider the finitely generated ideal $I=(f_1,\dots,f_l)$ of $B$. It contains a power of $p$ as $I[1/p]=B[1/p]$. Consider the affine blowup algebra $B[I/f_i]$ (see \ref{para:blowup}) and its normalization $B_i$. It is clear that $B_i[1/p]=B[1/pf_i]$. As the image of $B[I/f_i]\otimes_BB[I/f_j]$ in $B_i\otimes_B B_j[1/p]=B[1/pf_if_j]$ is the affine blowup algebra $B[I^2/f_if_j]$, we see that the integral closure $B_{ij}$ of $B_i\otimes_B B_j$ in $B_i\otimes_B B_j[1/p]$ coincides with the normalization of $B[I^2/f_if_j]$ for any $1\leq i,j\leq l$. Thus, $B_i,B_{ij}$ define naturally quasi-adequate $\ca{O}_K$-algebras with the same chart induced by $B$ \eqref{eq:B-chart} (cf. \ref{rem:quasi-adequate-exmp}). For the simplicity of our arguments below, we consider the integral closures of $B_i,B_{ij}$ in $\ca{L}_{ur}$ (the fraction field of $\overline{B}$), and denote them abusively by $\overline{B_i},\overline{B_{ij}}$ (which may be smaller than the corresponding maximal unramified extensions). Therefore, $\overline{B_i}$ is the integral closure of $\overline{B}[I/f_i]$ in $\overline{B}[I/f_i][1/p]=\overline{B}[1/pf_i]$, which is almost pre-perfectoid over $\ca{O}_{\overline{K}}$ by \ref{thm:blow-up-perfd}. Since $\overline{B_{ij}}$ is the integral closure of $\overline{B_i}\otimes_{\overline{B}}\overline{B_j}$ in $\overline{B_i}\otimes_{\overline{B}}\overline{B_j}[1/p]=\overline{B}[1/pf_if_j]$, it is almost pre-perfectoid and whose $p$-adic completion is almost isomorphic to that of $\overline{B_i}\otimes_{\overline{B}}\overline{B_j}$ by \cite[5.33, 5.30]{he2021coh}. Since $\{\spec(\overline{B}[I/f_i])\to\spec(\overline{B})\}_{1\leq i\leq l}$ is the composition of a blowup with a Zariski open covering, $\{\spec(\overline{B_i})\to\spec(\overline{B})\}_{1\leq i\leq l}$ is a v-covering by \cite[3.15]{he2021coh}. Thus, the descent of perfectoid algebras in the arc-topology (due to Bhatt-Scholze, cf. \cite[5.35]{he2021coh}) induces an exact augmented \v Cech complex
	\begin{align}
		\xymatrix{
			0\ar[r]&\widehat{\overline{B}}[\frac{1}{p}]\ar[r]&\prod_{1\leq i\leq l}\widehat{\overline{B_i}}[\frac{1}{p}]\ar@<0.3ex>[r]\ar@<-0.3ex>[r]& \prod_{1\leq i,j\leq l}\widehat{\overline{B_{ij}}}[\frac{1}{p}].
		}
	\end{align}
	On the other hand, let $Y_i$ (resp. $\overline{Y_i}$) be the log scheme with underlying scheme $\spec(B_i)$ (resp. $\spec(\overline{B_i})$) whose log structure is the inverse image of that of $Y$ (resp. $\overline{Y}$). Although the maximal unramified extension of the quasi-adequate algebra $B_i$ may be bigger than $\overline{B_i}$, the $\widehat{\overline{B_i}}$-module $V_p(\Omega^1_{\overline{Y_i}/Y_i})$ is finite free with basis $\{(\df\log(t_{i,p^n}))_{n\in\bb{N}}\}_{0\leq i\leq d}$ by the arguments of \ref{prop:B_nm-fal-ext} as $\overline{B_i}$ contains $(B_i)^{\overline{K}}_{\underline{\infty}}$ and almost purity still holds for it (cf. \ref{prop:B_nm-almost-purity}). Similar result also holds for $B_{ij}$. Thus, there is a commutative diagram
	\begin{align}\label{diam:prop:dlogs}
		\xymatrix{
			0\ar[r]&\widehat{\overline{B}}[\frac{1}{p}]^{1+d}\ar[r]\ar[d]^-{\wr}&\prod_{1\leq i\leq l}\widehat{\overline{B_i}}[\frac{1}{p}]^{1+d}\ar@<0.3ex>[r]\ar@<-0.3ex>[r]\ar[d]^-{\wr}& \prod_{1\leq i,j\leq l}\widehat{\overline{B_{ij}}}[\frac{1}{p}]^{1+d}\ar[d]^-{\wr}\\
			0\ar[r]&V_p(\Omega^1_{\overline{Y}/Y})\ar[r]&\prod_{1\leq i\leq l}V_p(\Omega^1_{\overline{Y_i}/Y_i})\ar@<0.3ex>[r]\ar@<-0.3ex>[r]& \prod_{1\leq i,j\leq l}V_p(\Omega^1_{\overline{Y_{ij}}/Y_{ij}})
		}
	\end{align}
	where the vertical maps are isomorphisms. Thus, the second row is also exact. Recall that in each $B_i[1/p]$, we can write $s=ut_1^{a_1}\cdots t_d^{a_d}$ for some $u\in B_i[1/p]^\times$ and $a_1,\dots,a_d\in\bb{N}$. Notice that $s_{p^n}t_{1,p^n}^{-a_1}\cdots t_{d,p^n}^{-a_d}\in\ca{L}_{\mrm{ur}}$. We put $u_{p^n}=s_{p^n}t_{1,p^n}^{-a_1}\cdots t_{d,p^n}^{-a_d}$, which is a $p^n$-th root of $u$ and thus lies in $\overline{B_i}[1/p]$. We see that the element 
	\begin{align}
		(\df\log(u_{p^n}))_{n\in\bb{N}}+\sum_{i=1}^d a_i(\df\log(t_{i,p^n}))_{n\in\bb{N}}\in V_p(\Omega^1_{\overline{Y_i}/Y_i})
	\end{align}
	is the unique element $\omega_i$ whose multiplication by $s$ coincides with $(s_{p^n}^{p^n-1}\df s_{p^n})_{n\in\bb{N}}$ (cf. \ref{lem:dlogs-unique}). The unicity implies that $\omega_i$ and $\omega_j$ coincides in $V_p(\Omega^1_{\overline{Y_{ij}}/Y_{ij}})$. Therefore, by the exactness of the second row in \eqref{diam:prop:dlogs}, we obtain an element $\omega\in V_p(\Omega^1_{\overline{Y}/Y})$ such that $s\omega= (s_{p^n}^{p^n-1}\df s_{p^n})_{n\in\bb{N}}$, which completes the proof.
\end{proof}

\begin{mythm}\label{thm:B-fal-ext}
	Let $B$ be a quasi-adequate $\ca{O}_K$-algebra, $Y_K$ the log scheme with underlying scheme $\spec(B[1/p])$ with compactifying log structure associated to $\spec(B_{\triv})\to\spec(B[1/p])$, $\scr{E}_B$ the $\widehat{\overline{B}}[1/p]$-submodule of $\prod_{\ak{E}(B)}\scr{E}_{\ca{O}_E}$ (see {\rm\ref{para:B-fal-ext-compare}}) generated by the subset
	\begin{align}\label{eq:thm:B-fal-ext-subset}
		\{(\df\log(s_{p^n}))_{n\in\bb{N}}\ |\ s_1\in B[1/p]\cap B_{\triv}^\times,\ s_{p^n}\in \overline{B}[1/p],\ s_{p^{n+1}}^p=s_{p^n},\ \forall n\in\bb{N}\}.
	\end{align}
	Then, $\scr{E}_B$ is stable under the natural $\widehat{\overline{B}}[1/p]$-semi-linear action of $G=\gal(\ca{L}_{\mrm{ur}}/\ca{L})$ on $\prod_{\ak{E}(B)}\scr{E}_{\ca{O}_E}$ (induced by the right $G$-action on $\ak{E}(B)$, cf. \eqref{eq:lem:A-inj}), and there is a canonical $G$-equivariant exact sequence of $\widehat{\overline{B}}[1/p]$-modules,
	\begin{align}\label{eq:thm:B-fal-ext}
		0\longrightarrow \widehat{\overline{B}}[\frac{1}{p}](1)\stackrel{\iota}{\longrightarrow}\scr{E}_B\stackrel{\jmath}{\longrightarrow} \widehat{\overline{B}}\otimes_{B}\Omega^1_{Y_K/K}\longrightarrow 0,
	\end{align}
	satisfying the following properties:
	\begin{enumerate}
		\renewcommand{\labelenumi}{{\rm(\theenumi)}}
		\item We have $\iota(1\otimes (\zeta_{p^n})_{n\in \bb{N}})=(\df\log(\zeta_{p^n}))_{n\in \bb{N}}$.\label{item:thm:B-fal-ext-1}
		\item For any element $s\in B[1/p]\cap B_{\triv}^\times$ and any compatible system of $p$-power roots $(s_{p^n})_{n\in \bb{N}}$ of $s$ in $\overline{B}[1/p]$, $\jmath((\df\log(s_{p^n}))_{n\in\bb{N}})=\df \log(s)$.\label{item:thm:B-fal-ext-2}
		\item If we take a quasi-adequate chart of $B$ and fix compatible systems of $p$-power roots $(t_{i,p^n})_{n\in \bb{N}}\subseteq \overline{B}[1/p]$ of the coordinates $t_1,\dots,t_d\in B[1/p]$, then the sequence \eqref{eq:thm:B-fal-ext} identifies with the image of the vertical maps in \eqref{eq:B-fal-ext-compare}. In particular, the $\widehat{\overline{B}}[1/p]$-linear surjection $\jmath$ admits a section sending $\df\log(t_i)$ to $(\df \log(t_{i,p^n}))_{n\in\bb{N}}$ for any $1\leq i\leq d$. \label{item:thm:B-fal-ext-3}
	\end{enumerate} 
	In particular, $\scr{E}_B$ is a finite free $\widehat{\overline{B}}[\frac{1}{p}]$-module with basis $\{(\df \log(t_{i,p^n}))_{n\in\bb{N}}\}_{0\leq i\leq d}$, where $t_{0,p^n}=\zeta_{p^n}$, on which $G$ acts continuously with respect to the canonical topology (where $\widehat{\overline{B}}[\frac{1}{p}]$ is endowed with the $p$-adic topology defined by $\widehat{\overline{B}}$).
\end{mythm}
\begin{proof}
	As the subset \eqref{eq:thm:B-fal-ext-subset} is $G$-stable, we see that $\scr{E}_B$ is $G$-stable and thus endowed with a natural $G$-action. By \ref{prop:dlogs}, $V_p(\Omega^1_{\overline{Y}/Y})$ identifies with $\scr{E}_B$ as a submodule of $\prod_{\ak{E}(B)}\scr{E}_{\ca{O}_E}$ via \eqref{eq:B-fal-ext-compare}. The conclusion follows directly from \ref{prop:B_nm-fal-ext} and \ref{para:B-fal-ext-compare}, where the ``in particular part'' follows from the same arguments of \ref{thm:fal-ext}.
\end{proof}

\begin{mydefn}\label{defn:B-fal-ext}
	We call the canonical sequence \eqref{eq:thm:B-fal-ext} the \emph{Faltings extension} of $B$.
\end{mydefn}

\begin{myrem}\label{rem:B-fal-ext}
	The Faltings extension \eqref{eq:thm:B-fal-ext} is functorial in the following sense: let $K'$ be a complete discrete valuation field extension of $K$ with perfect residue field, $B'$ a quasi-adequate $\ca{O}_{K'}$-algebra. Consider a commutative diagram of $(K,\ca{O}_K,\ca{O}_{\overline{K}})$-triples (see \ref{defn:triple}) 
	\begin{align}
		\xymatrix{
			(B_{\triv},B,\overline{B})\ar[r]^-{f}& (B'_{\triv},B',\overline{B'})\\
			(K,\ca{O}_K,\ca{O}_{\overline{K}})\ar[u]\ar[r]&(K',\ca{O}_{K'},\ca{O}_{\overline{K'}})\ar[u]
		}
	\end{align}
	Then, $f$ induces a natural map $\ak{E}(B')\to \ak{E}(B)$ sending $v'$ to $v'\circ f$ (cf. \ref{para:B-fal-ext-compare}). It induces further a natural map
	\begin{align}
		\prod_{\ak{E}(B)}\scr{E}_{\ca{O}_E}\longrightarrow\prod_{\ak{E}(B')}\scr{E}_{\ca{O}_{E'}},\ (\omega_v)_{v\in \ak{E}(B)}\mapsto (\omega_{v'\circ f})_{v'\in \ak{E}(B')},
	\end{align}
	which maps $\scr{E}_B$ to $\scr{E}_{B'}$ and is compatible with the Faltings extensions.
	\begin{align}\label{diam:rem:B-fal-ext}
		\xymatrix{
			0\ar[r]& \widehat{\overline{B}}[\frac{1}{p}](1)\ar[r]^-{\iota}\ar[d]&\scr{E}_B\ar[r]^-{\jmath}\ar[d]&\widehat{\overline{B}}\otimes_{B}\Omega^1_{Y_K/K}\ar[r]\ar[d]& 0\\
			0\ar[r]& \widehat{\overline{B'}}[\frac{1}{p}](1)\ar[r]^-{\iota}&\scr{E}_{B'}\ar[r]^-{\jmath}&\widehat{\overline{B'}}\otimes_{B'}\Omega^1_{Y'_{K'}/K'}\ar[r]& 0
		}
	\end{align}
	Moreover, if $B'$ is the integral closure of $B$ in a finite extension of $\ca{L}$ contained in $\ca{L}_{\mrm{ur}}$, then $Y'_{K'}$ is \'etale over $Y_K$ by \cite[\Luoma{9}.2.1]{gabber2014travaux}. Assuming that $K'$ is finite over $K$, we see that the natural map $B'\otimes_B\Omega^1_{Y_K/K}\to \Omega^1_{Y'_{K'}/K'}$ is an isomorphism, and thus the vertical maps in \eqref{diam:rem:B-fal-ext} are isomorphisms. In particular, $\scr{E}_B$ contains $(\df\log(s_{p^n}))_{n\in\bb{N}}$ for any element $(s_{p^n})_{n\in \bb{N}}$ of $\plim_{x\mapsto x^p}\overline{B}[1/p]\cap \overline{B}_{\triv}^\times$.
\end{myrem}

\begin{mypara}\label{para:B-fal-ext-dual}
	As in \ref{para:fal-ext-dual}, taking a Tate twist of the dual of the Faltings extension \eqref{eq:thm:B-fal-ext} of $B$, we obtain a canonical exact sequence of finite projective $\widehat{\overline{B}}[1/p]$-representations of $G$, which splits as a sequence of $\widehat{\overline{B}}[1/p]$-modules,
	\begin{align}\label{eq:B-fal-ext-dual}
		0\longrightarrow \ho_{B[\frac{1}{p}]}(\Omega^1_{Y_K/K}(-1),\widehat{\overline{B}}[\frac{1}{p}])\stackrel{\jmath^*}{\longrightarrow} \scr{E}^*_B(1)\stackrel{\iota^*}{\longrightarrow}\widehat{\overline{B}}[\frac{1}{p}]\longrightarrow 0
	\end{align} 
	where $\scr{E}^*_B=\ho_{\widehat{\overline{B}}[1/p]}(\scr{E}_B,\widehat{\overline{B}}[1/p])$. There is a canonical $G$-equivariant $\widehat{\overline{B}}[1/p]$-linear Lie algebra structure on $\scr{E}^*_{\ca{O}_K}(1)$ associated to the linear form $\iota^*$,
	\begin{align}
		[f_1,f_2]=\iota^*(f_1)f_2-\iota^*(f_2)f_1,\ \forall f_1,f_2\in \scr{E}^*_B(1).
	\end{align}
	Thus, $\ho_{B[1/p]}(\Omega^1_{Y_K/K}(-1),\widehat{\overline{B}}[1/p])$ is a Lie ideal of $\scr{E}^*_{\ca{O}_K}(1)$, and $\widehat{\overline{B}}[1/p]$ is the quotient, and  the induced Lie algebra structures on them are trivial. Any $\widehat{\overline{B}}[1/p]$-linear splitting of \eqref{eq:B-fal-ext-dual} identifies $\scr{E}^*_B(1)$ with the semi-direct product of Lie algebras of $\widehat{\overline{B}}[1/p]$ acting on $\ho_{B[1/p]}(\Omega^1_{Y_K/K}(-1),\widehat{\overline{B}}[1/p])$ by multiplication. Let $\{T_i=(\df\log(t_{i,p^n}))_{n\in\bb{N}}\otimes\zeta^{-1}\}_{0\leq i\leq d}$ (where $t_{0,p^n}=\zeta_{p^n}$) denote the basis of $\scr{E}_B(-1)$, and let $\{T_i^*\}_{0\leq i\leq d}$ be the dual basis of $\scr{E}^*_B(1)$. Then, we see that the Lie bracket on $\scr{E}^*_B(1)$ is determined by
	\begin{align}
		[T_0^*,T_i^*]=T_i^*\quad\trm{ and }\quad [T_i^*,T_j^*]=0,
	\end{align}
	for any $1\leq i,j\leq d$. Indeed, this dual basis induces an isomorphism of $\widehat{\overline{B}}[1/p]$-linear Lie algebras
	\begin{align}
		\widehat{\overline{B}}[\frac{1}{p}]\otimes_{\bb{Q}_p}\lie(\bb{Z}_p\ltimes\bb{Z}_p^d)\iso \scr{E}^*_B(1),\ 1\otimes \partial_i\mapsto T_i^*,
	\end{align}
	where $\{\partial_i\}_{0\leq i\leq d}$ is the standard basis of $\lie(\bb{Z}_p\ltimes\bb{Z}_p^d)$ (cf. \ref{para:gamma-basis}).
\end{mypara}

\section{Descent of Representations of Arithmetic Fundamental Groups after Tsuji}\label{sec:descent}

Tsuji \cite[\textsection14]{tsuji2018localsimpson} studied the descent of representations of the arithmetic fundamental group of an adequate algebra. In this section, we show that his arguments still work for quasi-adequate algebras. 

\begin{mypara}\label{para:notation-B-galois}
	We construct a Kummer tower from the coverings of a quasi-adequate algebra defined in \ref{para:notation-quasi-adequate-tower}. The following notation will be used in this section.  We fix a complete discrete valuation field $K$ of characteristic $0$ with perfect residue field of characteristic $p>0$, an algebraic closure $\overline{K}$ of $K$, and a compatible system of primitive $n$-th roots of unity $(\zeta_n)_{n\in\bb{N}}$ in $\overline{K}$. Let $B$ be a quasi-adequate $\ca{O}_K$-algebra, $\ca{L}$ its fraction field, $\ca{L}_{\mrm{ur}}$ the fraction field of $\overline{B}$. We fix a quasi-adequate chart of $B$, $A$ the associated adequate $\ca{O}_K$-algebra defined in \ref{rem:quasi-adequate}, and a system of coordinates $t_1,\dots,t_d\in A[1/p]$. Let $\ca{K}$ (resp. $\ca{K}_{\mrm{ur}}$) be the fraction field of $A$ (resp. $\overline{A}$). For any $1\leq i\leq d$, we fix a compatible system of $k$-th roots $(t_{i,k})_{k\in \bb{N}}$ of $t_i$ in $\overline{A}[1/p]$.
	
	Let $J$ be the subset of $\bb{N}_{>0}^d$ consisting of elements $\underline{N}=(N_1,\dots,N_d)$ with $N_i$ prime to $p$ for any $1\leq i\leq d$. We endow $J$ with the partial order defined by the divisibility relation (cf. \ref{para:product}). For $\underline{N}\in J$, $n\in\bb{N}$ and $\underline{m}=(m_1,\dots,m_d)\in \bb{N}^d$, we define finite field extensions of $K$ and $\ca{K}$ in $\overline{K}$ and $\ca{K}_{\mrm{ur}}$ respectively by
	\begin{align}
		K^{(\underline{N})}_n=K(\zeta_{p^nN_i}\ |\ 1\leq i\leq d),\qquad \ca{K}^{(\underline{N})}_{n,\underline{m}}=K^{(\underline{N})}_n\ca{K}(t_{i,p^{m_i}N_i}\ |\ 1\leq i\leq d).
	\end{align}
	It is clear that these fields $\ca{K}^{(\underline{N})}_{n,\underline{m}}$ form a system of fields over the directed partially ordered set $J\times\bb{N}\times \bb{N}^d$ (cf. \ref{para:product}). We extend this notation for one of the components of $\underline{N},n,\underline{m}$ being $\infty$ by taking the filtered union, and we omit the index $\underline{N}$ or $n$ or $\underline{m}$ if $\underline{N}=\underline{1}$ or $n=0$ or $\underline{m}=\underline{0}$ respectively. We remark that if we take again the notation in \ref{para:notation-quasi-adequate-tower}, then $\ca{K}^{(\underline{N})}_{n,\underline{m}}=\ca{K}^{K^{(\underline{N})}_n}_{p^{\underline{m}}\underline{N}}$.
	We set $\ca{L}^{(\underline{N})}_{n,\underline{m}}=\ca{L}\ca{K}^{(\underline{N})}_{n,\underline{m}}$ and let $A^{(\underline{N})}_{n,\underline{m}}$ (resp. $B^{(\underline{N})}_{n,\underline{m}}$) be the integral closure of $A$ in $\ca{K}^{(\underline{N})}_{n,\underline{m}}$ (resp. of $B$ in $\ca{L}^{(\underline{N})}_{n,\underline{m}}$).
	
	For any $\underline{N}\in J$, the system $(B^{(\underline{N})}_{n,\underline{m}})_{(n,\underline{m})\in\bb{N}^{1+d}}$ is the Kummer tower of $B^{(\underline{N})}$ defined by $\zeta_{p^n},t_{1,p^n},\dots,t_{d,p^n}$ (cf. \ref{defn:kummer-tower}). Following \ref{defn:tower-completion}, for any $(n,\underline{m})\in(\bb{N}\cup\{\infty\})^{1+d}$, we denote by $\widehat{B}^{(\underline{N})}_{n,\underline{m}}$ the $p$-adic completion of $B^{(\underline{N})}_{n,\underline{m}}$, and we set
	\begin{align}\label{eq:tilde-B}
		\widetilde{B}^{(\underline{N})}_{n,\underline{m}}=\colim_{(n',\underline{m'})\in(\bb{N}^{1+d})_{\leq (n,\underline{m})}} \widehat{B}^{(\underline{N})}_{n',\underline{m'}}.
	\end{align} 
	We remark that the transition maps in the colimit of \eqref{eq:tilde-B} are closed embeddings with respect to the $p$-adic topology, and that $\widetilde{B}^{(\underline{N})}_{n,\underline{m}}$ identifies with a topological subring of $\widehat{B}^{(\underline{N})}_{n,\underline{m}}$ (both endowed with the $p$-adic topology) by \ref{defn:tower-completion}. We name some Galois groups as indicated in the following diagram:
	\begin{align}
		\xymatrix{
			\ca{L}_{\mrm{ur}}&&\\
			\ca{L}^{(\underline{N})}_{\infty,\underline{\infty}}\ar[u]^-{H^{(\underline{N})}_{\underline{\infty}}}&&\\
			\ca{L}^{(\underline{N})}_{\infty,\underline{m}}\ar[u]^-{\Delta^{(\underline{N})}_{\underline{m}}}\ar@/^3pc/[uu]^-{H^{(\underline{N})}_{\underline{m}}}&\ca{L}^{(\underline{N})}_{n,\underline{m}} \ar[l]^-{\Sigma^{(\underline{N})}_{n,\underline{m}}}\ar[lu]|-{\Gamma^{(\underline{N})}_{n,\underline{m}}}&\ca{L}\ar@/_1pc/[lluu]|{G}\ar[l]\ar@/_1pc/[llu]|(0.6){\Xi^{(\underline{N})}}
		}
	\end{align}
\end{mypara}

\begin{mylem}\label{lem:B_nm-kummer}
	For any $\underline{N}\in J$, the Kummer tower $(B^{(\underline{N})}_{n,\underline{m}})_{(n,\underline{m})\in \bb{N}^{1+d}}$ satisfies the condition {\rm\ref{prop:kummer-tower-lem}.(\ref{item:prop:kummer-tower-lem-1})}.
\end{mylem}
\begin{proof}
	For any $\ak{p}\in\ak{S}_p(B^{(\underline{N})})$, if we denote by $E$ the completion of the discrete valuation field $B^{(\underline{N})}_{\ak{p}}[1/p]$, then we need to show that the Kummer tower $(\ca{O}_{E_{n,\underline{m}}})_{(n,\underline{m})\in \bb{N}^{1+d}}$ (defined by $\zeta_{p^n},t_{1,p^n},\dots,t_{d,p^n}$) satisfies the condition {\rm\ref{prop:kummer-tower-lem}.(\ref{item:prop:kummer-tower-lem-1})}. It suffices to check that $\df t_1,\dots,\df t_d$ form an $E$-basis of $\widehat{\Omega}^1_{B^{(\underline{N})}_{\ak{p}}}[1/p]=\widehat{\Omega}^1_{\ca{O}_E}[1/p]$ by \ref{prop:rank}. Since $B^{(\underline{N})}$ is a quasi-adequate $\ca{O}_{K^{(\underline{N})}}$-algebra with a system of coordinates $t_{1,N_1},\dots,t_{d,N_d}$ by \ref{prop:B-adequate}.(\ref{item:prop:B-adequate-1}), the conclusion follows from \ref{prop:quasi-adequate-diff}.
\end{proof}

\begin{myprop}\label{prop:notation-n_0}
	There exists $n_0\in\bb{N}$ such that the following statements hold for any $\underline{N}\in J$ and $(n,\underline{m})\in \bb{N}_{\geq n_0}^{1+d}$:
	\begin{enumerate}
		\renewcommand{\labelenumi}{{\rm(\theenumi)}}
		\item We have $\ca{L}^{(\underline{N})}_{\infty,\underline{\infty}}=\ca{L}^{(\underline{N})}_{n,\underline{m}}\otimes_{\ca{K}_{n,\underline{m}}} \ca{K}_{\infty,\underline{\infty}}$.\label{item:prop:notation-n_0-1}
		\item The cyclotomic character \eqref{eq:cycl-char} $\chi:G\to \bb{Z}_p^\times$ (describing the action on $\zeta_{p^n}$) and the $p$-adic logarithm map \eqref{eq:p-adic-log} $\log:\bb{Z}_p^\times\to \bb{Z}_p$ induce an isomorphism
		\begin{align}\label{eq:prop:notation-n_0-chi}
			\log\circ\chi: \Sigma^{(\underline{N})}_{n,\underline{\infty}}\iso p^n\bb{Z}_p.
		\end{align}\label{item:prop:notation-n_0-2}
		\item The continuous $1$-cocycle \eqref{eq:cont-cocycle} $\xi:G\to \bb{Z}_p^d$ (describing the action on $t_{1,p^n},\dots,t_{d,p^n}$) induces an isomorphism 
		\begin{align}\label{eq:prop:notation-n_0-xi}
			\xi:\Delta^{(\underline{N})}_{\underline{m}}\iso p^{m_1}\bb{Z}_p\times\cdots\times p^{m_d}\bb{Z}_p
		\end{align}
		where $\underline{m}=(m_1,\dots,m_d)$.\label{item:prop:notation-n_0-3}
		\item The natural map $\ak{S}_p(B^{(\underline{N})}_{n,\underline{m}})\to \ak{S}_p(B^{(\underline{N})}_{n_0,\underline{n_0}})$ is a bijection.\label{item:prop:notation-n_0-4}
	\end{enumerate}
\end{myprop}
\begin{proof}
	Since the Kummer towers $(B_{n,\underline{m}})_{(n,\underline{m})\in \bb{N}^{1+d}}$ and $(A_{n,\underline{m}})_{(n,\underline{m})\in \bb{N}^{1+d}}$ both satisfy the condition {\rm\ref{prop:kummer-tower-lem}.(\ref{item:prop:kummer-tower-lem-1})} by \ref{lem:B_nm-kummer}, there exists $n_0\in \bb{N}$ such that the conditions \ref{prop:kummer-tower-lem}.(\ref{item:prop:kummer-tower-lem-2}, \ref{item:prop:kummer-tower-lem-3}) hold for any $(n,\underline{m})\in \bb{N}_{\geq n_0}^{1+d}$. In particular, the natural map $\gal(\ca{L}_{\infty,\underline{\infty}}/\ca{L}_{n,\underline{m}})\to\gal(\ca{K}_{\infty,\underline{\infty}}/\ca{K}_{n,\underline{m}})$ is an isomorphism of pro-$p$ groups. Thus, $\ca{L}_{\infty,\underline{\infty}}=\ca{L}_{n,\underline{m}}\otimes_{\ca{K}_{n,\underline{m}}} \ca{K}_{\infty,\underline{\infty}}$ by Galois theory. As $[\ca{L}^{(\underline{N})}_{n,\underline{m}}:\ca{L}_{n,\underline{m}}]$ is prime to $p$, we have $\ca{L}^{(\underline{N})}_{\infty,\underline{\infty}}=\ca{L}^{(\underline{N})}_{n,\underline{m}}\otimes_{\ca{L}_{n,\underline{m}}}\ca{L}_{\infty,\underline{\infty}}$, which implies (\ref{item:prop:notation-n_0-1}). Then, we deduce easily the other statements by Galois theory.
\end{proof}

\begin{mylem}\label{lem:B_nm-perfd}
	For any $\underline{N}\in J$, the $\ca{O}_{K^{(\underline{N})}_{\infty}}$-algebra $B^{(\underline{N})}_{\infty,\underline{\infty}}$ is almost pre-perfectoid.
\end{mylem}
\begin{proof}
	As $K^{(\underline{N})}_{\infty}$ is a pre-perfectoid field (cf. the proof of \ref{lem:H-coh}), the conclusion follows from the same argument of \ref{lem:B-perfd} (where admitting roots in the prime-to-$p$ part is unnecessary).
\end{proof}

\begin{myprop}\label{prop:B_nm-p-iso}
	For any $\underline{N}\in J$, there exists $k_0\in\bb{N}$ such that the natural map
	\begin{align}
		B^{(\underline{N})}_{n,\underline{m}}\otimes_{A^{(\underline{N})}_{n,\underline{m}}}A^{(\underline{N})}_{n',\underline{m'}}\longrightarrow B^{(\underline{N})}_{n',\underline{m'}}
	\end{align}
	is a $p^{k_0}$-isomorphism for any elements $(n',\underline{m'})\geq (n,\underline{m})$ in $(\bb{N}_{\geq n_0}\cup\{\infty\})^{1+d}$, where $n_0\in\bb{N}$ is defined in {\rm\ref{prop:notation-n_0}}. 
\end{myprop}
\begin{proof}
	It follows from the same argument of \ref{prop:B-p-iso} using \ref{lem:B_nm-perfd} instead of \ref{lem:B-perfd}.
\end{proof}

\begin{mylem}[{cf. \cite[14.11]{tsuji2018localsimpson}}]\label{lem:tate-sen-2}
	For any $\underline{N}\in J$, $n\in \bb{N}_{\geq n_0}\cup\{\infty\}$ and any elements $\underline{m'}\geq \underline{m}$ in $(\bb{N}_{\geq n_0}\cup\{\infty\})^d$ with $\underline{m'}\in\underline{m}+\bb{N}^d\subseteq (\bb{N}\cup\{\infty\})^d$, where $n_0\in \bb{N}$ is defined in {\rm\ref{prop:notation-n_0}}, the natural homomorphism
	\begin{align}
		\widehat{B}^{(\underline{N})}_{n,\underline{m}}\otimes_{B^{(\underline{N})}_{n,\underline{m}}}B^{(\underline{N})}_{n,\underline{m'}}[\frac{1}{p}]\longrightarrow \widehat{B}^{(\underline{N})}_{n,\underline{m'}}[\frac{1}{p}]
	\end{align}
	is an isomorphism, and $B^{(\underline{N})}_{n,\underline{m'}}[1/p]$ is a finite free $B^{(\underline{N})}_{n,\underline{m}}[1/p]$-module.
\end{mylem}
\begin{proof}
	Let $I=\{(k_1,\dots,k_d)\in\bb{N}^d\ |\ 0\leq k_i< p^{m_i'-m_i},\ 1\leq i\leq d\}$. We deduce from the conclusion of \ref{prop:X-tower-str} (cf. \ref{lem:A-X}) for finite indexes that there exists $k_0\in\bb{N}$ independent of $n,\underline{m},\underline{m'}$ such that
	\begin{align}
		p^{k_0}A^{(\underline{N})}_{n,\underline{m'}}\subseteq \bigoplus_{\underline{k}\in I} A^{(\underline{N})}_{n,\underline{m}}\cdot\prod_{i=1}^d t_{i,p^{m_i'}N_i}^{k_i}\subseteq p^{-k_0}A^{(\underline{N})}_{n,\underline{m'}}.
	\end{align}
	This shows that $A^{(\underline{N})}_{n,\underline{m'}}$ is $p^{2k_0}$-isomorphic to a finite free $A^{(\underline{N})}_{n,\underline{m}}$-module. The same result holds for $B^{(\underline{N})}_{n,\underline{m'}}$ over $B^{(\underline{N})}_{n,\underline{m}}$ by \ref{prop:B_nm-p-iso}, which completes the proof.
\end{proof}

\begin{mylem}[{cf. \cite[14.10, 14.8]{tsuji2018localsimpson}}]\label{lem:tate-sen-1}
	For any $\underline{N}\in J$, there exists $k_1\in\bb{N}$ such that for any elements $(n,\underline{m})\in (\bb{N}_{\geq n_0}\cup\{\infty\})^{1+d}$, where $n_0\in \bb{N}$ is defined in {\rm\ref{prop:notation-n_0}}, the following statements hold:
	\begin{enumerate}
		\renewcommand{\labelenumi}{{\rm(\theenumi)}}
		\item If $(n,\underline{m})\in \bb{N}^{1+d}$, for any $r\in\bb{N}$, let $\sigma$ be a generator of $\gal(\ca{L}^{(\underline{N})}_{n+r,\underline{m}}/\ca{L}^{(\underline{N})}_{n,\underline{m}})\cong \bb{Z}/p^r\bb{Z}$ by \eqref{eq:prop:notation-n_0-chi}, then we have
		\begin{align}
			B^{(\underline{N})}_{n+r,\underline{m}}\subseteq p^{-k_1}(B^{(\underline{N})}_{n,\underline{m}}+(\sigma-1)(B^{(\underline{N})}_{n+r,\underline{m}})).
		\end{align}\label{item:lem:tate-sen-1-1}
		\item If the $i$-th component of $\underline{m}$ is an integer for some $1\leq i\leq d$, for any $r\in \bb{N}$, let $\tau$ be a generator of $\gal(\ca{L}^{(\underline{N})}_{\infty,\underline{m}+\underline{r}_i}/\ca{L}^{(\underline{N})}_{\infty,\underline{m}})\cong \bb{Z}/p^r\bb{Z}$ by \eqref{eq:prop:notation-n_0-xi}, then we have
		\begin{align}
			B^{(\underline{N})}_{\infty,\underline{m}+\underline{r}_i}\subseteq p^{-k_1}(B^{(\underline{N})}_{\infty,\underline{m}}+(\tau-1)(B^{(\underline{N})}_{\infty,\underline{m}+\underline{r}_i})).
		\end{align}\label{item:lem:tate-sen-1-2}
	\end{enumerate}
\end{mylem}
\begin{proof}
	The conclusion holds for the case $B=A$ by \cite[14.8, 14.10]{tsuji2018localsimpson}. It remains true in general due to \ref{prop:B_nm-p-iso}.
\end{proof}

\begin{mydefn}[{cf. \cite[7.4, 7.5, 7.6]{tsuji2018localsimpson}}]\label{defn:tower-tate-sen}
	We say that a tower $(R_n)_{n\in \bb{N}}$ of normal domains flat over $\ca{O}_K$ (see \ref{defn:tower}) is \emph{Tate-Sen} if it satisfies the following conditions:
	\begin{enumerate}
		\renewcommand{\labelenumi}{{\rm(\theenumi)}}
		\item There exists a Noetherian normal domain $R_{-1}$ over $\ca{O}_K$ contained in $R_0$ with $R_{-1}/ pR_{-1}\neq 0$ such that $R_0$ is integral over $R_{-1}$ (thus \ref{rem:int-clos-fini} applies to $R_{-1}\to R_n$ for any $n\in\bb{N}\cup \{\infty\}$).\label{item:defn:tower-tate-sen-noether}
		\item The tower $(R_n)_{n\in \bb{N}}$ is a $\bb{Z}_p$-tower in the following sense: if $\ca{E}_n$ denotes the fraction field of $R_n$, then $\ca{E}_\infty$ is a Galois extension of $\ca{E}_0$ with Galois group isomorphic to $\bb{Z}_p$ and $\ca{E}_n$ is the $p^n\bb{Z}_p$-invariant part of $\ca{E}_\infty$ for any $n\in \bb{N}$.\label{item:defn:tower-tate-sen-Z_p}
		\item There exists $k_1\in \bb{N}$ such that for any $n,r\in\bb{N}$ and any generator $\sigma\in \gal(\ca{E}_{n+r}/\ca{E}_n)$, we have
		\begin{align}
			R_{n+r}\subseteq p^{-k_1}(R_n+(\sigma-1)(R_{n+r})).
		\end{align}\label{item:defn:tower-tate-sen-bound}
		\item For any $n\in \bb{N}$, $\widehat{R_n}[1/p]$ is finite over $\widehat{R_0}[1/p]$.\label{item:defn:tower-tate-sen-fini}
		\item The set $\ak{S}_p(R_\infty)$ of primes ideals of $R_\infty$ of height $1$ containing $p$ is finite.\label{item:defn:tower-tate-sen-ht1}
	\end{enumerate}
\end{mydefn}

Tsuji has established a series of decompletion results for Tate-Sen towers in \cite[\textsection7]{tsuji2018localsimpson}.

\begin{myprop}[{cf. \cite[14.12]{tsuji2018localsimpson}}]\label{prop:tate-sen}
	For any $\underline{N}\in J$ and $(n,\underline{m})$ in $(\bb{N}_{\geq n_0}\cup\{\infty\})^{1+d}$, where $n_0\in \bb{N}$ is defined in {\rm\ref{prop:notation-n_0}}, the following statements hold:
	\begin{enumerate}
		\renewcommand{\labelenumi}{{\rm(\theenumi)}}
		\item If $(n,\underline{m})\in \bb{N}^{1+d}$, then the tower $(B^{(\underline{N})}_{n+r,\underline{m}})_{r\in\bb{N}}$ is Tate-Sen.\label{item:prop:tate-sen-1}
		\item If the $i$-th component of $\underline{m}$ is an integer for some $1\leq i\leq d$, then the tower $(B^{(\underline{N})}_{\infty,\underline{m}+\underline{r}_i})_{r\in \bb{N}}$ is Tate-Sen.\label{item:prop:tate-sen-2}
	\end{enumerate}
\end{myprop}
\begin{proof}
	It follows from the fact that $B/pB\neq 0$, \ref{prop:notation-n_0}, \ref{lem:tate-sen-2}, and \ref{lem:tate-sen-1}. 
\end{proof}

\begin{myrem}\label{rem:tate-sen}
	With the notation in {\rm\ref{para:notation-K}}, the results of \ref{prop:tate-sen} remain true for the Kummer tower $(\ca{O}_{K_{n,\underline{m}}})_{(n,\underline{m})\in \bb{N}^{1+d}}$. Indeed, one can check firstly for the subfield $K'$ defined in \ref{lem:cohen} with the aid of \ref{lem:cohen-str}, and then deduce the general case as we did above from $A$ to $B$.
\end{myrem}

\begin{mylem}[{cf. \cite[14.13]{tsuji2018localsimpson}}]\label{lem:decompletion}
	For any $\underline{N}\in J$, the following statements hold:
	\begin{enumerate}
		\renewcommand{\labelenumi}{{\rm(\theenumi)}}
		\item The $(\Xi^{(\underline{N})},\widehat{B}[1/p])$-finite part of $\widehat{B}^{(\underline{N})}_{\infty,\underline{\infty}}[1/p]$ is $\widetilde{B}^{(\underline{N})}_{\infty,\underline{\infty}}[1/p]$ (see {\rm\ref{defn:repn}}).\label{item:lem:decompletion-1}
		\item Let $V$ be an object of $\repnpr(\Xi^{(\underline{N})},\widetilde{B}^{(\underline{N})}_{\infty,\underline{\infty}}[1/p])$. Then, $V$ is the $(\Xi^{(\underline{N})},\widehat{B}[1/p])$-finite part of $\widehat{B}^{(\underline{N})}_{\infty,\underline{\infty}}\otimes_{\widetilde{B}^{(\underline{N})}_{\infty,\underline{\infty}}}V$.\label{item:lem:decompletion-2}
	\end{enumerate}
\end{mylem}
\begin{proof}
	We follow the proof of \cite[14.13]{tsuji2018localsimpson} and take $n_0\in \bb{N}$ defined in {\rm\ref{prop:notation-n_0}}.
	
	(\ref{item:lem:decompletion-1}) Notice that for any $(n,\underline{m})\in \bb{N}^{1+d}$, $B^{(\underline{N})}_{n,\underline{m}}$ is finite over $B$. Thus, $\widehat{B}^{(\underline{N})}_{n,\underline{m}}[1/p]$ is a $\Xi^{(\underline{N})}$-stable finitely generated $\widehat{B}[1/p]$-submodule of $\widehat{B}^{(\underline{N})}_{\infty,\underline{\infty}}[1/p]$ as $B$ is Noetherian. Thus, $\widetilde{B}^{(\underline{N})}_{\infty,\underline{\infty}}[1/p]$ lies in the  $(\Xi^{(\underline{N})},\widehat{B}[1/p])$-finite part of $\widehat{B}^{(\underline{N})}_{\infty,\underline{\infty}}[1/p]$. 
	
	For the converse, let $M$ be a $\Xi^{(\underline{N})}$-stable finitely generated $\widehat{B}[1/p]$-submodule of $\widehat{B}^{(\underline{N})}_{\infty,\underline{\infty}}[1/p]$. For any $1\leq i\leq d$ and $\underline{m}\in \bb{N}_{\geq n_0}^i\times\{\infty\}^{d-i}$, applying \cite[7.14]{tsuji2018localsimpson} to the Tate-Sen tower $(B^{(\underline{N})}_{\infty,\underline{m}+\underline{r}_i})_{r\in \bb{N}}$ (\ref{prop:tate-sen}.(\ref{item:prop:tate-sen-2})), we see that the condition $M\subseteq \widehat{B}^{(\underline{N})}_{\infty,\underline{m}+\underline{\infty}_i}[\frac{1}{p}]$ implies that $M\subseteq \widehat{B}^{(\underline{N})}_{\infty,\underline{m}+\underline{r}_i}[\frac{1}{p}]$ for some $r\in \bb{N}$. Applying this argument in the order $i=1,\dots,d$, we obtain an element $\underline{m}\in \bb{N}_{\geq n_0}^d$ such that $M\subseteq \widehat{B}^{(\underline{N})}_{\infty,\underline{m}}[1/p]$. Then, applying \cite[7.14]{tsuji2018localsimpson} to the Tate-Sen tower $(B^{(\underline{N})}_{n_0+r,\underline{m}})_{r\in \bb{N}}$ (\ref{prop:tate-sen}.(\ref{item:prop:tate-sen-1})), we obtain an element $n\in \bb{N}_{\geq n_0}$ such that $M\subseteq \widehat{B}^{(\underline{N})}_{n,\underline{m}}[1/p]$. This proves the converse part.
	
	(\ref{item:lem:decompletion-2}) Since $\widetilde{B}^{(\underline{N})}_{\infty,\underline{\infty}}$ is the colimit of $\widehat{B}^{(\underline{N})}_{n,\underline{m}}$ and $\Xi^{(\underline{N})}$ is topologically finitely generated (cf. \ref{prop:notation-n_0}), there exists $(n,\underline{m})\in \bb{N}_{\geq n_0}^{1+d}$ and an object $V'$ of $\repnpr(\Xi^{(\underline{N})},\widehat{B}^{(\underline{N})}_{n,\underline{m}}[1/p])$ such that $V=\widetilde{B}^{(\underline{N})}_{\infty,\underline{\infty}}\otimes_{\widehat{B}^{(\underline{N})}_{n,\underline{m}}}V'$ (\cite[5.2.(1)]{tsuji2018localsimpson}). As $\widehat{B}^{(\underline{N})}_{n,\underline{m}}$ is finite over $B$, by (\ref{item:lem:decompletion-1}) and \cite[7.3.(3)]{tsuji2018localsimpson}, we see that the $(\Xi^{(\underline{N})},\widehat{B}[1/p])$-finite part of $\widehat{B}^{(\underline{N})}_{\infty,\underline{\infty}}\otimes_{\widehat{B}^{(\underline{N})}_{n,\underline{m}}}V'$ is $\widetilde{B}^{(\underline{N})}_{\infty,\underline{\infty}}\otimes_{\widehat{B}^{(\underline{N})}_{n,\underline{m}}}V'$.
\end{proof}

\begin{myprop}[{cf. \cite[14.15]{tsuji2018localsimpson}}]\label{prop:decompletion}
	For any $\underline{N}\in J$, the functor
	\begin{align}\label{eq:cor:decompletion}
		\repnpr(\Xi^{(\underline{N})},\widetilde{B}^{(\underline{N})}_{\infty,\underline{\infty}}[\frac{1}{p}])\longrightarrow \repnpr(\Xi^{(\underline{N})},\widehat{B}^{(\underline{N})}_{\infty,\underline{\infty}}[\frac{1}{p}]),\ V \mapsto \widehat{B}^{(\underline{N})}_{\infty,\underline{\infty}}\otimes_{\widetilde{B}^{(\underline{N})}_{\infty,\underline{\infty}}} V
	\end{align}
	is an equivalence of categories, and a quasi-inverse is obtained by taking the $(\Xi^{(\underline{N})},\widehat{B}[1/p])$-finite part.
\end{myprop}
\begin{proof}
	We follow the proof of \cite[14.15]{tsuji2018localsimpson}. Let $W$ be an object of $\repnpr(\Xi^{(\underline{N})},\widehat{B}^{(\underline{N})}_{\infty,\underline{\infty}}[1/p])$. For any $1\leq i\leq d$ and $\underline{m}\in \bb{N}_{\geq n_0}^i\times\{\infty\}^{d-i}$, we can apply \cite[7.23]{tsuji2018localsimpson} to the Tate-Sen tower $(B^{(\underline{N})}_{\infty,\underline{m}+\underline{r}_i})_{r\in \bb{N}}$ (\ref{prop:tate-sen}.(\ref{item:prop:tate-sen-2})). Applying this argument to $W$ in the order $i=1,\dots,d$, we obtain an object $W'$ of $\repnpr(\Xi^{(\underline{N})},\widehat{B}^{(\underline{N})}_{\infty,\underline{m}}[1/p])$ for some $\underline{m}\in\bb{N}_{\geq n_0}^d$ such that $W=\widehat{B}^{(\underline{N})}_{\infty,\underline{\infty}}\otimes_{\widehat{B}^{(\underline{N})}_{\infty,\underline{m}}} W'$. Then, applying \cite[7.24]{tsuji2018localsimpson} to the Tate-Sen tower $(B^{(\underline{N})}_{n_0+r,\underline{m}})_{r\in \bb{N}}$ (\ref{prop:tate-sen}.(\ref{item:prop:tate-sen-1})), we obtain an object $V$ of $\repnpr(\Xi^{(\underline{N})},\widehat{B}^{(\underline{N})}_{n,\underline{m}}[\frac{1}{p}])$ for some $n\in\bb{N}_{\geq n_0}$ such that $W'=\widehat{B}^{(\underline{N})}_{\infty,\underline{m}}\otimes_{\widehat{B}^{(\underline{N})}_{n,\underline{m}}} V$. This shows that the functor \eqref{eq:cor:decompletion} is essentially surjective. The conclusion follows from \ref{lem:decompletion}.(\ref{item:lem:decompletion-2}).
\end{proof}

\begin{myprop}[{cf. \cite[14.16]{tsuji2018localsimpson}}]\label{prop:descent-further}
	For any $\underline{N}\in J$ and $\underline{m}\in \bb{N}_{\geq n_0}^d$, where $n_0\in \bb{N}$ is defined in {\rm\ref{prop:notation-n_0}}, the functor (cf. {\rm\ref{defn:analytic}})
	\begin{align}\label{eq:prop:descent-further}
		\repnan{\Delta^{(\underline{N})}_{\underline{m}}}(\Xi^{(\underline{N})},\widetilde{B}^{(\underline{N})}_{\infty,\underline{m}}[\frac{1}{p}])\longrightarrow \repnpr(\Xi^{(\underline{N})},\widetilde{B}^{(\underline{N})}_{\infty,\underline{\infty}}[\frac{1}{p}]),\ V \mapsto \widetilde{B}^{(\underline{N})}_{\infty,\underline{\infty}}\otimes_{\widetilde{B}^{(\underline{N})}_{\infty,\underline{m}}} V
	\end{align}
	is fully faithful. Moreover, any object of $\repnpr(\Xi^{(\underline{N})},\widetilde{B}^{(\underline{N})}_{\infty,\underline{\infty}}[1/p])$ lies in the essential image of the above functor for some $\underline{m}\in \bb{N}_{\geq n_0}^d$.
\end{myprop}
\begin{proof}
	We follow the proof of \cite[14.16]{tsuji2018localsimpson}. For any object $W$ of $\repnpr(\Xi^{(\underline{N})},\widetilde{B}^{(\underline{N})}_{\infty,\underline{\infty}}[1/p])$, we consider the $\widetilde{B}^{(\underline{N})}_{\infty,\underline{\infty}}[1/p]$-linear endomorphism $\varphi_\tau$ on $W$ given by the infinitesimal action of $\tau\in\Delta^{(\underline{N})}$, which is nilpotent by \ref{prop:operator-nilpotent} (whose assumptions are satisfied by \ref{prop:notation-n_0}). Then, we obtain a continuous semi-linear action $\overline{\rho}$ of $\Delta^{(\underline{N})}$ on $W$ defined by $\overline{\rho}(\tau)=\exp(-\varphi_\tau)\rho(\tau)$ as in \ref{prop:special-functor}.(\ref{item:prop:special-functor-1}) (in fact, $\overline{\rho}$ can be extended to $\Gamma^{(\underline{N})}$, but we don't need this). Notice that if $W=\widetilde{B}^{(\underline{N})}_{\infty,\underline{\infty}}\otimes_{\widetilde{B}^{(\underline{N})}_{\infty,\underline{m}}} V$ for some object $V$ of $\repnan{\Delta^{(\underline{N})}_{\underline{m}}}(\Xi^{(\underline{N})},\widetilde{B}^{(\underline{N})}_{\infty,\underline{m}}[1/p])$, then $\overline{\rho}|_{\Delta^{(\underline{N})}_{\underline{m}}}$ acts trivially on $V$, which implies that
	\begin{align}
		(W,\overline{\rho})^{\Delta^{(\underline{N})}_{\underline{m}}}=(\widetilde{B}^{(\underline{N})}_{\infty,\underline{\infty}})^{\Delta^{(\underline{N})}_{\underline{m}}}\otimes_{\widetilde{B}^{(\underline{N})}_{\infty,\underline{m}}} V=V,
	\end{align}
	where the last identity follows from \ref{lem:tate-sen-2}. This shows that \eqref{eq:prop:descent-further} is fully faithful.
	
	Since $\widetilde{B}^{(\underline{N})}_{\infty,\underline{\infty}}$ is the filtered colimit of $\widetilde{B}^{(\underline{N})}_{\infty,\underline{m}}$, there exists $\underline{m}\in \bb{N}_{\geq n_0}^d$ and an object $V$ of  $\repnpr(\Delta^{(\underline{N})},\widetilde{B}^{(\underline{N})}_{\infty,\underline{m}}[1/p])$ such that $(W,\overline{\rho})=\widetilde{B}^{(\underline{N})}_{\infty,\underline{\infty}}\otimes_{\widetilde{B}^{(\underline{N})}_{\infty,\underline{m}}} V$ (\cite[5.2.(1)]{tsuji2018localsimpson}). Moreover, since any $w\in W$ is fixed by an open subgroup of $\Delta^{(\underline{N})}$ via $\overline{\rho}$ by \ref{prop:derivative}, after enlarging $\underline{m}$, we may assume that $\overline{\rho}|_{\Delta^{(\underline{N})}_{\underline{m}}}$ acts trivially on $V$. By the discussion above, we have $(W,\overline{\rho})^{\Delta^{(\underline{N})}_{\underline{m}}}=V$. We claim that $V$ is $\Xi^{(\underline{N})}$-stable under $\rho$. Indeed, for any $g\in \Xi^{(\underline{N})}$, $\tau\in\Delta^{(\underline{N})}_{\underline{m}}$ and $v\in V$, if we set $\tau'=g^{-1}\tau g\in \Delta^{(\underline{N})}_{\underline{m}}$, then
	\begin{align}
		\rho(\tau)(\rho(g)v)=\rho(g)(\rho(\tau')v)=\rho(g)(\exp(\varphi_{\tau'})v)=\exp(\varphi_\tau)(\rho(g)v),
	\end{align}
	where the second equality follows from $\overline{\rho}(\tau')(v)=v$, and the last equality follows from \ref{lem:infinitesimal}.(\ref{item:infinitesimal-1}). This shows that $\rho(g)v\in V=(W,\overline{\rho})^{\Delta^{(\underline{N})}_{\underline{m}}}$, and hence $V$ is $\Xi^{(\underline{N})}$-stable under $\rho$. As $\widetilde{B}^{(\underline{N})}_{\infty,\underline{m}}[1/p]\to \widetilde{B}^{(\underline{N})}_{\infty,\underline{\infty}}[1/p]$ is a closed embedding (\ref{defn:tower-completion}), so is $V\to W$, which implies that $\Xi^{(\underline{N})}$ acts continuously on $V$. Moreover, $V$ is $\Delta^{(\underline{N})}_{\underline{m}}$-analytic by definition, which completes the proof.
\end{proof}

\begin{mylem}[{cf. \cite[14.5]{tsuji2018localsimpson}}]\label{lem:trace}
	Let $\ca{L}',\ca{L}''$ be two finite extensions of $\ca{L}_{\infty,\underline{\infty}}$ in $\ca{L}_{\mrm{ur}}$ with $\ca{L}'\subseteq \ca{L}''$, $B'$ and $B''$ the integral closures of $B$ in $\ca{L}'$ and $\ca{L}''$ respectively. Then, the inclusion $\mrm{Tr}_{\ca{L}''/\ca{L}'}(B'')\subseteq B'$ is almost surjective.
\end{mylem}
\begin{proof}
	Consider $\ca{L}^{(\underline{\infty})}=\colim_{\underline{N}\in J}\ca{L}^{(\underline{N})}$, and let $B'^{(\underline{\infty})}$ (resp. $B''^{(\underline{\infty})}$) be the integral closures of $B$ in $\ca{L}'^{(\underline{\infty})}=\ca{L}'\ca{L}^{(\underline{\infty})}$ (resp. $\ca{L}''^{(\underline{\infty})}=\ca{L}''\ca{L}^{(\underline{\infty})}$). Notice that $B''^{(\underline{\infty})}$ is almost finite \'etale over $B'^{(\underline{\infty})}$ by \ref{prop:B_nm-almost-purity}. In particular, $\mrm{Tr}_{\ca{L}''^{(\underline{\infty})}/\ca{L}'^{(\underline{\infty})}}(B''^{(\underline{\infty})})\subseteq B'^{(\underline{\infty})}$ is almost surjective (\cite[\Luoma{5}.7.12]{abbes2016p}). Thus, for any $x\in B'$ and $\pi\in\ak{m}_{K_\infty}$, there exists $\underline{N}\in J$ and $y\in B''^{(\underline{N})}$ such that $\pi x=\mrm{Tr}_{\ca{L}''^{(\underline{N})}/\ca{L}'^{(\underline{N})}}(y)$. Notice that $l=[\ca{L}'^{(\underline{N})}:\ca{L}']$ is prime to $p$. We take $y'=l^{-1}\mrm{Tr}_{\ca{L}''^{(\underline{N})}/\ca{L}''}(y)\in B''$. Thus, $\pi x=\mrm{Tr}_{\ca{L}''/\ca{L}'}(y')$.
\end{proof}

\begin{myprop}[{cf. \cite[14.7]{tsuji2018localsimpson}}]\label{prop:almost-et-descent}
	For any $\underline{N}\in J$, the functor
	\begin{align}\label{eq:prop:almost-et-descent}
		\repnpr(\Xi^{(\underline{N})},\widehat{B}^{(\underline{N})}_{\infty,\underline{\infty}}[\frac{1}{p}])\longrightarrow \repnpr(G,\widehat{\overline{B}}[\frac{1}{p}]),\ V \mapsto \widehat{\overline{B}}\otimes_{\widehat{B}^{(\underline{N})}_{\infty,\underline{\infty}}} V
	\end{align}
	is fully faithful. Moreover, any object of $\repnpr(G,\widehat{\overline{B}}[1/p])$ lies in the essential image of the above functor for some $\underline{N}\in J$.
\end{myprop}
\begin{proof}
	We follow the proof of \cite[14.7]{tsuji2018localsimpson}. Firstly, $\widehat{B}^{(\underline{N})}_{\infty,\underline{\infty}}\to (\widehat{\overline{B}})^{H^{(\underline{N})}_{\underline{\infty}}}$ is an almost isomorphism by almost Galois descent \cite[6.4]{tsuji2018localsimpson} (whose assumptions \cite[6.1, 6.2]{tsuji2018localsimpson} are satisfied by \ref{lem:trace}). It follows immediately that the functor \eqref{eq:prop:almost-et-descent} is fully faithful (cf. \cite[6.5]{tsuji2018localsimpson}). For an object $W$ of $\repnpr(G,\widehat{\overline{B}}[1/p])$, by almost Galois descent \cite[6.10.(1)]{tsuji2018localsimpson}, there exists an open subgroup $H'$ of $H_{\underline{\infty}}$ such that for the integral closure $B'$ of $B$ in $\ca{L}'=\ca{L}_{\mrm{ur}}^{H'}$, $W^{H'}$ is a finite projective $\widehat{B'}[1/p]$-module such that $W=\widehat{\overline{B}}\otimes_{\widehat{B'}} W^{H'}$. This remains true for any open subgroup of $H'$ by \cite[6.10.(2)]{tsuji2018localsimpson} so that we may assume that $\ca{L}'$ is Galois over $\ca{L}_{\infty,\underline{\infty}}$. By Abhyankar's lemma \ref{lem:abhyankar}, there exists $\underline{N}\in J$ such that $B'^{(\underline{N})}[1/p]$ is finite \'etale over $B^{(\underline{N})}_{\infty,\underline{\infty}}[1/p]$ (cf. \cite[8.21]{he2021coh}). After replacing $\ca{L}'$ by $\ca{L}'^{(\underline{N})}$, we may assume that $B'[1/p]$ is finite \'etale and Galois over $B^{(\underline{N})}_{\infty,\underline{\infty}}[1/p]$. Thus, $\widehat{B'}[1/p]$ is also finite \'etale and Galois over $\widehat{B}^{(\underline{N})}_{\infty,\underline{\infty}}[1/p]$ by \cite[6.15]{tsuji2018localsimpson}. By Galois descent, $W^{H^{(\underline{N})}_{\underline{\infty}}}$ is a finite projective $\widehat{B}^{(\underline{N})}_{\infty,\underline{\infty}}[1/p]$-representation of $\Xi^{(\underline{N})}$ such that $W=\widehat{\overline{B}}\otimes_{\widehat{B}^{(\underline{N})}_{\infty,\underline{\infty}}} W^{H^{(\underline{N})}_{\underline{\infty}}}$.
\end{proof}

\begin{mythm}[{cf. \cite[14.2]{tsuji2018localsimpson}}]\label{thm:descent}
	For any $\underline{N}\in J$ and $\underline{m}\in \bb{N}_{\geq n_0}^d$, where $n_0\in \bb{N}$ is defined in {\rm\ref{prop:notation-n_0}}, the functor
	\begin{align}
		\repnan{\Delta^{(\underline{N})}_{\underline{m}}}(\Xi^{(\underline{N})},\widetilde{B}^{(\underline{N})}_{\infty,\underline{m}}[\frac{1}{p}])\longrightarrow \repnpr(G,\widehat{\overline{B}}[\frac{1}{p}]),\  V\mapsto \widehat{\overline{B}}\otimes_{\widetilde{B}^{(\underline{N})}_{\infty,\underline{m}}} V
	\end{align}
	is fully faithful. Moreover, any object of $\repnpr(G,\widehat{\overline{B}}[1/p])$ lies in the essential image of the above functor for some $\underline{N}\in J$ and $\underline{m}\in \bb{N}_{\geq n_0}^d$.
\end{mythm}
\begin{proof}
	It follows from \ref{prop:almost-et-descent}, \ref{prop:decompletion} and \ref{prop:descent-further}. 
\end{proof}

\section{Sen Operators over Quasi-adequate Algebras}\label{sec:sen-B}
In this section, we fix a complete discrete valuation field $K$ of characteristic $0$ with perfect residue field of characteristic $p>0$, an algebraic closure $\overline{K}$ of $K$, and a quasi-adequate $\ca{O}_K$-algebra $B$ of relative dimension $d$ with fraction field $\ca{L}$ (see \ref{defn:quasi-adequate-alg}). Let $Y_K$ denote the log scheme with underlying scheme $\spec(B[1/p])$ with compactifying log structure associated to $\spec(B_{\triv})\to\spec(B[1/p])$.

\begin{mypara}\label{para:general-Bnm-tower}
	In the following subsections, we introduce some notation that will be used in our construction \ref{thm:sen-brinon-B} of Sen operators. We fix a compatible system of primitive $n$-th roots of unity $(\zeta_n)_{n\in\bb{N}}$ contained in $\overline{K}$. As in \ref{para:notation-K-tilde}, we fix $e\in\bb{N}$ and let $t_1,\dots,t_e$ be finitely many elements of $B[1/p]\cap B_{\triv}^\times$ with compatible systems of $k$-th roots $(t_{i,k})_{k\in\bb{N}_{>0}}$ contained in $\overline{B}[1/p]$ for any $1\leq i\leq e$. We define the tower $(B^{(\underline{N})}_{n,\underline{m}})_{(\underline{N},n,\underline{m})\in J\times\bb{N}\times\bb{N}^e}$ and name the Galois groups as in \ref{para:notation-B-galois}:
	\begin{align}
		\xymatrix{
			\ca{L}_{\mrm{ur}}&&\\
			\ca{L}^{(\underline{N})}_{\infty,\underline{\infty}}\ar[u]^-{H^{(\underline{N})}_{\underline{\infty}}}&&\\
			\ca{L}^{(\underline{N})}_{\infty,\underline{m}}\ar[u]^-{\Delta^{(\underline{N})}_{\underline{m}}}\ar@/^3pc/[uu]^-{H^{(\underline{N})}_{\underline{m}}}&\ca{L}^{(\underline{N})}_{n,\underline{m}} \ar[l]^-{\Sigma^{(\underline{N})}_{n,\underline{m}}}\ar[lu]|-{\Gamma^{(\underline{N})}_{n,\underline{m}}}&\ca{L}\ar@/_1pc/[lluu]|{G}\ar[l]\ar@/_1pc/[llu]|(0.6){\Xi^{(\underline{N})}}
		}
	\end{align}
	We remark that for any $\underline{N}\in J$, the system $(B^{(\underline{N})}_{n,\underline{m}})_{(n,\underline{m})\in\bb{N}^{1+d}}$ is the Kummer tower of $B^{(\underline{N})}$ defined by $\zeta_{p^n},t_{1,p^n},\dots,t_{e,p^n}$ (cf. \ref{defn:kummer-tower}). 
\end{mypara}

\begin{myprop}[{cf. \ref{prop:rank}}]\label{prop:B-rank}
	With the notation in {\rm\ref{para:general-Bnm-tower}}, the following conditions are equivalent:
	\begin{enumerate}
		\renewcommand{\labelenumi}{{\rm(\theenumi)}}
		\item The Kummer tower $(B_{n,\underline{m}})_{(n,\underline{m})\in \bb{N}^{1+e}}$ satisfies the condition {\rm\ref{prop:kummer-tower-lem}.(\ref{item:prop:kummer-tower-lem-1})}.\label{item:prop:B-rank-1}
		\item The $e$ elements $\df t_1,\dots,\df t_e$ of $\Omega^1_{\ca{L}/K}$ are linearly independent over $\ca{L}$.\label{item:prop:B-rank-3}
	\end{enumerate}
\end{myprop}
\begin{proof}
	For any $\ak{p}\in\ak{S}_p(B)$, let $E_{\ak{p}}$ be the completion of $\ca{L}$ with respect to the discrete valuation ring $B_{\ak{p}}$. Recall that the $\ca{L}$-module $\Omega^1_{\ca{L}/K}$ and the $E_{\ak{p}}$-module $\widehat{\Omega}^1_{B_{\ak{p}}}[1/p]$ are both finite free with the same basis $\df t_1',\dots,\df t_d'$ given by a system of coordinates $t_1',\dots,t_d'\in B[1/p]$ of $B$ by \eqref{eq:rem:adequate-chart} and \ref{prop:quasi-adequate-diff} respectively. In particular, the natural map $\Omega^1_{B_{\ak{p}}/\ca{O}_K}\to \widehat{\Omega}^1_{B_{\ak{p}}/\ca{O}_K}$ induces a natural isomorphism by inverting $p$,
	\begin{align}
		E_{\ak{p}}\otimes_{\ca{L}}\Omega^1_{\ca{L}/K}\iso \widehat{\Omega}^1_{B_{\ak{p}}}[\frac{1}{p}]=\widehat{\Omega}^1_{B_{\ak{p}}/\ca{O}_K}[\frac{1}{p}].
	\end{align}
	Thus, $\df t_1,\dots,\df t_e$ are $\ca{L}$-linearly independent in $\Omega^1_{\ca{L}/K}$ if and only if they are $E_{\ak{p}}$-linearly independent in $\widehat{\Omega}^1_{B_{\ak{p}}}[1/p]$. The conclusion follows from \ref{prop:rank}.
\end{proof}

\begin{mypara}\label{para:general-Bnm-tower-2}
	Following \ref{para:general-Bnm-tower}, we assume that the equivalent conditions in \ref{prop:B-rank} hold. Let $\partial_0\in\lie(\Sigma_{0,\underline{\infty}})$ and $\partial_1,\dots,\partial_e\in\lie(\Delta)$ be the standard bases defined in \ref{para:gamma-basis} for the Kummer tower $(B_{n,\underline{m}})_{(n,\underline{m})\in \bb{N}^{1+e}}$ defined by $\zeta_{p^n},t_{1,p^n},\dots,t_{e,p^n}$. We remark that there are natural identifications of Lie algebras for any $\underline{N}\in J$, $n\in \bb{N}$ and $\underline{m}\in\bb{N}^e$,
	\begin{align}\label{eq:para:general-Bnm-tower-lie}
		\lie(\Delta^{(\underline{N})}_{\underline{m}})=\lie(\Delta),\quad \lie(\Xi^{(\underline{N})})=\lie(\Gamma),\quad \lie(\Sigma^{(\underline{N})}_{n,\underline{m}})=\lie(\Sigma).
	\end{align}
	We define $1+e$ elements of the finite projective $\widehat{\overline{B}}[1/p]$-module $\scr{E}_B(-1)$ defined in \ref{thm:B-fal-ext},  
	\begin{align}\label{eq:para:general-Bnm-tower}
		T_0=(\df\log(\zeta_{p^n}))_{n\in\bb{N}}\otimes\zeta^{-1},\ T_1=(\df\log (t_{1,p^n}))_{n\in\bb{N}}\otimes\zeta^{-1},\  \dots,\ T_e=(\df\log (t_{e,p^n}))_{n\in\bb{N}}\otimes\zeta^{-1},
	\end{align}
	where $\zeta=(\zeta_{p^n})_{n\in\bb{N}}$.
\end{mypara}

\begin{mythm}[{cf. \ref{thm:sen-brinon-operator}}]\label{thm:sen-brinon-B}
	Let $B$ be a quasi-adequate $\ca{O}_K$-algebra with fraction field $\ca{L}$, $G=\gal(\ca{L}_{\mrm{ur}}/\ca{L})$. Then, for any object $W$ of $\repnpr(G,\widehat{\overline{B}}[1/p])$, there is a canonical homomorphism of $\widehat{\overline{B}}[1/p]$-linear Lie algebras (see {\rm\ref{para:B-fal-ext-dual}}),
	\begin{align}\label{eq:sen-brinon-B}
		\varphi_{\sen}|_W:\scr{E}_B^*(1)\longrightarrow \mrm{End}_{\widehat{\overline{B}}[\frac{1}{p}]}(W),
	\end{align}
	which is $G$-equivariant with respect to the canonical action on $\scr{E}_B^*(1)$ defined in {\rm\ref{para:B-fal-ext-dual}} and the adjoint action on $\mrm{End}_{\widehat{\overline{B}}[1/p]}(W)$ (i.e. $g\in G$ sends an endomorphism $\phi$ to $g\circ \phi \circ g^{-1}$), and functorial in $W$, i.e. it defines a canonical functor 
	\begin{align}\label{eq:sen-functor-B}
		\varphi_{\sen}:\repnpr(G,\widehat{\overline{B}}[\frac{1}{p}])\longrightarrow \mbf{Rep}^{\mrm{proj}}(\scr{E}_B^*(1),\widehat{\overline{B}}[\frac{1}{p}]),
	\end{align}
	from the category of finite projective (continuous semi-linear) $\widehat{\overline{B}}[1/p]$-representations of the profinite group $G$ to the category of finite projective $\widehat{\overline{B}}[1/p]$-linear representations of the Lie algebra $\scr{E}_B^*(1)$.
	
	Moreover, under the assumption in {\rm\ref{para:general-Bnm-tower-2}} and with the same notation, assume that there exists an object $V$ of $\repnan{\Delta^{(\underline{N})}_{\underline{m}}}(\Xi^{(\underline{N})},\widetilde{B}^{(\underline{N})}_{\infty,\underline{m}}[1/p])$ for some $\underline{N}\in J$ and $\underline{m}\in\bb{N}^d$ with $W= \widehat{\overline{B}}\otimes_{\widetilde{B}^{(\underline{N})}_{\infty,\underline{m}}}V$. Then, for any $f\in \scr{E}_B^*(1)=\ho_{\widehat{\overline{B}}[1/p]}(\scr{E}_B(-1),\widehat{\overline{B}}[1/p])$,
	\begin{align}\label{eq:thm:sen-brinon-B}
		\varphi_{\sen}|_W(f)=\sum_{i=0}^e f(T_i)\otimes\varphi_{\partial_i}|_V,
	\end{align}
	where $\varphi_{\partial_i}|_V\in\mrm{End}_{\widetilde{B}^{(\underline{N})}_{\infty,\underline{m}}[1/p]}(V)$ is the infinitesimal action of $\partial_i\in\lie(\Xi^{(\underline{N})})$ on $V$ defined in {\rm\ref{cor:operator}}.
\end{mythm}
\begin{proof}
	For any $\ak{q}\in\ak{S}_p(\overline{B})$ with image $\ak{p}\in\ak{S}_p(B)$, consider the element $(B_{\triv},B,\overline{B})\to (E_{\ak{p}},\ca{O}_{E_{\ak{p}}},\ca{O}_{\overline{E}_{\ak{q}}})$ of $\ak{E}(B)$ defined in \ref{para:notation-A-inj}. Consider the diagram (cf. \ref{rem:B-fal-ext})
	\begin{align}\label{diam:thm:sen-brinon-B}
		\xymatrix{
			\scr{E}_B^*(1)\ar[r]\ar@{.>}[d]^-{\alpha}&\prod_{\ak{q}}\widehat{\overline{E}}_{\ak{q}}\otimes_{\widehat{\overline{B}}}\scr{E}_B^*(1)\ar@{.>}[d]^-{(\alpha_{\ak{q}})_{\ak{q}}}&\prod_{\ak{q}}\scr{E}_{\ca{O}_{E_{\ak{p}}}}^*(1)\ar[l]_-{\sim}\ar[d]^-{\varphi_{\sen}}\\
			 \mrm{End}_{\widehat{\overline{B}}[\frac{1}{p}]}(W)\ar[r]&\prod_{\ak{q}}\widehat{\overline{E}}_{\ak{q}}\otimes_{\widehat{\overline{B}}}\mrm{End}_{\widehat{\overline{B}}[\frac{1}{p}]}(W)&\prod_{\ak{q}} \mrm{End}_{\widehat{\overline{E}}_{\ak{q}}}(\widehat{\overline{E}}_{\ak{q}}\otimes_{\widehat{\overline{B}}}W)\ar[l]_-{\sim}
		}
	\end{align}
	where the product is taken over $\ak{q}\in\ak{S}_p(\overline{B})$, and the right vertical arrow is the canonical Lie algebra action defined in \ref{thm:sen-brinon-operator}. Notice that the horizontal arrows in the right square are isomorphisms, since $W$ is finite projective and both $\scr{E}_B$ and $\scr{E}_{\ca{O}_{E_\ak{q}}}$ admit the same basis induced by a system of coordinates of $B$ (cf. \ref{para:B-fal-ext-compare}). On the other hand, the horizontal arrows in the left square are injective by \ref{prop:ht1-prime-inj}. Therefore, there is at most one map $\alpha$ making this diagram commutative together with its base change $\alpha_{\ak{q}}$. 
	
	We claim that taking $\alpha$ to be the map defined by \eqref{eq:thm:sen-brinon-B} (under the corresponding assumptions) makes the diagram commute. Indeed, for any $f\in \scr{E}^*_B(1)$, if $f_{\ak{q}}\in \scr{E}_{\ca{O}_{E_{\ak{p}}}}^*(1)$ denotes its base change, then we need to check that
	\begin{align}
		\varphi_{\sen}|_{\widehat{\overline{E}}_{\ak{q}}\otimes_{\widehat{\overline{B}}}W}(f_{\ak{q}})=\sum_{i=0}^ef_{\ak{q}}(T_{i,\ak{q}})\otimes\varphi_{\partial_i}|_V,
	\end{align}
	where $T_{i,\ak{q}}\in  \scr{E}_{\ca{O}_{E_{\ak{p}}}}(-1)$ is the image of the element $T_i\in \scr{E}_B(-1)$ defined in \eqref{eq:para:general-Bnm-tower}. For simplicity, we omit the subscripts $\ak{p}$ and $\ak{q}$. With the notation in \ref{para:general-Bnm-tower}, there is a commutative diagram of fields defined by $t_{i,k}$ similarly as in \ref{para:notation-B-galois},
	\begin{align}
		\xymatrix{
			E\ar[r]&E^{(\underline{N})}_{0,\underline{m}}\ar[r]_-{\Sigma'}\ar@/^2.5pc/[rrr]|-{G'}\ar@/^1.5pc/[rr]|-{\Gamma'}&E^{(\underline{N})}_{\infty,\underline{m}}\ar[r]_-{\Delta'}&E^{(\underline{N})}_{\infty,\underline{\infty}}\ar[r]&\overline{E}\\
			\ca{L}\ar[u]\ar[r]\ar@/_2pc/[rrrr]|-{G}\ar@/_1pc/[rrr]|-{\Xi^{(\underline{N})}}&\ca{L}^{(\underline{N})}_{0,\underline{m}}\ar[r]^-{\Sigma^{(\underline{N})}_{0,\underline{m}}}\ar[u]&\ca{L}^{(\underline{N})}_{\infty,\underline{m}}\ar[r]^-{\Delta^{(\underline{N})}_{\underline{m}}}\ar[u]&\ca{L}^{(\underline{N})}_{\infty,\underline{\infty}}\ar[r]\ar[u]&\overline{\ca{L}}\ar[u]
		}
	\end{align}
	where we put the notation for corresponding Galois groups on the arrows. For $\underline{N}$ and $\underline{m}$ in the statement, we set $E'=E^{(\underline{N})}_{0,\underline{m}}$ and consider the Kummer tower $(\ca{O}_{E'_{n,\underline{m'}}})_{(n,\underline{m'})\in\bb{N}^{1+e}}$ defined by $\zeta_{p^n},t_{1,p^n},\dots,t_{e,p^n}$ (thus $E'_\infty=E^{(\underline{N})}_{\infty,\underline{m}}$ and $E'_{\infty,\underline{\infty}}=E^{(\underline{N})}_{\infty,\underline{\infty}}$), which satisfies the condition {\rm\ref{prop:kummer-tower-lem}.(\ref{item:prop:kummer-tower-lem-1})} by \ref{prop:B-rank} (as $E'$ is a finite extension of $E$). Then, $W'=\widehat{\overline{E}}\otimes_{\widehat{\overline{B}}} W$ is an object of $\repnpr(G',\widehat{\overline{E}})$ and $V'=E'_{\infty}\otimes_{\widetilde{B}^{(\underline{N})}_{\infty,\underline{m}}}V$ is an object of $\repnan{\Delta'}(\Gamma',E'_{\infty})$ such that $W'=\widehat{\overline{E}}\otimes_{E'_{\infty}}V'$. By \ref{prop:B-rank} and \eqref{eq:para:general-Bnm-tower-lie}, we have natural identifications of Lie algebras
	\begin{align}
		\lie(\Delta')=\lie(\Delta),\quad \lie(\Gamma')=\lie(\Gamma),\quad \lie(\Sigma')=\lie(\Sigma).
	\end{align} 
	Moreover, the images of $\partial_0,\dots,\partial_e$ in $\lie(\Gamma')$ are the standard basis $\partial_0',\dots,\partial_e'$ defined in \ref{para:gamma-basis} for the Kummer tower $(\ca{O}_{E'_{n,\underline{m'}}})_{(n,\underline{m'})\in\bb{N}^{1+e}}$ defined by $\zeta_{p^n},t_{1,p^n},\dots,t_{e,p^n}$. Therefore, applying the ``moreover'' part of \ref{thm:sen-brinon-operator} to $E'$ (whose assumptions are satisfied by \ref{prop:B-rank}), for any $f'\in \scr{E}^*_{\ca{O}_{E'}}(1)$ we have
	\begin{align}
		\varphi_{\sen}|_{W'}(f')=\sum_{i=0}^ef'(T_i')\otimes\varphi_{\partial_i'}|_{V'},
	\end{align}
	where $T_i'\in \scr{E}_{\ca{O}_{E'}}(-1)$ is the image of $T_i\in \scr{E}_B(-1)$. Notice that $\varphi_{\partial_i'}|_{V'}=\id_{E'_{\infty}}\otimes\varphi_{\partial_i}|_V$ by \ref{rem:derivative}. The claim follows from the following canonical commutative diagram given by \ref{prop:sen-brinon-operator-func} (as $E'$ is a finite extension of $E$)
	\begin{align}
		\xymatrix{
			\scr{E}^*_{\ca{O}_{E'}}(1)\ar[rr]^-{\varphi_{\sen}|_{W'}}\ar[d]^-{\wr}&&\mrm{End}_{\widehat{\overline{E}}}(W')\\
			\scr{E}^*_{\ca{O}_E}(1)
			\ar[rr]^-{\varphi_{\sen}|_{\widehat{\overline{E}}\otimes_{\widehat{\overline{B}}} W}}&&\mrm{End}_{\widehat{\overline{E}}}(\widehat{\overline{E}}\otimes_{\widehat{\overline{B}}} W)
			\ar@{=}[u]
		}
	\end{align}
	
	Finally, the uniqueness of $\alpha$ implies that \eqref{eq:thm:sen-brinon-B} does not depend on choice of $V$ or $t_i$, and the descent \ref{thm:descent} guarantees its existence and functoriality. Its $G$-equivariance follows from the same argument of that of \eqref{eq:sen-brinon-operator}.
\end{proof}

\begin{myrem}\label{rem:sen-brinon-B-field}
	The same argument also shows that the $\widehat{\overline{B}}[1/p]$-linear map
	\begin{align}
		W\longrightarrow W\otimes_{\widehat{\overline{B}}}\scr{E}_B(-1)
	\end{align}
	sending $x$ to $\sum_{i=0}^e (\id_{\widehat{\overline{B}}}\otimes \varphi_{\partial_i}|_V)(x)\otimes T_i$, is $G$-equivariant and does not depend on the choice of $V$ or $t_i$. It naturally induces the map $\varphi_{\sen}|_W$ \eqref{eq:sen-brinon-B}. We note that it is not a Higgs field.
\end{myrem}

\begin{mydefn}\label{defn:sen-brinon-B}
	Let $W$ be an object of $\repnpr(G,\widehat{\overline{B}}[1/p])$. We denote by $\Phi(W)$ the image of $\varphi_{\sen}|_W$, and by $\Phi^{\geo}(W)$ the image of $\ho_{B[1/p]}(\Omega^1_{Y_K/K}(-1),\widehat{\overline{B}}[1/p])$ under $\varphi_{\sen}|_W$. We call an element of $\Phi(W)\subseteq \mrm{End}_{\widehat{\overline{B}}[1/p]}(W)$ a \emph{Sen operator} of $W$. We call an element of $\Phi^{\geo}(W)\subseteq \mrm{End}_{\widehat{\overline{B}}[1/p]}(W)$ a \emph{geometric Sen operator} of $W$. And we call the image of $1\in\widehat{\overline{B}}[1/p]$ in $\Phi^{\ari}(W)=\Phi(W)/\Phi^{\geo}(W)$ the \emph{arithmetic Sen operator} of $W$.
\end{mydefn}

\begin{align}
	\xymatrix{
		0\ar[r]&\ho_{B[\frac{1}{p}]}(\Omega^1_{Y_K/K}(-1),\widehat{\overline{B}}[\frac{1}{p}])\ar[r]^-{\jmath^*}\ar@{->>}[d]&\scr{E}^*_B(1)\ar[r]^-{\iota^*}\ar@{->>}[d]&\widehat{\overline{B}}[\frac{1}{p}]\ar@{->>}[d]\ar[r]&0\\
		0\ar[r]&\Phi^{\geo}(W)\ar[r]&\Phi(W)\ar[r]&\Phi^{\ari}(W)\ar[r]&0
	}
\end{align}

\begin{myprop}[{cf. \ref{prop:sen-brinon-operator-func}}]\label{prop:sen-brinon-B-func}
	Let $K'$ be a complete discrete valuation field extension of $K$ with perfect residue field, $B'$ a quasi-adequate $\ca{O}_{K'}$-algebra with fraction field $\ca{L}'$. Consider a commutative diagram of $(K,\ca{O}_K,\ca{O}_{\overline{K}})$-triples (see {\rm\ref{defn:triple}}) 
	\begin{align}
		\xymatrix{
			(B_{\triv},B,\overline{B})\ar[r]& (B'_{\triv},B',\overline{B'})\\
			(K,\ca{O}_K,\ca{O}_{\overline{K}})\ar[u]\ar[r]&(K',\ca{O}_{K'},\ca{O}_{\overline{K'}})\ar[u]
		}
	\end{align}
	with $\overline{B}\to \overline{B'}$ injective. Let $W$ be an object of $\repnpr(G,\widehat{\overline{B}}[1/p])$, $W'=\widehat{\overline{B'}}\otimes_{\widehat{\overline{B}}}W$ the associated object of $\repnpr(G',\widehat{\overline{B'}}[1/p])$, where $G'=\gal(\ca{L}'_{\mrm{ur}}/\ca{L}')$. Assume that $\ca{L}'\otimes_{\ca{L}}\Omega^1_{\ca{L}/K}\to \Omega^1_{\ca{L'}/K'}$ is injective. Then, there is a natural commutative diagram
	\begin{align}\label{diam:sen-brinon-B-func}
		\xymatrix{
			\scr{E}^*_{B'}(1)\ar[rr]^-{\varphi_{\sen}|_{W'}}\ar[d]&&\mrm{End}_{\widehat{\overline{B'}}[\frac{1}{p}]}(W')\\
			\widehat{\overline{B'}}\otimes_{\widehat{\overline{B}}}\scr{E}^*_B(1)
			\ar[rr]^-{\id_{\widehat{\overline{B'}}}\otimes\varphi_{\sen}|_W}&&\widehat{\overline{B'}}\otimes_{\widehat{\overline{B}}}\mrm{End}_{\widehat{\overline{B}}[\frac{1}{p}]}(W)
			\ar[u]_-{\wr}
		}
	\end{align}
	where $\varphi_{\sen}$ are the canonical Lie algebra actions defined in {\rm\ref{thm:sen-brinon-B}}, the left vertical arrow is induced by taking dual of the natural map $\widehat{\overline{B'}}\otimes_{\widehat{\overline{B}}}\scr{E}_B(-1)\to \scr{E}_{B'}(-1)$ (cf. {\rm\ref{rem:B-fal-ext}}), and the right vertical arrow is the canonical isomorphism. Moreover, if we denote by $\widehat{\overline{B'}}\Phi(W)$ the image of $\widehat{\overline{B'}}\otimes_{\widehat{\overline{B}}}\Phi(W)$ in $\widehat{\overline{B'}}\otimes_{\widehat{\overline{B}}}\mrm{End}_{\widehat{\overline{B}}[\frac{1}{p}]}(W)$, then the inverse of the right vertical arrow induces a natural isomorphism 
	\begin{align}\label{eq:prop:sen-brinon-B-func}
		\Phi(W')\iso\widehat{\overline{B'}}\Phi(W),
	\end{align}
	which is compatible with geometric and arithmetic Sen operators.
\end{myprop}
\begin{proof}
	We follow the same argument of \ref{prop:sen-brinon-operator-func} using \ref{prop:B-rank} instead of \ref{prop:rank}. We may assume that we are in the situation of the ``moreover'' part of \ref{thm:sen-brinon-B} by the descent theorem \ref{thm:descent}. Let $t'_{i,k}\in \overline{B'}[1/p]$ be the image of $t_{i,k}\in \overline{B}[1/p]$. With the notation in \ref{para:general-Bnm-tower}, there is a commutative diagram of fields
	\begin{align}
		\xymatrix{
			\ca{L}'\ar[r]\ar@/^2pc/[rrrr]|-{G'}\ar@/^1pc/[rrr]|-{\Xi'^{(\underline{N})}}&\ca{L}'^{(\underline{N})}_{0,\underline{m}}\ar[r]&\ca{L}'^{(\underline{N})}_{\infty,\underline{m}}\ar[r]_-{\Delta'^{(\underline{N})}_{\underline{m}}}&\ca{L}'^{(\underline{N})}_{\infty,\underline{\infty}}\ar[r]&\overline{\ca{L}'}\\
			\ca{L}\ar[u]\ar[r]\ar@/_2pc/[rrrr]|-{G}\ar@/_1pc/[rrr]|-{\Xi^{(\underline{N})}}&\ca{L}^{(\underline{N})}_{0,\underline{m}}\ar[r]\ar[u]&\ca{L}^{(\underline{N})}_{\infty,\underline{m}}\ar[r]^-{\Delta^{(\underline{N})}_{\underline{m}}}\ar[u]&\ca{L}^{(\underline{N})}_{\infty,\underline{\infty}}\ar[r]\ar[u]&\overline{\ca{L}}\ar[u]
		}
	\end{align}
	Since $\df t_1',\dots,\df t_e'$ are $\ca{L}'$-linearly independent in $\Omega^1_{\ca{L}'/K'}$ by assumption, the Kummer tower $(B'_{n,\underline{m}})_{(n,\underline{m})\in \bb{N}^{1+e}}$ defined by $\zeta_{p^n},t'_{1,p^n},\dots,t'_{e,p^n}$ also satisfies the condition {\rm\ref{prop:kummer-tower-lem}.(\ref{item:prop:kummer-tower-lem-1})} by \ref{prop:B-rank}. By the discussion in \ref{para:general-Bnm-tower-2}, we have a natural isomorphism $\lie(\Xi'^{(\underline{N})})\iso \lie(\Xi^{(\underline{N})})$ which identifies their standard bases $\{\partial'_i\}_{0\leq i\leq e}$ and $\{\partial_i\}_{0\leq i\leq e}$. Moreover, $V'=\widetilde{B'}^{(\underline{N})}_{\infty,\underline{m}}\otimes_{\widetilde{B}^{(\underline{N})}_{\infty,\underline{m}}}V$ is an object of $\repnan{\Delta'^{(\underline{N})}_{\underline{m}}}(\Xi'^{(\underline{N})},\widetilde{B'}^{(\underline{N})}_{\infty,\underline{m}}[1/p])$ with $W'=\widehat{\overline{B'}}\otimes_{\widehat{\overline{B}}}W= \widehat{\overline{B'}}\otimes_{\widetilde{B'}^{(\underline{N})}_{\infty,\underline{m}}}V'$. By \ref{rem:derivative}, the natural identification $\mrm{End}_{\widehat{\overline{B'}}}(W')=\widehat{\overline{B'}}\otimes_{\widehat{\overline{B}}}\mrm{End}_{\widehat{\overline{B}}[1/p]}(W)$ identifies $\id_{\widehat{\overline{B'}}}\otimes \varphi_{\partial'_i}|_{V'}$ with $\id_{\widehat{\overline{B'}}}\otimes \varphi_{\partial_i}|_V$. This shows that the diagram \eqref{diam:sen-brinon-B-func} is commutative which induces an isomorphism \eqref{eq:prop:sen-brinon-B-func}.
\end{proof}

\begin{myrem}\label{rem:sen-brinon-B-func}
	One can also replace $B'$ by a complete discrete valuation ring extension $\ca{O}_E$ of $\ca{O}_K$ whose residue field admits a finite $p$-basis. Assuming that $E\otimes_{\ca{L}}\Omega^1_{\ca{L}/K}\to \widehat{\Omega}^1_{\ca{O}_E}[1/p]$ is injective, the result of \ref{prop:sen-brinon-B-func} still holds for $\ca{O}_E$.
\end{myrem}

\begin{mycor}\label{cor:quasi-adequate-cofinal}
	Let $\scr{F}^{\mrm{fini}}_{\ca{L}_{\mrm{ur}}/\ca{L}}$ be the directed system of finite field extensions $\ca{L}'$ of $\ca{L}$ in $\ca{L}_{\mrm{ur}}$, $\scr{F}^{\mrm{qa}}_{\ca{L}_{\mrm{ur}}/\ca{L}}$ the subsystem consists of $\ca{L}'$ such that the integral closure $B'$ of $B$ in $\ca{L}'$ is a quasi-adequate $\ca{O}_K$-algebra, where we put $B'_{\triv}=B_{\triv}\otimes_BB'$ and $\overline{B'}=\overline{B}$.
	\begin{enumerate}
		\renewcommand{\labelenumi}{{\rm(\theenumi)}}
		\item The subsystem $\scr{F}^{\mrm{qa}}_{\ca{L}_{\mrm{ur}}/\ca{L}}$ is cofinal in $\scr{F}^{\mrm{fini}}_{\ca{L}_{\mrm{ur}}/\ca{L}}$.\label{item:cor:quasi-adequate-cofinal-1}
		\item Let $\ca{L}'\in\scr{F}^{\mrm{qa}}_{\ca{L}_{\mrm{ur}}/\ca{L}}$ with Galois group $G'=\gal(\ca{L}_{\mrm{ur}}/\ca{L}')\subseteq G$, $W$ an object of $\repnpr(G,\widehat{\overline{B}}[1/p])$ and $W'$ the object of $\repnpr(G',\widehat{\overline{B}}[1/p])$ defined by restricting the $G$-action on $W$. Then, there is a natural commutative diagram
		\begin{align}
			\xymatrix{
				\scr{E}^*_{B'}(1)\ar[rr]^-{\varphi_{\sen}|_{W'}}\ar[d]^-{\wr}&&\mrm{End}_{\widehat{\overline{B}}[\frac{1}{p}]}(W')\\
				\scr{E}^*_B(1)
				\ar[rr]^-{\varphi_{\sen}|_W}&&\mrm{End}_{\widehat{\overline{B}}[\frac{1}{p}]}(W)
				\ar@{=}[u]
			}
		\end{align}\label{item:cor:quasi-adequate-cofinal-2}
	\end{enumerate}
\end{mycor}
\begin{proof}
	(\ref{item:cor:quasi-adequate-cofinal-1}) follows from Abhyankar's lemma \ref{lem:abhyankar}, and (\ref{item:cor:quasi-adequate-cofinal-2}) follows from \ref{rem:B-fal-ext} and \ref{prop:sen-brinon-B-func} (as $\ca{L}'$ is \'etale over $\ca{L}$).
\end{proof}

\begin{myrem}\label{rem:quasi-adequate-cofinal}
	 Let $\ca{L}'\in\scr{F}^{\mrm{fini}}_{\ca{L}_{\mrm{ur}}/\ca{L}}$ with Galois group $G'=\gal(\ca{L}_{\mrm{ur}}/\ca{L}')\subseteq G$, $\ca{L}''\in \scr{F}^{\mrm{qa}}_{\ca{L}_{\mrm{ur}}/\ca{L}}$ containing $\ca{L}'$ with Galois group $G''=\gal(\ca{L}_{\mrm{ur}}/\ca{L}'')\subseteq G'$, $W'$ an object of $\repnpr(G',\widehat{\overline{B}}[1/p])$, $W''$ the object of $\repnpr(G'',\widehat{\overline{B}}[1/p])$ defined by restricting the $G'$-action on $W'$. Then, we define a Lie algebra action $\varphi_{\sen}|_{W'}$ of $\scr{E}_B^*(1)$ on $W'$ by assigning it to be $\varphi_{\sen}|_{W''}$. This definition of $\varphi_{\sen}|_{W'}$ does not depend on the choice of $\ca{L}''$ by \ref{cor:quasi-adequate-cofinal}. One can check that $\varphi_{\sen}|_{W'}$ is $G'$-equivariant by the same arguments in \ref{thm:sen-brinon-operator}.
\end{myrem}

\begin{mylem}\label{lem:G-H-fixed}
	With the notation in {\rm\ref{para:notation-B-galois}}, let $V$ be an object of $\repnan{\Delta^{(\underline{N})}_{\underline{m}}}(\Xi^{(\underline{N})},\widetilde{B}^{(\underline{N})}_{\infty,\underline{m}}[1/p])$ for some $\underline{N}\in J$ and $\underline{m}\in \bb{N}^d$, $W=\widehat{\overline{B}}\otimes_{\widetilde{B}^{(\underline{N})}_{\infty,\underline{m}}}V$ the associated object of $\repnpr(G,\widehat{\overline{B}}[1/p])$.
	\begin{enumerate}
		\renewcommand{\labelenumi}{{\rm(\theenumi)}}
		\item Let $V'=\widetilde{B}^{(\underline{N})}_{\infty,\underline{\infty}}\otimes_{\widetilde{B}^{(\underline{N})}_{\infty,\underline{m}}}V$ be the associated object of $\repnpr(\Xi^{(\underline{N})},\widetilde{B}^{(\underline{N})}_{\infty,\underline{\infty}}[1/p])$. Then, $W^G=V'^G$. In particular, $(\widehat{\overline{B}}[1/p])^G=\widehat{B}[1/p]$.\label{item:lem:G-H-fixed-1}
		\item Let $C=\colim_{\underline{m'}\in\bb{N}^d}\widehat{B}^{(\underline{N})}_{\infty,\underline{m'}}$, $V''=C\otimes_{\widetilde{B}^{(\underline{N})}_{\infty,\underline{m}}}V$ the associated object of $\repnpr(\Delta^{(\underline{N})},C[1/p])$. Then, $W^H=V''^H$.\label{item:lem:G-H-fixed-2}
	\end{enumerate}
\end{mylem}
\begin{proof}
	(\ref{item:lem:G-H-fixed-1}) Recall that by the proof of \ref{prop:almost-et-descent} we have $W^{H^{(\underline{N})}_{\underline{\infty}}}=\widehat{B}^{(\underline{N})}_{\infty,\underline{\infty}}\otimes_{\widetilde{B}^{(\underline{N})}_{\infty,\underline{m}}}V$, and by \ref{lem:decompletion}.(\ref{item:lem:decompletion-2}) its $(\Xi^{(\underline{N})},\widehat{B}[1/p])$-finite part is $V'$. Thus, we have $W^G\subseteq V'^G$ (and hence $W^G=V'^G$ as $V'\subseteq W$).	In particular, we have $(\widehat{\overline{B}})^G\subseteq (\widetilde{B}^{(\underline{N})}_{\infty,\underline{\infty}})^G$. Since $\widehat{B}^{(\underline{N})}_{n,\underline{m}}=\widehat{B}\otimes_B B^{(\underline{N})}_{n,\underline{m}}$ (as $B^{(\underline{N})}_{n,\underline{m}}$ is finite over the Noetherian ring $B$), we have $(\widetilde{B}^{(\underline{N})}_{\infty,\underline{\infty}})^G=\colim (\widehat{B}^{(\underline{N})}_{n,\underline{m}})^G=\colim \widehat{B}\otimes_B (B^{(\underline{N})}_{n,\underline{m}})^G=\widehat{B}$ (as $\widehat{B}$ is flat over $B$).
	
	(\ref{item:lem:G-H-fixed-2}) The arguments of \ref{lem:decompletion} also show that $(\Delta^{(\underline{N})},\widehat{B}^{(\underline{N})}_\infty[1/p])$-finite part of $W^{H^{(\underline{N})}_{\underline{\infty}}}$ is $V''$. Thus, we have $W^H\subseteq V''^H$ (and hence $W^H=V''^H$ as $V''\subseteq W$).
\end{proof}

\begin{myprop}\label{prop:sen-fixed}
	Let $W$ be an object of $\repnpr(G,\widehat{\overline{B}}[1/p])$. Then, any Sen operator of $W$ vanishes on $W^G$, and any geometric Sen operator of $W$ vanishes on $W^H$.
\end{myprop}
\begin{proof}
	By the descent theorem \ref{thm:descent}, we may put ourselves in the situation of \ref{lem:G-H-fixed}. By \ref{prop:derivative}, the infinitesimal Lie algebra action of $\lie(\Xi^{(\underline{N})})$ (resp. $\lie(\Delta^{(\underline{N})})$) on $V'$ (resp. $V''$) is well-defined and vanishes on $V'^{\Xi^{(\underline{N})}}$ (resp. $V''^{\Delta^{(\underline{N})}}$), and thus vanishes on $W^G$ (resp. $W^H$) by \ref{lem:G-H-fixed}. On the other hand, the infinitesimal action of $\lie(\Xi^{(\underline{N})})$ (resp. $\lie(\Delta^{(\underline{N})})$) on $V'$ (resp. $V''$) is the base change of that on $V$ by \ref{rem:derivative}. Thus, these infinitesimal actions induce the (resp. geometric) Sen operators on $W$ by extending scalars to $\widehat{\overline{B}}[1/p]$, which completes the proof.
\end{proof}

\begin{mypara}
	Recall that for any finite projective module $M$ over a ring $R$, the trace map $\mrm{Tr}:\mrm{End}_R(M)\to R$ is the composition of the second map with the inverse of the first map:
	\begin{align}
		\mrm{End}_R(M)\stackrel{\sim}{\longleftarrow} M\otimes_R \ho_R(M,R)\longrightarrow R,
	\end{align}
	where for any $x\in M$ and $f\in \ho_R(M,R)$, the element $x\otimes f$ is mapped to $xf$ by the first arrow and to $f(x)$ by the second arrow. For any $R$-linear endomorphism $\phi$ of $M$, its determinant is defined as follows (\cite[\textsection 1]{goldman1961det}): we write $R^n=M\oplus M'$ for some $n\in \bb{N}$ and define $\det(\phi)$ to be the determinant of the endomorphism $\phi\oplus \id_{M'}$ on $R^n$, which does not depend on the choice of $M'$ and is compatible with base change. Then, the characteristic polynomial $\det(T-\phi)\in R[T]$ of $\phi$ is defined as the determinant of the $R[T]$-linear endomorphism $T\otimes \id_M-\id_{R[T]}\otimes\phi$ of $R[T]\otimes_R M$ (\cite[\textsection 2]{goldman1961det}). Similarly, one can define the reverse characteristic polynomial $\det(1-T\phi)\in R[T]$ as the determinant of the $R[T]$-linear endomorphism $\id_{R[T]}\otimes \id_M-T\otimes\phi$ of $R[T]\otimes_R M$. One can check by taking localizations of $R$ that the reverse characteristic polynomial is also given by the formula (which is indeed a finite sum)
	\begin{align}\label{eq:reverse-char-poly}
		\det(1-T\phi)=\sum_{k=0}^\infty (-1)^k \mrm{Tr}(\wedge^k \phi)T^k,
	\end{align}
	where $\wedge^k \phi:\wedge^k M\to \wedge^k M$ is the $R$-linear endomorphism of the $k$-th exterior power of $M$ induced by $\phi$.	By the multiplicativity of determinants, we have  $\det(T-\phi)=\det(T)\det(1-T^{-1}\phi)$ in $R[T^{\pm 1}]$. We remark that the coefficients of the characteristic polynomial $\det(T)$ of the zero endomorphism are mutually orthogonal idempotents of $R$, and $\det(T)$ may not be a monomial if $\spec(R)$ is not connected (\cite[2.2, 2.3]{goldman1961det}).
\end{mypara}

\begin{mylem}\label{lem:char-poly-conjugate}
	Let $R$ be a ring endowed with an action of a group $G$, $M$ a finite projective $R$-module endowed with a semi-linear action of $G$, $\phi$ an $R$-linear endomorphism of $M$. Then, for any $g\in G$, $\det(1-T(g\circ\phi\circ g^{-1}))=g(\det(1-T\phi))$.
\end{mylem}
\begin{proof}
	It is clear from the definitions that $\mrm{Tr}(g\circ\phi\circ g^{-1})=g(\mrm{Tr}(\phi))$ and that $\wedge^k (g\circ\phi\circ g^{-1})=g\circ(\wedge^k\phi)\circ g^{-1}$. The conclusion follows from \eqref{eq:reverse-char-poly}.
\end{proof}

\begin{myprop}\label{prop:sen-B-char-poly}
	Let $W$ be an object of $\repnpr(G,\widehat{\overline{B}}[1/p])$. Then, any lifting $\phi\in\Phi(W)$ of the arithmetic Sen operator of $W$ has the same (resp. reverse) characteristic polynomial. Moreover, the coefficients of the reverse characteristic polynomial lie in $\widehat{B}[1/p]$.
\end{myprop}
\begin{proof}
	The first assertion follows from the diagram \eqref{diam:thm:sen-brinon-B} and \ref{prop:sen-char-poly}. For the second, notice that for any $g\in G$, $g\circ \phi\circ g^{-1}$ is also a lifting of the arithmetic Sen operator by the $G$-equivariance of $\varphi_{\sen}|_W$. Thus, $\phi$ and $g\circ \phi\circ g^{-1}$ have the same (resp. reverse) characteristic polynomials. By \ref{lem:char-poly-conjugate}, we see that the coefficients of the reverse characteristic polynomial $\det(1-T\phi)$ lie in $\widehat{\overline{B}}[1/p]^G=\widehat{B}[1/p]$ by \ref{lem:G-H-fixed}.(\ref{item:lem:G-H-fixed-1}).
\end{proof}

\begin{myrem}\label{rem:sen-B-char-poly}
	\begin{enumerate}
		\renewcommand{\labelenumi}{{\rm(\theenumi)}}
		\item If $B$ is adequate, then the operator ``$\varphi_{\sen}$'' on the associated Higgs bundle defined by Tsuji (\cite[page 876]{tsuji2018localsimpson}) induces a lifting of the arithmetic Sen operator of $W$ by extending scalars to $\widehat{\overline{B}}[1/p]$ (cf. \cite[15.1.(4)]{tsuji2018localsimpson}). In particular, the characteristic polynomial of $\phi$ coincides with that of ``$\varphi_{\sen}$''. Thus, in general, we call the roots of the characteristic polynomial of $\phi$ the \emph{Hodge-Tate weights} of $W$.
		\item If $W$ is defined over $\widehat{B}[1/p]$, i.e. there exists an object $V$ of $\repnpr(G,\widehat{B}[1/p])$ such that $W=\widehat{\overline{B}}\otimes_{\widehat{B}}V$, then the characteristic polynomial $\det(T)$ of the zero endomorphism of $W$ has coefficients in $\widehat{B}[1/p]$. The identity $\det(T-\phi)=\det(T)\det(1-T^{-1}\phi)$ in $\widehat{\overline{B}}[1/p][T^{\pm 1}]$ implies that the coefficients of the characteristic polynomial of $\phi$ also lie in $\widehat{B}[1/p]$.
	\end{enumerate}
	
\end{myrem}

\begin{mypara}\label{para:inertia-subgroup}
	Let $\ak{q}\in\ak{S}_p(\overline{B})$ with image $\ak{p}\in\ak{S}_p(B)$, consider the element $(B_{\triv},B,\overline{B})\to (E_{\ak{p}},\ca{O}_{E_{\ak{p}}},\ca{O}_{\overline{E}_{\ak{q}}})$ of $\ak{E}(B)$ defined in \ref{para:notation-A-inj}. Let $I_{\ak{q}}\subseteq G=\gal(\ca{L}_{\mrm{ur}}/\ca{L})$ be the image of the inertial subgroup of the absolute Galois group of $E_{\ak{p}}$. It is a closed subgroup of $G$, which we call the \emph{inertia subgroup} of $G$ at $\ak{q}\in\ak{S}_p(\overline{B})$.
\end{mypara}

\begin{mythm}[{cf. \ref{thm:sen-lie}}]\label{thm:sen-lie-B}
	Let $(V,\rho)$ be an object of $\repnpr(G,\bb{Q}_p)$, $W=\widehat{\overline{B}}[1/p]\otimes_{\bb{Q}_p}V$ the associated object of $\repnpr(G,\widehat{\overline{B}}[1/p])$. Then, $\sum_{\ak{q}\in\ak{S}_p(\overline{B})}\lie(\rho(I_{\ak{q}}))$ is the smallest $\bb{Q}_p$-subspace $S$ of $\mrm{End}_{\bb{Q}_p}(V)$ such that the $\widehat{\overline{B}}[1/p]$-module of Sen operators $\Phi(W)$ is contained in the submodule $\widehat{\overline{B}}[1/p]\otimes_{\bb{Q}_p} S$ of $\mrm{End}_{\widehat{\overline{B}}[1/p]}(W)=\widehat{\overline{B}}[1/p]\otimes_{\bb{Q}_p}\mrm{End}_{\bb{Q}_p}(V)$.
\end{mythm}
\begin{proof}
	By \ref{lem:smallest}, it suffices to show that for any $\bb{Q}_p$-linear form $f$ on $\mrm{End}_{\bb{Q}_p}(V)$, we have $f(\lie(\rho(I_{\ak{q}})))=0$ for any $\ak{q}\in\ak{S}_p(\overline{B})$ if and only if $f'(\Phi(W))=0$, where $f'$ is the $\widehat{\overline{B}}[1/p]$-linear form on $\mrm{End}_{\widehat{\overline{B}}[1/p]}(W)$ defined by extending scalars from $f$. We set $W_{\ak{q}}=\widehat{\overline{E}}_{\ak{q}}\otimes_{\widehat{\overline{B}}}W=\widehat{\overline{E}}_{\ak{q}}\otimes_{\bb{Q}_p}V$ and let $f'_{\ak{q}}$ be the $\widehat{\overline{E}}_{\ak{q}}$-linear form on $\mrm{End}_{\widehat{\overline{E}}_{\ak{q}}}(W_{\ak{q}})$ defined by extending scalars from $f$. Consider the commutative diagram 
	\begin{align}
		\xymatrix{
			\mrm{End}_{\widehat{\overline{B}}[1/p]}(W)\ar[d]_-{f'}\ar[r]^-{(\iota_\ak{q})_{\ak{q}}}&\prod_{\ak{q}\in\ak{S}_p(\overline{B})}\mrm{End}_{\widehat{\overline{E}}_{\ak{q}}}(W_{\ak{q}})\ar[d]^-{(f'_{\ak{q}})_{\ak{q}}}\\
			\widehat{\overline{B}}[1/p]\ar[r]&\prod_{\ak{q}\in\ak{S}_p(\overline{B})} \widehat{\overline{E}}_{\ak{q}}
		}
	\end{align}
	where $\iota_\ak{q}$ is defined by extending scalars. Notice that the horizontal arrows are injective by \ref{prop:ht1-prime-inj}. Thus, $f'(\Phi(W))=0$ if and only if $f_{\ak{q}}'(\iota_{\ak{q}}(\Phi(W)))=0$ for any $\ak{q}\in\ak{S}_p(\overline{B})$. Notice that $f(\lie(\rho(I_{\ak{q}})))=0$ if and only if $f_{\ak{q}}'(\Phi(W_{\ak{q}}))=0$ by \ref{thm:sen-lie}, and that $\Phi(W_{\ak{q}})$ is generated by $\iota_{\ak{q}}(\Phi(W))$ by \eqref{eq:prop:sen-brinon-B-func} and \ref{rem:sen-brinon-B-func} (or directly from \eqref{diam:thm:sen-brinon-B}). This completes the proof.
\end{proof}

\begin{mycor}\label{cor:sen-lie-B-zero}
	With the notation in {\rm\ref{thm:sen-lie-B}}, the $\widehat{\overline{B}}[1/p]$-module of Sen operators $\Phi(W)$ is zero if and only if $\rho(I_{\ak{q}})$ is finite for any $\ak{q}\in\ak{S}_p(\overline{B})$.
\end{mycor}

\begin{mycor}\label{cor:sen-lie-B}
	With the notation in {\rm\ref{thm:sen-lie-B}}, the $\widehat{\overline{B}}[1/p]$-module of Sen operators $\Phi(W)$ is contained in $\widehat{\overline{B}}[1/p]\otimes_{\bb{Q}_p} \lie(\rho(G))$, and the $\widehat{\overline{B}}[1/p]$-module of geometric Sen operators $\Phi^{\geo}(W)$ is contained in $\widehat{\overline{B}}[1/p]\otimes_{\bb{Q}_p} \lie(\rho(H))$, where $H=\gal(\ca{L}_{\mrm{ur}}/\ca{L}_\infty)$. 
\end{mycor}
\begin{proof}
	The first assertion is a direct result of \ref{thm:sen-lie-B}. Note that $\Phi^{\geo}(W)=[\Phi(W),\Phi(W)]$ by the Lie algebra structure on $\scr{E}^*_B(1)$ (cf. \ref{para:B-fal-ext-dual}), and that $[\lie(\rho(G)),\lie(\rho(G))]\subseteq \lie (\rho(H))$ as $G/H=\Sigma$ is abelian. Thus, the second assertion follows from the first.
\end{proof}

\begin{mythm}[{cf. \cite[Theorem 12]{sen1980sen}}]\label{thm:sen-lie-lift-B}
	Let $\ca{G}$ be a quotient of $G$ which is a $p$-adic analytic group. Then, there exists a unique homomorphism of $\widehat{\overline{B}}[1/p]$-linear Lie algebras $\varphi_{\sen}|_{\ca{G}}:\scr{E}^*_B(1)\to \widehat{\overline{B}}[1/p]\otimes_{\bb{Q}_p}\lie(\ca{G})$ making the following diagram commutative for any object $V$ of $\repnpr(\ca{G},\bb{Q}_p)$,
	\begin{align}\label{diam:thm:sen-lie-lift-B}
		\xymatrix{
			\scr{E}^*_B(1)\ar[r]^-{\varphi_{\sen}|_{\ca{G}}}\ar[d]_-{\varphi_{\sen}|_W}&\widehat{\overline{B}}[\frac{1}{p}]\otimes_{\bb{Q}_p}\lie(\ca{G})\ar[d]^-{\id_{\widehat{\overline{B}}[\frac{1}{p}]}\otimes\varphi|_V}\\
			\mrm{End}_{\widehat{\overline{B}}[\frac{1}{p}]}(W)&\widehat{\overline{B}}[\frac{1}{p}]\otimes_{\bb{Q}_p}\mrm{End}_{\bb{Q}_p}(V)\ar[l]_-{\sim}
		}
	\end{align}
	where $W=\widehat{\overline{B}}[1/p]\otimes_{\bb{Q}_p}V$ is the associated object of $\repnpr(G,\widehat{\overline{B}}[1/p])$, $\varphi_{\sen}|_W$ is the canonical Lie algebra action defined in {\rm\ref{thm:sen-brinon-B}}, and $\varphi|_V$ is the infinitesimal Lie algebra action of $\lie(\ca{G})$ on $V$ (cf. {\rm\ref{cor:operator}}). 
\end{mythm}
\begin{proof}
	Firstly, as $\ca{G}$ is a compact $p$-adic analytic group, it admits a faithful finite projective $\bb{Q}_p$-representation $V$ by \ref{thm:compact-lie-group}. The faithfulness implies that the map $\varphi|_V:\lie(\ca{G})\to \mrm{End}_{\bb{Q}_p}(V)$ is injective (cf. \ref{prop:derivative}.(\ref{item:prop:derivative-3})). Thus, the uniqueness of $\varphi_{\sen}|_{\ca{G}}$ is clear. 
	
	Consider an injective morphism $V\to V'$ of faithful finite projective $\bb{Q}_p$-representations of $\ca{G}$. Note that $W=\widehat{\overline{B}}[1/p]\otimes_{\bb{Q}_p}V$ is still a subrepresentation of $W'=\widehat{\overline{B}}[1/p]\otimes_{\bb{Q}_p}V'$. We claim that the natural surjection $\Phi(W')\to \Phi(W)$ (of the images of $\varphi_{\sen}|_{W'}$ and $\varphi_{\sen}|_W$) defined by restriction is also injective. Indeed, we regard $\widehat{\overline{B}}[1/p]\otimes_{\bb{Q}_p}\lie(\ca{G})$ as a subset of $\mrm{End}_{\widehat{\overline{B}}[1/p]}(W)$ via $\id_{\widehat{\overline{B}}[1/p]}\otimes\varphi|_V$. Thus, the restriction from $V'$ to $V$ induces the identity map on $\widehat{\overline{B}}[1/p]\otimes_{\bb{Q}_p}\lie(\ca{G})$. On the other hand, $\Phi(W)$ (resp. $\Phi(W')$) is contained in $\widehat{\overline{B}}[1/p]\otimes_{\bb{Q}_p}\lie(\ca{G})$ by \ref{thm:sen-lie-B}, which shows that the surjective map $\Phi(W')\to \Phi(W)$ induced by the restriction is injective.
	
	Therefore, we take a faithful finite projective $\bb{Q}_p$-representation $V$ of $\ca{G}$, and we define $\varphi_{\sen}|_{\ca{G}}$ to be the composition of  
	\begin{align}
		\scr{E}^*_B(1)\stackrel{\varphi_{\sen}|_W}{\longrightarrow}\Phi(W)\subseteq \im(\id_{\widehat{\overline{B}}[\frac{1}{p}]}\otimes\varphi|_V)=\widehat{\overline{B}}[\frac{1}{p}]\otimes_{\bb{Q}_p}\lie(\ca{G}).
	\end{align}
	As any two faithful representations $V$ and $V'$ are both contained in $V\oplus V'$, we deduce easily from the above discussions that this definition of $\varphi_{\sen}|_{\ca{G}}$ does not depend on the choice of $V$. It follows immediately that the diagram \eqref{diam:thm:sen-lie-lift-B} is commutative for faithful representations.
	
	In general, we take a faithful finite projective $\bb{Q}_p$-representation $V'$ of $\ca{G}$. Then, for any object $V$ of $\repnpr(\ca{G},\bb{Q}_p)$, $V\oplus V'$ is also a faithful finite projective $\bb{Q}_p$-representation of $\ca{G}$. The previous result shows that there is a canonical commutative diagram
	\begin{align}\label{diam:thm:sen-lie-lift-B-2}
		\xymatrix{
			\scr{E}^*_B(1)\ar[r]^-{\varphi_{\sen}|_{\ca{G}}}\ar[d]_-{\varphi_{\sen}|_{W\oplus W'}}&\widehat{\overline{B}}[\frac{1}{p}]\otimes_{\bb{Q}_p}\lie(\ca{G})\ar[d]^-{\id_{\widehat{\overline{B}}[\frac{1}{p}]}\otimes\varphi|_{V\oplus V'}}\\
			\mrm{End}_{\widehat{\overline{B}}[\frac{1}{p}]}(W\oplus W')&\widehat{\overline{B}}[\frac{1}{p}]\otimes_{\bb{Q}_p}\mrm{End}_{\bb{Q}_p}(V\oplus V')\ar[l]_-{\sim}
		}
	\end{align}
	where $W=\widehat{\overline{B}}[1/p]\otimes_{\bb{Q}_p}V$ and $W'=\widehat{\overline{B}}[1/p]\otimes_{\bb{Q}_p}V'$. Notice that the image of $\varphi_{\sen}|_{W\oplus W'}=\varphi_{\sen}|_W\oplus \varphi_{\sen}|_{W'}$ lies in $\mrm{End}_{\widehat{\overline{B}}[\frac{1}{p}]}(W)\oplus\mrm{End}_{\widehat{\overline{B}}[\frac{1}{p}]}(W')$ (cf. \ref{thm:sen-brinon-B}), and that the image of $\varphi|_{V\oplus V'}=\varphi|_V\oplus\varphi|_{V'}$ lies in $\mrm{End}_{\bb{Q}_p}(V)\oplus \mrm{End}_{\bb{Q}_p}(V')$. By looking at the first component, the commutativity of \eqref{diam:thm:sen-lie-lift-B-2} implies that of \eqref{diam:thm:sen-lie-lift-B}.
\end{proof}

\begin{mydefn}\label{defn:sen-lie-lift-B}
	Let $\ca{G}$ be a quotient of $G$ which is a $p$-adic analytic group. We denote by $\Phi_{\ca{G}}$ the image of $\varphi_{\sen}|_{\ca{G}}$ (\ref{thm:sen-lie-lift-B}), and by $\Phi^{\geo}_{\ca{G}}$ the image of $\ho_{B[1/p]}(\Omega^1_{Y_K/K}(-1),\widehat{\overline{B}}[1/p])$ under $\varphi_{\sen}|_{\ca{G}}$. We call an element of $\Phi_{\ca{G}}\subseteq \widehat{\overline{B}}[1/p]\otimes_{\bb{Q}_p}\lie(\ca{G})$ a \emph{Sen operator} of $\ca{G}$. We call an element of $\Phi^{\geo}_{\ca{G}}$ a \emph{geometric Sen operator} of $\ca{G}$. And we call the image of $1\in\widehat{\overline{B}}[1/p]$ in $\Phi^{\ari}_{\ca{G}}=\Phi_{\ca{G}}/\Phi^{\geo}_{\ca{G}}$ the \emph{arithmetic Sen operator} of $\ca{G}$.
\end{mydefn}

\begin{mycor}\label{cor:sen-lie-lift-B}
	Let $\ca{G}$ be a quotient of $G$ which is a $p$-adic analytic group.
	\begin{enumerate}
		\renewcommand{\labelenumi}{{\rm(\theenumi)}}
		\item The canonical morphism $\varphi_{\sen}|_{\ca{G}}:\scr{E}^*_B(1)\to \widehat{\overline{B}}[1/p]\otimes_{\bb{Q}_p}\lie(\ca{G})$ is $G$-equivariant with respect to the canonical action on $\scr{E}^*_B(1)$, $\widehat{\overline{B}}[1/p]$ and the adjoint action on $\lie(\ca{G})$ defined in {\rm\ref{para:adjoint}}.\label{item:cor:sen-lie-lift-B-1}
		\item For any $\ak{q}\in\ak{S}_p(\overline{B})$, let $\ca{G}_{I_{\ak{q}}}\subseteq \ca{G}$ be the image of the inertia subgroup $I_{\ak{q}}\subseteq G$ at $\ak{q}$ (see {\rm\ref{para:inertia-subgroup}}). Then, $\sum_{\ak{q}\in\ak{S}_p(\overline{B})}\lie(\ca{G}_{I_{\ak{q}}})$ is the smallest $\bb{Q}_p$-subspace $S$ of $\lie(\ca{G})$ such that $\Phi_{\ca{G}}$ is contained in $\widehat{\overline{B}}[1/p]\otimes_{\bb{Q}_p} S$. \label{item:cor:sen-lie-lift-B-2}
		\item Let $\ca{G}_H\subseteq \ca{G}$ be the image of $H=\gal(\ca{L}_{\mrm{ur}}/\ca{L}_{\infty})\subseteq G$. Then, the Lie algebra $\Phi^{\geo}_{\ca{G}}$ of geometric Sen operators of $\ca{G}$ is contained in $\widehat{\overline{B}}[1/p]\otimes_{\bb{Q}_p}\lie(\ca{G}_H)$.\label{item:cor:sen-lie-lift-B-3}
		\item Let $K'$ be a complete discrete valuation field extension of $K$ with perfect residue field, $B'$ a quasi-adequate $\ca{O}_{K'}$-algebra with fraction field $\ca{L}'$ and Galois group $G'=\gal(\ca{L}'_{\mrm{ur}}/\ca{L}')$. Consider a commutative diagram of $(K,\ca{O}_K,\ca{O}_{\overline{K}})$-triples (see {\rm\ref{defn:triple}}) 
		\begin{align}
			\xymatrix{
				(B_{\triv},B,\overline{B})\ar[r]& (B'_{\triv},B',\overline{B'})\\
				(K,\ca{O}_K,\ca{O}_{\overline{K}})\ar[u]\ar[r]&(K',\ca{O}_{K'},\ca{O}_{\overline{K'}})\ar[u]
			}
		\end{align}
		with $\overline{B}\to \overline{B'}$ injective, and let $\ca{G}'$ be the image of the composition of $G'\to G\to \ca{G}$. Then, there is a natural commutative diagram
		\begin{align}
			\xymatrix{
				\scr{E}^*_{B'}(1)\ar[rr]^-{\varphi_{\sen}|_{\ca{G}'}}\ar[d]&&\widehat{\overline{B'}}[\frac{1}{p}]\otimes_{\bb{Q}_p}\lie(\ca{G}')\ar[d]\\
				\widehat{\overline{B'}}\otimes_{\widehat{\overline{B}}}\scr{E}^*_B(1)
				\ar[rr]^-{\id_{\widehat{\overline{B'}}}\otimes\varphi_{\sen}|_{\ca{G}}}&&\widehat{\overline{B'}}[\frac{1}{p}]\otimes_{\bb{Q}_p}\lie(\ca{G})
			}
		\end{align}
		Moreover, if we denote by $\widehat{\overline{B'}}\Phi_{\ca{G}}$ the image of $\widehat{\overline{B'}}\otimes_{\widehat{\overline{B}}}\Phi_{\ca{G}}$ in $\widehat{\overline{B'}}[\frac{1}{p}]\otimes_{\bb{Q}_p}\lie(\ca{G})$, then the right vertical arrow induces a natural isomorphism
		\begin{align}
			\Phi_{\ca{G}'}\iso\widehat{\overline{B'}}\Phi_{\ca{G}}
		\end{align}
		which is compatible with geometric and arithmetic Sen operators.\label{item:cor:sen-lie-lift-B-4}
		\item Let $\ca{L}'$ be an element of the directed system $\scr{F}^{\mrm{qa}}_{\ca{L}_{\mrm{ur}}/\ca{L}}$ defined in {\rm\ref{cor:quasi-adequate-cofinal}}, $G'=\gal(\ca{L}_{\mrm{ur}}/\ca{L}')$, $\ca{G}'\subseteq \ca{G}$ the image of $G'\subseteq G$. Then, there is a natural commutative diagram
		\begin{align}
			\xymatrix{
				\scr{E}^*_{B'}(1)\ar[rr]^-{\varphi_{\sen}|_{\ca{G}'}}\ar[d]^-{\wr}&&\widehat{\overline{B}}[\frac{1}{p}]\otimes_{\bb{Q}_p}\lie(\ca{G}')\ar[d]^-{\wr}\\
				\scr{E}^*_B(1)\ar[rr]^-{\varphi_{\sen}|_{\ca{G}}}&&\widehat{\overline{B}}[\frac{1}{p}]\otimes_{\bb{Q}_p}\lie(\ca{G})
			}
		\end{align}
		In particular, $\Phi_{\ca{G}'}=\Phi_{\ca{G}}$. \label{item:cor:sen-lie-lift-B-5}
		\item  Let $\ca{G}'$ be a quotient of $\ca{G}$. Then, there is a natural commutative diagram 
		\begin{align}
			\xymatrix{
				\scr{E}^*_B(1)\ar[rr]^-{\varphi_{\sen}|_{\ca{G}}}\ar@{=}[d]&&\widehat{\overline{B}}[\frac{1}{p}]\otimes_{\bb{Q}_p}\lie(\ca{G})\ar@{->>}[d]\\
				\scr{E}^*_B(1)\ar[rr]^-{\varphi_{\sen}|_{\ca{G}'}}&&\widehat{\overline{B}}[\frac{1}{p}]\otimes_{\bb{Q}_p}\lie(\ca{G}')
			}
		\end{align}
		In particular, it induces a surjection $\Phi_{\ca{G}}\to \Phi_{\ca{G}'}$.
		\label{item:cor:sen-lie-lift-B-6}
	\end{enumerate}
\end{mycor}
\begin{proof}
	For (\ref{item:cor:sen-lie-lift-B-1}), the $\ca{G}$-equivariance of $\varphi_{\sen}|_{\ca{G}}$ follows from that of the other three arrows in the diagram \eqref{diam:thm:sen-lie-lift-B} (cf. \ref{thm:sen-brinon-B}, \ref{para:adjoint}). (\ref{item:cor:sen-lie-lift-B-2}) holds by applying \ref{thm:sen-lie-B} to a faithful finite projective $\bb{Q}_p$-representation $V$ of $\ca{G}$. (\ref{item:cor:sen-lie-lift-B-3}) follows from the arguments of \ref{cor:sen-lie-B}. 
	
	For (\ref{item:cor:sen-lie-lift-B-4}), as $\ca{G}'$ is a closed subgroup of $\ca{G}$, a faithful finite projective $\bb{Q}_p$-representation $V$ of $\ca{G}$ defines a faithful finite projective $\bb{Q}_p$-representation $V'$ of $\ca{G}'$ by restricting the action. Combining \ref{prop:sen-brinon-B-func} with \ref{thm:sen-lie-lift-B}, we obtain a natural diagram 
	\begin{align}
		\xymatrix{
			\scr{E}^*_{B'}(1)\ar[rr]^-{\varphi_{\sen}|_{\ca{G}'}}\ar[d]&&\widehat{\overline{B'}}[\frac{1}{p}]\otimes_{\bb{Q}_p}\lie(\ca{G}')\ar[d]\ar[rr]^-{\id_{\widehat{\overline{B'}}[\frac{1}{p}]}\otimes\varphi|_{V'}}&&\widehat{\overline{B'}}[\frac{1}{p}]\otimes_{\bb{Q}_p}\mrm{End}_{\bb{Q}_p}(V')\ar[d]^-{\wr}\\
			\widehat{\overline{B'}}\otimes_{\widehat{\overline{B}}}\scr{E}^*_B(1)
			\ar[rr]^-{\id_{\widehat{\overline{B'}}}\otimes\varphi_{\sen}|_{\ca{G}}}&&\widehat{\overline{B'}}[\frac{1}{p}]\otimes_{\bb{Q}_p}\lie(\ca{G})\ar[rr]^-{\id_{\widehat{\overline{B'}}[\frac{1}{p}]}\otimes\varphi|_V}&&\widehat{\overline{B'}}[\frac{1}{p}]\otimes_{\bb{Q}_p}\mrm{End}_{\bb{Q}_p}(V)
		}
	\end{align}
	Notice that the left square is commutative, since the right square and the big rectangle are commutative and the horizontal arrows in the right square are injective. Moreover, \eqref{eq:prop:sen-brinon-B-func} implies that the image $\Phi_{\ca{G}'}$ of $\varphi_{\sen}|_{\ca{G}'}$ coincides with the image $\widehat{\overline{B'}}\Phi_{\ca{G}}$ of  $\id_{\widehat{\overline{B'}}}\otimes\varphi_{\sen}|_{\ca{G}}$ via the middle vertical arrow (which is injective). This completes the proof of (\ref{item:cor:sen-lie-lift-B-4}), and (\ref{item:cor:sen-lie-lift-B-5}) is a special case of (\ref{item:cor:sen-lie-lift-B-4}).
	
	For (\ref{item:cor:sen-lie-lift-B-6}), a faithful finite projective $\bb{Q}_p$-representation $V$ of $\ca{G}'$ can be regarded as an object of $\repnpr(\ca{G},\bb{Q}_p)$. Thus, we obtain a natural diagram
	\begin{align}
		\xymatrix{
			\scr{E}^*_B(1)\ar[r]^-{\varphi_{\sen}|_{\ca{G}}}\ar[d]_-{\varphi_{\sen}|_W}&\widehat{\overline{B}}[\frac{1}{p}]\otimes_{\bb{Q}_p}\lie(\ca{G})\ar[d]^-{\id_{\widehat{\overline{B}}[\frac{1}{p}]}\otimes\varphi|_V}\ar@{->>}[r]&\widehat{\overline{B}}[\frac{1}{p}]\otimes_{\bb{Q}_p}\lie(\ca{G}')\ar[d]^-{\id_{\widehat{\overline{B}}[\frac{1}{p}]}\otimes\varphi|_V}\\
			\mrm{End}_{\widehat{\overline{B}}[\frac{1}{p}]}(W)&\widehat{\overline{B}}[\frac{1}{p}]\otimes_{\bb{Q}_p}\mrm{End}_{\bb{Q}_p}(V)\ar[l]_-{\sim}&\widehat{\overline{B}}[\frac{1}{p}]\otimes_{\bb{Q}_p}\mrm{End}_{\bb{Q}_p}(V)\ar@{=}[l]
		}
	\end{align}
	Notice that the left square is commutative by \ref{thm:sen-lie-lift-B} and the right square is obviously commutative. Since the right vertical arrow is injective, the composition of the two horizontal arrows on the top is the unique map making the big rectangle commutative, which thus coincides with $\varphi_{\sen}|_{\ca{G}'}$ by \ref{thm:sen-lie-lift-B}.
\end{proof}

\begin{mypara}\label{para:inf-repn-B}
	Let $\ca{G}$ be a quotient of $G$ which is a $p$-adic analytic group. The universal Lie algebra homomorphism $\varphi_{\sen}|_{\ca{G}}:\scr{E}^*_B(1)\to \widehat{\overline{B}}[1/p]\otimes_{\bb{Q}_p}\lie(\ca{G})$ allows us to canonically extend Sen operators to certain infinite-dimensional representations as follows. 
	
	Let $V$ be a $\bb{Q}_p$-Banach space endowed with a $\bb{Q}_p$-linear action of $\ca{G}$ such that $\ca{G}$ preserves the norm of $V$ and induces a trivial action on $V^{\leq 1}/p^2V^{\leq 1}$, where $V^{\leq 1}$ is the closed ball of radius $1$ in $V$. Such a smallness condition implies that the infinitesimal action of $\ca{G}$ on $V^{\leq 1}$ is well-defined and given by the following formula for any $g\in \ca{G}$ and $x\in V^{\leq 1}$ (cf. \ref{para:small})
	\begin{align}\label{eq:para:inf-repn-B}
		\log(g)|_{V^{\leq 1}}(x)=\sum_{n=1}^\infty \frac{(-1)^{n-1}}{n}(g-1)^n(x).
	\end{align}
	It defines a Lie algebra homomorphism
	\begin{align}\label{eq:para:inf-repn-B-0}
		\varphi|_V:\lie(\ca{G})\longrightarrow \mrm{End}_{\bb{Q}_p}(V),
	\end{align}
	sending $\log_{\ca{G}}(g)$ to $\id_{\bb{Q}_p}\otimes\log(g)|_{V^{\leq 1}}$ for any $g\in \ca{G}$ (cf. \ref{lem:finite-dim-repn-analytic}, \ref{cor:analytic-derivative}). 
	
	We put $W^+=(\widehat{\overline{B}}\otimes_{\bb{Z}_p}V^{\leq 1})^\wedge$ and $W=W^+[1/p]$. They are naturally endowed with continuous group actions of $G$ with respect to the $p$-adic topology defined by $W^+$, cf. \ref{lem:system-repn}.(\ref{item:lem:system-repn-1}). Notice that the $\widehat{\overline{B}}$-linear endomorphism $\id_{\widehat{\overline{B}}}\otimes \log(g)|_{V^{\leq 1}}$ on $\widehat{\overline{B}}\otimes_{\bb{Z}_p}V^{\leq 1}$ extends to an endomorphism on $W^+$ by taking $p$-adic completion. Thus, we also obtain a Lie algebra homomorphism after inverting $p$,
	\begin{align}\label{eq:para:inf-repn-B-2}
		\widehat{\overline{B}}[\frac{1}{p}]\otimes_{\bb{Q}_p}\lie(\ca{G})\longrightarrow \mrm{End}_{\widehat{\overline{B}}[\frac{1}{p}]}(W).
	\end{align}
	We define $\varphi_{\sen}|_W:\scr{E}^*_B(1)\to\mrm{End}_{\widehat{\overline{B}}[\frac{1}{p}]}(W)$ to be the composition of 
	\begin{align}\label{eq:para:inf-repn-B-1}
		\xymatrix{
			\scr{E}^*_B(1)\ar[rr]^-{\varphi_{\sen}|_{\ca{G}}}&& \widehat{\overline{B}}[1/p]\otimes_{\bb{Q}_p}\lie(\ca{G})\ar[rr]^-{\eqref{eq:para:inf-repn-B-2}}&&\mrm{End}_{\widehat{\overline{B}}[\frac{1}{p}]}(W).
		}
	\end{align}
	It is a $G$-equivariant homomorphism of $\widehat{\overline{B}}[1/p]$-linear Lie algebras. We denote by $\Phi(W)$ its image, and by $\Phi^{\geo}(W)$ the image of $\ho_{B[1/p]}(\Omega^1_{Y_K/K}(-1),\widehat{\overline{B}}[1/p])$. It follows from the construction that any element of $\Phi(W)$ acts continuously on $W$ with respect to the $p$-adic topology defined by $W^+$.
\end{mypara}

\begin{mylem}\label{lem:inf-repn-B-basic}
	With the notation in {\rm\ref{para:inf-repn-B}}, for any object $V_0$ of $\repnpr(\ca{G},\bb{Q}_p)$ and any $\ca{G}$-equivariant continuous $\bb{Q}_p$-linear homomorphism $V_0\to V$, the Lie algebra action $\varphi_{\sen}|_W$ of $\scr{E}_B^*(1)$ on $W$ defined by \eqref{eq:para:inf-repn-B-1} is compatible with the canonical Lie algebra action $\varphi_{\sen}|_{W_0}$ of $\scr{E}_B^*(1)$ on $W_0$ defined in {\rm\ref{thm:sen-brinon-B}}, where $W_0=\widehat{\overline{B}}[1/p]\otimes_{\bb{Q}_p}V_0$ is the associated object of $\repnpr(G,\widehat{\overline{B}}[1/p])$. 
\end{mylem}
\begin{proof}
	We take a $\bb{Q}_p$-basis $e_1,\dots,e_r$ of $V_0$ whose images under $V_0\to V$ lie in $V^{\leq 1}$. We put $V_0^+=\bb{Z}_pe_1\oplus\cdots\oplus \bb{Z}_pe_r\subseteq V_0$. Let $\ca{G}_0$ be a sufficiently small open subgroup of $\ca{G}$ whose image under the continuous group homomorphism $\rho:\ca{G}\to \mrm{GL}_r(\bb{Q}_p)$ (induced by the $\ca{G}$-action on $V_0$) is contained in $\id+p^2\mrm{M}_r(\bb{Z}_p)$. Thus, the lattice $V_0^+=\bb{Z}_p^r\subseteq \bb{Q}_p^r=V_0$ is $\ca{G}_0$-stable, and $\ca{G}_0$ acts trivially on $V_0^+/p^2V_0^+$. Notice that the infinitesimal action of $\ca{G}_0$ on $V_0^+$ is also well-defined and given by the same formula as in \eqref{eq:para:inf-repn-B}. Then, the $\ca{G}_0$-equivariant homomorphism $V_0^+\to V^{\leq 1}$ guarantees that the map $\widehat{\overline{B}}[1/p]\otimes_{\bb{Q}_p}\lie(\ca{G})\to\mrm{End}_{\widehat{\overline{B}}[\frac{1}{p}]}(W_0)$ induced by the infinitesimal action of $\ca{G}$ on $V_0^+$ is compatible with \eqref{eq:para:inf-repn-B-2}. Thus, the conclusion follows from \ref{thm:sen-lie-lift-B} and \eqref{eq:para:inf-repn-B-1}.
\end{proof}

\begin{mythm}\label{thm:sen-operator-B-fix}
	Let $\ca{G}$ be a quotient of $G$ which is a $p$-adic analytic group, $V$ a $\bb{Q}_p$-Banach space endowed with a $\bb{Q}_p$-linear action of $\ca{G}$ satisfying the following conditions:
	\begin{enumerate}
		\renewcommand{\labelenumi}{{\rm(\theenumi)}}
		\item The $(\ca{G},\bb{Q}_p)$-finite part of $V$ is dense in $V$ (cf. {\rm\ref{defn:repn}}).\label{item:thm:sen-operator-B-fix-1}
		\item The $\ca{G}$-action preserves the norm on $V$, and induces a trivial action on $V^{\leq 1}/p^3V^{\leq 1}$, where $V^{\leq 1}$ is the closed ball of radius $1$ in $V$.\label{item:thm:sen-operator-B-fix-2}
	\end{enumerate}
	We set $W=(\widehat{\overline{B}}\otimes_{\bb{Z}_p}V^{\leq 1})^\wedge[1/p]$ endowed with the natural $\widehat{\overline{B}}[1/p]$-semi-linear action of $G$. Then, the canonical $\widehat{\overline{B}}[1/p]$-linear Lie subalgebra $\Phi^{\geo}_{\ca{G}}$ of $\widehat{\overline{B}}[1/p]\otimes_{\bb{Q}_p}\lie(\ca{G})$ defined in {\rm\ref{defn:sen-lie-lift-B}} acts trivially on $W^H$ via \eqref{eq:para:inf-repn-B-2}, where $H=\gal(\ca{L}_{\mrm{ur}}/\ca{L}_\infty)\subseteq G$.
\end{mythm}
\begin{proof}
	The strategy is to reduce to \ref{thm:sen-operator-fix-infty}. We need to construct a suitable directed system of $p^3\bb{Z}_p$-small objects in  $\repnpr(\ca{G},\widehat{\overline{B}})$ in order to reduce to the situation of \ref{thm:sen-operator-fix-infty}. By condition (\ref{item:thm:sen-operator-B-fix-1}), we fix a directed system $(V_\lambda)_{\lambda\in\Lambda}$ of finite dimensional $\ca{G}$-stable $\bb{Q}_p$-subspaces of $V$ such that $V_\infty=\bigcup_{\lambda\in\Lambda}V_\lambda$ is dense in $V$. Recall that for any $\lambda\in\Lambda\cup\{\infty\}$ and $n\in\bb{N}$, we have (cf. \ref{lem:p-adic-separated1})
	\begin{align}
		p^nV_\lambda^{\leq 1}=V_{\lambda}^{\leq p^{-n}}=V_\lambda\cap p^n V^{\leq 1}.
	\end{align}
	In particular, $V_\lambda^{\leq 1}/p^nV_\lambda^{\leq 1}\to V^{\leq 1}/p^nV^{\leq 1}$ is injective, and $V_\infty^{\leq 1}/p^nV_\infty^{\leq 1}\to V^{\leq 1}/p^nV^{\leq 1}$ is an isomorphism as $V_\infty$ is dense in $V$. Thus, $V^{\leq 1}$ identifies naturally with the $p$-adic completion of $V_\infty^{\leq 1}$, and the $\ca{G}$-action on $V_\lambda^{\leq 1}/p^3V_\lambda^{\leq 1}\subseteq V^{\leq 1}/p^3V^{\leq 1}$ is also trivial by condition (\ref{item:thm:sen-operator-B-fix-2}). In conclusion, $(V_\lambda^{\leq 1})_{\lambda\in\Lambda}$ forms a directed system of $p^3\bb{Z}_p$-small objects in  $\repnpr(\ca{G},\bb{Z}_p)$ (cf. \ref{lem:banach-sub}, \ref{defn:small}) whose transition maps are injective and injective modulo $p^n$ for any $n\in\bb{N}$. 
	
	For any $\lambda\in\Lambda\cup\{\infty\}$, we put $W_\lambda^+=\widehat{\overline{B}}\otimes_{\bb{Z}_p}V_\lambda^{\leq 1}$. As $\widehat{\overline{B}}$ is flat over $\bb{Z}_p$, $(W_\lambda^+)_{\lambda\in\Lambda}$ forms a directed system of $p^3\bb{Z}_p$-small objects in  $\repnpr(\ca{G},\widehat{\overline{B}})$ (cf. \ref{defn:small}) whose transition maps are injective and injective modulo $p^n$ for any $n\in\bb{N}$. We put $W_\lambda=W_\lambda^+[1/p]$ and $\widehat{W_\infty}=\widehat{W_\infty^+}[1/p]$, where $\widehat{W_\infty^+}$ is the $p$-adic completion of $W_\infty^+$. Remark that we have $G$-equivariant natural identifications
	\begin{align}
		W^+=(\widehat{\overline{B}}\otimes_{\bb{Z}_p}V^{\leq 1})^\wedge=\widehat{W_\infty^+}.
	\end{align}
	Thus, $W=W^+[1/p]=\widehat{W_\infty}$.
	
	For any $\ak{q}\in\ak{S}_p(\overline{B})$ with image $\ak{p}\in\ak{S}_p(B)$, consider the element $(B_{\triv},B,\overline{B})\to (E_{\ak{p}},\ca{O}_{E_{\ak{p}}},\ca{O}_{\overline{E}_{\ak{q}}})$ of $\ak{E}(B)$ defined in \ref{para:notation-A-inj}. We set $G_{\ak{q}}=\gal(\overline{E}_{\ak{q}}/E_{\ak{p}})$ and $H_{\ak{q}}=\gal(\overline{E}_{\ak{q}}/E_{\ak{p},\infty})$. Consider the natural commutative diagram
	\begin{align}
		\xymatrix{
			\repnpr(\ca{G},\bb{Z}_p)\ar[r]\ar[d]& \repnpr(G,\widehat{\overline{B}})\ar[r]\ar[d]&\repnpr(G_{\ak{q}},\ca{O}_{\widehat{\overline{E}}_{\ak{q}}})\ar[d]\\
			\repnpr(\ca{G},\bb{Q}_p)\ar[r]&\repnpr(G,\widehat{\overline{B}}[\frac{1}{p}])\ar[r]&\repnpr(G_{\ak{q}},\widehat{\overline{E}}_{\ak{q}})
		}
	\end{align}
	Let $W_{\lambda,\ak{q}}^+=\ca{O}_{\widehat{\overline{E}}_{\ak{q}}}\otimes_{\widehat{\overline{B}}} W_\lambda^+=\ca{O}_{\widehat{\overline{E}}_{\ak{q}}}\otimes_{\bb{Z}_p} V_\lambda^{\leq 1}$ and $W_{\lambda,\ak{q}}=W_{\lambda,\ak{q}}^+[1/p]$. Since $\ca{O}_{\widehat{\overline{E}}_{\ak{q}}}$ is flat over $\bb{Z}_p$, $(W_{\lambda,\ak{q}}^+)_{\lambda\in\Lambda}$ forms a directed system of $p^3\bb{Z}_p$-small objects in $\repnpr(G_{\ak{q}},\ca{O}_{\widehat{\overline{E}}_{\ak{q}}})$ whose transition maps are injective and injective modulo $p^n$ for any $n\in\bb{N}$. Thus, we are in the situation of \ref{thm:sen-operator-fix-infty}. Taking colimit of the diagram \eqref{diam:thm:sen-brinon-B}, we get a canonical commutative diagram
	\begin{align}
		\xymatrix{
			\scr{E}^*_B(1)\ar[d]_-{\varphi_{\sen}|_{W_\infty}}\ar[r]&\prod_{\ak{q}}\scr{E}^*_{\ca{O}_{E_{\ak{p}}}}(1)\ar[d]^-{(\varphi_{\sen}|_{W_{\infty,\ak{q}}})_{\ak{q}}}\\
			\mrm{End}_{\widehat{\overline{B}}[\frac{1}{p}]}(W_\infty)\ar[r]^-{(\iota_{\ak{q}})_{\ak{q}}}&\prod_{\ak{q}}\mrm{End}_{\widehat{\overline{E}}_{\ak{q}}}(W_{\infty,\ak{q}})
		}
	\end{align}
	where the product is taken over $\ak{q}\in\ak{S}_p(\overline{B})$, $\iota_{\ak{q}}$ is defined by extending scalars, and the horizontal arrows are injective. Notice that the Lie algebra action $\varphi_{\sen}|_W$ defined by \eqref{eq:para:inf-repn-B-1} is the unique continuation of $\varphi_{\sen}|_{W_\infty}$ by \ref{lem:inf-repn-B-basic}. The variant of \ref{lem:inf-repn-B-basic} for valuation ring case shows that $\varphi_{\sen}|_{W_{\infty,\ak{q}}}$ extends uniquely to $\varphi_{\sen}|_{\widehat{W_{\infty,\ak{q}}}}$ by continuation, whose geometric part thus coincides with that defined in \ref{thm:sen-operator-fix-infty}. In particular, $\iota_{\ak{q}}(\Phi(\widehat{W_\infty}))\subset\Phi(\widehat{W_{\infty,\ak{q}}})$ (resp. $\iota_{\ak{q}}(\Phi^{\geo}(\widehat{W_\infty}))\subset\Phi^{\geo}(\widehat{W_{\infty,\ak{q}}})$). Taking filtered colimit on $\lambda\in\Lambda$ and then inverse limit on $n\in\bb{N}$ of the natural injection $W_\lambda^+/p^n\to \prod_{\ak{q}}W_{\lambda,\ak{q}}^+/p^n$ given by \ref{prop:ht1-prime-inj}, we obtain a natural injection
	\begin{align}
		\widehat{W_\infty}\longrightarrow \prod_{\ak{q}}\widehat{W_{\infty,\ak{q}}},
	\end{align}
	which particularly sends $(\widehat{W_\infty})^G$ (resp. $(\widehat{W_\infty})^H$) into $\prod_{\ak{q}}(\widehat{W_{\infty,\ak{q}}})^{G_{\ak{q}}}$ (resp. $\prod_{\ak{q}}(\widehat{W_{\infty,\ak{q}}})^{H_{\ak{q}}}$). Since $\Phi^{\geo}(\widehat{W_{\infty,\ak{q}}})$ acts trivially on $(\widehat{W_{\infty,\ak{q}}})^{H_{\ak{q}}}$ for each $\ak{q}$ by \ref{thm:sen-operator-fix-infty}, we see that $\Phi^{\geo}(\widehat{W_\infty})$ acts trivially on $(\widehat{W_\infty})^H$.
\end{proof}

\section{Application to Locally Analytic Vectors}\label{sec:analytic}
This section is devoted to generalizing a result of Pan \cite[3.1.2]{pan2021locally} to higher dimension.

\begin{mypara}\label{para:analytic-function}
	 We mainly follow \cite[2.1]{pan2021locally} to briefly review the notion of locally analytic vectors. Let $M$ be a finite free $\bb{Z}_p$-module with a basis $\mbf{e}_1,\dots,\mbf{e}_n$, $V$ a $\bb{Q}_p$-Banach space with norm $|\ |$. We say that a map $f:M\to V$ is (strictly) \emph{analytic} (cf. \cite[Definition 6.17]{ddms1999lie}) if there exists $v_{\underline{m}}\in V$ for any $\underline{m}=(m_1,\dots,m_d)\in\bb{N}^d$ such that ($|\underline{m}|=m_1+\cdots +m_d$)
	\begin{align}
		\lim_{|\underline{m}|\to \infty}|v_{\underline{m}}|=0,
	\end{align}
	and for any $a_1,\dots,a_d\in\bb{Z}_p$, we have
	\begin{align}
		f(a_1\mbf{e}_1+\cdots+a_d\mbf{e}_d)=\sum_{\underline{m}\in\bb{N}^d}a_1^{m_1}\cdots a_d^{m_d}v_{\underline{m}}.
	\end{align}
	Notice that $f=0$ if and only if $v_{\underline{m}}$ is zero for any $\underline{m}\in\bb{N}^d$. If $\mbf{e}_1',\dots,\mbf{e}_d'$ form another basis of $M$, then we can write $f(a_1\mbf{e}_1'+\cdots+a_d\mbf{e}_d')=\sum_{\underline{m}\in\bb{N}^d}a_1^{m_1}\cdots a_d^{m_d}v_{\underline{m}}'$, where $v_{\underline{m}}'$ is a $\bb{Z}_p$-linear combination of finitely many $v_{\underline{m}}$. In particular, we get $\sup_{\underline{m}\in\bb{N}^d}|v_{\underline{m}}'|\leq \sup_{\underline{m}\in\bb{N}^d}|v_{\underline{m}}|$, and thus they are actually equal by symmetry. In conclusion, the definition on the analyticity of $f$ does not depend on the choice of the basis of $M$. We denote by $\scr{C}^{\mrm{an}}(M,V)$ the $\bb{Q}_p$-vector space of $V$-valued analytic functions on $M$. We note that after fixing a basis of $M$, taking the coefficients $v_{\underline{m}}$ identifies $\scr{C}^{\mrm{an}}(M,V)$ with a subspace of $\prod_{{\underline{m}\in\bb{N}^d}}V$. We define a norm on $\scr{C}^{\mrm{an}}(M,V)$ by setting
	\begin{align}
		|f|=\sup_{\underline{m}\in\bb{N}^d}|v_{\underline{m}}|\in\bb{R}_{\geq 0},
	\end{align}	
	which does not depend on the choice of the basis of $M$ and makes $\scr{C}^{\mrm{an}}(M,V)$ a $\bb{Q}_p$-Banach space.
\end{mypara}

\begin{mylem}\label{lem:analytic-function-basic}
	Let $M$ be a finite free $\bb{Z}_p$-module, $V$ a $\bb{Q}_p$-Banach space. Then, for any $n\in\bb{N}$, the natural map
	\begin{align}
		(V^{\leq 1}\otimes_{\bb{Z}_p} \scr{C}^{\mrm{an}}(M,\bb{Q}_p)^{\leq 1})/p^n\longrightarrow \scr{C}^{\mrm{an}}(M,V)^{\leq 1}/p^n.
	\end{align}
	is an isomorphism.
\end{mylem}
\begin{proof}
	After fixing a basis of $M$, by taking coefficients we identify $\scr{C}^{\mrm{an}}(M,V)$ with
	\begin{align}
		\{(v_{\underline{m}})\in \prod_{{\underline{m}\in\bb{N}^d}}V\ |\ \lim_{|\underline{m}|\to \infty}|v_{\underline{m}}|=0\}.
	\end{align}
	Notice that $\scr{C}^{\mrm{an}}(M,V)^{\leq 1}$ identifies with
	\begin{align}
		\{(v_{\underline{m}})\in \prod_{{\underline{m}\in\bb{N}^d}}V^{\leq 1}\ |\ \lim_{|\underline{m}|\to \infty}|v_{\underline{m}}|=0\}=(\oplus_{{\underline{m}\in\bb{N}^d}}V^{\leq 1})^\wedge,
	\end{align}
	where the completion is $p$-adic. Taking $V=\bb{Q}_p$, we get $\scr{C}^{\mrm{an}}(M,\bb{Q}_p)^{\leq 1}=(\oplus_{{\underline{m}\in\bb{N}^d}}\bb{Z}_p)^\wedge$. Thus, $(V^{\leq 1}\otimes_{\bb{Z}_p} \scr{C}^{\mrm{an}}(M,\bb{Q}_p)^{\leq 1})/p^n=\oplus_{{\underline{m}\in\bb{N}^d}}V^{\leq 1}/p^n=\scr{C}^{\mrm{an}}(M,V)^{\leq 1}/p^n$, which completes the proof.
\end{proof}

\begin{mylem}\label{lem:analytic-function-pullback}
	Let $V$ be a $\bb{Q}_p$-Banach space, $g:M\to N$ a $\bb{Z}_p$-linear homomorphism of finite free $\bb{Z}_p$-modules. Then, the pullback of functions induces a map 
	\begin{align}
		g^*:\scr{C}^{\mrm{an}}(N,V)\longrightarrow \scr{C}^{\mrm{an}}(M,V),\ f\mapsto f\circ g
	\end{align}
	which decreases the norm. Moreover, if the cokernel of $g$ is finite, then $g^*$ is injective; if $g$ is injective, then $g^*$ has dense image.
\end{mylem}
\begin{proof}
	As any submodule of a finite free $\bb{Z}_p$-module is still finite free, we may decompose $g$ as an injection composed with a surjection, so that we can treat the two cases separately. 
	
	Assume firstly that $g$ is surjective. We write $M=N\oplus L$, and choose a basis $\mbf{e}_1,\dots,\mbf{e}_d$ for $M$ such that $\mbf{e}_1,\dots,\mbf{e}_c$ form a basis of $N$, where $0\leq c\leq d$ are integers. We see that an analytic function $f(a_1\mbf{e}_1+\cdots+a_c\mbf{e}_c)=\sum_{\underline{m}\in\bb{N}^c}a_1^{m_1}\cdots a_c^{m_c}v_{\underline{m}}$ on $N$ is pulled back to an analytic function $(f\circ g)(a_1\mbf{e}_1+\cdots+a_d\mbf{e}_d)=\sum_{\underline{m}\in\bb{N}^c\times\{0\}^{d-c}}a_1^{m_1}\cdots a_d^{m_d}v_{\underline{m}}$. Thus, $g^*$ is injective and preserves norms.
	
	Assume that $g$ is injective. Since $N/M$ is a direct sum of a finite free $\bb{Z}_p$-module with a finite $\bb{Z}_p$-module, by writing $g$ as a composition of two injective maps, we can treat separately the case where $g$ admits a retraction and the case where the cokernel of $g$ is finite. For the first case, we can write $N=M\oplus L$. By an argument as before, we see that $g^*$ is surjective and decreases the norm. For the second case, by expressing a basis of $M$ as $\bb{Z}_p$-linear combination of that of $N$, we see that $g^*$ decreases the norm as in \ref{para:analytic-function}. Conversely, as $N[1/p]=M[1/p]$, we can express a basis of $N$ as $\bb{Q}_p$-linear combination of that of $M$. Thus, any analytic function on $N$ with only finitely many coefficients $v_{\underline{m}}$ non-zero is a restriction of such an analytic function on $M$. This shows that $g^*$ has dense image.
\end{proof}

\begin{mydefn}\label{defn:analytic-function-lie}
	Let $\ca{G}$ be a uniform pro-$p$ group, $L_\ca{G}$ its corresponding powerful Lie algebra over $\bb{Z}_p$ with identity map $\exp:L_{\ca{G}}\to\ca{G}$ (see \ref{para:L_G}), $V$ a $\bb{Q}_p$-Banach space. We say that a $V$-valued function $f:\ca{G}\to V$ is \emph{analytic} if $f\circ\exp:L_\ca{G}\to V$ is analytic (as $L_\ca{G}$ is a finite free $\bb{Z}_p$-module). We set $\scr{C}^{\mrm{an}}(\ca{G},V)=\scr{C}^{\mrm{an}}(L_\ca{G},V)$ and denote its norm by $|\ |_\ca{G}$.
\end{mydefn}
	
\begin{myrem}\label{rem:analytic-function-lie}
	 One can also use a system of coordinates of the second kind \eqref{eq:coordinates-second} to define analyticity of a function as in \cite[2.1.1]{pan2021locally}. Notice that the transition map between the coordinates of the first kind and the second kind is a homeomorphism $\psi:\bb{Z}_p^d\to \bb{Z}_p^d$ such that $\psi$ and $\psi^{-1}$ are both analytic by \cite[Exercise 8.3]{ddms1999lie} (or by \cite[34.1]{schneider2011lie}, as $\ca{G}$ is $p$-saturable by \cite[Notes on page 81, Theorem 7.7, Exercise 7.10]{ddms1999lie}). Thus, the two definitions of $\scr{C}^{\mrm{an}}(\ca{G},V)$ and its norm coincide (cf. \cite[Theorem 6.35]{ddms1999lie}).
\end{myrem}

\begin{mypara}\label{para:three-action}
	Let $V$ be a $\bb{Q}_p$-vector space endowed with a $\bb{Q}_p$-linear action of a group $\ca{G}$, $\scr{C}(\ca{G},V)=\prod_{\ca{G}}V$ the $\bb{Q}_p$-vector space of $V$-valued functions on $\ca{G}$. We mainly consider three $\ca{G}$-actions $\{1\otimes \rho_{\dl}, 1\otimes \rho_{\rr},\rho_V\otimes \rho_{\dl}\}$ on $\scr{C}(\ca{G},V)$, called the left translation action, right translation action and diagonal action respectively, defined as follows: for any $g,g'\in \ca{G}$ and $f\in \scr{C}(\ca{G},V)$,
	\begin{align}\label{eq:para:three-action}
		(\rho(g)f)(g')=\left\{ \begin{array}{ll}
			f(g^{-1}g') & \textrm{if }\rho=1\otimes \rho_{\dl},\\
			f(g'g) & \textrm{if }\rho=1\otimes \rho_{\rr},\\
			g(f(g^{-1}g')) & \textrm{if }\rho=\rho_V\otimes \rho_{\dl}.
		\end{array} \right.
	\end{align}
\end{mypara}

\begin{mylem}[{\cite[2.1.2]{pan2021locally}}]\label{lem:group-action-analytic}
	Let $\ca{G}$ be a uniform pro-$p$ group, $V$ a $\bb{Q}_p$-Banach space. Then, $\scr{C}^{\mrm{an}}(\ca{G},V)$ is stable under the left and right translation actions of $\ca{G}$ which preserve the norm $|\ |_{\ca{G}}$. Moreover, the left and right translation actions of $\ca{G}^{p^n}$ on $\scr{C}^{\mrm{an}}(\ca{G},V)^{\leq 1}/p^n\scr{C}^{\mrm{an}}(\ca{G},V)^{\leq 1}$ are trivial for any $n\in\bb{N}$.
\end{mylem}
\begin{proof}
	For any element $g\in \ca{G}$, let $n$ be the maximal integer such that $g\in\ca{G}^{p^n}$ and we set $g_1=g^{p^{-n}}\in\ca{G}\setminus \ca{G}^p$ (as $\ca{G}\to \ca{G}^{p^n}$ sending $x$ to $x^{p^n}$ is a homeomorphism, cf. \ref{para:L_G}). Then, we take $g_2,\dots,g_d\in\ca{G}$ such that $(g_1,\dots,g_d)$ forms a minimal topological generating set of $\ca{G}$ (i.e. it forms a $\bb{Z}_p$-basis of $L_{\ca{G}}$, cf. \ref{para:L_G}). We obtain a system of coordinates of second kind \eqref{eq:coordinates-second},
	\begin{align}
		\bb{Z}_p^d&\longrightarrow \ca{G},\ (a_1,\dots,a_d)\mapsto g_1^{a_1}\cdots g_d^{a_d}.
	\end{align}
	The left translation by $g=g_1^{p^n}$ sends the coordinate $(a_1,a_2,\dots,a_d)$ to $(a_1-p^n,a_2,\dots,a_d)$. Thus, it preserves the analyticity of a $V$-valued function on $\ca{G}$ by \ref{rem:analytic-function-lie} as well as the norm $|\ |_{\ca{G}}$, and acts trivially on $\scr{C}^{\mrm{an}}(\ca{G},V)^{\leq 1}/p^n\scr{C}^{\mrm{an}}(\ca{G},V)^{\leq 1}$. Thus, we obtain the conclusion for the left translation. The proof for the right translation is analogous. 
\end{proof}

\begin{myprop}[{\cite[2.1.3]{pan2021locally}}]\label{prop:analytic-approx}
	Let $\ca{G}$ be a uniform pro-$p$ group isomorphic to a closed subgroup of $\id+p^2\mrm{M}_d(\bb{Z}_p)$ for some $d\in\bb{N}$. Then, there exists a directed system of finite-dimensional $\bb{Q}_p$-subspaces $(V_k)_{k\geq 1}$ of $\scr{C}^{\mrm{an}}(\ca{G},\bb{Q}_p)$ stable under the left and right translation action of $\ca{G}$ such that $\bigcup_{k\geq 1}V_k$ is dense in $\scr{C}^{\mrm{an}}(\ca{G},\bb{Q}_p)$.
\end{myprop}
\begin{proof}
	We follow the proof of \cite[2.1.3]{pan2021locally}. Consider the commutative diagram (cf. \ref{exmp:gl_n})
	\begin{align}
		\xymatrix{
			\ca{G}\ar[d]_-{\log_{\ca{G}}}\ar[r]&\id+p^2\mrm{M}_d(\bb{Z}_p)\ar@<0.3pc>[d]^-{\log}\\
			L_{\ca{G}}\ar[r]^-{\iota}&p^2\mrm{M}_d(\bb{Z}_p)\ar@<0.3pc>[u]^-{\exp}
		}
	\end{align}
	We remark that the pullback by the injection $\iota$ induces a map $\scr{C}^{\mrm{an}}(p^2\mrm{M}_d(\bb{Z}_p),\bb{Q}_p)\to \scr{C}^{\mrm{an}}(L_{\ca{G}},\bb{Q}_p)$ with dense image by \ref{lem:analytic-function-pullback}. For any $1\leq i,j\leq d$, let $X_{ij}:\mrm{M}_d(\bb{Z}_p)\to\bb{Z}_p$ be the map taking the value of the $(i,j)$-component. Let $W_k$ (resp. $V_k$) be the space of $\bb{Q}_p$-valued functions on $\id+p^2\mrm{M}_d(\bb{Z}_p)$ (resp. $\ca{G}$) of the form $P(X_{ij}|1\leq i,j\leq d)$ where $P$ is a polynomial with coefficients in $\bb{Q}_p$ of degree $\leq k$. Since $f\circ\exp\in \scr{C}^{\mrm{an}}(p^2\mrm{M}_d(\bb{Z}_p),\bb{Q}_p)$ for any $f\in W_k$, we have $W_k\subseteq \scr{C}^{\mrm{an}}(\id+p^2\mrm{M}_d(\bb{Z}_p),\bb{Q}_p)$. Moreover, since $\log\circ \exp$ is the identity on $p^2\mrm{M}_d(\bb{Z}_p)$, the set $\{f\circ \exp\ |\ f\in W_k,\ k\in \bb{N}\}$ is dense in $\scr{C}^{\mrm{an}}(p^2\mrm{M}_d(\bb{Z}_p),\bb{Q}_p)$. Therefore, $\bigcup_{k\in\bb{N}}W_k$ is dense in $\scr{C}^{\mrm{an}}(\id+p^2\mrm{M}_d(\bb{Z}_p),\bb{Q}_p)$. By pulling back, we see that $\bigcup_{k\in\bb{N}}V_k$ is dense in $\scr{C}^{\mrm{an}}(\ca{G},\bb{Q}_p)$ by \ref{lem:analytic-function-pullback}. Moreover, it follows from the construction that each $V_k$ is stable under the left and right translation action of $\ca{G}$. 
\end{proof}

\begin{mydefn}[{\cite[2.1.4]{pan2021locally}}]\label{defn:analytic-vector}
	Let $\ca{G}$ be a uniform pro-$p$ group, $V$ a $\bb{Q}_p$-Banach space endowed with a $\bb{Q}_p$-linear continuous action of $\ca{G}$. An element $v\in V$ is called \emph{$\ca{G}$-analytic} if the function $f_v:\ca{G}\to V$ sending $g$ to $gv$ is analytic. We denote by $V^{\ca{G}\trm{-}\mrm{an}}$ the subset of $V$ consisting of $\ca{G}$-analytic vectors.
\end{mydefn}
	We remark that $V^{\ca{G}\trm{-}\mrm{an}}$ is stable under the action of $\ca{G}$ on $V$, since the right translation of the analytic function $f_v$ is still analytic by \ref{lem:group-action-analytic}. For any continuous group homomorphism of uniform pro-$p$ groups $\ca{G}'\to \ca{G}$, regarding $V$ also as a $\bb{Q}_p$-representation of $\ca{G}'$, then we have $V^{\ca{G}\trm{-}\mrm{an}}\subseteq V^{\ca{G}'\trm{-}\mrm{an}}$ by \ref{lem:analytic-function-pullback}. The subset $V^{\ca{G}\trm{-}\mrm{an}}$ is actually a $\bb{Q}_p$-subspace of $V$ by the following lemma.
	
\begin{mylem}[{\cite[2.1.5]{pan2021locally}}]\label{lem:analytic-vector-pre}
	Let $\ca{G}$ be a uniform pro-$p$ group, $V$ a $\bb{Q}_p$-Banach space endowed with a $\bb{Q}_p$-linear continuous action of $\ca{G}$. Then, the evaluation map at $1\in \ca{G}$ induces a bijection
	\begin{align}\label{eq:analytic-vector-pre}
		\scr{C}^{\mrm{an}}(\ca{G},V)^{\rho_V\otimes \rho_{\dl}=1}\iso V^{\ca{G}\trm{-an}},
	\end{align}
	where $\rho_V\otimes\rho_{\dl}:\scr{C}^{\mrm{an}}(\ca{G},V)\to \scr{C}(\ca{G},V)$ is the diagonal action (see {\rm\ref{para:three-action}}). Moreover, the inverse of this bijection induces a $\ca{G}$-equivariant inclusion
	\begin{align}\label{eq:analytic-vector-pre-2}
		V^{\ca{G}\trm{-an}}\longrightarrow (\scr{C}^{\mrm{an}}(\ca{G},V),1\otimes\rho_{\rr}),\ v\mapsto (f_v:g\mapsto gv),
	\end{align}
	where $1\otimes\rho_{\rr}$ is the right translation of $\ca{G}$ on $\scr{C}^{\mrm{an}}(\ca{G},V)$ (see {\rm\ref{lem:group-action-analytic}}).
\end{mylem}
\begin{proof}
	Notice that an element $f\in \scr{C}(\ca{G},V)$ is fixed by $\rho_V\otimes \rho_{\dl}$ if and only if $f(g)=g(f(1))$ for any $g\in\ca{G}$. Thus, in this case $f$ is analytic if and only if $f(1)\in V$ is $\ca{G}$-analytic by definition, so that we have the bijection \eqref{eq:analytic-vector-pre}. It follows from the definition that \eqref{eq:analytic-vector-pre-2} is $\ca{G}$-equivariant.
\end{proof}

	The smallness condition will give us enough analytic vectors.
	
\begin{mylem}[{\cite[2.1.9]{pan2021locally}}]\label{lem:finite-dim-repn-analytic}
	Let $\ca{G}$ be a uniform pro-$p$ group, $V$ a $\bb{Q}_p$-Banach space endowed with a $\bb{Q}_p$-linear action of $\ca{G}$. Assume that $\ca{G}$ preserves the norm of $V$ and induces a trivial action on $V^{\leq 1}/p^2V^{\leq 1}$. Then, $V=V^{\ca{G}\trm{-}\mrm{an}}$. In particular, for any object $V_0$ of $\repnpr(\ca{G},\bb{Q}_p)$, there exists a uniform pro-$p$ open subgroup $\ca{G}_0$ of $\ca{G}$ such $V_0=V_0^{\ca{G}_0\trm{-}\mrm{an}}$.
\end{mylem}
\begin{proof}
	For any $g\in \ca{G}$, as $(g-1)(V^{\leq 1})\subseteq p^2V^{\leq 1}$, we have $g=\exp(\log(g))$ by \ref{para:small} where $\exp$ and $\log$ are defined in loc.cit. We take a minimal topological generating set $(g_1,\dots,g_d)$ of $\ca{G}$ (cf. \ref{para:L_G}) so that we obtain a system of coordinates of second kind \eqref{eq:coordinates-second},
	\begin{align}
		\bb{Z}_p^d&\longrightarrow \ca{G},\ (a_1,\dots,a_d)\mapsto g_1^{a_1}\cdots g_d^{a_d}.
	\end{align}
	For any $v\in V$, we have
	\begin{align}
		g_1^{a_1}\cdots g_d^{a_d}(v)=&\exp(a_1\log(g_1))\cdots\exp(a_d\log(g_d))(v)\quad\trm{(cf. \ref{para:small})}\\
		=&\left(\sum_{n=0}^{\infty}\frac{a_1^n}{n!}\log(g_1)^n\right)\cdots\left(\sum_{n=0}^{\infty}\frac{a_d^n}{n!}\log(g_d)^n\right)(v),\nonumber
	\end{align}
	which is clearly an analytic $V$-valued function on variables $(a_1,\dots,a_d)\in \bb{Z}_p^d$ (as $\log(g_i)\in p^2\mrm{End}_{\bb{Z}_p}(V^{\leq 1})$). Thus, the function $f_v:\ca{G}\to V$ sending $g$ to $gv$ is analytic by \ref{rem:analytic-function-lie}, i.e. $v$ is $\ca{G}$-analytic.
	
	For the ``in particular'' part, we fix a $\bb{Q}_p$-basis of $V_0$ so that we obtain a continuous group homomorphism $\rho:\ca{G}\to \mrm{GL}_d(\bb{Q}_p)$. Let $\ca{G}_0$ be a sufficiently small uniform pro-$p$ open subgroup of $\ca{G}$ whose image under $\rho$ is contained in $\id+p^2\mrm{M}_d(\bb{Z}_p)$. Thus, the lattice $V_0^+=\bb{Z}_p^d\subseteq \bb{Q}_p^d=V_0$ is $\ca{G}_0$-stable, and $\ca{G}_0$ acts trivially on $V_0^+/p^2V_0^+$. By the assertion we just proved above, we see that $V_0=V_0^{\ca{G}_0\trm{-}\mrm{an}}$.
\end{proof}

	Moreover, one can get a slightly stronger result for the analytic vectors in $\scr{C}^{\mrm{an}}(\ca{G},V)$.

\begin{mylem}\label{lem:analytic-function-space-an}
	Let $\ca{G}$ be a uniform pro-$p$ group, $V$ a $\bb{Q}_p$-Banach space. Endowing the $\bb{Q}_p$-Banach space $\scr{C}^{\mrm{an}}(\ca{G},V)$ with the left or right translation action of $\ca{G}$, then $\scr{C}^{\mrm{an}}(\ca{G},V)=\scr{C}^{\mrm{an}}(\ca{G},V)^{\ca{G}\trm{-}\mrm{an}}$.
\end{mylem}
\begin{proof}
	Firstly, the action of $\ca{G}$ on $\scr{C}^{\mrm{an}}(\ca{G},V)$ is well-defined and continuous by \ref{lem:group-action-analytic}. Thus, we can talk about $\ca{G}$-analytic vectors in $\scr{C}^{\mrm{an}}(\ca{G},V)$ by \ref{defn:analytic-vector}. We fix a $\bb{Z}_p$-basis of $L_{\ca{G}}$. For any $f\in \scr{C}^{\mrm{an}}(L_\ca{G},V)$, consider the function $\psi:L_\ca{G}\times L_\ca{G}\to V$ sending $(x,y)$ to $f(x*y)$, where the operation $x*y$ is given by the Baker-Campbell-Hausdorff formula \eqref{eq:para:L-*} (which defines the multiplication in $\ca{G}$). Notice that $\psi$ is analytic by \cite[Lemma 9.12]{ddms1999lie}, i.e. there exists a unique element $v_{\underline{k},\underline{l}}\in V$ for any $\underline{k},\underline{l}\in\bb{N}^d$ such that $|v_{\underline{k},\underline{l}}|\to 0$ when $|\underline{k}|+|\underline{l}|\to \infty$, and that for any $x=(x_1,\dots,x_d),y=(y_1,\dots,y_d)\in\bb{Z}_p^d$, we have
	\begin{align}\label{eq:lem:group-action-analytic}
		f(x*y)=\sum_{\underline{k},\underline{l}\in\bb{N}^d} x_1^{k_1}\cdots x_d^{k_d}y_1^{l_1}\cdots y_d^{l_d} v_{\underline{k},\underline{l}}.
	\end{align}
	In particular, the function $f_{\underline{k}}$ sending $y$ to $\sum_{\underline{l}\in\bb{N}^d}y_1^{l_1}\cdots y_d^{l_d} v_{\underline{k},\underline{l}}$ is analytic, and we have $|f_{\underline{k}}|\to 0$ when $|\underline{k}|\to \infty$. Thus, the function on $L_{\ca{G}}$ sending $x$ to ${}_xf=\sum_{\underline{k}\in\bb{N}^d} x_1^{k_1}\cdots x_d^{k_d}f_{\underline{k}}$ is analytic, which shows that $f$ is a $\ca{G}$-analytic vector of $\scr{C}^{\mrm{an}}(\ca{G},V)$ with respect to the left translation. The proof for right translation is analogous.
\end{proof}

The continuity condition of the $\ca{G}$-action on $V$ in definition \ref{defn:analytic-vector} is used to define infinitesimal actions as follows.

\begin{myprop}\label{prop:analytic-derivative}
	Let $\ca{G}$ be a uniform pro-$p$ group, $V$ a $\bb{Q}_p$-Banach space endowed with a $\bb{Q}_p$-linear continuous action of $\ca{G}$. Then, for any $g\in G$, there exists a unique $\bb{Q}_p$-linear endomorphism $\varphi_g$ of $V^{\ca{G}\trm{-}\mrm{an}}$ defined by
	\begin{align}\label{eq:prop:analytic-derivative-defn}
		\varphi_g(v)=\lim_{\bb{Z}_p\setminus\{0\}\ni a\to 0} a^{-1}(g^a-1)(v),\quad \forall v\in V^{\ca{G}\trm{-}\mrm{an}},
	\end{align}
	such that for any $a\in\bb{Z}_p$,
	\begin{align}\label{eq:prop:analytic-derivative}
		g^a(x)=\exp(a\varphi_g)(v)=\sum_{k=0}^\infty \frac{a^k}{k!}(\underbrace{\varphi_g\circ\cdots \circ\varphi_g}_{k\trm{ copies}})(v).
	\end{align}
\end{myprop}
\begin{proof}
	Let $v\in V^{\ca{G}\trm{-}\mrm{an}}$ and $g\in \ca{G}$. As $\bb{Z}_p\to\ca{G}$ sending $a$ to $g^a$ is a continuous homomorphism of uniform pro-$p$ groups, $v$ is also $\bb{Z}_p$-analytic, i.e. there exists a unique element $v_k\in V$ for any $k\in\bb{N}$ such that $|v_k|\to 0$ when $k\to \infty$, and that for any $a\in \bb{Z}_p$,
	\begin{align}
		g^av=\sum_{k=0}^\infty a^kv_k.
	\end{align}
	We see easily that $v_0=v$ by taking $a=0$ and that
	\begin{align}
		v_1=\lim_{\bb{Z}_p\setminus\{0\}\ni a\to 0} a^{-1}(g^a-1)(v).
	\end{align}
	Thus, $\varphi_g(v)=v_1$ is a well-defined element of $V$. We claim that $v_1\in V^{\ca{G}\trm{-}\mrm{an}}$ (so that $\varphi_g$ is a $\bb{Q}_p$-linear endomorphism of $V^{\ca{G}\trm{-}\mrm{an}}$) and $v_k=\varphi_g^k(v)/k!$ for any $k\geq 0$ (so that \eqref{eq:prop:analytic-derivative} follows). Indeed, for any $b\in\bb{Z}_p$, we have
	\begin{align}
		g^{a+b}v=g^b\left(\sum_{k=0}^\infty a^kv_k\right)=\sum_{k=0}^\infty a^kg^bv_k
	\end{align}
	where the second equality follows from the continuity of the $\ca{G}$-action on $V$; on the other hand,
	 \begin{align}
	 	g^{a+b}v=\sum_{k=0}^\infty (a+b)^kv_k=\sum_{k=0}^\infty a^k\left(\sum_{l=0}^\infty \binom{k+l}{l}b^lv_{k+l}\right)
	 \end{align}
 	where the second equality follows from the absolute convergence condition $|v_k|\to 0$ when $k\to \infty$. Combining the two expressions for $g^{a+b}v$, we get
 	\begin{align}
 		g^bv_k=\sum_{l=0}^\infty \binom{k+l}{l}b^lv_{k+l},
 	\end{align}
 	which shows that $v_k\in V^{\ca{G}\trm{-}\mrm{an}}$ and $v_{k+1}=\varphi_g(v_k)/(k+1)$. The claim follows by induction.
\end{proof}

\begin{mylem}\label{lem:analytic-derivative}
	With the notation in {\rm\ref{prop:analytic-derivative}}, we endow $V^{\ca{G}\trm{-}\mrm{an}}$ with the norm induced from the norm $|\ |_{\ca{G}}$ on $\scr{C}^{\mrm{an}}(\ca{G},V)$ via the canonical injection \eqref{eq:analytic-vector-pre-2}. Then, the map
	\begin{align}
		\phi:\bb{Z}_p\times\ca{G}\times V^{\ca{G}\trm{-}\mrm{an}}\longrightarrow V^{\ca{G}\trm{-}\mrm{an}},
	\end{align}
	sending $(0,g,v)$ to $\varphi_g(v)$ and sending $(a,g,v)$ to $a^{-1}(g^av-v)$ for $a\neq 0$, is continuous.
\end{mylem}
\begin{proof}
	We fix a $\bb{Z}_p$-basis $u_1,\dots,u_d$ of $L_{\ca{G}}$. By the homeomorphism $\exp:L_{\ca{G}}\to\ca{G}$, it suffices to verify the continuity of the map
	\begin{align}
		\psi:\bb{Z}_p\times\bb{Z}_p^d\times V^{\ca{G}\trm{-}\mrm{an}}\longrightarrow V^{\ca{G}\trm{-}\mrm{an}}
	\end{align}
	sending $(a_0,a_1,\dots,a_d,v)$ to $\phi(a_0,\exp(\sum_{i=1}^da_iu_i),v)$. For any $v\in V^{\ca{G}\trm{-}\mrm{an}}$, there exists a unique element $v_{\underline{m}}\in V$ for any $\underline{m}=(m_1,\dots,m_d)\in\bb{N}^d$ such that $|v_{\underline{m}}|\to 0$ when $|\underline{m}|\to \infty$, and that for any $(a_1,\dots,a_d)\in\bb{Z}_p^d$,
	\begin{align}
		\exp(\sum_{i=1}^da_iu_i)v=\sum_{\underline{m}\in\bb{N}^d}  a_1^{m_1}\cdots a_d^{m_d}v_{\underline{m}}.
	\end{align}
	Since $v_{\underline{0}}=v$ and $\exp(\sum_{i=1}^da_iu_i)^{a_0}=\exp(\sum_{i=1}^da_0a_iu_i)$, we see that for $a_0\neq 0$,
	\begin{align}\label{eq:lem:analytic-derivative}
		\psi(a_0,a_1,\dots,a_d,v)=\sum_{\underline{m}\in\bb{N}^d\setminus\{\underline{0}\}}  a_0^{|\underline{m}|-1}a_1^{m_1}\cdots a_d^{m_d}v_{\underline{m}}.
	\end{align}
	As the right hand side is continuous with respect to the variable $a_0$, we see that this formula remains valid for $a_0=0$ by the definition of $\varphi_g$ \eqref{eq:prop:analytic-derivative-defn}. Notice that $|v|_{\ca{G}}=\sup_{\underline{m}\in\bb{N}^d}|v_{\underline{m}}|$ by definition. Thus, one deduces easily from the formula \eqref{eq:lem:analytic-derivative} the continuity of $\psi$.
\end{proof}

\begin{mycor}\label{cor:analytic-derivative}
	With the notation in {\rm\ref{prop:analytic-derivative}}, there is a canonical morphism of $\bb{Q}_p$-linear Lie algebras induced by the infinitesimal action of $\ca{G}$,
	\begin{align}\label{eq:prop:analytic-derivative-lie}
		\varphi:\lie(\ca{G})\longrightarrow \mrm{End}_{\bb{Q}_p}(V^{\ca{G}\trm{-}\mrm{an}}).
	\end{align}
	More precisely, its composition with the logarithm map $\log_{\ca{G}}:\ca{G}\to \lie(\ca{G})$ is the map $\varphi:\ca{G}\to \mrm{End}_{\bb{Q}_p}(V^{\ca{G}\trm{-}\mrm{an}})$ sending $g$ to the infinitesimal action $\varphi_g$ \eqref{eq:prop:analytic-derivative-defn} of $g\in G$ on $V^{\ca{G}\trm{-}\mrm{an}}$.
\end{mycor}
\begin{proof}
	The proof is the same as that of \ref{cor:operator} by using \ref{lem:analytic-derivative} instead of \ref{lem:derivative-uniform-cont}.
\end{proof}

\begin{myrem}\label{rem:analytic-derivative}
	With the notation in {\rm\ref{prop:analytic-derivative}}, assume that $\ca{G}$ preserves the norm of $V$ and induces a trivial action on $V^{\leq 1}/p^2V^{\leq 1}$. Then, $V=V^{\ca{G}\trm{-}\mrm{an}}$ by \ref{lem:finite-dim-repn-analytic}. We remark that the infinitesimal action \eqref{eq:prop:analytic-derivative-lie} of $\ca{G}$ on $V$ coincides with the Lie algebra action \eqref{eq:para:inf-repn-B-0} by \eqref{eq:para:small-derivative}.
\end{myrem}

\begin{mydefn}[{cf. \cite[2.1.6]{pan2021locally}}]\label{defn:locally-analytic-vector}
	Let $\ca{G}$ be a $p$-adic analytic group, $V$ a $\bb{Q}_p$-Banach space endowed with a $\bb{Q}_p$-linear continuous action of $\ca{G}$. We say an element $v\in V$ is \emph{$\ca{G}$-locally analytic} if it is $\ca{G}_0$-analytic for some uniform pro-$p$ open subgroup $\ca{G}_0$ of $\ca{G}$. We set $V^{\ca{G}\trm{-}\mrm{la}}=\bigcup_{\ca{G}_0}V^{\ca{G}_0\trm{-}\mrm{an}}$ the $\bb{Q}_p$-linear subspace of $V$ consisting of $\ca{G}$-locally analytic vectors.
\end{mydefn}

We obtain from \ref{cor:analytic-derivative} a canonical morphism of Lie algebras over $\bb{Q}_p$,
\begin{align}\label{eq:analytic-derivative-lie}
	\varphi:\lie(\ca{G})\longrightarrow \mrm{End}_{\bb{Q}_p}(V^{\ca{G}\trm{-}\mrm{la}}),
\end{align} 
which is called the \emph{infinitesimal Lie algebra action on locally analytic vectors}. We remark that if $V$ is finite-dimensional, then $V=V^{\ca{G}\trm{-}\mrm{la}}$ by \ref{lem:finite-dim-repn-analytic}.

\begin{mypara}\label{para:locally-analytic-vector-B}
	Let $K$ be a complete discrete valuation field of characteristic $0$ with perfect residue field of characteristic $p>0$, $(B_{\triv},B,\overline{B})$ a $(K,\ca{O}_K,\ca{O}_{\overline{K}})$-triple (see \ref{defn:triple}), $\ca{F}$ a Galois extension of the fraction field $\ca{L}$ of $B$ contained in $\ca{L}_{\mrm{ur}}$ such that $\ca{G}=\gal(\ca{F}/\ca{L})$ is a $p$-adic analytic group. We denote by $\ca{G}_H$ the image of $H=\gal(\ca{L}_{\mrm{ur}}/\ca{L}_{\infty})$ under the surjection $G=\gal(\ca{L}_{\mrm{ur}}/\ca{L})\to \ca{G}$, where $\ca{L}_\infty=K_\infty\ca{L}$ is the cyclotomic extension of $\ca{L}$, and we denote by $\ca{F}_{\Sigma}$ the invariant subfield of $\ca{F}$ by $\ca{G}_H$. We name some Galois groups as indicated in the following diagram
	
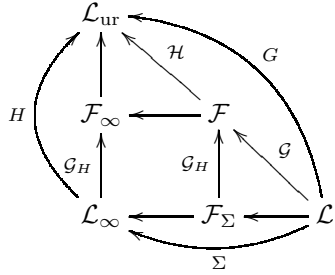
\begin{align}\label{diam:para:locally-analytic-vector-B}
		\xymatrix{
			\ca{L}_{\mrm{ur}}&&\\
			\ca{F}_\infty\ar[u]&\ca{F}\ar[l]\ar[lu]_-{\ca{H}}&\\
			\ca{L}_\infty\ar[u]^-{\ca{G}_H} \ar@/^2.1pc/[uu]^-{H} &\ca{F}_{\Sigma}\ar[l]\ar[u]^-{\ca{G}_H}&\ca{L}\ar[l]\ar@/^1pc/[ll]^-{\Sigma}\ar[lu]_-{\ca{G}}\ar@/_2pc/[lluu]_{G}
		}
	\end{align}
	Indeed, we have $\ca{F}_\infty=K_\infty\ca{F}=\ca{L}_\infty\otimes_{\ca{F}_{\Sigma}}\ca{F}$ by Galois theory. Consider the directed system $\scr{F}^{\mrm{fini}}_{\ca{L}_{\mrm{ur}}/\ca{L}}$ of finite field extension $\ca{L}'$ of $\ca{L}$ contained in $\ca{L}_{\mrm{ur}}$. For each $\ca{L}'$, we construct the above diagram for the Galois extension $\ca{F}'=\ca{L}'\ca{F}$ over $\ca{L}'$ contained in $\ca{L}_{\mrm{ur}}$ in the same way and add prime superscript to the notation. There is a natural commutative diagram of fields
	
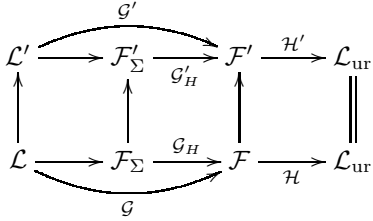
\begin{align}
		\xymatrix{
			\ca{L}'\ar[r]\ar@/^1pc/[rr]^-{\ca{G}'}&\ca{F}'_\Sigma\ar[r]_-{\ca{G}'_H}&\ca{F}'\ar[r]^-{\ca{H}'}&
			\ca{L}_{\mrm{ur}}\\
			\ca{L}\ar[r]\ar[u]\ar@/_1pc/[rr]_-{\ca{G}}&\ca{F}_\Sigma\ar[r]^-{\ca{G}_H}\ar[u]&\ca{F}\ar[r]_-{\ca{H}}\ar[u]&
			\ca{L}_{\mrm{ur}}\ar@{=}[u]
		}
	\end{align}
	We remark that each Galois group with prime superscript naturally identifies with an open subgroup of the corresponding Galois group without prime superscript.
	
	Recall that $\widehat{\overline{B}}[1/p]$ is a $\bb{Q}_p$-Banach algebra endowed with the canonical norm defined by $\widehat{\overline{B}}$ (see \ref{para:normed-module-2}) and the continuous action of $G$. Consider a $\widehat{\overline{B}}[1/p]$-Banach module $W$ endowed with a semi-linear continuous action of $G$. Then, the $\ca{H}$-invariant part $W^{\ca{H}}$ is a $(\widehat{\overline{B}}[1/p])^{\ca{H}}$-Banach module endowed with the induced continuous action of $\ca{G}$. The spaces of locally analytic vectors $(W^{\ca{H}'})^{\ca{G}'\trm{-}\mrm{la}}$ (resp. $(W^{\ca{H}'})^{\ca{G}'_H\trm{-}\mrm{la}}$) form a directed system of $\bb{Q}_p$-linear subspaces of $W$ over $\scr{F}^{\mrm{fini}}_{\ca{L}_{\mrm{ur}}/\ca{L}}$. We denote its colimit by
	\begin{align}
		W^{\ca{G}\trm{-}\mrm{la}}&=\colim_{\ca{L}'\in\scr{F}^{\mrm{fini}}_{\ca{L}_{\mrm{ur}}/\ca{L}}}(W^{\ca{H}'})^{\ca{G}'\trm{-}\mrm{la}}\\
		\trm{(resp. } W^{\ca{G}_H\trm{-}\mrm{la}}&=\colim_{\ca{L}'\in\scr{F}^{\mrm{fini}}_{\ca{L}_{\mrm{ur}}/\ca{L}}}(W^{\ca{H}'})^{\ca{G}'_H\trm{-}\mrm{la}}\trm{)}
	\end{align}
	and we call it the subspace of \emph{$\ca{G}$-locally analytic} (resp. \emph{$\ca{G}_H$-locally analytic}) \emph{vectors of $W$} (we take the colimit here for the flexibility of replacing $\ca{L}$ by a finite extension $\ca{L}'$ in the proof of \ref{thm:locally-analytic-vector}). It is $G$-stable and endowed with the infinitesimal Lie algebra action $\varphi^{\mrm{la}}$ of $\lie(\ca{G})$ (resp. $\lie(\ca{G}_H)$) by \ref{cor:analytic-derivative}. We extend this action $\widehat{\overline{B}}[1/p]$-linearly to
	\begin{align}
		1\otimes\varphi^{\mrm{la}}:(\widehat{\overline{B}}[1/p]\otimes_{\bb{Q}_p}\lie(\ca{G}))\times W^{\ca{G}\trm{-}\mrm{la}}&\longrightarrow W\label{eq:locally-analytic-vector-B-1}\\
		\trm{(resp. }1\otimes\varphi^{\mrm{la}}:(\widehat{\overline{B}}[1/p]\otimes_{\bb{Q}_p}\lie(\ca{G}_H))\times W^{\ca{G}_H\trm{-}\mrm{la}}&\longrightarrow W\trm{)}.\label{eq:locally-analytic-vector-B-2}
	\end{align}
\end{mypara}

\begin{mypara}\label{para:analytic-vector-different-act}
	With the notation in \ref{para:locally-analytic-vector-B}, assume that $\ca{G}$ is a uniform pro-$p$ group. Then, we consider three $G$-actions $\{1\otimes \rho_{\dl}, 1\otimes \rho_{\rr},\rho_W\otimes \rho_{\dl}\}$ on $\scr{C}^{\mrm{an}}(\ca{G},W)$ defined by the same formula as in \eqref{eq:para:three-action}.	By \ref{cor:analytic-derivative} and \ref{lem:analytic-function-space-an}, there are infinitesimal Lie algebra actions $\varphi_{\dl},\varphi_{\rr}$ of $\lie(\ca{G})$ on $\scr{C}^{\mrm{an}}(\ca{G},W)$ associated to the left and right translation of $\ca{G}$, which commute with each other. We also extend them $\widehat{\overline{B}}[1/p]$-linearly to
	\begin{align}\label{eq:para:analytic-vector-different-act}
		1\otimes\varphi_{\dl},1\otimes\varphi_{\rr}:(\widehat{\overline{B}}[1/p]\otimes_{\bb{Q}_p}\lie(\ca{G}))\times \scr{C}^{\mrm{an}}(\ca{G},W)\longrightarrow \scr{C}^{\mrm{an}}(\ca{G},W).
	\end{align}
\end{mypara}

\begin{mylem}[{cf. \cite[2.1.4, 3.3.5]{pan2021locally}}]\label{lem:analytic-vector-ari}
	With the notation in {\rm\ref{para:locally-analytic-vector-B}} and {\rm\ref{para:analytic-vector-different-act}}, assume that $\ca{G}$ is a uniform pro-$p$ group. Then, the map
	\begin{align}\label{eq:analytic-vector-ari}
		(W^{\ca{H}})^{\ca{G}\trm{-}\mrm{an}}\longrightarrow \scr{C}^{\mrm{an}}(\ca{G},W)^{(\rho_W\otimes \rho_{\dl})(G)=1},\ v\mapsto (f_v:g\mapsto gv)
	\end{align}
	is well-defined and bijective. Moreover, for any $v\in (W^{\ca{H}})^{\ca{G}\trm{-}\mrm{an}}$ and $\phi\in \widehat{\overline{B}}[1/p]\otimes_{\bb{Q}_p}\lie(\ca{G})$, we have
	\begin{align}
		(1\otimes\varphi^{\mrm{la}})(\phi,v)=-(1\otimes\varphi_{\dl})(\phi,f_v)(1)\in W.
	\end{align}
\end{mylem}
\begin{proof}
	The first part follows from the same argument of \ref{lem:analytic-vector-pre}. For the ``moreover'' part, by assumption, we write $\phi=\sum_{i=1}^na_i\otimes\log_{\ca{G}}(g_i)$ where $a_i\in \widehat{\overline{B}}[1/p]$ and $g_i\in\ca{G}$. Then, by the definition, we have
	\begin{align}
		(1\otimes\varphi_{\dl})(\phi,f_v)(1)=&\sum_{i=1}^n a_i\cdot\lim_{n\to\infty}p^{-n}(\rho_{\dl}(g_i^{p^n})f_v-f_v)(1)\\
		=&\sum_{i=1}^n a_i\cdot\lim_{n\to\infty}p^{-n}(g_i^{-p^n}v-v)\nonumber\\
		=&-\sum_{i=1}^n a_i\varphi_{g_i}(v)=-(1\otimes\varphi^{\mrm{la}})(\phi,v),\nonumber
	\end{align}
	where the first equality follows from the fact that taking limits in $\scr{C}^{\mrm{an}}(\ca{G},W)$ with respect to its norm commutes with evaluating at $1\in\ca{G}$.
\end{proof}

\begin{mylem}\label{lem:analytic-vector-geo}
	With the notation in {\rm\ref{para:locally-analytic-vector-B}} and {\rm\ref{para:analytic-vector-different-act}}, assume that $\ca{G}$ is a uniform pro-$p$ group and that $\ca{G}/\ca{G}_H$ is isomorphic to $0$ or $\bb{Z}_p$. We take $\sigma_0\in \ca{G}$ whose image in $\ca{G}/\ca{G}_H$ is a topological generator, and denote by $\sigma_0^{\bb{Z}_p}$ the closed subgroup of $\ca{G}$ generated by $\sigma_0$. Then, the map
	\begin{align}\label{eq:analytic-vector}
		(W^{\ca{H}})^{\ca{G}_H\trm{-}\mrm{an}}\longrightarrow\{f\in \scr{C}^{\mrm{an}}(\ca{G},W)\ |\ (\rho_W\otimes \rho_{\dl})(H)f=f,\ (1\otimes\rho_{\rr})(\sigma_0^{\bb{Z}_p})f=f\}
	\end{align}
	sending $v$ to $(f_v:h\sigma_0^a\mapsto hv)$, where $h\in \ca{G}_H$ and $a\in\bb{Z}_p$, is well-defined and bijective. Moreover, for any $v\in (W^{\ca{H}})^{\ca{G}_H\trm{-}\mrm{an}}$ and $\phi\in \widehat{\overline{B}}[1/p]\otimes_{\bb{Q}_p}\lie(\ca{G}_H)$, we have
	\begin{align}
		(1\otimes\varphi^{\mrm{la}})(\phi,v)=-(1\otimes\varphi_{\dl})(\phi,f_v)(1)\in W.
	\end{align}
\end{mylem}
\begin{proof}
	If $\ca{G}/\ca{G}_H=0$, then we reduce to \ref{lem:analytic-vector-ari}. Assume that $\ca{G}/\ca{G}_H=\bb{Z}_p$. Firstly, we see that there is a canonical bijection by the same argument of \ref{lem:analytic-vector-pre}, 
	\begin{align}
		(W^{\ca{H}})^{\ca{G}_H\trm{-}\mrm{an}}\iso \scr{C}^{\mrm{an}}(\ca{G}_H,W)^{(\rho_W\otimes \rho_{\dl})(H)=1},\ v\mapsto (f'_v:h\mapsto hv).
	\end{align}
	Notice that $\ca{G}_H$ is a uniform pro-$p$ group and there is an exact sequence of the powerful Lie algebras over $\bb{Z}_p$ by \ref{lem:uniform-ext}, 
	\begin{align}
		0\longrightarrow L_{\ca{G}_H}\longrightarrow L_{\ca{G}}\longrightarrow L_{\ca{G}/\ca{G}_H}\longrightarrow 0.
	\end{align}
	We take a $\bb{Z}_p$-linear basis $h_1,\dots,h_d$ of $L_{\ca{G}_H}$. Then, $\sigma_0,h_1,\dots,h_d$ form a $\bb{Z}_p$-linear basis of $L_{\ca{G}}$, and thus they also form a minimal topological generating set of $\ca{G}$ (see \ref{para:L-*}). We obtain two systems of coordinates of second kind \eqref{eq:coordinates-second},
	\begin{align}
		\bb{Z}_p^d&\longrightarrow \ca{G}_H,\ (a_1,\dots,a_d)\mapsto h_1^{a_1}\cdots h_d^{a_d}\\
		\bb{Z}_p^{1+d}&\longrightarrow \ca{G},\ (a_0,a_1,\dots,a_d)\mapsto h_1^{a_1}\cdots h_d^{a_d}\sigma_0^{a_0}.
	\end{align}
	By \ref{lem:analytic-function-pullback} and \ref{rem:analytic-function-lie}, the map of underlying sets $\ca{G}\to \ca{G}_H$ sending $h_1^{a_1}\cdots h_d^{a_d}\sigma_0^{a_0}$ to $h_1^{a_1}\cdots h_d^{a_d}$ induces an injective map by pullback
	\begin{align}
		\scr{C}^{\mrm{an}}(\ca{G}_H,W)\longrightarrow \scr{C}^{\mrm{an}}(\ca{G},W),\ f'\mapsto (h\sigma_0^a\mapsto f'(h)),
	\end{align}
	whose image is $\scr{C}^{\mrm{an}}(\ca{G},W)^{(1\otimes\rho_{\rr})(\sigma_0^{\bb{Z}_p})=1}$. This shows that \eqref{eq:analytic-vector} is a well-defined bijection. The ``moreover'' part follows from the same argument of that of \ref{lem:analytic-vector-ari}.
\end{proof}

\begin{mythm}[{cf. \cite[3.3.5]{pan2021locally}}]\label{thm:locally-analytic-vector}
	 With the notation in {\rm\ref{para:locally-analytic-vector-B}}, assume that $B$ is a quasi-adequate $\ca{O}_K$-algebra. Let $\Phi^{\geo}_{\ca{G}}$ be the canonical $\widehat{\overline{B}}[1/p]$-linear Lie subalgebra of $\widehat{\overline{B}}[1/p]\otimes_{\bb{Q}_p}\lie(\ca{G}_H)$ defined in {\rm\ref{defn:sen-lie-lift-B}}. Then, under the canonical infinitesimal action $1\otimes\varphi^{\mrm{la}}$ \eqref{eq:locally-analytic-vector-B-2} by taking $W=\widehat{\overline{B}}[1/p]$, $\Phi^{\geo}_{\ca{G}}$ annihilates $(\widehat{\overline{B}}[1/p])^{\ca{G}_H\trm{-}\mrm{la}}$.
\end{mythm}
\begin{proof}
	Consider the cofinal subsystem $\scr{F}^{\mrm{qa}}_{\ca{L}_{\mrm{ur}}/\ca{L}}$ of $\scr{F}^{\mrm{fini}}_{\ca{L}_{\mrm{ur}}/\ca{L}}$ defined in \ref{cor:quasi-adequate-cofinal}. Recall that for any $\ca{L}'\in\scr{F}^{\mrm{qa}}_{\ca{L}_{\mrm{ur}}/\ca{L}}$, we have natural identifications $\Phi^{\geo}_{\ca{G}}=\Phi^{\geo}_{\ca{G}'}$ by \ref{cor:sen-lie-B}.(\ref{item:cor:sen-lie-lift-B-5}). Thus, it suffices to show that $\Phi^{\geo}_{\ca{G}}$ annihilates $(\widehat{\overline{B}}[1/p]^{\ca{H}})^{\ca{G}_H\trm{-}\mrm{la}}$. Moreover, after replacing $\ca{G}$ by a sufficiently small uniform pro-$p$ open subgroup, it suffices to show that $\Phi^{\geo}_{\ca{G}}$ annihilates $(\widehat{\overline{B}}[1/p]^{\ca{H}})^{\ca{G}_H\trm{-}\mrm{an}}$ (so that $\ca{L}$ may not lies in $\scr{F}^{\mrm{qa}}_{\ca{L}_{\mrm{ur}}/\ca{L}}$ from now on). Recall that there exists a sufficiently small uniform pro-$p$ open subgroup $\ca{G}_0$ of $\ca{G}$ such that $\ca{G}_{0H}=\ca{G}_0\cap \ca{G}_H$ and $\ca{G}_0/\ca{G}_{0H}$ are both uniform by \ref{lem:uniform-ext-2}. Replacing $\ca{G}$ by $\ca{G}_0$, we obtain the following conditions on $\ca{G}$: 
	\begin{enumerate}
		\renewcommand{\labelenumi}{{\rm(\theenumi)}}
		\item $\ca{G}$ is a uniform pro-$p$ group isomorphic to a closed subgroup of $\id+p^2\mrm{M}_d(\bb{Z}_p)$ for some $d\in\bb{N}$ (using \ref{thm:compact-lie-group}), and
		\item $\ca{G}/\ca{G}_H$ is isomorphic to $0$ or $\bb{Z}_p$ (as it is a uniform pro-$p$ subquotient of $G/H=\Sigma\subseteq\bb{Z}_p^\times$).
	\end{enumerate}
	
	\begin{align}\label{diam:thm:locally-analytic-vector}
		\xymatrix{
			\ca{L}'\ar[r]\ar@/^1pc/[rr]^-{\ca{G}'}&\ca{F}'_\Sigma\ar[r]_-{\ca{G}'_H}&\ca{F}'\ar[r]^-{\ca{H}'}&
			\ca{L}_{\mrm{ur}}\\
			\ca{L}\ar[r]\ar[u]\ar@/_1pc/[rr]_-{\ca{G}}&\ca{F}_\Sigma\ar[r]^-{\ca{G}_H}\ar[u]&\ca{F}\ar[r]_-{\ca{H}}\ar[u]&
			\ca{L}_{\mrm{ur}}\ar@{=}[u]
		}
	\end{align}
	We take $\ca{L}'\in \scr{F}^{\mrm{qa}}_{\ca{L}_{\mrm{ur}}/\ca{L}}$ sufficiently large such that the image of the natural injection $\ca{G}'\to \ca{G}$ is contained in $\ca{G}^{p^3}$ (recall that $\ca{F}'=\ca{L}'\ca{F}$ by definition). Then, $\ca{G}'$ acts trivially on $\scr{C}^{\mrm{an}}(\ca{G},\bb{Q}_p)^{\leq 1}/p^3\scr{C}^{\mrm{an}}(\ca{G},\bb{Q}_p)^{\leq 1}$ via the left and right translation actions by \ref{lem:group-action-analytic}. Combining with \ref{prop:analytic-approx}, we can apply \ref{thm:sen-operator-B-fix} to the $\bb{Q}_p$-Banach space $\scr{C}^{\mrm{an}}(\ca{G},\bb{Q}_p)$ endowed with the left translation action $1\otimes\rho_{\dl}$ of $\ca{G}'$, so that $\Phi^{\geo}_{\ca{G}'}$ acts trivially on $((\widehat{\overline{B}}\otimes_{\bb{Z}_p}\scr{C}^{\mrm{an}}(\ca{G},\bb{Q}_p)^{\leq 1})^\wedge[1/p])^{H'}$. 
	
	Notice that there is a natural identification $(\widehat{\overline{B}}\otimes_{\bb{Z}_p}\scr{C}^{\mrm{an}}(\ca{G},\bb{Q}_p)^{\leq 1})^\wedge[1/p]=\scr{C}^{\mrm{an}}(\ca{G},\widehat{\overline{B}}[1/p])$ by \ref{lem:analytic-function-basic}, which satisfies the following properties (by firstly checking over the submodule $\widehat{\overline{B}}\otimes_{\bb{Z}_p}\scr{C}^{\mrm{an}}(\ca{G},\bb{Q}_p)^{\leq 1}$ and then taking $p$-adic completion): 
	\begin{enumerate}
		\renewcommand{\labelenumi}{{\rm(\theenumi)}}
		\item the action of $G'$ on $(\widehat{\overline{B}}\otimes_{\bb{Z}_p}\scr{C}^{\mrm{an}}(\ca{G},\bb{Q}_p)^{\leq 1})^\wedge[1/p]$ defined in \ref{para:inf-repn-B} coincides with the diagonal action $\rho_{\widehat{\overline{B}}[1/p]}\otimes\rho_{\dl}$ on $\scr{C}^{\mrm{an}}(\ca{G},\widehat{\overline{B}}[1/p])$, and
		\item the action of $\Phi_{\ca{G}'}$ on $(\widehat{\overline{B}}\otimes_{\bb{Z}_p}\scr{C}^{\mrm{an}}(\ca{G},\bb{Q}_p)^{\leq 1})^\wedge[1/p]$ defined by \eqref{eq:para:inf-repn-B-2} coincides with the infinitesimal action $1\otimes\varphi_{\dl}$ \eqref{eq:para:analytic-vector-different-act} on $\scr{C}^{\mrm{an}}(\ca{G},\widehat{\overline{B}}[1/p])$ induced by the left translation of $\ca{G}$ (cf. \ref{rem:analytic-derivative}).
	\end{enumerate}
	Therefore, applying \ref{lem:analytic-vector-geo} we see that for any $v\in ((\widehat{\overline{B}}[1/p])^{\ca{H}})^{\ca{G}_H\trm{-}\mrm{an}}$ and any $\phi\in \Phi^{\geo}_{\ca{G}}$,
	\begin{align}
		(1\otimes\varphi^{\mrm{la}})(\phi,v)=-(1\otimes\varphi_{\dl})(\phi,f_v)(1)=0,
	\end{align}
	as $(f_v:h\sigma_0^a\mapsto hv)$ is an element of $\scr{C}^{\mrm{an}}(\ca{G},\widehat{\overline{B}}[1/p])^{(\rho_{\widehat{\overline{B}}[1/p]}\otimes \rho_{\dl})(H)=1}\subseteq ((\widehat{\overline{B}}\otimes_{\bb{Z}_p}\scr{C}^{\mrm{an}}(\ca{G},\bb{Q}_p)^{\leq 1})^\wedge[1/p])^{H'}$ killed by $\Phi^{\geo}_{\ca{G}'}=\Phi^{\geo}_{\ca{G}}$.
\end{proof}

\begin{myrem}\label{rem:locally-analytic-vector}
	We don't know whether $\Phi_{\ca{G}}$ annihilates $(\widehat{\overline{B}}[1/p])^{\ca{G}\trm{-}\mrm{la}}$ or not, cf. \ref{rem:thm:sen-operator-fix-infty}.
\end{myrem}

\section{Appendix: Hyodo's Computation of Galois Cohomologies}
This appendix is devoted to detailed proofs of \ref{cor:fal-ext-connect} and \ref{cor:hyodo-ring-coh}, using essentially Hyodo's arguments in \cite{hyodo1986hodge} and \cite{hyodo1989variation}. We take again the notation in \ref{para:notation-K}.
\begin{align}
	\xymatrix{
		\overline{K}&\\
		K_{\infty,\underline{\infty}}\ar[u]&\\
		K_\infty\ar[u]^-{\Delta}\ar@/^2pc/[uu]^-{H}&K\ar[l]^-{\Sigma}\ar[lu]_-{\Gamma}\ar@/_1pc/[luu]_{G}
	}
\end{align}

\begin{mylem}[{cf. \cite[2-1]{hyodo1986hodge}}]\label{lem:H-coh}
	Assume that $K$ satisfies the assumption {\rm($\ast$)} in {\rm\ref{para:notation-K-good}}. Let $r,q\in\bb{N}$.
	\begin{enumerate}
		\renewcommand{\labelenumi}{{\rm(\theenumi)}}
		\item The natural map $H^q(\Delta,\ca{O}_{K_{\infty,\underline{\infty}}}/p^r\ca{O}_{K_{\infty,\underline{\infty}}})\to H^q(H,\ca{O}_{\overline{K}}/p^r\ca{O}_{\overline{K}})$ is an almost isomorphism.\label{item:lem:H-coh-1}
		\item The natural map induced by the cup product $\wedge^q \ho(\Delta,\ca{O}_{K_\infty}/p^r\ca{O}_{K_\infty})\to H^q(\Delta,\ca{O}_{K_{\infty,\underline{\infty}}}/p^r\ca{O}_{K_{\infty,\underline{\infty}}})$ is injective, and admits a natural retraction with kernel killed by $\zeta_p-1$.\label{item:lem:H-coh-2}
	\end{enumerate}
\end{mylem}
\begin{proof}
	(\ref{item:lem:H-coh-1}) If we take $K'$ as in \ref{lem:cohen}, then $\ca{O}_{K'_{\infty,\underline{\infty}}}=\colim_{n\in\bb{N}}\ca{O}_{K'}[\zeta_{p^n},t_{1,p^n},\dots,t_{d,p^n}]$ (by \ref{lem:cohen-str}) is a non-discrete valuation ring of height $1$ such that the Frobenius map on $\ca{O}_{K'_{\infty,\underline{\infty}}}/p\ca{O}_{K'_{\infty,\underline{\infty}}}$ is surjective. In other words, $K'_{\infty,\underline{\infty}}$ is a pre-perfectoid field. As $K$ is finite over $K'$, $\ca{O}_{K_{\infty,\underline{\infty}}}$ is almost  \'etale over $\ca{O}_{K'_{\infty,\underline{\infty}}}$ by almost purity (\ref{thm:almost-purity}), and the conclusion follows by almost Galois descent (cf. \cite[\Luoma{2}.6.24]{abbes2016p}).
	
	(\ref{item:lem:H-coh-2}) By \eqref{eq:para:notation-K-good}, we can decompose $\ca{O}_{K_{\infty,\underline{\infty}}}$ into a direct sum of free $\ca{O}_{K_\infty}$-submodules of rank $1$,
	\begin{align}
		\ca{O}_{K_{\infty,\underline{\infty}}}=\bigoplus_{\underline{m}\in\bb{N}^d}\bigoplus_{\underline{k}\in J_{\leq p^{\underline{m}}}} \ca{O}_{K_\infty}t_{1,p^{m_1}}^{k_1}\cdots t_{d,p^{m_d}}^{k_d},
	\end{align}
	where $J\subseteq \bb{N}_{>0}^d$ is the subset of tuples with components prime to $p$. It induces a natural retraction of the inclusion $\ca{O}_{K_\infty}\to \ca{O}_{K_{\infty,\underline{\infty}}}$. Notice that $\Delta$ acts on $\ca{O}_{K_\infty}t_{1,p^{m_1}}^{k_1}\cdots t_{d,p^{m_d}}^{k_d}$ by the multiplication by a group homomorphism from $\Delta$ to the group of roots of unity contained in $\overline{K}^\times$, which is trivial if and only if $\underline{m}= 0$. By \cite[\Luoma{2}.8.1]{abbes2016p}, we see that $\wedge^q \ho(\Delta,\ca{O}_{K_\infty}/p^r)= H^q(\Delta,\ca{O}_{K_\infty}/p^r)$ and that $H^q(\Delta,\ca{O}_{K_\infty}t_{1,p^{m_1}}^{k_1}\cdots t_{d,p^{m_d}}^{k_d}/p^r)$ is killed by $\zeta_p-1$ if $\underline{m}\neq 0$.
\end{proof}

\begin{myprop}\label{app:prop:fal-ext-connect}
	The connecting map of the Faltings extension \eqref{eq:fal-ext} induces a canonical $\widehat{K_\infty}$-linear isomorphism
	\begin{align}\label{eq:app:prop:fal-ext-connect}
		\widehat{K_\infty}\otimes_{\ca{O}_K}\widehat{\Omega}^1_{\ca{O}_K}\iso H^1(H,\widehat{\overline{K}}(1)),
	\end{align}
	sending $\df\log(t_i)$ to $\xi_i\otimes\zeta$, where $\zeta=(\zeta_{p^n})_{n\in\bb{N}}\in \bb{Z}_p(1)$ and $\xi=(\xi_1,\dots,\xi_d):H\to \bb{Z}_p^d$ is the continuous $1$-cocycle \eqref{eq:cont-cocycle}. Moreover, for any $q\in\bb{N}$, the cup product induces a natural isomorphism
	\begin{align}\label{eq:app:prop:fal-ext-connect-1}
		\widehat{K_\infty}\otimes_{\ca{O}_K}\wedge^q_{\ca{O}_K}\widehat{\Omega}^1_{\ca{O}_K}\iso H^q(H,\widehat{\overline{K}}(q)).
	\end{align}
\end{myprop}
\begin{proof}
	The statement itself defines a natural map $\widehat{K_\infty}\otimes_{\ca{O}_K}\wedge^q_{\ca{O}_K}\widehat{\Omega}^1_{\ca{O}_K}\to H^q(H,\widehat{\overline{K}}(q))$ for any $q\in\bb{N}$. We only need to prove that it is an isomorphism. For $q=0$, this follows from Ax-Sen-Tate's theorem \cite{ax1970ax}. 
	
	Assume that $K$ satisfies the assumption {\rm($\ast$)} in {\rm\ref{para:notation-K-good}}. Then, the natural map
	\begin{align}\label{eq:app:prop:fal-ext-connect-2}
		\wedge^q \ho(\Delta,K_\infty/p^r\ca{O}_{K_\infty})\longrightarrow H^q(H,\overline{K}/p^r\ca{O}_{\overline{K}}) 
	\end{align} 
	is a $p^2$-isomorphism by \ref{lem:H-coh} for any $q,r\in\bb{N}$. Consider the canonical exact sequence for any $q\in\bb{N}$ (\cite[\Luoma{2}.3.10.4, \Luoma{2}.3.10.5]{abbes2016p})
	\begin{align}
		0\to \rr^1\lim_{r\in\bb{N}}H^{q-1}(H,\overline{K}/p^r\ca{O}_{\overline{K}})\to H^q(H,\widehat{\overline{K}})\to \lim_{r\in\bb{N}}H^q(H,\overline{K}/p^r\ca{O}_{\overline{K}})\to 0.
	\end{align}
	Since the inverse system $(\wedge^{q-1} \ho(\Delta,K_\infty/p^r\ca{O}_{K_\infty}))_{r\in\bb{N}}$ satisfies the Mittag-Leffler condition, we have $\rr^1\lim_{r\in\bb{N}}\wedge^{q-1} \ho(\Delta,K_\infty/p^r\ca{O}_{K_\infty})=0$. Thus, $\rr^1\lim_{r\in\bb{N}}H^{q-1}(H,\overline{K}/p^r\ca{O}_{\overline{K}})$ is killed by $p^4$ by the $p^2$-isomorphism \eqref{eq:app:prop:fal-ext-connect-2} and \ref{rem:pi-iso-retract}.(\ref{item:rem:pi-iso-retract-2}). Moreover, it is actually zero since multiplication by $p$ is invertible on $H^q(H,\widehat{\overline{K}})$. Then, we get $H^q(H,\widehat{\overline{K}})=\wedge^q \ho(\Delta,\widehat{K_\infty})$ by a similar argument. Unwinding the definitions, we get the conclusion in this case.
	
	In general, there exists a finite Galois extension $K'$ of $K$ which satisfies the assumption {\rm($\ast$)} in {\rm\ref{para:notation-K-good}}. Remark that the map $K'\otimes_{\ca{O}_K}\widehat{\Omega}^1_{\ca{O}_K}\to \widehat{\Omega}^1_{\ca{O}_{K'}}[1/p]$ is an isomorphism. By Ax-Sen-Tate's theorem, we have $(\widehat{K'_\infty} )^{H/H'}=\widehat{K_\infty}$, where $H'=\gal(\overline{K}/K'_\infty)$. Thus,
	\begin{align}
		\widehat{K_\infty}\otimes_{\ca{O}_K}\widehat{\Omega}^1_{\ca{O}_K}\longrightarrow (\widehat{K'_\infty}\otimes_{\ca{O}_{K'}}\widehat{\Omega}^1_{\ca{O}_{K'}})^{H/H'}
	\end{align}
	is an isomorphism. In particular, the restriction map
	\begin{align}
		\mrm{Res}:H^q(H,\widehat{\overline{K}}(q))\longrightarrow H^q(H',\widehat{\overline{K}}(q))^{H/H'}
	\end{align}
	is surjective, since the natural map $\widehat{K'_\infty}\otimes_{\ca{O}_{K'}}\widehat{\Omega}^1_{\ca{O}_{K'}}\to H^q(H',\widehat{\overline{K}}(q))$ is an isomorphism by applying the discussion above to $K'$. It is also injective, since there is a co-restriction map $\mrm{Cor}:H^q(H',\widehat{\overline{K}}(q))^{H/H'}\to H^q(H,\widehat{\overline{K}}(q))$ such that $\mrm{Cor}\circ \mrm{Res}=[K':K]\cdot\id$ (cf. \cite[\textsection 2]{tate1976galcoh}). Thus, we deduce the proposition for $K$ from the special case for $K'$.
\end{proof}

\begin{mypara}
	For simplicity, we put $\widehat{\Omega}^1=K\otimes_{\ca{O}_K}\widehat{\Omega}^1_{\ca{O}_K}(-1)$ and $\widehat{\Omega}^i=\wedge_K^i\widehat{\Omega}^1$ for any $i\in\bb{Z}$ (where $\widehat{\Omega}^i=0$ if $i<0$). Recall that there is a natural exact sequence \eqref{eq:para:hyodo-ring} for any $n\in\bb{Z}$ (where $\mrm{Sym}^n=0$ if $n<0$),
	\begin{align}\label{eq:app:cor:hyodo-ring-coh}
		0\longrightarrow \mrm{Sym}^{n-1}_{\widehat{\overline{K}}} (\scr{E}_{\ca{O}_K}(-1))\longrightarrow \mrm{Sym}^n_{\widehat{\overline{K}}} (\scr{E}_{\ca{O}_K}(-1))\longrightarrow \widehat{\overline{K}}\otimes_K\mrm{Sym}^n_K\widehat{\Omega}^1\longrightarrow 0.
	\end{align}
	It induces a natural exact sequence
	\begin{align}\label{eq:app:cor:hyodo-ring-coh-2}
		0\longrightarrow \widehat{\overline{K}}\otimes_K\mrm{Sym}^{n-1}_K\widehat{\Omega}^1\longrightarrow \mrm{Sym}^n_{\widehat{\overline{K}}} (\scr{E}_{\ca{O}_K}(-1))/\mrm{Sym}^{n-2}_{\widehat{\overline{K}}} (\scr{E}_{\ca{O}_K}(-1))\longrightarrow \widehat{\overline{K}}\otimes_K\mrm{Sym}^n_K\widehat{\Omega}^1\longrightarrow 0.
	\end{align}
	For any $i,j\in\bb{Z}$, we set
	\begin{align}
		E_1^{i,j}=H^{i+j}(H,\widehat{\overline{K}}\otimes_K\mrm{Sym}^{-i}_K\widehat{\Omega}^1)
	\end{align}
	and we denote by $\df_1^{i,j}:E_1^{i,j}\to E_1^{i+1,j}$ the $(i+j)$-th connecting map associated to \eqref{eq:app:cor:hyodo-ring-coh-2} for $n=-i$. We remark that if we endow each $\mrm{Sym}^n_{\widehat{\overline{K}}}(\scr{E}_{\ca{O}_K}(-1))$ with the finite decreasing filtration $(\mrm{F}^i)_{i\in\bb{Z}}$ given by
	\begin{align}
		\mrm{F}^i\mrm{Sym}^n_{\widehat{\overline{K}}}(\scr{E}_{\ca{O}_K}(-1))=\left\{ \begin{array}{ll}
			\mrm{Sym}^n_{\widehat{\overline{K}}}(\scr{E}_{\ca{O}_K}(-1)) & \textrm{if $i<-n$,}\\
			\mrm{Sym}^{-i}_{\widehat{\overline{K}}}(\scr{E}_{\ca{O}_K}(-1)) & \textrm{otherwise,}
		\end{array} \right.
	\end{align}
	then the associated spectral sequence of the group cohomology (\cite[\href{https://stacks.math.columbia.edu/tag/012M}{012M}]{stacks-project})
	\begin{align}\label{eq:app:cor:hyodo-ring-coh-ss}
		{}_nE_1^{i,j}\Rightarrow H^{i+j}(H,\mrm{Sym}^n_{\widehat{\overline{K}}}(\scr{E}_{\ca{O}_K}(-1)))
	\end{align}
	is convergent and is given by
	\begin{align}
		{}_nE_1^{i,j}=\left\{ \begin{array}{ll}
			0 & \textrm{if $i<-n$,}\\
			E_1^{i,j} & \textrm{otherwise,}
		\end{array} \right.\quad {}_n\df_1^{i,j}=\left\{ \begin{array}{ll}
		0 & \textrm{if $i<-n$,}\\
		\df_1^{i,j} & \textrm{otherwise.}
	\end{array} \right.
	\end{align}
\end{mypara}

\begin{mylem}[{cf. \cite[\Luoma{3}.5.7, \Luoma{3}.11.11]{abbes2016p}}]\label{app:lem:hyodo-ring-coh}
	For any $i,j\in\bb{Z}$, we define a $K$-linear map $\phi^{i,j}:\mrm{Sym}_K^i\widehat{\Omega}^1\otimes_K\widehat{\Omega}^j\to \mrm{Sym}_K^{i-1}\widehat{\Omega}^1\otimes_K\widehat{\Omega}^{j+1}$ by
	\begin{align}\label{diam:app:lem:hyodo-ring-coh-0}
		[x_1\otimes\cdots\otimes x_i]\otimes \omega\mapsto \sum_{k=1}^i[x_1\otimes\cdots\otimes x_{k-1}\otimes x_{k+1}\otimes\cdots\otimes x_i]\otimes(x_k\wedge \omega),
	\end{align}
	for any $x_1,\dots,x_i\in \widehat{\Omega}^1$ and $\omega\in \widehat{\Omega}^j$. Then, there is a natural commutative diagram
	\begin{align}\label{diam:app:lem:hyodo-ring-coh}
		\xymatrix{
			\widehat{K_\infty}\otimes_K(\mrm{Sym}_K^{-i}\widehat{\Omega}^1)\otimes_K\widehat{\Omega}^{i+j}\ar[d]^-{\wr}_-{\alpha^{i,j}}\ar[rr]^-{1\otimes \phi^{-i,i+j}}&&\widehat{K_\infty}\otimes_K(\mrm{Sym}_K^{-i-1}\widehat{\Omega}^1)\otimes_K\widehat{\Omega}^{i+j+1}\ar[d]^-{\wr}_-{\alpha^{i+1,j}}\\
			E_1^{i,j}\ar[rr]^-{\df_1^{i,j}}&&E_1^{i+1,j}
		}
	\end{align}
	where the vertical maps are the natural isomorphisms induced by \eqref{eq:app:prop:fal-ext-connect-1}.
\end{mylem}
\begin{proof}
	As in \ref{cor:hyodo-ring}, we set $T_k=(\df\log(t_{k,p^n}))_{n\in\bb{N}}\otimes\zeta^{-1}\in\scr{E}_{\ca{O}_K}(-1)$ for any $1\leq k\leq d$. We remark that their images $\df T_k$ in $\widehat{\Omega}^1$ form a $K$-basis, and that  $h(T_k)=\xi_k(h)+T_k$ for any $h\in H$. For any $1\leq r_1,\dots,r_i,s_1,\dots,s_j\leq d$, we have 
	\begin{align}
		&(\df_1^{-i,i+j}\circ\alpha^{-i,i+j})([\df T_{r_1} \otimes\cdots\otimes \df T_{r_i} ]\otimes(\df T_{s_1} \wedge\cdots\wedge\df T_{s_j} ))\\
		=& \df_1^{-i,i+j}([\df T_{r_1} \otimes\cdots\otimes \df T_{r_i} ]\otimes(\xi_{s_1}\cup\cdots\cup\xi_{s_j})) \nonumber\\
		=&(h\mapsto [(\xi_{r_1}(h)+T_{r_1})\otimes\cdots\otimes(\xi_{r_i}(h)+T_{r_i})-T_{r_1} \otimes\cdots\otimes T_{r_i}])\cup \xi_{s_1}\cup\cdots\cup\xi_{s_j}\nonumber\\
		=&\sum_{k=1}^i[\df T_{r_1}\otimes\cdots\otimes \df T_{r_{k-1}}\otimes \df T_{r_{k+1}}\otimes\cdots\otimes \df T_{r_i}]\otimes (\xi_{r_k}\cup\xi_{s_1}\cup\cdots\cup\xi_{s_j})\nonumber\\
		=&(\alpha^{-i+1,i+j}\circ(1\otimes\phi^{i,j}))([\df T_{r_1} \otimes\cdots\otimes \df T_{r_i} ]\otimes(\df T_{s_1} \wedge\cdots\wedge\df T_{s_j} )).\nonumber
	\end{align}
\end{proof}

\begin{mycor}[{cf. \cite[\Luoma{3}.11.12]{abbes2016p}}]\label{app:cor:hyodo-ring-coh}
	We have
	\begin{align}
		\colim_{n\in\bb{N}}H^q(H,\mrm{Sym}^n_{\widehat{\overline{K}}} (\scr{E}_{\ca{O}_K}(-1))) = \left\{ \begin{array}{ll}
			\widehat{K_\infty} & \textrm{if $q=0$,}\\
			0 & \textrm{otherwise.}
		\end{array} \right.
	\end{align} 
	In particular, $(\hyodoring)^H=\widehat{K_\infty}$.
\end{mycor}
\begin{proof}
	 For any $n\in\bb{N}_{>0}$, the differential map \eqref{diam:app:lem:hyodo-ring-coh-0} defined in \ref{app:lem:hyodo-ring-coh} induces a natural sequence
	 \begin{align}
	 	0\longrightarrow \mrm{Sym}_K^n\widehat{\Omega}^1\stackrel{\phi^{n,0}}{\longrightarrow}\mrm{Sym}_K^{n-1}\widehat{\Omega}^1\otimes_K\widehat{\Omega}^1\stackrel{\phi^{n-1,1}}{\longrightarrow} \cdots\stackrel{\phi^{2,n-2}}{\longrightarrow}\widehat{\Omega}^1\otimes_K\widehat{\Omega}^{n-1}\stackrel{\phi^{1,n-1}}{\longrightarrow} \widehat{\Omega}^n\longrightarrow 0.
	 \end{align}
	 which is exact by \cite[Lemma 1.2]{hyodo1989variation} (cf. \cite[\Luoma{3}.5.1]{abbes2016p}). Therefore, by \ref{app:lem:hyodo-ring-coh}, the sequence $(E_1^{\bullet,j},\df_1^{\bullet,j})$ is exact for any $j\neq 0$, and nonzero only at $E_1^{0,0}=\widehat{K_\infty}$ if $j=0$. Then, we see that for any $n\in \bb{N}$, the nonzero terms ${}_nE_2^{i,j}$ of the second page of the spectral sequence \eqref{eq:app:cor:hyodo-ring-coh-ss} appear on the positions
	 \begin{align}
	 	(i,j)\in\{(0,0),(-n,n+m)\ |\ m\in\bb{N}\}.
	 \end{align}
 	 In particular, the spectral sequence degenerates at the second page, and we see that $H^0(H,\mrm{Sym}^n_{\widehat{\overline{K}}} (\scr{E}_{\ca{O}_K}(-1)))=\widehat{K_\infty}$ and that $H^q(H,\mrm{Sym}^n_{\widehat{\overline{K}}} (\scr{E}_{\ca{O}_K}(-1)))\to H^q(H,\mrm{Sym}^{n+1}_{\widehat{\overline{K}}} (\scr{E}_{\ca{O}_K}(-1)))$ is zero for any $q\neq 0$.
\end{proof}

\section{Appendix: Faltings Extensions as Graded Pieces of De Rham Period Rings}
The Faltings extension and Hyodo ring (the Hodge-Tate period ring) are also constructed as graded pieces of de Rham period ring in the literature in various $p$-adic geometric settings. See \cite[\textsection5]{brinon2008derham} for the good reduction case, \cite[\textsection2]{tsuji2011purity} for the semi-stable reduction case, \cite[\textsection15]{tsuji2018localsimpson} for the adequate case, and \cite[\textsection6]{scholze2013hodge}, \cite{scholze2016erratum} for smooth adic spaces. This appendix is devoted to a comparison between our construction of the Faltings extension with theirs. We fix a complete discrete valuation field $K$ of characteristic $0$ with perfect residue field of characteristic $p>0$, an algebraic closure $\overline{K}$ of $K$, and a compatible system of primitive $n$-th roots of unity $(\zeta_n)_{n\in\bb{N}}$ in $\overline{K}$.

\begin{mylem}\label{lem:G-fix-A1}
	Let $(A_{\triv},A,\overline{A})$ be a $(K,\ca{O}_K,\ca{O}_{\overline{K}})$-triple in the sense of {\rm\ref{defn:triple}}, $\ca{K}$ the fraction field of $A$, $G=\gal(\ca{K}_{\mrm{ur}}/\ca{K})$. Then, $\widehat{\overline{A}}[\frac{1}{p}](1)^G=0$, where $(1)$ denotes the first Tate twist.
\end{mylem}
\begin{proof}
	By \ref{lem:A-inj}, we reduce to the case where $A$ is a complete discrete valuation ring extension of $\ca{O}_K$, which is proved by Tate (\cite[Theorem 2]{tate1967p}, cf. \cite[Theorem 1]{hyodo1986hodge}).
\end{proof}

\begin{mylem}\label{lem:fal-ext-compare}
	Let $(A_{\triv},A,\overline{A})$ be a $(K,\ca{O}_K,\ca{O}_{\overline{K}})$-triple in the sense of {\rm\ref{defn:triple}}, $\ca{K}$ the fraction field of $A$, $G=\gal(\ca{K}_{\mrm{ur}}/\ca{K})$. Consider an isomorphism $f:\scr{E}\to \scr{E}'$ of extensions of an $\widehat{\overline{A}}[1/p]$-representation $W$ of $G$ by $\widehat{\overline{A}}[1/p](1)$ making the following diagram commutative.
	\begin{align}
		\xymatrix{
			0\ar[r]&\widehat{\overline{A}}[\frac{1}{p}](1)\ar[r]^-{\iota}\ar@{=}[d]&\scr{E}\ar[r]^-{\jmath}\ar[d]_-{f}^-{\wr}&W\ar[r]\ar[d]^-{\cdot(-1)}&0\\
			0\ar[r]&\widehat{\overline{A}}[\frac{1}{p}](1)\ar[r]^-{\iota'}&\scr{E}'\ar[r]^-{\jmath'}&W\ar[r]&0
		}
	\end{align}
	Let $x$ and $x'$ be elements of $\scr{E}$ and $\scr{E}'$ respectively. Assume that $\jmath(x)=-\jmath'(x')\in W^G$ and $g(x)-x=g(x')-x'\in \widehat{\overline{A}}[1/p](1)$ for any $g\in G$. Then, $f(x)=x'$.
\end{mylem}
\begin{proof}
	If we set $y=x'-f(x)\in\scr{E}'$, then $\jmath'(y)=0$ and $g(y)=y$ for any $g\in G$ by assumption, which implies that $y=0$ by \ref{lem:G-fix-A1}.
\end{proof}

\begin{mypara}\label{para:derham-period}
	For simplicity, we only consider Tsuji's construction of de Rham period ring \cite[\textsection15]{tsuji2018localsimpson}. The comparison with other constructions of the Faltings extension should be clear from our arguments in the following. We quickly review Tsuji's construction and state properties without proofs. Let $A$ be an adequate $\ca{O}_K$-algebra of relative dimension $d$ satisfying the following condition:
	\begin{enumerate}
		\renewcommand{\labelenumi}{{\rm(\theenumi)}}
		\item We set $X=(\spec(A),\alpha_{\spec(A_{\triv})\to \spec(A)})$ (cf. \ref{rem:quasi-adequate}). Then, the monoid $\Gamma(X,\ca{M}_X)/A^\times$ is finitely generated, and the identity of $\Gamma(X,\ca{M}_X)$ induces a chart of $X$ (cf. \cite[1.3.2]{tsuji1999p}).
	\end{enumerate}
	We remark that any adequate $\ca{O}_K$-algebra Zariski locally satisfies this condition (\cite[1.3.3]{tsuji1999p}). Let $\ca{K}$ be the fraction field of $A$, $G=\gal(\ca{K}_{\mrm{ur}}/\ca{K})$. Consider the tilt $\overline{A}^\flat=\plim_{x\mapsto x^p}\overline{A}/p\overline{A}$ of $\overline{A}$, and we put $W_{\ca{O}_K}(\overline{A}^\flat)=\ca{O}_K\otimes_{W(\kappa)}W(\overline{A}^\flat)$, where $\kappa$ is the residue field of $K$, and $W(-)$ is taking the ring of Witt vectors. Let $[-]:\overline{A}^\flat\to W(\overline{A}^\flat)$ denote the multiplicative lift, $\pi$ a uniformizer of $K$ with compatible system of $p^n$-th root $(\pi_{p^n})_{n\in\bb{N}}$ contained in $\ca{O}_{\overline{K}}$, $\xi_{\pi}=\pi\otimes 1-1\otimes [(\pi_{p^n})_{n\in\bb{N}}]\in W_{\ca{O}_K}(\overline{A}^\flat)$. There is a canonical exact sequence 
	\begin{align}
		\xymatrix{
			0\ar[r]& W_{\ca{O}_K}(\overline{A}^\flat)\ar[r]^-{\cdot \xi_\pi}&W_{\ca{O}_K}(\overline{A}^\flat)\ar[r]^-{\vartheta_{\ca{O}_K}}&\widehat{\overline{A}}\ar[r]& 0
		}
	\end{align}
	where $\vartheta_{\ca{O}_K}$ is the homomorphism of $\ca{O}_K$-algebras characterized by $\vartheta_{\ca{O}_K}(1\otimes[(a_n)_{n\in\bb{N}}])=\lim_{n\to \infty} \widetilde{a}_n^{p^n}$ (where $\widetilde{a}_n\in \overline{A}$ is a lift of $a_n\in \overline{A}/p\overline{A}$).
	
	Let $\overline{X}$ be the log scheme with underlying scheme $\spec(\widehat{\overline{A}})$ whose log structure is associated to $\Gamma(X,\ca{M}_X)\to \widehat{\overline{A}}$ (different to the notation in \ref{para:notation-Ybar}). Consider the fibred product of monoids 
	\begin{align}
		\xymatrix{
			\overline{Q}\ar[d]\ar[r]&\Gamma(X,\ca{M}_X)\ar[d]\\
			\plim_{x\mapsto x^p}\overline{A}\ar[r]&\overline{A}
		}
	\end{align}
	and let $\overline{D}$ be the log scheme with underlying scheme $\spec(W(\overline{A}^\flat))$ whose log structure is associated to the composition of $\overline{Q}\to \plim_{x\mapsto x^p}\overline{A}\to \overline{A}^\flat\stackrel{[-]}{\longrightarrow}W(\overline{A}^\flat)$. The condition above on the log structure of $X$ implies that $\overline{D}$ is an fs log scheme, and the natural maps $\vartheta:W(\overline{A}^\flat)\to \widehat{\overline{A}}$ and $\overline{Q}\to\Gamma(X,\ca{M}_X)$ induce an exact closed immersion $i_{\overline{D}}:\overline{X}\to \overline{D}$. We put $S=(\spec(\ca{O}_K),\alpha_{\spec(K)\to\spec(\ca{O}_K)})$ and $\overline{D}_S=S\times_{\spec(W(\kappa))}\overline{D}$, where $\spec(W(\kappa))$ is endowed with the trivial log structure. Consider the induced closed immersion of fs log schemes $i_{\overline{D}_S}:\overline{X}\to \overline{D}_S$ (not exact). For any $r,m\in\bb{N}$, let $\overline{D}_{S,m}^{(r)}$ be the $r$-th infinitesimal neighbourhood of the reduction mod $p^m$ of the closed immersion $i_{\overline{D}_S}$ in the category of fine log schemes in the sense of \cite[5.8]{kato1989log}. Then, the natural map
	\begin{align}
		(W_{\ca{O}_K}(\overline{A}^\flat)/\xi_\pi^{r+1}W_{\ca{O}_K}(\overline{A}^\flat))[\frac{1}{p}]\longrightarrow (\lim_{m\to\infty}\Gamma(\overline{D}_{S,m}^{(r)},\ca{O}_{\overline{D}_{S,m}^{(r)}}) )[\frac{1}{p}]
	\end{align}
	is an isomorphism (\cite[page 870, equation (33)]{tsuji2018localsimpson}, cf. \cite[2.5]{tsuji2011purity}). We define 
	\begin{align}
		B_{\mrm{dR}}^+(\overline{A})=\lim_{r\to\infty}\left((\lim_{m\to\infty}\Gamma(\overline{D}_{S,m}^{(r)},\ca{O}_{\overline{D}_{S,m}^{(r)}}) )[\frac{1}{p}]\right)
	\end{align}
	which is endowed with the natural action of $G$ and a canonical $G$-stable decreasing filtration by ideals
	\begin{align}
		\mrm{Fil}^rB_{\mrm{dR}}^+(\overline{A})=\ke\left(B_{\mrm{dR}}^+(\overline{A})\to (\lim_{m\to\infty}\Gamma(\overline{D}_{S,m}^{(r-1)},\ca{O}_{\overline{D}_{S,m}^{(r-1)}}) )[\frac{1}{p}]\right)=\xi_\pi^rB_{\mrm{dR}}^+(\overline{A}),
	\end{align}
	where $r\in\bb{N}_{>0}$ and we put $\mrm{Fil}^rB_{\mrm{dR}}^+(\overline{A})=B_{\mrm{dR}}^+(\overline{A})$ for $r\leq 0$.
	
	We put $\overline{D}_X=X\times_{\spec(W(\kappa))} \overline{D}$. Consider the induced closed immersion of fs log schemes $i_{\overline{D}_X}:\overline{X}\to\overline{D}_X$. For any $r,m\in\bb{N}$, let $\overline{D}_{X,m}^{(r)}$ be the $r$-th infinitesimal neighbourhood of the reduction mod $p^m$ of the closed immersion $ i_{\overline{D}_X}$ in the category of fine log schemes in the sense of \cite[5.8]{kato1989log}. We define
	\begin{align}
		\scr{B}_{\mrm{dR}}^+(\overline{A})=\lim_{r\to\infty}\left((\lim_{m\to\infty}\Gamma(\overline{D}_{X,m}^{(r)},\ca{O}_{\overline{D}_{X,m}^{(r)}}) )[\frac{1}{p}]\right)
	\end{align}
	which is endowed with the natural action of $G$ and a canonical $G$-stable decreasing filtration by ideals
	\begin{align}
		\mrm{Fil}^r\scr{B}_{\mrm{dR}}^+(\overline{A})=\ke\left(\scr{B}_{\mrm{dR}}^+(\overline{A})\to (\lim_{m\to\infty}\Gamma(\overline{D}_{X,m}^{(r-1)},\ca{O}_{\overline{D}_{X,m}^{(r-1)}}) )[\frac{1}{p}]\right),
	\end{align}
	where $r\in\bb{N}_{>0}$ and we put $\mrm{Fil}^r\scr{B}_{\mrm{dR}}^+(\overline{A})=\scr{B}_{\mrm{dR}}^+(\overline{A})$ for $r\leq 0$. Moreover, Tsuji \cite[page 871]{tsuji2018localsimpson} defines a $G$-equivariant $B_{\mrm{dR}}^+(\overline{A})$-linear integrable connection $\nabla:\scr{B}_{\mrm{dR}}^+(\overline{A})\to \scr{B}_{\mrm{dR}}^+(\overline{A})\otimes_A\Omega^1_{X/S}$ compatible with $\df:A\to \Omega^1_{X/S}$ and $\df\log:\Gamma(X,\ca{M}_X)\to \Omega^1_{X/S}$ satisfying that $\nabla(\mrm{Fil}^r\scr{B}_{\mrm{dR}}^+(\overline{A}))\subseteq \mrm{Fil}^{r-1}\scr{B}_{\mrm{dR}}^+(\overline{A})\otimes_A\Omega^1_{X/S}$. We summarize some properties of $\scr{B}_{\mrm{dR}}^+(\overline{A})$ proved by Tsuji in the following.
\end{mypara}

\begin{myprop}[{\cite[\textsection15]{tsuji2018localsimpson}}]\label{prop:derham-period}
	We keep the notation in {\rm\ref{para:derham-period}}.
	\begin{enumerate}
		\renewcommand{\labelenumi}{{\rm(\theenumi)}}
		\item The element $t=\log([(\zeta_{p^n})_{n\in\bb{N}}])\in B_{\mrm{dR}}^+(\overline{A})$ is a non-zero divisor such that $t^rB_{\mrm{dR}}^+(\overline{A})=\mrm{Fil}^rB_{\mrm{dR}}^+(\overline{A})$ for any $r\in\bb{N}$. Moreover, the map $\vartheta_{\ca{O}_K}$ induces a $G$-equivariant isomorphism
		\begin{align}
			t^rB_{\mrm{dR}}^+(\overline{A})/t^{r+1}B_{\mrm{dR}}^+(\overline{A})\iso \widehat{\overline{A}}[1/p](r).
		\end{align}\label{item:prop:derham-period-1}
		\item Let $(s_{p^n})_{n\in\bb{N}}\in \overline{Q}$ with image $s=s_1\in \Gamma(X,\ca{M}_X)$. Then, for any $r,m\in\bb{N}$, there is a unique element $u^{(r)}_m\in 1+\ke\left(\Gamma(\overline{D}_{X,m}^{(r)},\ca{O}_{\overline{D}_{X,m}^{(r)}})\to\Gamma(\overline{D}_{X,m}^{(r-1)},\ca{O}_{\overline{D}_{X,m}^{(r-1)}})\right)\subseteq \Gamma(\overline{D}_{X,m}^{(r)},\ca{O}_{\overline{D}_{X,m}^{(r)}}^\times)$ such that $s=[(s_{p^n})_{n\in\bb{N}}]\cdot u^{(r)}_m$ in $\Gamma(\overline{D}_{X,m}^{(r)},\ca{M}_{\overline{D}_{X,m}^{(r)}})$. It induces a canonical homomorphism of monoids
		\begin{align}
			q:\overline{Q}\longrightarrow 1+\mrm{Fil}^1\scr{B}_{\mrm{dR}}^+(\overline{A}),\ (s_{p^n})_{n\in\bb{N}}\mapsto u=(u^{(r)}_m)_{r,m\in\bb{N}}.
		\end{align}
		Moreover, $\nabla(u)=u\otimes \df\log(s)$.
		\label{item:prop:derham-period-2}
	\item Let $\{t_i\}_{1\leq i\leq d}$ be a system of coordinates of the adequate $\ca{O}_K$-algebra $A$ with compatible systems of $p$-power roots $(t_{i,p^n})_{n\in\bb{N}}$ contained in $\overline{A}[1/p]$. For sufficiently large $k\in \bb{N}$ such that $\pi^kt_i\in A$, the element $u_i=q((\pi_{p^n})_{n\in\bb{N}})^{-k}q((\pi_{p^n}^kt_{i,p^n})_{n\in\bb{N}})\in 1+\mrm{Fil}^1\scr{B}_{\mrm{dR}}^+(\overline{A})$ does not depend on the choice of $k$. Moreover, $\scr{B}_{\mrm{dR}}^+(\overline{A})$ is the $B_{\mrm{dR}}^+(\overline{A})$-algebra of formal power series with $d$ variables $u_1-1,\dots, u_d-1$, and for any $r\in \bb{Z}$ we have
	\begin{align}
		\mrm{Fil}^r\scr{B}_{\mrm{dR}}^+(\overline{A})=\prod_{\underline{k}\in\bb{N}^d}\mrm{Fil}^{r-|\underline{k}|}B_{\mrm{dR}}^+(\overline{A})\cdot(u_1-1)^{k_1}\cdots(u_d-1)^{k_d}.
	\end{align}\label{item:prop:derham-period-3}
	\end{enumerate}
\end{myprop}

\begin{mycor}\label{cor:fal-ext-derham}
	With the notation in {\rm\ref{para:derham-period}}, there is a canonical $G$-equivariant exact sequence of $\widehat{\overline{A}}[1/p]$-modules,
	\begin{align}
		0\longrightarrow \widehat{\overline{A}}[\frac{1}{p}](1)\stackrel{\iota}{\longrightarrow}\mrm{Gr}^1\scr{B}_{\mrm{dR}}^+(\overline{A})\stackrel{\jmath}{\longrightarrow}\widehat{\overline{A}}[\frac{1}{p}]\otimes_{A}\Omega^1_{X/S}\longrightarrow 0,
	\end{align}
	where $\mrm{Gr}^1\scr{B}_{\mrm{dR}}^+(\overline{A})=\mrm{Fil}^1\scr{B}_{\mrm{dR}}^+(\overline{A})/\mrm{Fil}^2\scr{B}_{\mrm{dR}}^+(\overline{A})$, satisfying the following properties:
	\begin{enumerate}
		\renewcommand{\labelenumi}{{\rm(\theenumi)}}
		\item The map $\iota$ is induced by taking $\mrm{Gr}^1$ on the map $B_{\mrm{dR}}^+(\overline{A})\to\scr{B}_{\mrm{dR}}^+(\overline{A})$. In particular, we have $\iota(1\otimes(\zeta_{p^n})_{n\in\bb{N}})=\log([(\zeta_{p^n})_{n\in\bb{N}}])=1-u_0$, where $u_0\in 1+\mrm{Fil}^1\scr{B}_{\mrm{dR}}^+(\overline{A})$ is the element associated to $(\zeta_{p^n})_{n\in\bb{N}}$ defined in {\rm\ref{prop:derham-period}.(\ref{item:prop:derham-period-2})}.\label{item:cor:fal-ext-derham-1}
		\item The map $\jmath$ is induced by taking $\mrm{Gr}^1$ on the connection $\nabla:\scr{B}_{\mrm{dR}}^+(\overline{A})\to \scr{B}_{\mrm{dR}}^+(\overline{A})\otimes_A\Omega^1_{X/S}$. In particular, for any element $s\in A[1/p]\cap A_{\triv}^\times$ and any compatible system of $p$-power roots $(s_{p^n})_{n\in\bb{N}}$ of $s$ contained in $\overline{A}[1/p]$, if we denote by $u\in 1+\mrm{Fil}^1\scr{B}_{\mrm{dR}}^+(\overline{A})$ the element associated to $(s_{p^n})_{n\in\bb{N}}$ defined as in {\rm\ref{prop:derham-period}.(\ref{item:prop:derham-period-2}, \ref{item:prop:derham-period-3})}, then we have $\jmath(u-1)=1\otimes\df\log(s)$.\label{item:cor:fal-ext-derham-2}
		\item With the notation in {\rm\ref{prop:derham-period}.(\ref{item:prop:derham-period-3})}, the $\widehat{\overline{A}}[1/p]$-linear surjection $\jmath$ admits a section sending $1\otimes\df\log(t_i)$ to $u_i-1$ for any $1\leq i\leq d$.\label{item:cor:fal-ext-derham-3}
	\end{enumerate}
	In particular, $\mrm{Gr}^1\scr{B}_{\mrm{dR}}^+(\overline{A})$ is a finite free $\widehat{\overline{A}}[1/p]$-module with basis $\{1-u_i\}_{0\leq i\leq d}$.
\end{mycor}

\begin{myprop}\label{prop:fal-ext-comparison}
	With the notation in {\rm\ref{para:derham-period}}, there is a unique $G$-equivariant $\widehat{\overline{A}}[1/p]$-linear isomorphism
	\begin{align}
		f:\scr{E}_A\iso \mrm{Gr}^1\scr{B}_{\mrm{dR}}^+(\overline{A}),
	\end{align}
	where $\scr{E}_A$ is defined in {\rm\ref{thm:B-fal-ext}}, such that for any element $s\in A[1/p]\cap A_{\triv}^\times$ and any compatible system of $p$-power roots $(s_{p^n})_{n\in\bb{N}}$ of $s$ contained in $\overline{A}[1/p]$, we have
	\begin{align}
		f((\df\log(s_{p^n}))_{n\in\bb{N}})=1-u,
	\end{align}
	where $u\in 1+\mrm{Fil}^1\scr{B}_{\mrm{dR}}^+(\overline{A})$ is the element associated to $(s_{p^n})_{n\in\bb{N}}$ defined as in {\rm\ref{prop:derham-period}.(\ref{item:prop:derham-period-2}, \ref{item:prop:derham-period-3})}. In particular, it induces a canonical isomorphism of Faltings extensions
	\begin{align}\label{diam:prop:fal-ext-comparison}
		\xymatrix{
			0\ar[r]&\widehat{\overline{A}}[\frac{1}{p}](1)\ar[r]^-{\iota}\ar@{=}[d]&\scr{E}_A\ar[r]^-{\jmath}\ar[d]_-{f}^-{\wr}&\widehat{\overline{A}}[\frac{1}{p}]\otimes_{A}\Omega^1_{X/S}\ar[r]\ar[d]^-{\cdot(-1)}&0\\
			0\ar[r]&\widehat{\overline{A}}[\frac{1}{p}](1)\ar[r]^-{\iota}&\mrm{Gr}^1\scr{B}_{\mrm{dR}}^+(\overline{A})\ar[r]^-{\jmath}&\widehat{\overline{A}}[\frac{1}{p}]\otimes_{A}\Omega^1_{X/S}\ar[r]&0
		}
	\end{align}
\end{myprop}
\begin{proof}
	We take the notation in {\rm\ref{prop:derham-period}.(\ref{item:prop:derham-period-3})}. Recall that $\scr{E}_A$ is a finite free $\widehat{\overline{A}}[1/p]$-module with basis $\{\alpha_i=\df\log(t_{i,p^n})_{n\in\bb{N}}\}_{0\leq i\leq d}$, where $t_{0,p^n}=\zeta_{p^n}$. Thus, the uniqueness of $f$ is clear. For its existence, we define $f$ to be the $\widehat{\overline{A}}[1/p]$-linear isomorphism sending $\alpha_i\in \scr{E}_A$ to $1-u_i\in\mrm{Gr}^1\scr{B}_{\mrm{dR}}^+(\overline{A})$ for any $0\leq i\leq d$. 
	
	We claim that $f$ is $G$-equivariant. Indeed, let $u\in 1+\mrm{Fil}^1\scr{B}_{\mrm{dR}}^+(\overline{A})$ be the element associated to $(s_{p^n})_{n\in\bb{N}}$ defined as in {\rm\ref{prop:derham-period}.(\ref{item:prop:derham-period-2}, \ref{item:prop:derham-period-3})}. For any $g\in G$, there is a unique element $\xi(g)\in\bb{Z}_p$ such that $g(s_{p^n})_{n\in\bb{N}}=(\zeta_{p^n}^{\xi(g)}s_{p^n})_{n\in\bb{N}}$. As $s=[(s_{p^n})_{n\in\bb{N}}]\cdot u^{(r)}_m$, we have $u=[(\zeta_{p^n})_{n\in\bb{N}}]^{\xi(g)}g(u)$. Taking logarithm and using the identity $\log([(\zeta_{p^n})_{n\in\bb{N}}])=1-u_0$ in $\mrm{Gr}^1\scr{B}_{\mrm{dR}}^+(\overline{A})$ (see \ref{cor:fal-ext-derham}.(\ref{item:cor:fal-ext-derham-1})), we obtain that 
	\begin{align}\label{eq:prop:fal-ext-comparison}
		g(1-u)=\xi(g)(1-u_0)+1-u\in \mrm{Gr}^1\scr{B}_{\mrm{dR}}^+(\overline{A}).
	\end{align}
	As $f$ is $\widehat{\overline{A}}[1/p]$-linear and sends the basis $\alpha_0,\dots,\alpha_d$ to $1-u_0,\dots,1-u_d$, we see that $f$ is $G$-equivariant by \eqref{eq:prop:fal-ext-comparison} and \eqref{eq:fal-ext-action}.
	
	We claim that $f(\alpha)=1-u$, where $\alpha=(\df\log(s_{p^n})_{n\in\bb{N}})$. Indeed, our definition of $f$ makes the diagram \eqref{diam:prop:fal-ext-comparison} commute by \ref{thm:B-fal-ext} and \ref{cor:fal-ext-derham}. Notice that $\jmath(\alpha)=\df\log(s)=-\jmath(1-u)\in \Omega^1_{X/S}$ and that $g(\alpha)-\alpha=\xi(g)\alpha_0=\xi(g)(1-u_0)=g(1-u)-(1-u)\in \widehat{\overline{A}}[1/p](1)$. Thus, we can apply \ref{lem:fal-ext-compare} to verify this claim.
	
	Therefore, $f$ satisfies all the requirements, which completes the proof.
\end{proof}

\bibliographystyle{myalpha}
\bibliography{bibli}


\end{document}